\documentclass[preprint,3p]{elsarticle}

\makeatletter
\def\ps@pprintTitle{%
\let\@oddhead\@empty
\let\@evenhead\@empty
\def\@oddfoot{\reset@font\hfil\thepage\hfil}
\let\@evenfoot\@oddfoot
}
\makeatother

\usepackage{amssymb}
\usepackage{latexsym}

\usepackage{url}
\usepackage{xcolor}
\definecolor{newcolor}{rgb}{.8,.349,.1}

\usepackage{amsmath}
\usepackage{ulem}
\usepackage{booktabs}
\usepackage{multirow}
\usepackage{epstopdf}

\usepackage{caption}
\usepackage{subcaption}
\usepackage{amsthm}
\theoremstyle{remark}
\newtheorem*{remark}{Remarks}
\theoremstyle{plain}
\newtheorem{proposition}{Proposition}
\newtheorem{corollary}{Corollary}

\DeclareMathOperator*{\argmin}{arg\,min\,}

\let\min\relax

\DeclareMathOperator*{\min}{min\,}

\begin{document}


\begin{frontmatter}

\title{On the stability of projection-based model order reduction for convection-dominated laminar and turbulent flows}

\author[1]{Sebastian Grimberg\corref{cor1}}
\cortext[cor1]{Corresponding author.}
\ead{sjg@stanford.edu}
\author[1,2,3]{Charbel Farhat}
\ead{cfarhat@stanford.edu}
\author[1]{Noah Youkilis}
\ead{nbyouk@stanford.edu}

\address[1]{Department of Aeronautics and Astronautics, Stanford University, Stanford, CA 94305, U.S.A.}
\address[2]{Department of Mechanical Engineering, Stanford University, Stanford, CA 94305, U.S.A.}
\address[3]{Institute for Computational and Mathematical Engineering, Stanford University, Stanford, CA 94305, U.S.A.}



\begin{abstract}
\small
In the literature on projection-based nonlinear model order reduction for fluid dynamics problems, it is often claimed that due to modal truncation, a projection-based reduced-order model (PROM) does not resolve 
the dissipative regime of the turbulent energy cascade and therefore is numerically unstable. Efforts at addressing this claim have ranged from attempting to model the effects of the truncated modes to 
enriching the classical subspace of approximation in order to account for the truncated phenomena. This paper challenges this claim. Exploring the relationship between projection-based model order 
reduction and semi-discretization and using numerical evidence from three relevant flow problems, it argues in an orderly manner that the real culprit behind most if not all reported numerical 
instabilities of PROMs for turbulence and convection-dominated turbulent flow problems is the Galerkin framework that has been used for constructing the PROMs. The paper also shows that alternatively, 
a Petrov-Galerkin framework can be used to construct numerically stable PROMs for convection-dominated laminar as well as turbulent flow problems that are numerically stable and accurate, without 
resorting to additional closure models or tailoring of the subspace of approximation. It also shows that such alternative PROMs deliver significant speedup factors.
\end{abstract}



\end{frontmatter}


\section{Introduction}
\label{sec:intro}

The construction of stable, accurate, and computationally efficient nonlinear projection-based reduced-order models (PROMs) for the numerical simulation of turbulence and turbulent flows has been a topic of recent interest in the model order reduction community.
The convection-dominated nature of such problems and the presence of multiscale phenomena over large spatial and temporal ranges pose issues for traditional projection-based model order reduction (PMOR) methods, due to the large Kolmogorov $n$-width \cite{pinkus1985} of the solution manifold associated with the high-dimensional model (HDM).
The consequence is that a large number of modes may be required in order for the PROM to maintain solution accuracy, which in turn limits computational efficiency, or the PROM may suffer a severe loss of information due to modal truncation.

It is often suggested that the truncation of the higher-order modes in the construction of a PROM for turbulent flow problems leads to its numerical instability.
The argument provided for this instability is as follows: Following the Kolmogorov hypotheses \cite{kolmogorov1941}, a high-Reynolds number turbulent flow is characterized by an energy cascade 
in which the large, energy-containing flow scales break down and transfer energy to progressively smaller scales until the length scales are small enough, the viscous forces are dominant, and turbulence
kinetic energy is dissipated \cite{pope2000}.
Typically, it is first argued that the modes used to construct the PROM are analogous to their Fourier space counterparts, which has been verified using numerical evidence for certain problems \cite{couplet2003}.
Then, it is concluded that by truncating the higher-order, low-energy modes in the subspace approximation, the PROM loses the ability to resolve the dissipation range of the 
turbulent energy spectrum. This implies the inability of the PROM to sufficiently dissipate energy, leads to excessive energy levels, and eventually to instability.

However, much of the numerical evidence provided to support the aforementioned {\it physical} argument does not consider {\it numerical} factors involved in construction of the PROM that contribute to 
instability, even for nonturbulent flow problems. This work seeks to address this issue in light of the aforementioned physical argument for PROM instability.

\subsection{Problem of interest and literature review}

Let $\mathbf{u}(t; \boldsymbol{\mu})$ denote the solution of a parametric, time-dependent, partial-differential equation (PDE) semi-discretized by an HDM of dimension $N$, where 
$\boldsymbol{\mu} \in \mathcal{P} \subset \mathbb{R}^p$ is a parameter vector of dimension $p$ and $\mathcal{P}$ denotes the bounded parameter space of interest. In the traditional framework for PMOR,
$\mathbf{u}(t; \boldsymbol{\mu})$ is approximated in a low-dimensional affine subspace of dimension $n$, with $n \ll N$. That is,
\begin{equation}
\label{eqn:subspace}
\mathbf{u}(t; \boldsymbol{\mu}) \approx \mathbf{u}_0 + \mathbf{V} \mathbf{y}(t; \boldsymbol{\mu})
\end{equation}
where the affine offset $\mathbf{u}_0 \in \mathbb{R}^N$ and the {\it right} reduced-order basis (ROB) $\mathbf{V} \in \mathbb{R}^{N \times n}$ define the subspace approximation and are for now assumed to be time- and parameter-independent, and $\mathbf{y}(t; \boldsymbol{\mu}) \in \mathbb{R}^n$ is the vector of reduced (or generalized) coordinates.
Note that the preceding expressions and remainder of the discussion in this paper are easily extended to the case where the HDM is formulated in the complex plane $\mathbb{C}$, by replacing each instance of $\mathbb{R}$ by $\mathbb{C}$.

For problems characterized by a large Kolmogorov $n$-width of the HDM-based solution manifold, popular methods for constructing the right ROB $\mathbf{V}$, such as the proper orthogonal decomposition (POD) 
method of snapshots \cite{sirovich1987} and the reduced-basis method \cite{prudhomme2002}, will exhibit slow convergence. Hence, such methods will result in a dimension $n$ of $\mathbf{V}$ that quickly 
approaches $O(N)$ in order to attain a high degree of accuracy in the subspace approximation (\ref{eqn:subspace}). It follows that PMOR methods of this type must anticipate that, for the sake of 
computational efficiency and therefore maintaining $n \ll N$, the PROM will always incur a degree of underresolution by sacrificing some degree of accuracy in the interest of achieving real-time 
(or near real-time) performance.

Alternatively, several PMOR methods that modify the traditional subspace approximation (\ref{eqn:subspace}) have been proposed to break Kolmogorov $n$-width barriers in order to increase the accuracy of 
the PROM for a given dimensionality. These include: partitioning the state- or parameter-space using multiple local, piecewise affine subspace approximations instead of a single global approximation 
\cite{dihlmann2011, amsallem2012, peherstorfer2014}; transforming the PROM solution basis to improve the affine approximation for convection-dominated flow problems 
\cite{ohlberger2013, gerbeau2014, cagniart2018, nair2019}; and, more recently, adopting convolutional autoencoders to determine low-dimensional nonlinear manifolds that replace the traditional affine 
subspace approximation of the HDM solution space \cite{lee2018}.
Even with such improved frameworks for obtaining low-dimensional approximations of the HDM-based solution manifold however, some amount of truncation of the physical scales captured by the HDM 
is bound to happen for complex problems of industrial or scientific interest.
Thus, the focus of this work is on the implications of the truncation that results from approximating the HDM solution using a low-dimensional subspace, regardless of the method selected for 
constructing the approximation.

Given the aforementioned argument for PROM instability and the above observation regarding the necessity of physical scale truncation in constructing the PROM subspace approximation, many stabilization 
methods have been proposed to remedy observed instability when constructing PROMs for turbulent flow problems. Such methods include attempts to model the effects of truncated modes at the PROM level via 
closure models \cite{aubry1988, wang2012, iliescu2014, osth2014,  benosman2017, rebollo2017, stabile2019}, and enriching the modes of the subspace approximation in order to better account for 
truncated phenomena \cite{bergmann2009, balajewicz2012, balajewicz2013, balajewicz2016, akkari2019}. The primary downsides of these approaches is that they can negatively impact PROM accuracy, for example, 
by destroying the PROM {\it consistency} property via the introduction of modeled closure terms, or resorting to modifications of the PROM approximation subspace which are often chosen to minimize 
the error associated with approximating the HDM-based solution manifold. The notion of consistency of a PROM approximation introduced in \cite{carlberg2011} states that a PROM is consistent if, 
when constructed using a right ROB $\mathbf{V}$ that perfectly captures the HDM solution manifold over the desired parameter space, for example by collecting snapshots at every time-step of the simulation performed using the HDM and 
constructing $\mathbf{V}$ without data compression, the PROM introduces no additional error in the solution of the problem over the same parameter space.

Perhaps a more fundamental issue, however, is that the large body of literature that demonstrates numerically the need for the aforementioned PROM stabilization methods mainly considers as a starting 
point the ``standard'' PROM, which is based on a Galerkin projection (and the $\ell^2$ inner product) of the HDM in order to construct the PROM. That is, the {\it left} ROB, 
$\mathbf{W} \in \mathbb{R}^{N \times n}$, employed in the construction of reduced vectors of the form $\mathbf{W}^T \mathbf{b}\in \mathbf{R}^{n}$ and reduced tangent matrices of the form 
$\mathbf{W}^T \mathbf{A} \mathbf{V} \in \mathbf{R}^{n \times n}$, where $\mathbf{b} \in \mathbb{R}^N$, $\mathbf{A} \in \mathbb{R}^{N \times N}$ and the superscript $T$ designates the transpose
are given by the HDM, is chosen as $\mathbf{W} = \mathbf{V}$.
This is unlike in the more general Petrov-Galerkin framework where $\mathbf{W} \neq \mathbf{V}$. In addition to enforcing the uniqueness of the approximate solution $\mathbf{y}(t; \boldsymbol{\mu})$,
the left ROB $\mathbf{W}$, which is also a reduced-order {\it test} basis corresponding to the reduced-order {\it trial} basis $\mathbf{V}$, enforces more general constraints for residual orthogonality 
that can be used to maintain some desirable numerical qualities in the PROM. The residual orthogonality constraints enforced by the Galerkin projection are appropriate for a large range of problems, for 
example, when the HDM is characterized by symmetric positive definite (SPD) tangent operators as in many structural dynamics and solid mechanics problems. However, the same cannot be said for 
convection-dominated flow problems involving nonsymmetric operators, in which case the flexibility of a different left subspace provided by the Petrov-Galerkin framework is well-known to be useful for
constructing optimal projectors for PMOR \cite{carlberg2011, buithanh2008}. For example for first-order linear time-invariant (LTI) systems, a method was presented in \cite{amsallem2012_2} to construct 
the left ROB $\mathbf{W}$ of the Petrov-Galerkin projection such that the corresponding linear PROM is guaranteed to satisfy the Lyapunov stability criterion for LTI systems and thus is asymptotically 
stable. Alternatively, replacing the standard $\ell^2$ inner product associated with the Galerkin projection with appropriate stabilizing alternatives has been proposed for the linearized Euler 
equations \cite{barone2009} as well as the compressible Navier-Stokes equations \cite{rowley2004}. Notably, by recognizing that the PROM is defined by the {\it pair} $(\mathbf{W}, \mathbf{V})$ as 
opposed to by $\mathbf{V}$ alone, it has been demonstrated that PROMs can be constructed in a manner appropriate for the problem at hand, without destroying their consistency property or modifying 
their approximation subspace. However, while the suitability of a Petrov-Galerkin approach for addressing stability issues of PMOR for convection-dominated viscous flow problems has been highlighted in 
\cite{carlberg2013}, it has not yet been demonstrated in the literature for truly turbulent flow problems.

Finally, it is important to note that the instability of Galerkin PROMs has been extensively documented for flow problems which do not involve turbulence, such as inviscid flows modeled by the
linearized Euler equations \cite{amsallem2012_2, barone2009} and laminar flows modeled by the incompressible as well as compressible Navier-Stokes equations \cite{rowley2004, rempfer2000}.
It has also been highlighted for Reynolds-averaged Navier-Stokes (RANS) models where no part of the turbulent energy spectrum is resolved by the spatial discretization 
\cite{carlberg2011, carlberg2017}. Collectively, this body of evidence suggests that the physical argument previously described to explain the instability of PROMs of turbulent flows, which has been 
justified to date using numerical examples involving Galerkin PROMs only, is strictly incorrect. The observed instabilities can be attributed instead to the reliance on a Galerkin projection in the 
presence of convection-dominated phenomena. Indeed, it is the numerical properties of the PROM that govern its behavior because the PROM, as well as the underlying HDM, only seek to {\it approximate} the 
physics of the strong-form of the equations. Thus, explaining a numerical behavior using physics-based arguments only is not necessarily justifiable.

\subsection{Contributions of this paper}

As alluded to in the previous section, the main objective of this paper is to demonstrate that the stability issues reported by many authors for PROMs of turbulent flows are likely not due to a loss 
of resolution in the approximation subspace, but instead to the reliance on the Galerkin framework for performing PMOR for convection-dominated flow problems. To this end, the remainder of this paper is 
organized as follows.

Section \ref{sec:semid} seeks to establish a connection between PMOR and other spatial discretization methods, in order to motivate Petrov-Galerkin approaches for the reduction of 
convection-dominated high-dimensional flow models. Section \ref{sec:pg} reviews a specific Petrov-Galerkin framework for nonlinear PMOR and advocates it as a better alternative to the standard
Galerkin framework for constructing PROMs for convection-dominated turbulent flow problems. In particular, this section emphasizes the construction of an appropriate left ROB $\mathbf{W}$. 
Section \ref{sec:preandhyper} addresses the issue of achieving computational efficiency for nonlinear Petrov-Galerkin PROMs, in order to introduce hyperreduction and demonstrate next that a 
hyperreduced Petrov-Galerkin PROM (HPROM) for turbulent flows can maintain the numerical stability achieved by its underlying PROM. Finally, Section \ref{sec:examples} reports on several
applications that: flaunt the numerical stability of Petrov-Galerkin HPROMs for convection-dominated laminar as well as turbulent flows, including in the presence of severe modal truncation;
demonstrate the instability of Galerkin PROMs and HPROMs, even for laminar but otherwise convection-dominated flows; and highlight on the other hand the stability of Galerkin PROMs and HPROMs
for diffusion-dominated turbulent flows. Hence, these examples support the conclusions formulated in Section \ref{sec:conclude}, namely, that the stability issues reported in the literature for the
PMOR of convection-dominated turbulent flow problems are not due to a loss of resolution induced by modal truncation, but to the Galerkin approach chosen for constructing in all reported cases
the PROMs.

\section{Spatial discretization of convection-dominated flow problems}
\label{sec:semid}

\subsection{Model order reduction as a semi-discretization method}

In the context of obtaining numerical solutions for PDEs, traditional spatial discretization methods reduce the infinite-dimensional, continuous solution space for the PDE of interest to a 
finite-dimensional subspace which, in the context of PMOR, represents the HDM-based solution manifold. Consider the classical Ritz method \cite{rayleigh1877, ritz1909}, which seeks to approximate the 
infinite-dimensional, exact solution of the continuous PDE by expanding an approximate solution using a set of $N$ {\it global} basis functions specified {\it a priori} -- that is, before any 
significant knowledge of the exact solution is known. This global, $N$-dimensional test basis and the corresponding global trial basis are then substituted as finite-dimensional approximations to 
the function spaces in the variational principle governing the problem of interest, which results in an HDM of dimension $N$.

In PMOR, the $N$-dimensional HDM solution manifold, referenced by the affine offset $\mathbf{u}_0$, is then further approximated using a basis of dimension $n$ spanned by the columns of $\mathbf{V}$.
The key difference with the Ritz procedure is that in this case, the basis functions spanning the approximation subspace, while still global, are constructed {\it a posteriori} -- that is, after some 
knowledge about the HDM-based solution manifold is acquired. For this reason, PMOR methods can be said to contain elements of machine learning, and more specifically, physics-based machine learning.
The low-dimensional test basis and the corresponding reduced-order trial basis spanned by the columns of the left ROB $\mathbf{W}$ are applied in typical Ritz fashion to the HDM to construct a 
correspondnig PROM of dimension $n$: first, by substituting the low-dimensional solution approximation; and then, by enforcing the orthogonality of the resulting residual to each element of the trial 
basis. This view of PMOR as another layer of semi-discretization applied to the HDM is important: it suggests that experience in developing stable and accurate spatial discretization methods for 
convection-dominated flow problems should be leveraged when constructing PROMs for such problems.

\subsection{Semi-discretization methods for convection-dominated flow problems}

In the context of the finite element (FE) method, it is well-understood that for convection-dominated flow problems governed by the advection-diffusion equation, the Euler equations, or 
the Navier-Stokes equations, the application of a standard continuous polynomial Galerkin FE formulation often produces solutions that are polluted by spurious oscillations. Such a pollution is 
recognized as a symptom of numerical instability and can, especially in the presence of nonlinearities, lead to the divergence of the solution process. In such cases, convection-dominated refers to large 
P\'eclet or Reynolds numbers and pollution arises when the semi-discretization is characterized by {\it element} P\'eclet or Reynolds numbers that exceed 1. For example, the error analysis of the 
Galerkin approximation based on local polynomial shape functions of order $k$ applied to the solution of the advection-diffusion equation yields
\begin{equation*}
\lVert \nabla e \rVert_{L^2(\Omega)} = O~\left( (1+Pe_h)\, h^k\right)
\end{equation*}
where $e$ is the approximation error measured with respect to the exact solution, and $h$ denotes the mesh element size \cite{hughes1987}. The appearance of $Pe_h$, the element P\'eclet number based on 
the reference length $h$, in the error constant reveals that convergence is inhibited for convection-dominated problems where $Pe_h$ is large. This behavior largely corresponds to the loss of the ``best 
approximation'' property of Galerkin FE methods, due to the nonsymmetry of the matrices associated with the semi-discretized convective terms.

Various stabilized FE methods have been developed in order to avoid excessive mesh refinement during the solution of convection-dominated flow problems, and therefore obtain numerically stable, 
oscillation-free solutions on coarse meshes. Among these methods, the most popular is the streamline-upwind/Petrov-Galerkin (SUPG) method \cite{brooks1982}. Its development has contributed to the 
understanding of the deficiencies of the standard Galerkin framework for convection-dominated flow problems, as well as the significance of the alternative Petrov-Galerkin framework for developing 
oscillation-free semi-discretization methods for these problems.

Interestingly, the earlier development of stabilized finite difference and finite volume semi-discretizations methods had followed a similar path. The idea of adding numerical dissipation in order to 
suppress spurious oscillations that arise when spatial derivatives in problems containing strong gradients and shocks are approximated by central differences dates back to 1950 \cite{vonneumann1950}.
(The reader is reminded that linear Galerkin FE semi-discretizations lead to second-order central-difference approximations of derivative operators, and thus their behavior is analogous to that 
discussed above in the context of FE semi-discretization methods). Schemes for computational fluid dynamics (CFD) have regularly relied on suitable numerical dissipation to eliminate spurious oscillations in the presence of convective 
phenomena without undermining solution accuracy \cite{jameson1981, boris1973, vanleer1979, shu1988}. Given that it is well-understood today that even in the absence of attempts at resolving turbulent 
scales, semi-discretizing convection-dominated flow problems using a Galerkin-type framework leads to solution instability, using a Galerkin approach for performing PMOR for convection-dominated HDMs
can be expected to lead to numerical instability.

\section{Petrov-Galerkin model order reduction}
\label{sec:pg}

In this section, Petrov-Galerkin frameworks for nonlinear PMOR are reviewed. In particular, the ability to utilize the left ROB $\mathbf W$ in order to provide optimal approximations of the HDM is 
presented, and the relationship between Petrov-Galerkin and Galerkin PROMs as well as their optimality are discussed.

Consider the $\boldsymbol{\mu}$-parametric, first-order, $N$-dimensional, semi-discrete HDM
\begin{equation}
\label{eqn:hdm}
\begin{split}
\mathbf{M}(\boldsymbol{\mu}) \, \dot{\mathbf{u}}(t; \boldsymbol{\mu}) + \mathbf{f}(\mathbf{u}(t; \boldsymbol{\mu}); \boldsymbol{\mu}) &= 0 \\
\mathbf{u}(0; \boldsymbol{\mu}) &= \mathbf{u}^0(\boldsymbol{\mu})
\end{split}
\end{equation}
where $t$ denotes time and the dot its derivative, $\mathbf{M}(\boldsymbol{\mu}) \in \mathbb{R}^{N \times N}$ is a parametric SPD mass matrix, and 
$\mathbf{f}(\mathbf{u}(t; \boldsymbol{\mu}); \boldsymbol{\mu}) \in \mathbb{R}^{N}$ is the nonlinear function resulting from the semi-discretization of the convective, diffusive, and source terms of a
convection-diffusion PDE. Assume that problem \eqref{eqn:hdm} is discretized by an implicit scheme. In this case, this nonlinear problem can be transformed at each $k$-th computational time-step into 
the residual form
\begin{equation}
\label{eqn:hdmres}
\mathbf{r}\left(\mathbf{u}(t^{\, k+1}; \boldsymbol{\mu}); \boldsymbol{\mu}\right) = \mathbf{M}(\boldsymbol{\mu}) \, \dot{\mathbf{u}}(t^{\, k+1}; \boldsymbol{\mu}) + 
	\mathbf{f}\left(\mathbf{u}(t^{\,k+1}; \boldsymbol{\mu}); \boldsymbol{\mu}\right) = 0
\end{equation}
where $t^{\,k+1} = t^{\,k} + \Delta t$, and $\Delta t$ is the time-step size.

In \eqref{eqn:hdmres} above, the specific relationship between $\mathbf{u}(t^{\,k+1}; \boldsymbol{\mu})$ and its time-derivative $\dot{\mathbf{u}}(t^{\,k+1}; \boldsymbol{\mu})$ is given by the 
chosen time-discretization method. It takes on a different form, for example, if the time-discretization is based on a linear multistep scheme such as a backward differentiation formula (BDF), or an 
implicit Runge-Kutta scheme. In either case however, the nonlinear system of equations to be solved at each time-step or implicit Runge-Kutta stage can be written in the form given in (\ref{eqn:hdmres}).

One approach for constructing a nonlinear Petrov-Galerkin PROM associated with the HDM (\ref{eqn:hdm}) consists in substituting the subspace approximation (\ref{eqn:subspace}) into the residual equation 
(\ref{eqn:hdmres}), and squaring the resulting overdetermined nonlinear system of equations by projecting it onto an appropriate left ROB $\mathbf W$ as follows
\begin{equation}
\label{eqn:promres}
\mathbf{W}^T \mathbf{r}\left(\mathbf{u}_0 + \mathbf{V}\mathbf{y}(t^{\,k+1}; \boldsymbol{\mu}); \boldsymbol{\mu}\right) = 0
\end{equation}
An alternative approach for constructing a nonlinear Petrov-Galerkin PROM solution $\mathbf{y}(t^{\,k+1}; \boldsymbol{\mu})$ at time-instance $t^{\,k+1}$ is to substitute (\ref{eqn:subspace}) into
(\ref{eqn:hdmres}) but then solve the following residual minimization problem
\begin{equation}
\label{eqn:promresmin} 
	\mathbf{y}(t^{\,k+1}; \boldsymbol{\mu}) = \argmin_{\mathbf{x} \, \in \, \mathbb{R}^{n}} \big\lVert \mathbf{r}\left(\mathbf{u}_0 + \mathbf{V}\mathbf{x}(t^{\,k+1}; \boldsymbol{\mu});
\boldsymbol{\mu}\right) \big\rVert_{\boldsymbol{\Theta}}^2
\end{equation}
where $\boldsymbol{\Theta} \in \mathbb{R}^{N \times N}$ is an SPD matrix defining the norm $\|\mathbf{x}\|_{\boldsymbol{\Theta}} = \sqrt{\mathbf{x}^T \boldsymbol{\Theta} \mathbf{x}}$.
In this case, the PROM solution is found by minimizing the $\boldsymbol{\Theta}$-norm of the nonlinear HDM residual over the approximation solution subspace.

There are two notable settings in which the two nonlinear PROMs (\ref{eqn:promres}) and (\ref{eqn:promresmin}) are equivalent. From \eqref{eqn:promresmin}, it follows that in each of these
settings, the Petrov-Galerkin PROM defined by the pair of ROBs $(\mathbf{W}, \mathbf{V})$ is {\it optimal} in the sense that it produces a vector of reduced coordinates that is at each time-step 
solution of an HDM residual minimization problem. The two settings can be described as follows:
\begin{enumerate}
\item The case where the Jacobian of the HDM residual
\begin{equation}
\label{eqn:jac}
\mathbf{J}\left(\mathbf{u}_0 + \mathbf{V}\mathbf{y}(t^{\,k+1}; \boldsymbol{\mu}); \boldsymbol{\mu}\right) = \frac{\partial \mathbf{r}}{\partial \mathbf{u}}\left(\mathbf{u}_0 + \mathbf{V}\mathbf{y}(t^{\,k+1}; \boldsymbol{\mu}); \boldsymbol{\mu}\right) \in \mathbb{R}^{N \times N}
\end{equation}
is SPD. In this case, solving (\ref{eqn:promresmin}) with $\boldsymbol{\Theta} = \mathbf{J}^{-1}\left(\mathbf{u}_0 + \mathbf{V}\mathbf{y}(t^{\,k+1}; \boldsymbol{\mu}); \boldsymbol{\mu}\right)$, which is 
also SPD, is provably equivalent to performing a Galerkin projection ($\mathbf{W} = \mathbf{V}$) and solving (\ref{eqn:promres}).
\item For HDMs associated with convection-diffusion problems however, $\mathbf{J}\left(\mathbf{u}_0 + \mathbf{V}\mathbf{y}(t^{\,k+1}; \boldsymbol{\mu}); \boldsymbol{\mu}\right)$ is not in general SPD. For 
such HDMs, $\mathbf{J}^{-1}\left(\mathbf{u}_0 + \mathbf{V}\mathbf{y}(t^{\,k+1}; \boldsymbol{\mu}); \boldsymbol{\mu}\right)$ does not define a norm, and therefore the solution of (\ref{eqn:promres}) 
obtained via a Galerkin projection does not satisfy any optimality property. In this case, one can employ instead the Petrov-Galerkin projection based on the left ROB 
$\mathbf{W} = \boldsymbol{\Psi} \mathbf{J}\left(\mathbf{u}_0 + \mathbf{V}\mathbf{y}(t^{\,k+1}; \boldsymbol{\mu}); \boldsymbol{\mu}\right) \mathbf{V}$, where $\boldsymbol{\Psi}$ is any SPD matrix, 
which makes solving the algebraic problem \eqref{eqn:promres} provably equivalent to solving the minimization problem (\ref{eqn:promresmin}).
\end{enumerate}
The equivalence conditions outlined above are proved in \ref{app:proofnl}. Here, it is noted that because the objective function of the minimization problem (\ref{eqn:promresmin}) is generally 
nonconvex, these equivalence conditions hold provided that the method and initial guess chosen for solving problem (\ref{eqn:promres}) and problem (\ref{eqn:promresmin}) ensure convergence to the same 
local minimum. 

The conditions for equivalence between the two approaches \eqref{eqn:promres} and \eqref{eqn:promresmin} for nonlinear PMOR are important because they reveal that when a PROM is 
built using a left ROB $\mathbf{W}$ that meets one of the two criteria outlined above, it delivers a solution for which the $\boldsymbol{\Theta}$-norm of the HDM residual decreases monotonically 
when the dimension $n$ of the subspace approximation is increased by appending new vectors to the right ROB $\mathbf{V}$. This is because in this case, the minimization occurs over a larger subspace of 
$\mathbb{R}^N$.

It is also useful to consider the solution of the nonlinear problem (\ref{eqn:promres}), for example, by the Newton-Raphson method. At each $p$-th iteration, this method generates a linearized
PROM of the form
\begin{equation}
\label{eqn:MISS1}
\mathbf{W}^T \mathbf{J}^{(p)} \mathbf{V} \boldsymbol{\Delta} \mathbf{y}^{(p)} = -\mathbf{W}^T \mathbf{r}^{(p)}
\end{equation}
where $\mathbf{J}^{(p)} = \mathbf{J}\left(\mathbf{u}_0 + \mathbf{V}\mathbf{y}^{(p)}(\boldsymbol{\mu}); \boldsymbol{\mu}\right)$, and 
$\mathbf{r}^{(p)} = \mathbf{r}\left(\mathbf{u}_0 + \mathbf{V} \mathbf{y}^{(p)}(\boldsymbol{\mu}); \boldsymbol{\mu}\right)$. Indeed, both previously described settings for PMOR
and nonlinear residual minimization lead to the minimization of the error ${e}^{(p)} = \mathbf{V}\Delta\mathbf{y}^{(p)} - \boldsymbol{\Delta}\mathbf{u}^{(p)}$ associated with the reduced-order 
approximation of the search direction $\boldsymbol{\Delta}\mathbf{u}^{(p)}$, where $\boldsymbol{\Delta}\mathbf{u}^{(p)}$ is the solution of the underlying linearized HDM 
\begin{equation}
\label{eqn:MISS2}
\mathbf{J}^{(p)} \boldsymbol{\Delta} \mathbf{u}^{(p)} = -\mathbf{r}^{(p)}
\end{equation}
and to the following results:
\begin{enumerate}
\item If the Jacobian matrix $\mathbf{J}^{(p)}$ is SPD and therefore defines a norm, choosing $\mathbf{W} = \mathbf{V}$ for the solution at each iteration $p$ of problem \eqref{eqn:MISS1} --
that is, constructing a Galerkin PROM of the linearized HDM \eqref{eqn:MISS2} -- leads to minimization of the step direction error
\begin{equation*}
\Delta \mathbf{y}^{(p)} = \argmin_{\mathbf{x} \, \in \, \mathbb{R}^{n}} \left \lVert \mathbf{V}\mathbf{x} - \boldsymbol{\Delta}\mathbf{u}^{(p)} \right \rVert_{\mathbf{J}^{(p)}}^2
\end{equation*}
\item Alternatively, constructing a Petrov-Galerkin PROM for \eqref{eqn:MISS2} with $\mathbf{W} = \boldsymbol{\Psi} \mathbf{J}^{(p)} \mathbf{V}$, where $\boldsymbol{\Psi}$ is any SPD matrix, leads to
\begin{equation*}
\Delta \mathbf{y}^{(p)} = \argmin_{\mathbf{x} \, \in \, \mathbb{R}^{n}} \left \lVert \mathbf{V}\mathbf{x} - \boldsymbol{\Delta}\mathbf{u}^{(p)} \right \rVert_{{\mathbf{J}^{(p)}}^T \boldsymbol{\Psi} 
\mathbf{J}^{(p)}}^2
\end{equation*}
which is valid even when the Jacobian matrix $\mathbf{J}^{(p)}$ is not SPD. Recognizing that upon convergence $\mathbf{y}^{(p)} = \mathbf{y}(t^{\,k+1}; \boldsymbol{\mu})$, it follows that
the iteration-dependent left ROB $\mathbf{W} = \boldsymbol{\Psi} \mathbf{J}^{(p)}\mathbf{V}$ is equivalent to the left ROB 
$\mathbf{W} = \boldsymbol{\Psi} \mathbf{J}\left(\mathbf{u}_0 + \mathbf{V}\mathbf{y}(t^{\,k+1}; \boldsymbol{\mu}); \boldsymbol{\mu}\right) \mathbf{V}$ introduced above in the context of the nonlinear 
residual minimization.
\end{enumerate}
The above results are proven in \ref{app:prooflin}. 

If the HDM (\ref{eqn:hdm}) is linear, these results can be used to bound the error of the constructed PROM, 
$\mathbf{e}(t^{\,k+1}; \boldsymbol{\mu}) = \mathbf{V}\mathbf{y}(t^{\,k+1}; \boldsymbol{\mu}) - \mathbf{u}(t^{\,k+1}; \boldsymbol{\mu})$. In this case, the Galerkin 
projection minimizes the error $\mathbf{e}(t^{\,k+1},\boldsymbol{\mu})$ in the $\mathbf{J}(\boldsymbol{\mu})$-norm when $\mathbf{J}(\boldsymbol{\mu})$ is SPD, while the Petrov-Galerkin 
projection minimizes this error in the $\mathbf{J}(\boldsymbol{\mu})^T \boldsymbol{\Psi} \mathbf{J}(\boldsymbol{\mu})$-norm. PROMs based on this Petrov-Galerkin formulation with 
$\boldsymbol{\Psi} = \mathbf{I}$, where $\mathbf I$ denotes the identity matrix, were studied in \cite{buithanh2008} for steady, linear problems.

For nonlinear HDMs of the form given in \eqref{eqn:hdm}, choosing $\boldsymbol{\Theta} = \mathbf{I}$ leads to the least-squares Petrov-Galerkin (LSPG) method introduced in \cite{carlberg2011}.
This method is characterized by the iteration-dependent left ROB
\begin{equation*}
\mathbf{W} = \mathbf{J}\left(\mathbf{u}_0 + \mathbf{V}\mathbf{y}(t^{\,k+1}; \boldsymbol{\mu}); \boldsymbol{\mu}\right) \mathbf{V}
\end{equation*}
and the minimization of the HDM residual (\ref{eqn:hdmres}) in the $\ell^2$-norm. In this case, the solution of the nonlinear system (\ref{eqn:promres}) produces identical iterates to those encountered when solving the nonlinear least-squares problem (\ref{eqn:promresmin}) using the Gauss-Newton method.

In \cite{abgrall2018}, PROMs were constructed based on the minimization of the HDM residual in the $\ell^1$-norm, without any connection to a Petrov-Galerkin approach. Here, a connection is established
to illustrate how general the Petrov-Galerkin approach can be. To this end, consider the matrix $\boldsymbol{\Theta}$ with diagonal entries defined as follows
\begin{equation*}
\boldsymbol{\Theta}_{ii}(\mathbf{x}; \boldsymbol{\mu}) = 
\begin{cases}
	\displaystyle \frac{1}{\big\lvert \mathbf{r}_i(\mathbf{u}_0 + \mathbf{V}\mathbf{x}; \boldsymbol{\mu}) \big\rvert}, &\textrm{if} \, \, \mathbf{r}_i(\mathbf{u}_0 + \mathbf{V}\mathbf{x}; \boldsymbol{\mu}) \neq 0 \\[14pt]
\displaystyle 1, & \textrm{otherwise} \\
\end{cases} \qquad \qquad i = 1, \ldots, N
\end{equation*}
where $\mathbf{r}_i(\mathbf{u}_0 + \mathbf{V}\mathbf{x}; \boldsymbol{\mu})$ denotes the $i$-th entry of the vector $\mathbf{r}(\mathbf{u}_0 + \mathbf{V}\mathbf{x}; \boldsymbol{\mu})$.
Clearly $\boldsymbol{\Theta}$ is SPD.
Given this definition, the objective function of the minimization problem (\ref{eqn:promresmin}) becomes
\begin{equation*}
\begin{split}
\big\lVert \mathbf{r}(\mathbf{u}_0 + \mathbf{V}\mathbf{x}; \boldsymbol{\mu}) \big\rVert_{\boldsymbol{\Theta}}^2 &= \sum_{i=1}^N \boldsymbol{\Theta}_{ii} \mathbf{r}_i(\mathbf{u}_0 + \mathbf{V}\mathbf{x}; \boldsymbol{\mu})^2 \\
&= \sum_{i=1}^N \big\lvert \mathbf{r}_i(\mathbf{u}_0 + \mathbf{V}\mathbf{x}; \boldsymbol{\mu}) \big\rvert \\
&= \big\lVert \mathbf{r}(\mathbf{u}_0 + \mathbf{V}\mathbf{x}; \boldsymbol{\mu}) \big\rVert_{1}
\end{split}
\end{equation*}
Thus, there exists an SPD matrix $\boldsymbol{\Theta}$ for which minimizing the $\ell^1$-norm of the HDM residual subject to the subspace approximation (\ref{eqn:subspace}) is equivalent to solving the 
$\boldsymbol{\Theta}$-weighted residual minimization problem (\ref{eqn:promresmin}), which in view of the equivalence conditions discussed above is equivalent to the Petrov-Galerkin PMOR
method (\ref{eqn:promres}). Note that the idea of a solution-dependent weighting metric $\boldsymbol{\Theta}$ has also been extensively leveraged in the development of the iteratively reweighted 
least squares (IRLS) algorithms for general $\ell^{p}$-minimization problems \cite{daubechies2010}.

In summary, for convection-dominated flow problems which in general involve nonsymmetric Jacobian matrices resulting from the semi-discretization of the convective terms, the superiority of the 
Petrov-Galerkin projection approach results from the optimality of the equivalent residual minimization approach. It is also noted that the concept of residual minimization for the purpose of achieving
a stable HDM is prevalent in FE developments -- for example, in the Galerkin/least-squares and least-squares FE methods \cite{hughes1989, bochev2009}. The extensive success of Galerkin PROMs for 
elliptic and other problems involving SPD Jacobian matrices, accompanied by the numerical evidence of their instability and inaccuracy for problems that are neither elliptic nor characterized by
SPD Jacobian matrices, suggests that the connection of a PROM approach to proper residual minimization is important for preserving numerical stability and accuracy. Just like Petrov-Galerkin 
semi-discretization methods can be designed to address numerical issues arising from the the numerical solution of convection-dominated problems by standard Galerkin semi-discretization methods,
Petrov-Galerkin PMOR approaches equipped with appropriate left ROBs can be designed to remedy the loss of residual minimization properties arising from the PMOR of convection-dominated HDMs using
standard Galerkin projections.

\section{Pre-computation and hyperreduction for accelerating the computation of reduced-order quantities}
\label{sec:preandhyper}

Whether taking the approach of building and solving the nonlinear algebraic system (\ref{eqn:promres}) or the nonlinear residual minimization problem (\ref{eqn:promresmin}), the computational complexity
of solving the resulting nonlinear PROM scales in general with the size of the underlying HDM $N$, even though the search space for computing the optimal solution $\mathbf{y}(t^{\,k+1}; \boldsymbol{\mu})$
is of dimension $n \ll N$. This is due to the need to compute at each nonlinear iteration within each time-step the reduced-order vectors 
$\mathbf{W}^T \mathbf{r}(\mathbf{u}_0 + \mathbf{V}\mathbf{y}; \boldsymbol{\mu})$ and reduced-order matrices $\mathbf{W}^T \mathbf{J}(\mathbf{u}_0 + \mathbf{V}\mathbf{y}; \boldsymbol{\mu}) \mathbf{V}$,
for various values of $\mathbf{y}$ and $\boldsymbol{\mu}$. Typically, these computations are performed by first computing the high-dimensional quantities 
$\mathbf{r}(\mathbf{u}_0 + \mathbf{V}\mathbf{y}; \boldsymbol{\mu})$ and $\mathbf{J}(\mathbf{u}_0 + \mathbf{V}\mathbf{y}; \boldsymbol{\mu}) \mathbf{V}$ then multiplying the results by $\mathbf{W}^T$,
which leads to computational complexities that scale as $O(Nn)$ and $O(N^2n)$, respectively.

There are two common approaches for mitigating the computational bottleneck highlighted above and achieving real-time or near real-time execution of the constructed nonlinear PROM: pre-computation; and 
hyperreduction. This section briefly overviews each of them, particularly in the context of the problems of interest to this work. For further details on this topic, the reader is referred to 
\cite{farhat2020}.

\subsection{Exact pre-computation-based approach}
\label{sec:precomp}

When the nonlinearity of the semi-discrete HDM (\ref{eqn:hdm}) is polynomial in the state $\mathbf{u}(t; \boldsymbol{\mu})$, an {\it exact} approach for resolving the aforementioned computational 
bottlenecks is possible and thoroughly described in \cite{farhat2020}. The approach involves decomposing the computation into offline and online parts. In the offline part, the quantities responsible for
the scaling of the computations with the dimension $N$ of the HDM are pre-computed. In the online part, the pre-computed quantities are exploited to reconstruct the desired reduced-order vectors and 
matrices with a computational complexity that scales only with the dimension $n \ll N$ of the PROM -- and therfore in real-time or near real-time. This approach is commonly employed for HDMs based on 
the semi-discretization of the incompressible Navier-Stokes equations, where for piecewise linear semi-discretizations, the nonlinearity in the semi-discrete state vector $\mathbf{u}(t; \boldsymbol{\mu})$
is quadratic. Note that in the parametric nonlinear setting, the aforementioned pre-computations must be performed for each queried parameter point $\boldsymbol{\mu}^* \in \mathcal{P}$, which could 
restrict the computational feasibility of this approach due to a prohibitively expensive offline part.

In the Petrov-Galerkin framework described in Section \ref{sec:pg}, the left ROB $\mathbf{W}$ depends by construction on the Jacobian matrix 
$\mathbf{J}(\mathbf{u}_0 + \mathbf{V}\mathbf{y}; \boldsymbol{\mu})$ (\ref{eqn:jac}) of the HDM. Given that the HDM residual (\ref{eqn:hdmres}) is typically a polynomial function of $\mathbf{u}$ of 
degree $d$ that is semi-discretization-dependent, its Jacobian is a polynomial function of $\mathbf{u}$ of degree $d-1$. 
Thus, the Petrov-Galerkin PROM (\ref{eqn:promres}) will include a polynomial nonlinearity in $\mathbf{y}$ of degree $2d-1$ that can be treated by pre-computations.
Furthermore after implicit temporal discretization is performed, the solution of the nonlinear system of equations (\ref{eqn:promres}) or that of the nonlinear residual minimization problem 
(\ref{eqn:promresmin}) requires the construction of the Jacobian of the PROM, which
can be written as
\begin{equation*}
\mathbf{W}^T \mathbf{J}(\mathbf{u}_0 + \mathbf{V}\mathbf{y}; \boldsymbol{\mu}) \mathbf{V}
\end{equation*}
and therefore is a polynomial function in $\mathbf{y}$ of degree $2d-2$: its computation can also be treated using the exact offline-online pre-computation-based approach.

\begin{remark}
The pre-computation-based approach introduces no additional approximations in order to evaluate the projected reduced-order quantities of interest in a computationally efficient manner. 
For this reason, it is desirable. It is by far the most prevalent approach, when it is applicable. For convection-diffusion flow problems, its applicability is limited however to the case of
incompressible flows. Even then, there are still many instances where the HDM contains a nonpolynomial dependence on the semi-discrete state vector $\mathbf{u}(t; \boldsymbol{\mu})$ and therefore
the  pre-computation-based approach is inapplicable. For example, spatial discretizations that introduce nonlinear numerical dissipation often destroy an otherwise polynomial dependence 
in $\mathbf{u}(t; \boldsymbol{\mu})$ and therefore lead to arbitrarily nonlinear HDMs. Moreover, an HDM may exhibit nonpolynomial nonlinearities depending on the choice of the state variables -- for
example, when the compressible Navier-Stokes equations are expressed using conservative variables. In some cases, lifting transformations may exist to transform a nonlinear HDM where the nonlinearity
in $\mathbf{u}(t; \boldsymbol{\mu})$ is nonpolynomial into an equivalent HDM to which the pre-computation-based approach can be applied \cite{kramer2019} for accelerating the computation of the 
reduced-order quantities. However, lifting and related approaches often lead to the introduction of a prohibitive number of auxiliary state variables and have not yet been successfully demonstrated
for large-scale CFD models. Finally, it should be noted that even when all nonlinearities in $\mathbf{u}(t; \boldsymbol{\mu})$ are polynomial of degree $d$, the offline part of the pre-computation-based 
approach is in practice computationally unaffordable for $d \ge 3$.
\end{remark}

\subsection{Hyperreduction approach}
\label{sec:hyper}

In the general nonlinear case where the pre-computation-based approach is not applicable due to any of the reasons outlined in Section \ref{sec:precomp}, an alternative approach known as hyperreduction 
may be used to accelerate the computation of the reduced-order quantities associated with a nonlinear PROM. This approach introduces a second layer of approximations that enables the evaluation of 
reduced-order vectors and matrices with a computational complexity that scales only with the size $n$ of the PROM -- that is, independently from the size $N$ of the HDM. The second layer of approximations
transforms the PROM into a hyperreduced PROM (HPROM) that delivers the sought-after overall computational speed-up, while still remaining accurate relative to the underlying PROM and HDM.

In this paper, hyperreduction is performed whenever needed, and unless otherwise specified, using a version of the energy-conserving sampling and weighting (ECSW) method \cite{farhat2014, farhat2015} 
that was recently adapted to Petrov-Galerkin PROMs \cite{grimberg2020}. ECSW is a member of the {\it project-then-approximate} class of hyperreduction methods, which seeks to directly approximate the 
projected reduced-order quantities $\mathbf{W}^T \mathbf{r}(\mathbf{u}_0 + \mathbf{V}\mathbf{y}; \boldsymbol{\mu})$ as well as $\mathbf{W}^T \mathbf{J}(\mathbf{u}_0 + 
\mathbf{V}\mathbf{y}; \boldsymbol{\mu}) \mathbf{V}$ and therefore avoid the construction of the high-dimensional quantities $\mathbf{r}(\mathbf{u}_0 + \mathbf{V}\mathbf{y}; \boldsymbol{\mu})$ and 
$\mathbf{J}(\mathbf{u}_0 + \mathbf{V}\mathbf{y}; \boldsymbol{\mu}) \mathbf{V}$, as well as the ensuing multiplication by $\mathbf{W}^T$.
First introduced in \cite{farhat2014} in the context of FE-based nonlinear PROMs for second-order dynamical systems, ECSW was shown in \cite{farhat2015} to preserve the Lagrangian structure associated 
with Hamilton's principle for second-order dynamical systems, and thus to preserve the numerical stability properties of the associated time-dependent HDMs.

ECSW and other project-then-approximate hyperreduction methods originated as an alternative to the {\it approximate-then-project} class of methods, which shares origins with the gappy POD method 
\cite{everson1995} developed in the context of image reconstruction. Methods of this type, including the well-known discrete empirical interpolation method (DEIM) \cite{chaturantabut2010}
and Gauss-Newton with approximated tensors (GNAT) method \cite{carlberg2011}, have had a longer history of successful applications to the hyperreduction of Galerkin as well as Petrov-Galerkin PROMs,
including those constructed for large-scale, RANS-modeled turbulent flow problems of industrial relevance \cite{carlberg2013}. However, there is substantially less evidence to date of successful 
applications of either class of hyperreduction methods to scale-resolving turbulent flow problems modeled using large eddy simulation (LES) or direct numerical simulation (DNS) methods.

In this work, ECSW is the hyperreduction method of choice for two main reasons. First, ECSW offers practical advantages over gappy POD-based approximate-then-project hyperreduction methods that are
described in \cite{grimberg2020}. From a user perspective, it enables the construction of HPROMs in a more streamlined fashion that requires less specification of problem-dependent parameters.
Second, given the successful demonstrations of ECSW for large-scale laminar and RANS flow simulations reported in \cite{grimberg2020}, it is desired here to provide further demonstration of ECSW 
for the hyperreduction of LES or DNS PROMs. 

For all practical details related to the implementation of ECSW for the hyperreduction of the Petrov-Galerkin PROMs presented in this work, the reader is referred to \cite{grimberg2020}.
Here, the emphasis is placed on the effects of the additional approximations introduced by the hyperreduction process on the numerical stability and accuracy of nonlinear PROMs
for convection-dominated laminar and turbulent flows.

\section{Numerical examples}
\label{sec:examples}

This section discusses the numerical stability of state-of-the-art PROMs, and where appropriate, corresponding HPROMs, constructed for three different CFD applications {\it designed to 
support the arguments made in this paper}. While the parametric aspect of a PROM is important for its relevance to a given application as well as its potential for wall-clock (or CPU) time reduction,
it is unrelated to the numerical stability aspects discussed in this paper. For this reason, in order to keep the discussion focused on numerial stability, and for the sake of simplicity, all three aforementioned CFD applications
are nonparametric. Hence, all PROMs and HPROMs presented herein are constructed/trained for a single parameter point $\boldsymbol \mu \in {\mathcal P}$. For a demonstration of the performance of similar 
PROMs and HPROMs for parametric CFD applications, the reader can consult, for example, \cite{zahr2015, washabaugh2016}.

In the first application, the performances of Galerkin and Petrov-Galerkin CFD PROMs and HPROMs are contrasted for a laminar flow problem, and numerical evidence about the instability of Galerkin PROMs for convection-dominated 
problems, even in the absence of turbulence, is presented. The purpose of this example application is to solidify the connection between the numerical instability of a Galerkin PROM and 
convection-dominated phenomena, demonstrate a weakness in the turbulent energy spectrum argument on PROM instability, as well as demonstrate the numerical stability of a Petrov-Galerkin framework 
for the PMOR of convection-dominated CFD problems. The second application example considers the construction of a PROM for a DNS of a fully turbulent flow problem, in the absence of a mean convective 
velocity. Finally, the third example demonstrates the successful construction of PROMs and HPROMs for the LES of a turbulent flow problem. In all cases, it is shown that stable, accurate, and 
computationally efficient Petrov-Galerkin PROMs or HPROMs are attainable for various classes of difficult CFD problems without sacrificing the notion of PROM consistency, or modifying the classical
design of an approximation subspace. Lastly, the approximations introduced for transforming PROMs into HPROMs in order to achieve computational efficiency are shown to preserve numerical stability 
without sacrificing accuracy.

All Petrov-Galerkin PROMs discussed in this section are constructed using the LSPG method \cite{carlberg2011, carlberg2013}. Whenever hyprereduction is required, it is performed using the
finite volume adaptation of ECSW described in \cite{grimberg2020} and equipped with the parallelized Lawson and Hanson nonnegative least-squares solver developed in \cite{chapman2017}.

In all cases, the accuracy of a PROM- or HPROM-based simulation relative to its counterpart HDM-based computation is assessed over the time-interval $[0, t_{max}]$ for a selection of scalar
quantities of interest (QoIs) using the relative error metric
\begin{equation*}
\mathbb{RE}_{Q} = \frac{\displaystyle \sqrt{\sum_{t \in P} (Q(t) - \widetilde{Q}(t))^2}}{\displaystyle \sqrt{\sum_{t \in P} Q(t)^2}} \times 100 \%
\end{equation*}
{%
\def\OldComma{,}%
\catcode`\,=13%
\def,{%
	\ifmmode%
	\OldComma\discretionary{}{}{}%
\else%
	\OldComma%
\fi%
}%
Here, $Q(t)$ and $\widetilde{Q}(t)$ are the predicted values of a QoI at time-instance $t$ using the HDM and PROM/HPROM of interest, respectively, and $P$ is the set of time-instances sampled
for computing the relative error $\mathbb{RE}_{Q}$ -- that is, $P = \big\{t \in \{0, \Delta s_{\mathbb{RE}}, 2 \Delta s_{\mathbb{RE}}, \ldots \} \mid t \leq t_{max}\big\}$ where 
$\Delta s_{\mathbb{RE}}$ is the sampling time-interval for the evaluation of $\mathbb{RE}_{Q}$, and $t_{max}$ is the final time-instance of the simulation time-interval.%
}
In each PROM- or HPROM-based simulation, the computation of $\widetilde{Q}(t)$ is performed by first reconstructing at each time-instance $t \in P$ the high-dimensional solution 
$\mathbf{u}(t; \boldsymbol{\mu}) = \mathbf{u}_0+\mathbf{V}\mathbf{y}(t; \boldsymbol{\mu})$, and then computing the QoI using the same procedure as in the 
case of the HDM-based simulation.

Finally, it is noted that all computations reported below are performed using a parallel Linux cluster where each compute node is equipped with two six-core Intel Xeon Gold 6128 processors running at 
3.40 GHz and 192 GB of memory, and the network fabric is Mellanox EDR 100G Infiniband.

\subsection{Computation of a laminar flow over a circular cylinder}
\label{sec:cylinder}

The first example focuses on the problem of predicting the two-dimensional, compressible, laminar flow over a right circular cylinder at $Re = 100$, and the free-stream Mach number $M_\infty = 0.2$.
At this Reynolds number, the flow is laminar: after a transient startup phase, it exhibits a K\'arm\'an vortex street. The purpose of this example is to demonstrate that even in the absence of turbulence,
convective phenomena can challenge the numerical stability of Galerkin PROMs and HPROMs when the constructed subspace approximation (\ref{eqn:subspace}) is rather ``coarse''. It is offered as 
a supplementary numerical evidence to that reported in \cite{carlberg2011, carlberg2017}.

\subsubsection{High-dimensional model}

For this problem, the CFD HDM is constructed by semi-discretizing the nondimensional form of the three-dimensional (3D), compressible Navier-Stokes equations using a mixed finite volume/finite element method.  
Specifically, the convective fluxes are approximated using a third-order, upwind, vertex-based finite volume method \cite{anderson1986, farhat1993}, and the diffusive fluxes are semi-discretized by
a piecewise linear Galerkin FE scheme.

The computational domain, shown in Figure \ref{fig:cylinder}(a), is chosen as a disk of diameter $40D$, where $D$ denotes the diameter of the circular cylinder. It is discretized by a 
one-element-thick unstructured mesh with $98,140$ vertices and $284,700$ tetrahedral elements, resulting in an HDM of dimension $N = 490,700$.
Symmetry boundary conditions are applied on the spanwise faces of the computational domain to ensure that the resulting flow is two-dimensional (2D), and a no-slip adiabatic wall boundary condition 
is applied on the surface of the cylinder. The chosen mesh resolution leads to good agreement between the computed flow solution and other numerical as well as experimental studies of this problem 
reported in \cite{wieselsberger1922, henderson1995}.

Time-discretization is performed using a second-order diagonally implicit Runge-Kutta (DIRK) scheme, and a fixed, nondimensional time-step $\Delta t = 1 \times 10^{-1}$. For the mesh outlined above,
this time-step corresponds to a CFL number of approximately $7,600$. In this case, the DIRK scheme offers improved nonlinear stability over the more commonly used multistep BDF schemes, even though 
both of their second-order variants are $L$-stable. In particlular, it allows a larger time-step. The Butcher tableau for this DIRK scheme is given in \ref{app:dirk}. The nonlinear system resulting at 
each Runge-Kutta stage is solved by a Newton-Krylov method equipped with GMRES as the linear solver, and the restricted additive Schwartz preconditioner based on an ILU(1) factorization \cite{cai1998}.

For all simulations reported herein, the same initial flow condition is applied and computed as follows. Using the HDM described above, the flow is impulsively started from a uniform 
state, then time-integrated until the onset of vortex shedding. 

All simulations are performed in the nondimensional time-interval $t \in [0, 200]$. At approximately $t = 100$, the HDM-based flow solution becomes periodic. 
For this relatively small-scale computational problem, the HDM-based simulation requires 65.4 hours wall-clock time on a single core of a compute node of the Linux cluster specified above.
Figure \ref{fig:cylinder}(b) shows a snapshot of the flow vorticity solution computed at the end of the simulation -- that is, at $t = 200$ -- using the HDM.

\begin{figure}[h!]
	\centering
	\begin{subfigure}[c]{0.275\textwidth}%
	\centering
	\includegraphics[width=\linewidth]{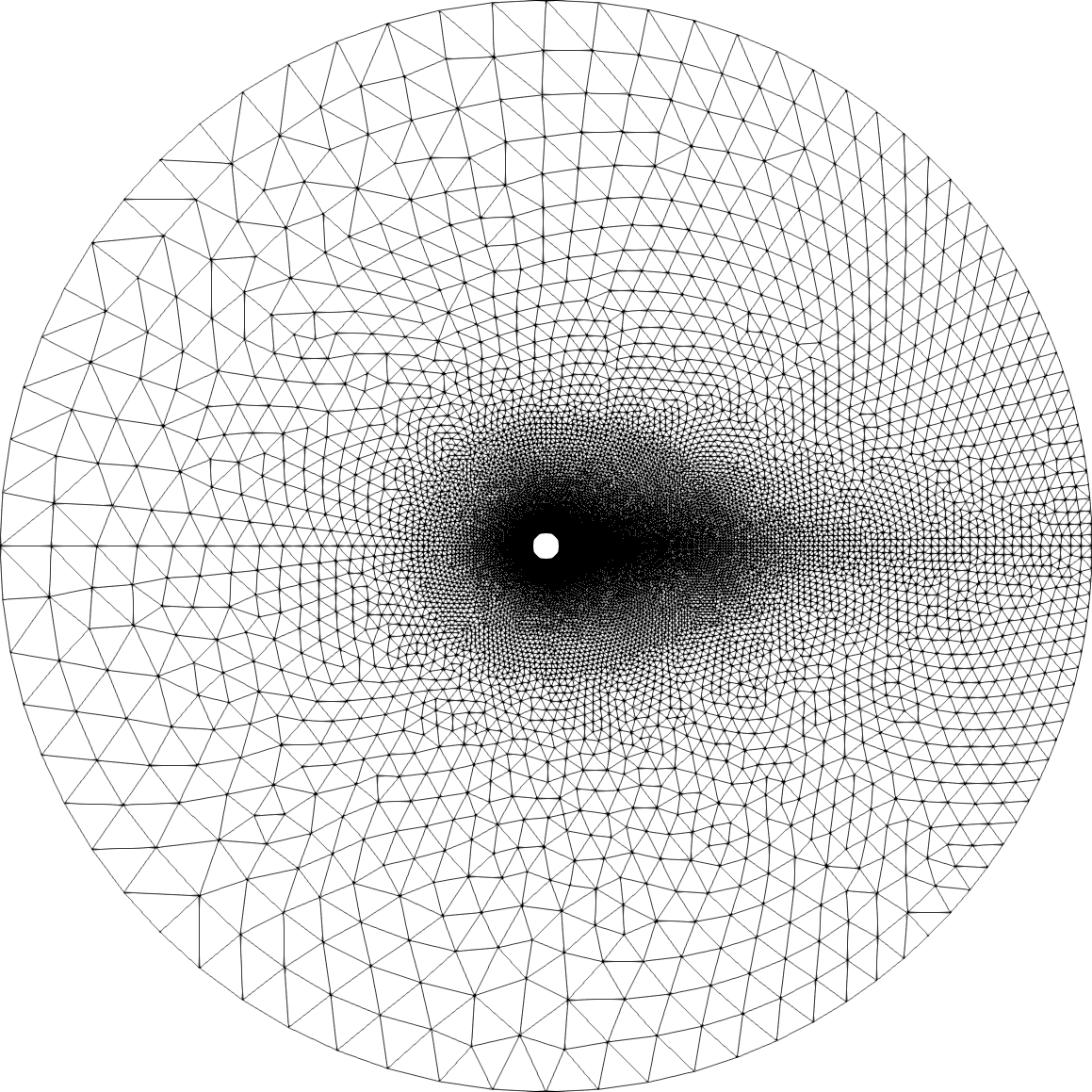}
	\caption{}
	\end{subfigure}%
	\hspace{1em}
	\begin{subfigure}[c]{0.60\textwidth}%
	\centering
	\includegraphics[width=\linewidth]{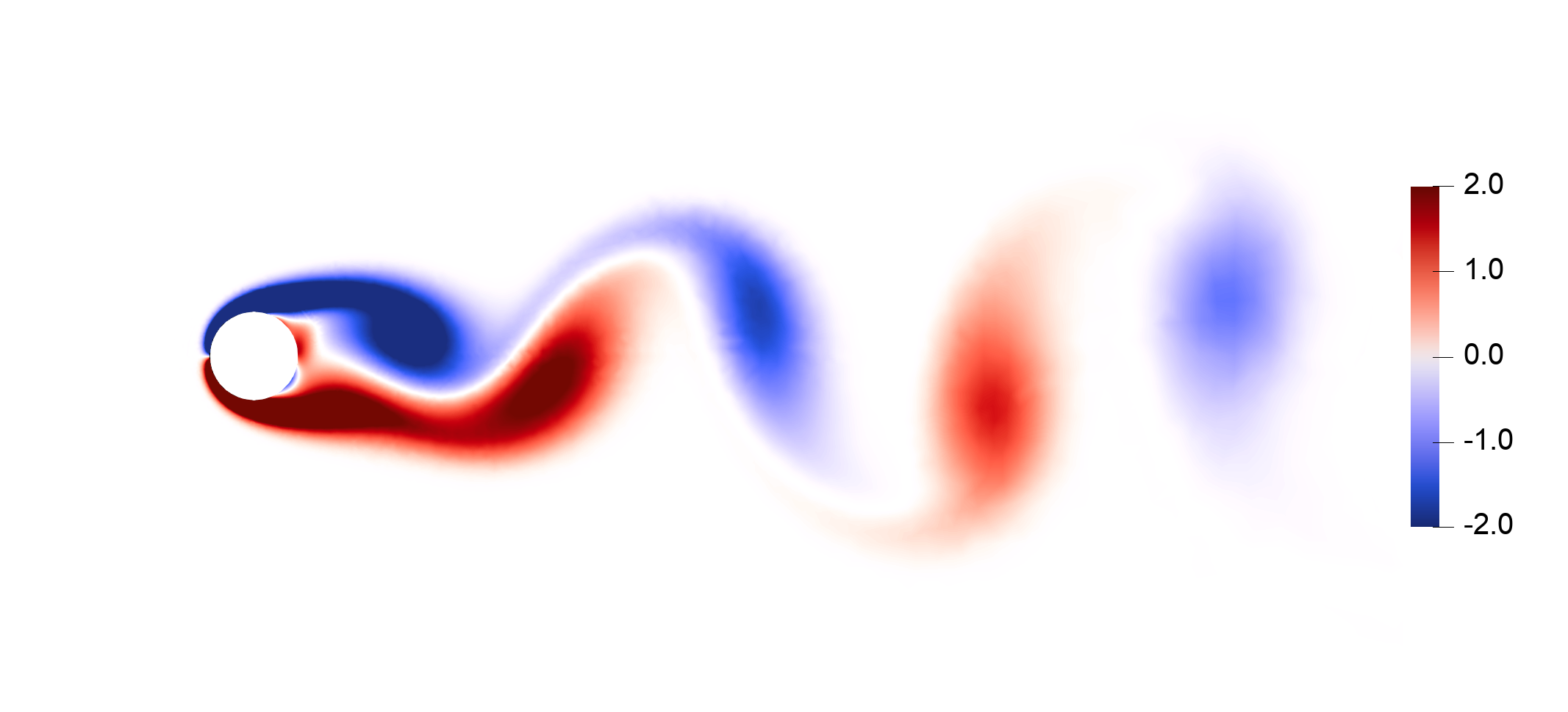}
	\caption{}
	\end{subfigure}%
	\caption{2D flow problem over a right circular cylinder: (a) Discretized computational domain; (b) HDM-based vorticity solution snapshot at $t = 200$.}
	\label{fig:cylinder}
\end{figure}

\subsubsection{Projection-based reduced-order models with hyperreduction}

HDM-based solution snapshots are collected in the first nondimensional subinterval $[0,150]$ at the sampling rate defined by $\Delta s = 2 \times 10^{-1}$. Hence, the second nondimensional 
subinterval $[150, 200]$ is used here to test the ability of the constructed PROMs and HPROMs to accurately predict the behavior of the solution outside of the trained time-interval, when the
solution happens to remain periodic with the same amplitude and frequency. The offset of the affine subspace approximation $\mathbf{u}_0$ is chosen as the initial condition of the problem, and the 
singular value decomposition (SVD) method is employed to compress the 751 collected solution snapshots after normalization and construct three right ROBs of dimension $n = 20$, $n = 35$, and $n = 55$.
These choices of $n$ lead to ROBs that capture $99.9\%$, $99.99\%$, and $99.999\%$ of the energy of the singular values of the snapshot matrix, respectively. Using these ROBs, both Galerkin and 
Petrov-Galerkin PROMs are constructed. For computational efficiency, the left ROB associated with the Petrov-Galerkin projection is computed only once per time-step, then re-used for all Gauss-Newton 
iterations associated with that single time-step.

In this case, the nonlinearity of the resulting semi-discrete equations defining the HDM is not low-order polynomial. Thus, hyperreduction is required in order to achieve the expected computational 
speed-up over the performance of the HDM.
To this end, the training snapshots for ECSW are collected at half the frequency mentioned above for constructing the right ROBs -- that is, a subset of 376 of the aforementioned collected solution
snapshots is used to train in the nondimensional time-interval $[0,150]$ the reduced mesh computed by ECSW -- and the training tolerance is set to $\varepsilon = 1\times10^{-2}$.

All constructed PROMs and HPROMs are time-discretized using the same second-order DIRK scheme as the HDM, and the same nondimensional time-step $\Delta t = 1 \times 10^{-1}$.
As in the case of the HDM-based simulation, all PROM- and HPROM-based simulations are performed on a single core of one of the Linux cluster's compute nodes.

For this problem, the QoIs chosen for assessing the accuracy of the constructed PROMs and HPROMs are the time-dependent lift and drag coefficients $c_L$ and $c_D$, as well as the streamwise velocity 
$v_x$ and pressure $p$, computed at a probe located $5D$ downstream from the cylinder's trailing edge. The nondimensional sampling time-interval for computing the relative errors is 
chosen to be $\Delta s_{\mathbb{RE}} = \Delta t = 1 \times 10^{-1}$.

\paragraph{Galerkin reduced-order models}

Figures \ref{fig:cylinderliftgal} and \ref{fig:cylinderdraggal} compare the time-histories of the lift and drag coefficients computed using the HDM and Galerkin PROMs and HPROMs for each dimension $n$ of the right ROB.
The time-histories of the streamwise velocity and pressure computed at the probe location are compared in Figures \ref{fig:cylinderprobevxgal} and \ref{fig:cylinderprobepgal}.
It is first observed that while the HDM-based solution reaches a limit cycle and becomes periodic after the transient startup phase, all Galerkin PROMs exhibit a drift away from this limit cycle which is most noticeable in the $c_D$ and $v_x$ time-histories.
While increasing the dimension $n$ decreases the rate of drifting and increases accuracy, even the case $n = 55$ exhibits this undesirable behavior.
Though the HDM is numerically stabilized using upwinding, it does not lead to numerical stability of the associated Galerkin PROM. This supports the argument that suitable approximation methods 
must be applied at each modeling stage, whether from the continuous level to the HDM, or from the HDM to the PROM.

\begin{figure}[h!]
	\centering
	\begin{subfigure}[c]{0.3\textwidth}%
	\centering
	\includegraphics[width=\linewidth]{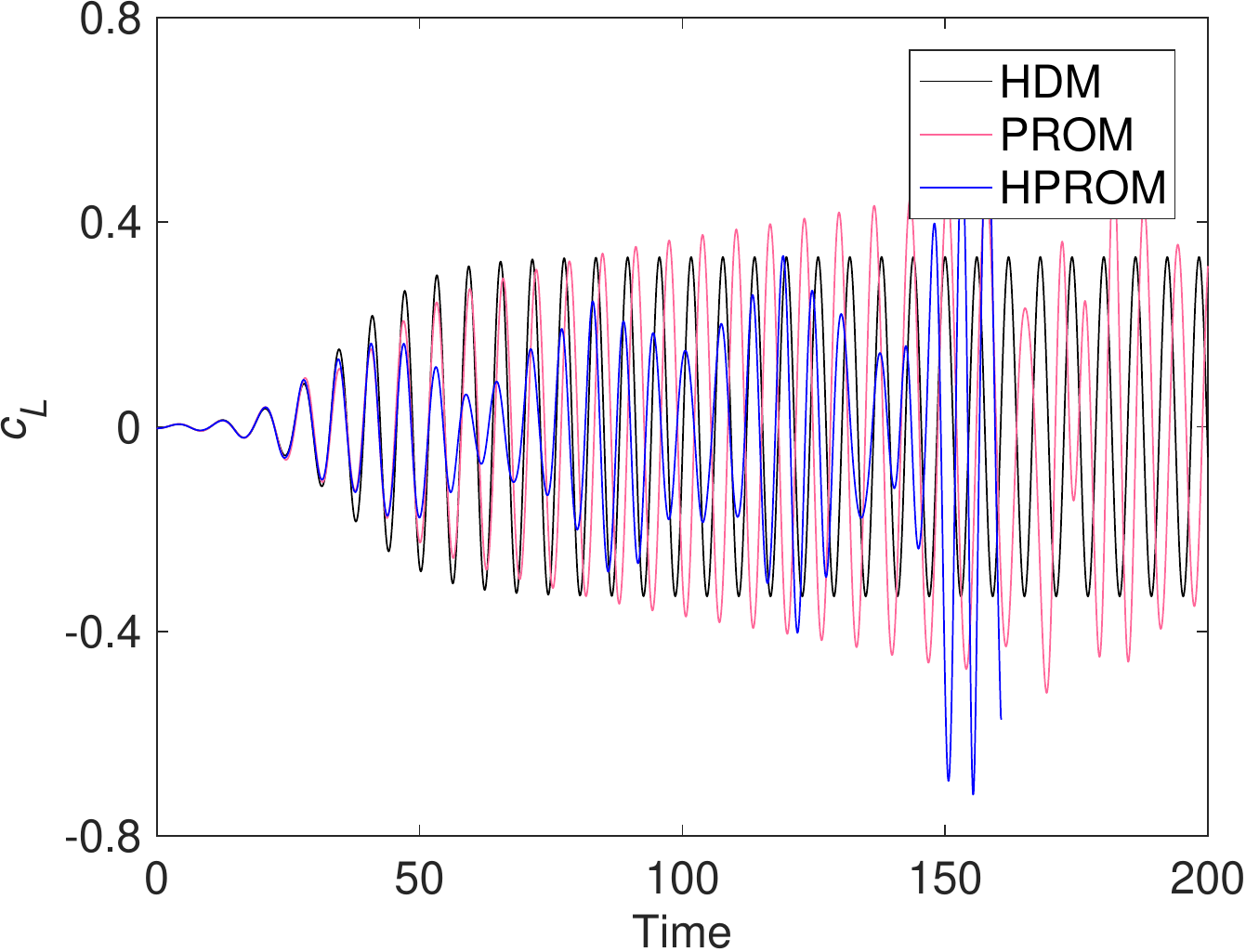}
	\caption{Galerkin, $n = 20$}
	\end{subfigure}%
	\hspace{1em}
	\begin{subfigure}[c]{0.3\textwidth}%
	\centering
	\includegraphics[width=\linewidth]{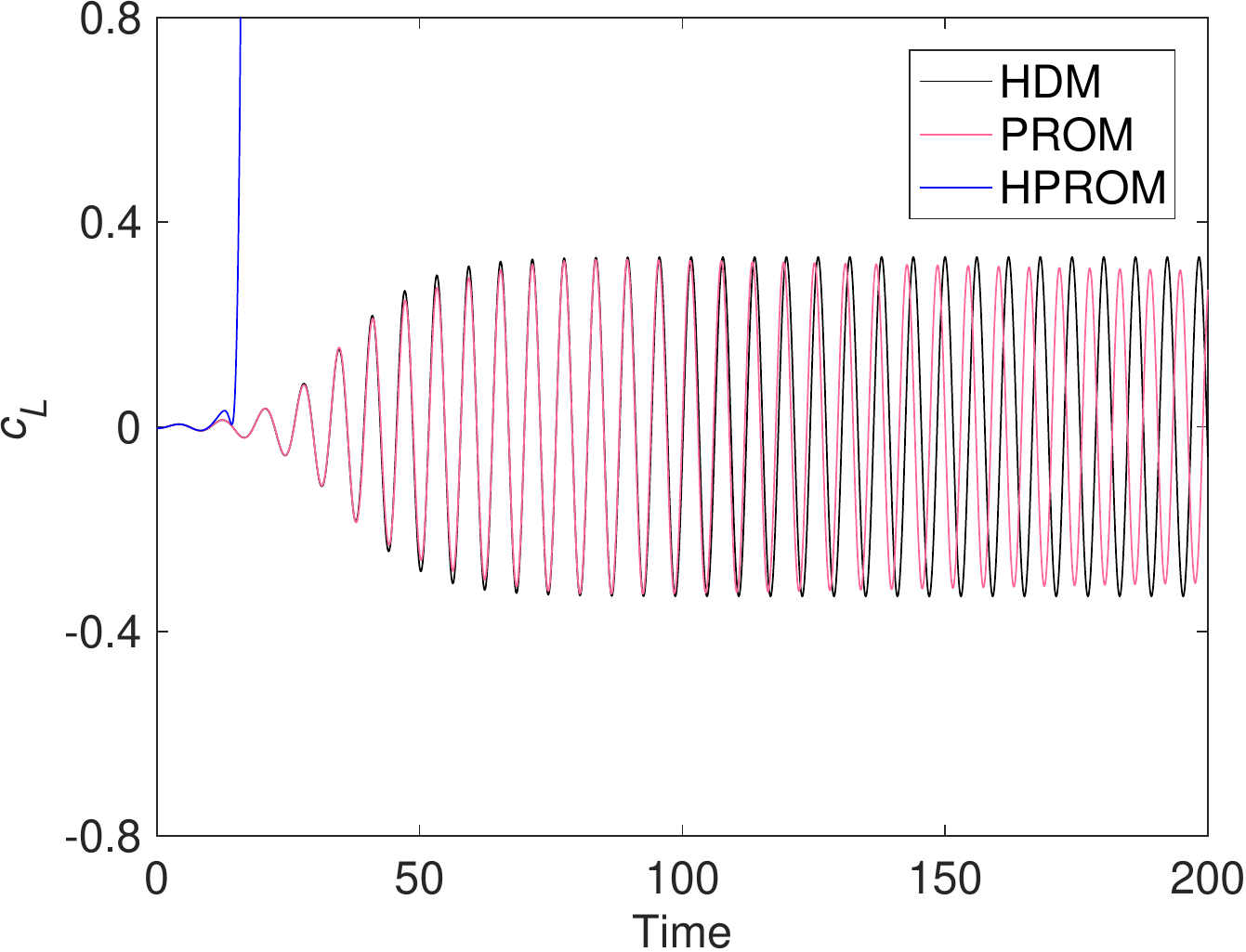}
	\caption{Galerkin, $n = 35$}
	\end{subfigure}%
	\hspace{1em}
	\begin{subfigure}[c]{0.3\textwidth}%
	\centering
	\includegraphics[width=\linewidth]{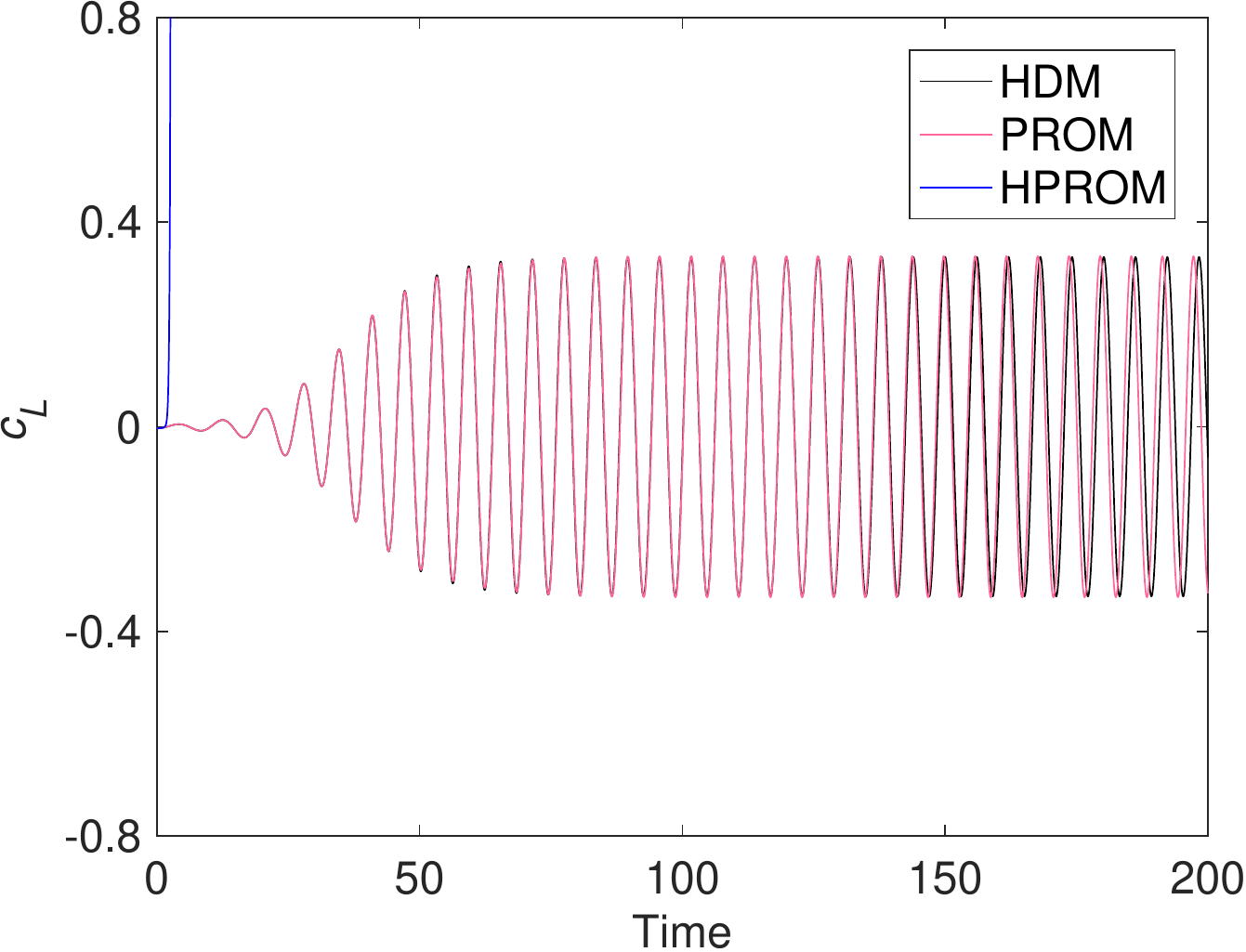}
	\caption{Galerkin, $n = 55$}
	\end{subfigure}%
	\caption{2D flow over a right circular cylinder: Time-histories of the lift coefficient computed using the HDM and Galerkin PROMs and HPROMs.}
	\label{fig:cylinderliftgal}
\end{figure}

\begin{figure}[h!]
	\centering
	\begin{subfigure}[c]{0.3\textwidth}%
	\centering
	\includegraphics[width=\linewidth]{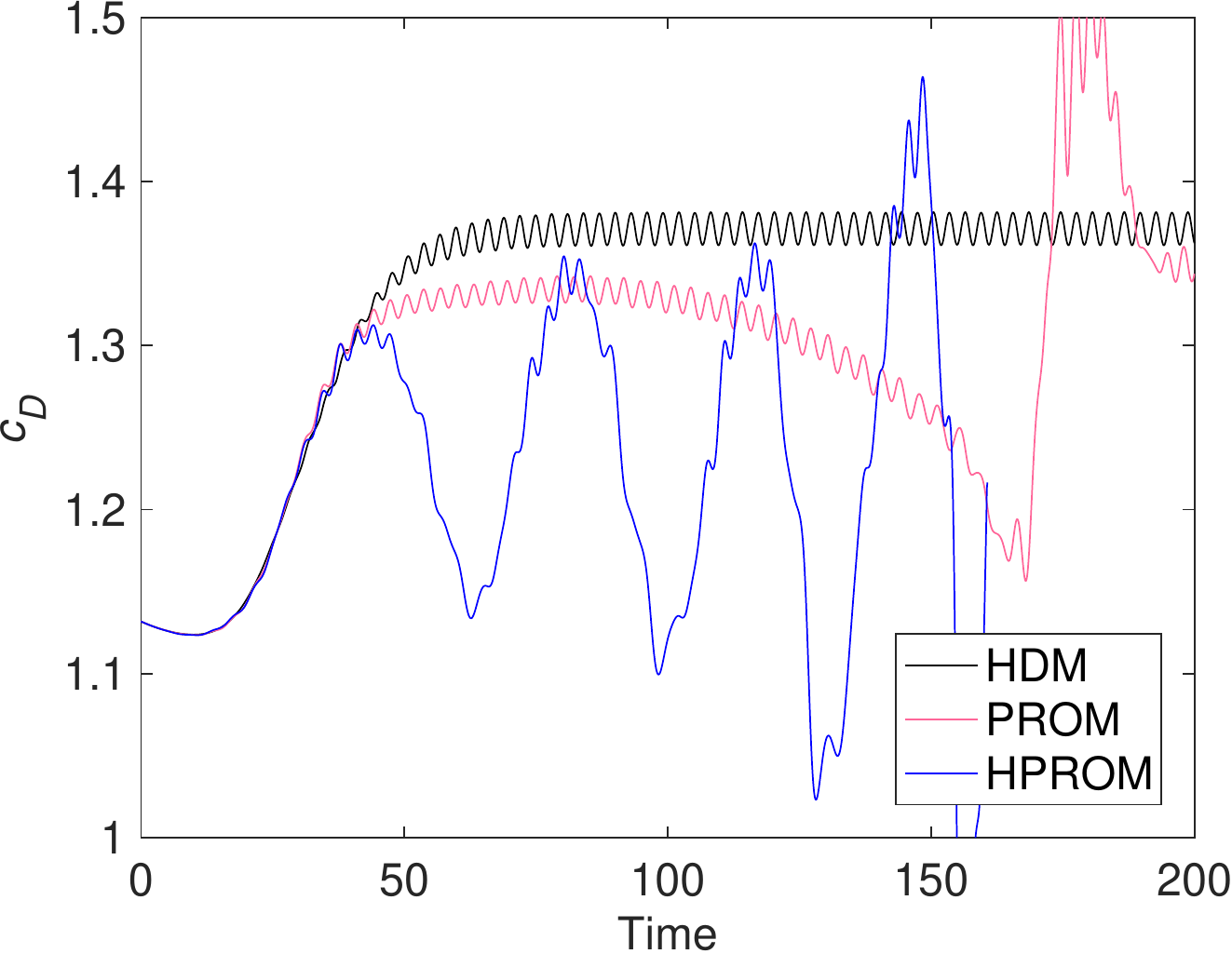}
	\caption{Galerkin, $n = 20$}
	\end{subfigure}%
	\hspace{1em}
	\begin{subfigure}[c]{0.3\textwidth}%
	\centering
	\includegraphics[width=\linewidth]{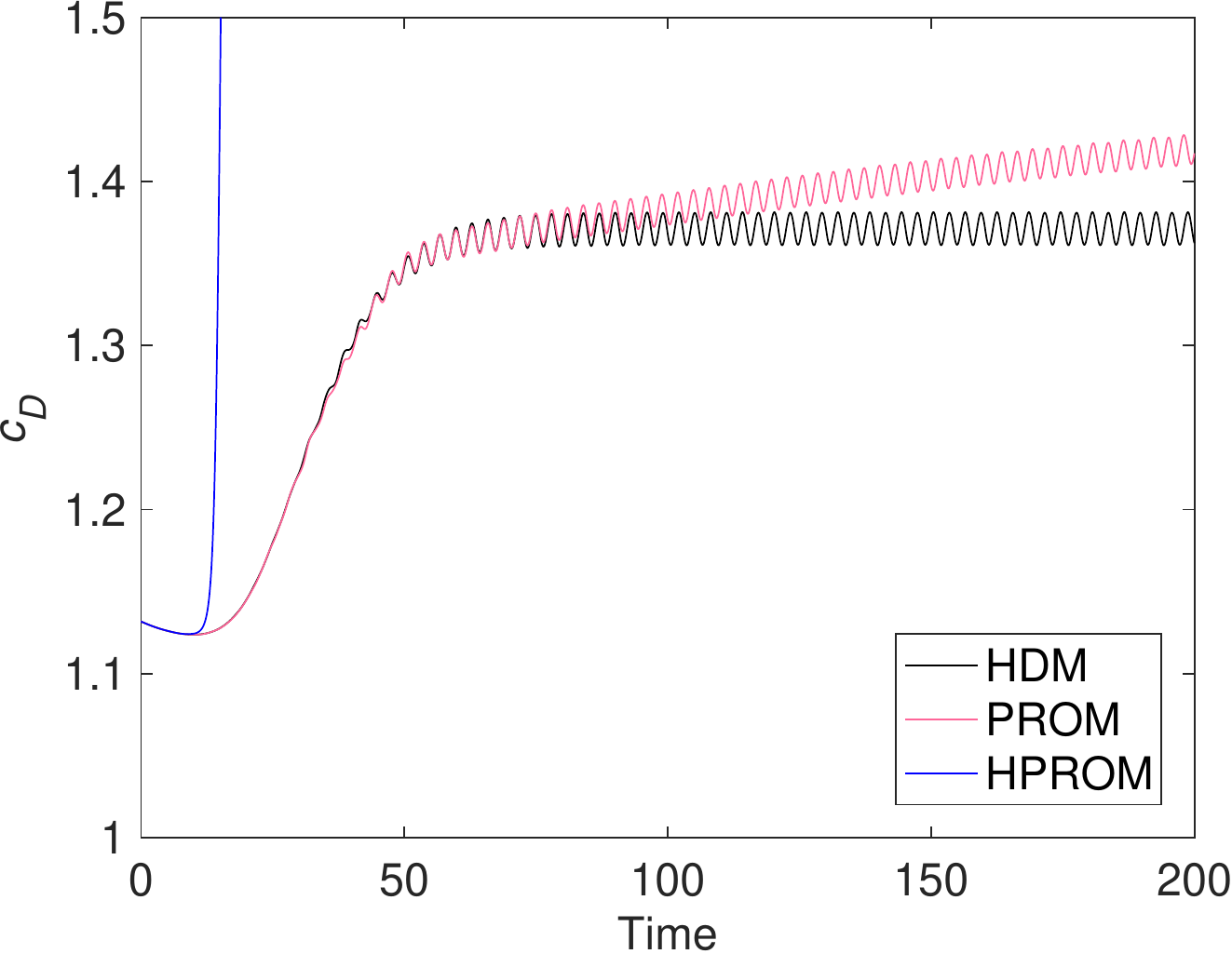}
	\caption{Galerkin, $n = 35$}
	\end{subfigure}%
	\hspace{1em}
	\begin{subfigure}[c]{0.3\textwidth}%
	\centering
	\includegraphics[width=\linewidth]{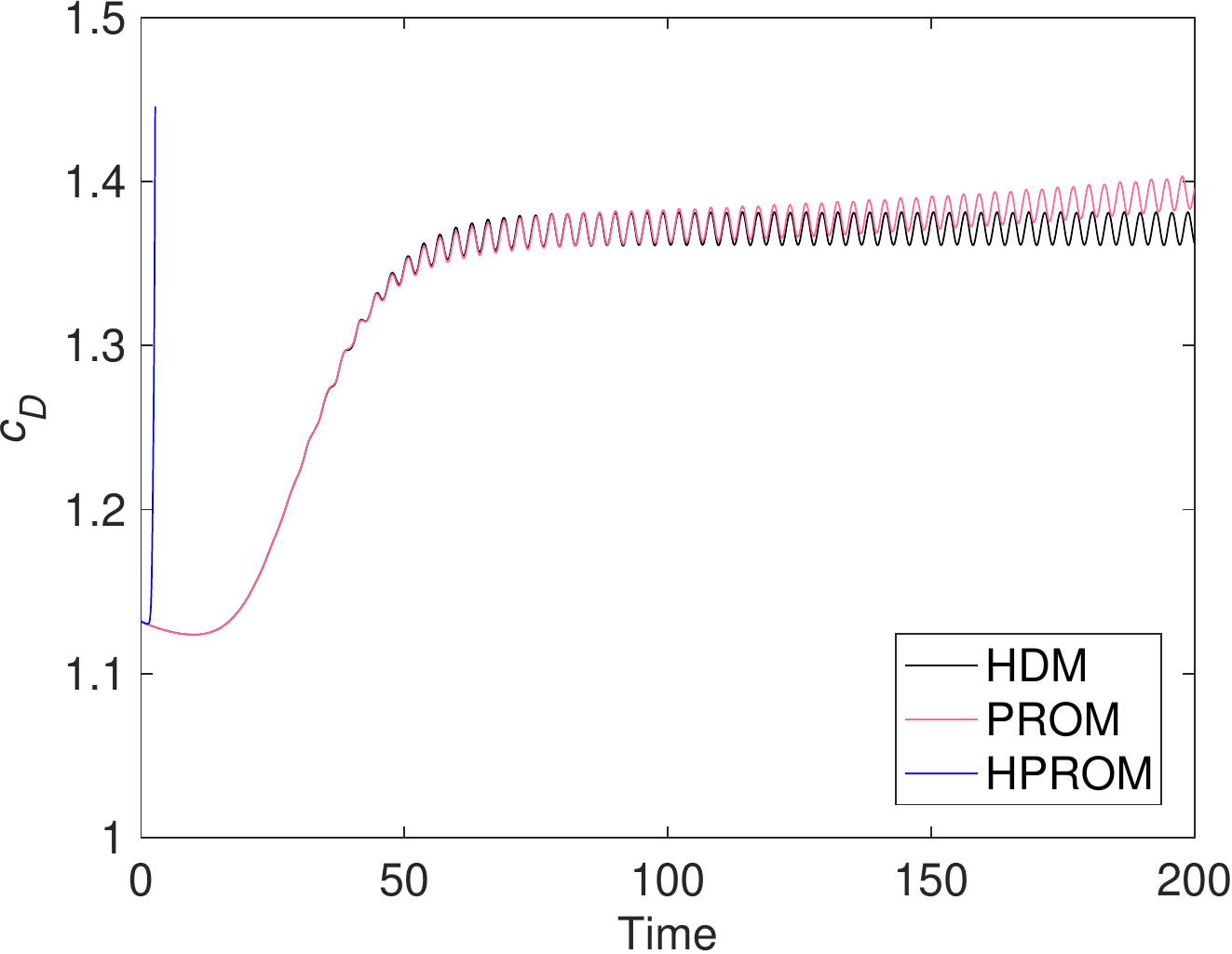}
	\caption{Galerkin, $n = 55$}
	\end{subfigure}%
	\caption{2D flow over a right circular cylinder: Time-histories of the drag coefficient computed using the HDM and Galerkin PROMs and HPROMs.}
	\label{fig:cylinderdraggal}
\end{figure}

\begin{figure}[h!]
	\centering
	\begin{subfigure}[c]{0.3\textwidth}%
	\centering
	\includegraphics[width=\linewidth]{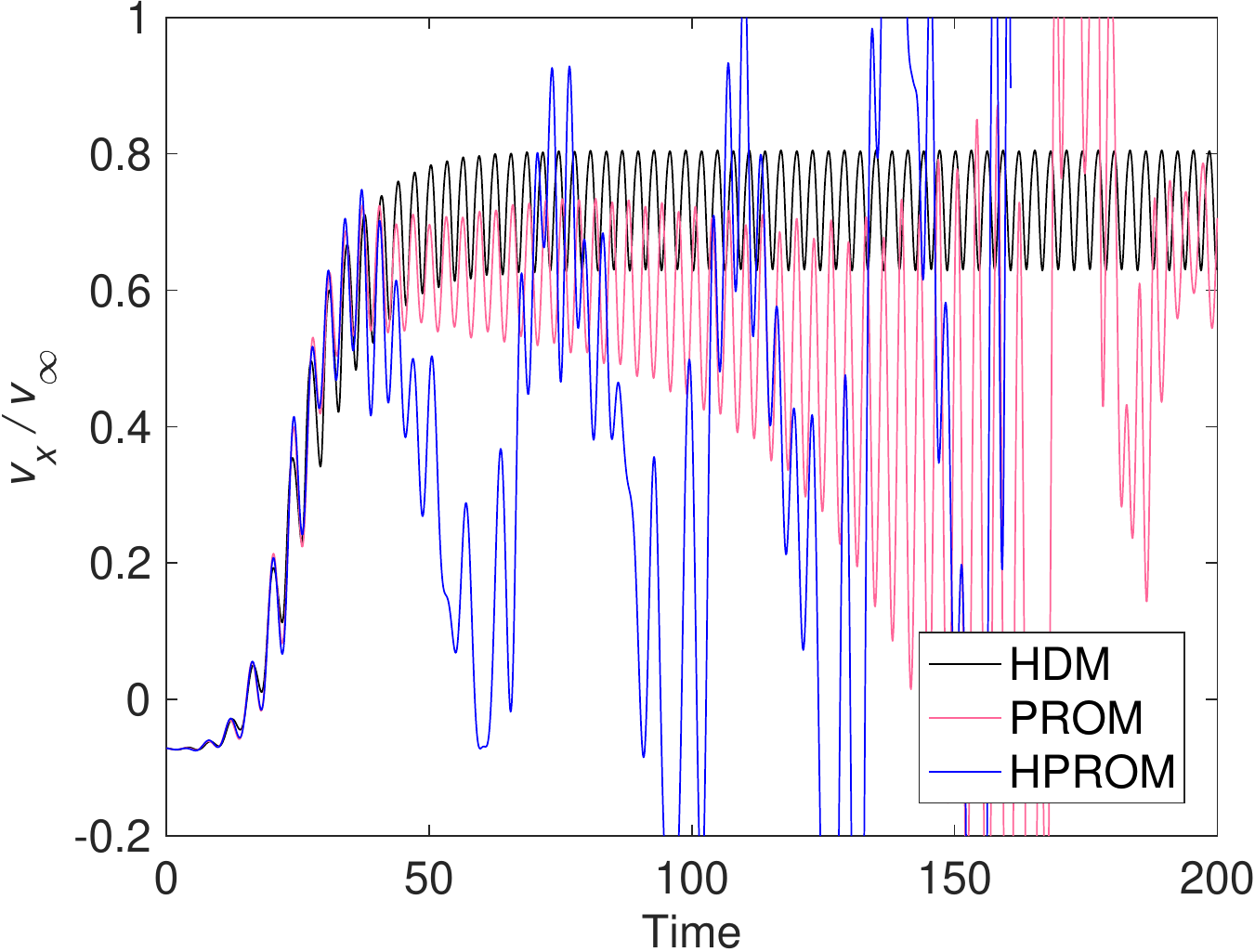}
	\caption{Galerkin, $n = 20$}
	\end{subfigure}%
	\hspace{1em}
	\begin{subfigure}[c]{0.3\textwidth}%
	\centering
	\includegraphics[width=\linewidth]{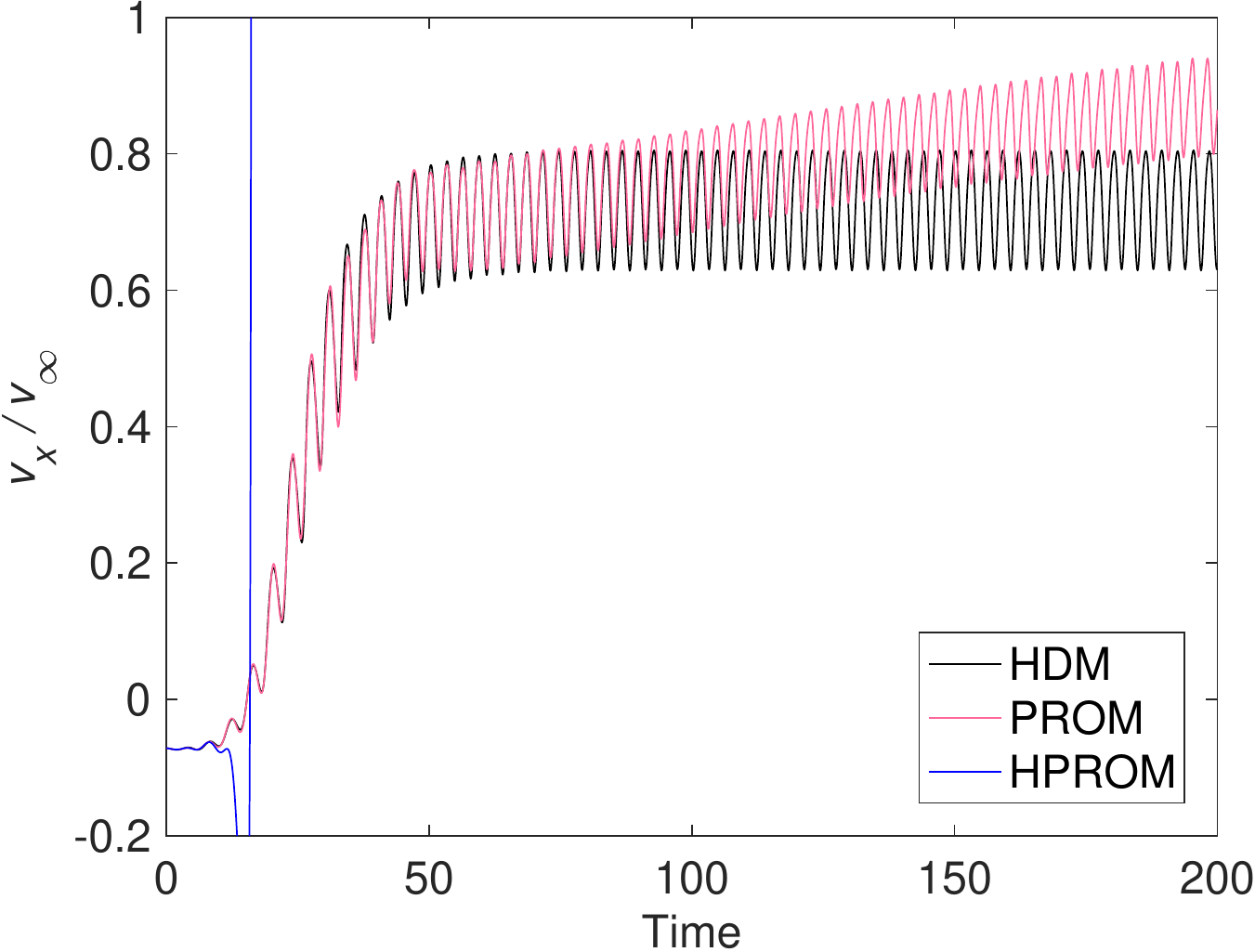}
	\caption{Galerkin, $n = 35$}
	\end{subfigure}%
	\hspace{1em}
	\begin{subfigure}[c]{0.3\textwidth}%
	\centering
	\includegraphics[width=\linewidth]{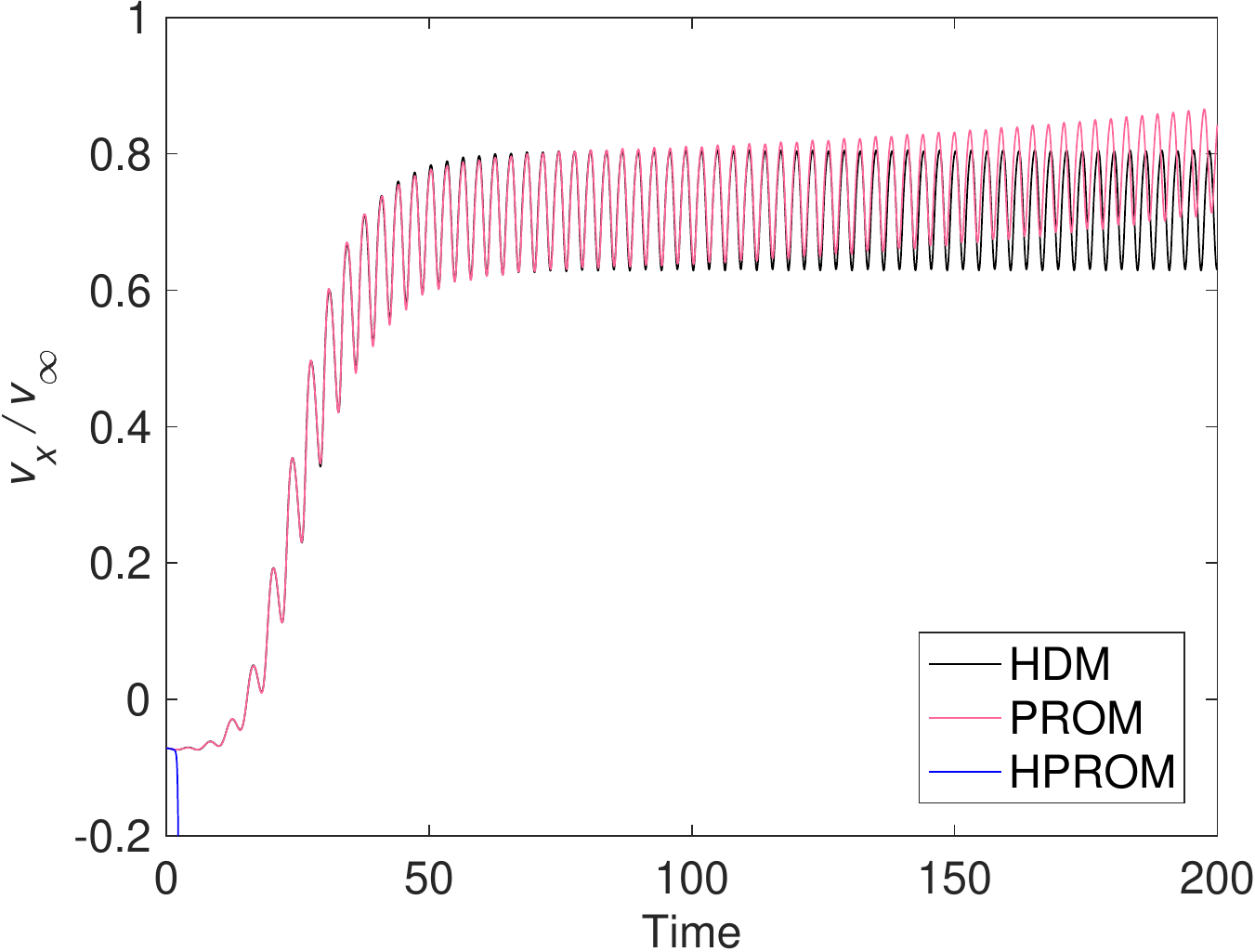}
	\caption{Galerkin, $n = 55$}
	\end{subfigure}%
	\caption{2D flow over a right circular cylinder: Time-histories of the streamwise velocity computed at a probe using the HDM and Galerkin PROMs and HPROMs.}
	\label{fig:cylinderprobevxgal}
\end{figure}

\begin{figure}[h!]
	\centering
	\begin{subfigure}[c]{0.3\textwidth}%
	\centering
	\includegraphics[width=\linewidth]{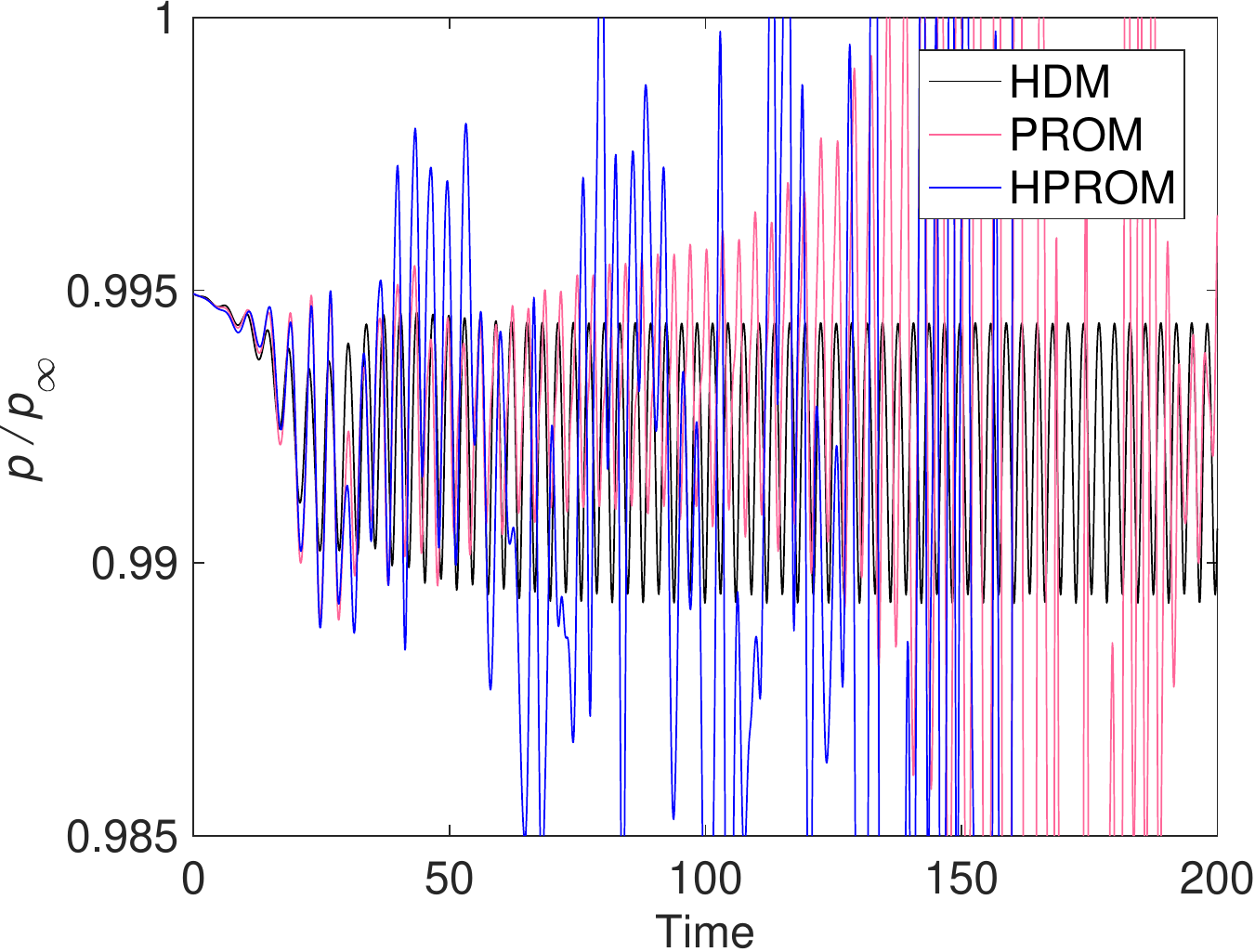}
	\caption{Galerkin, $n = 20$}
	\end{subfigure}%
	\hspace{1em}
	\begin{subfigure}[c]{0.3\textwidth}%
	\centering
	\includegraphics[width=\linewidth]{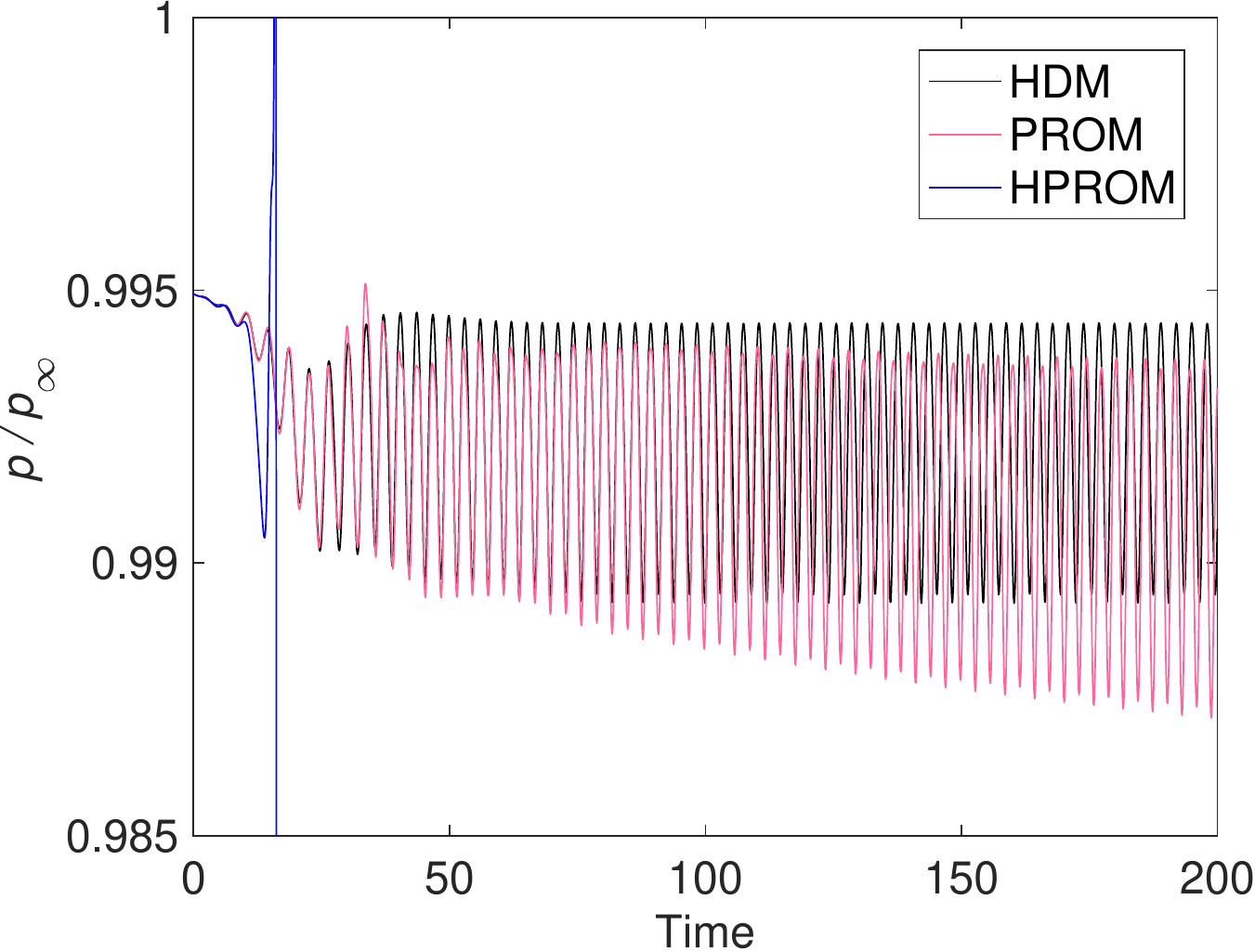}
	\caption{Galerkin, $n = 35$}
	\end{subfigure}%
	\hspace{1em}
	\begin{subfigure}[c]{0.3\textwidth}%
	\centering
	\includegraphics[width=\linewidth]{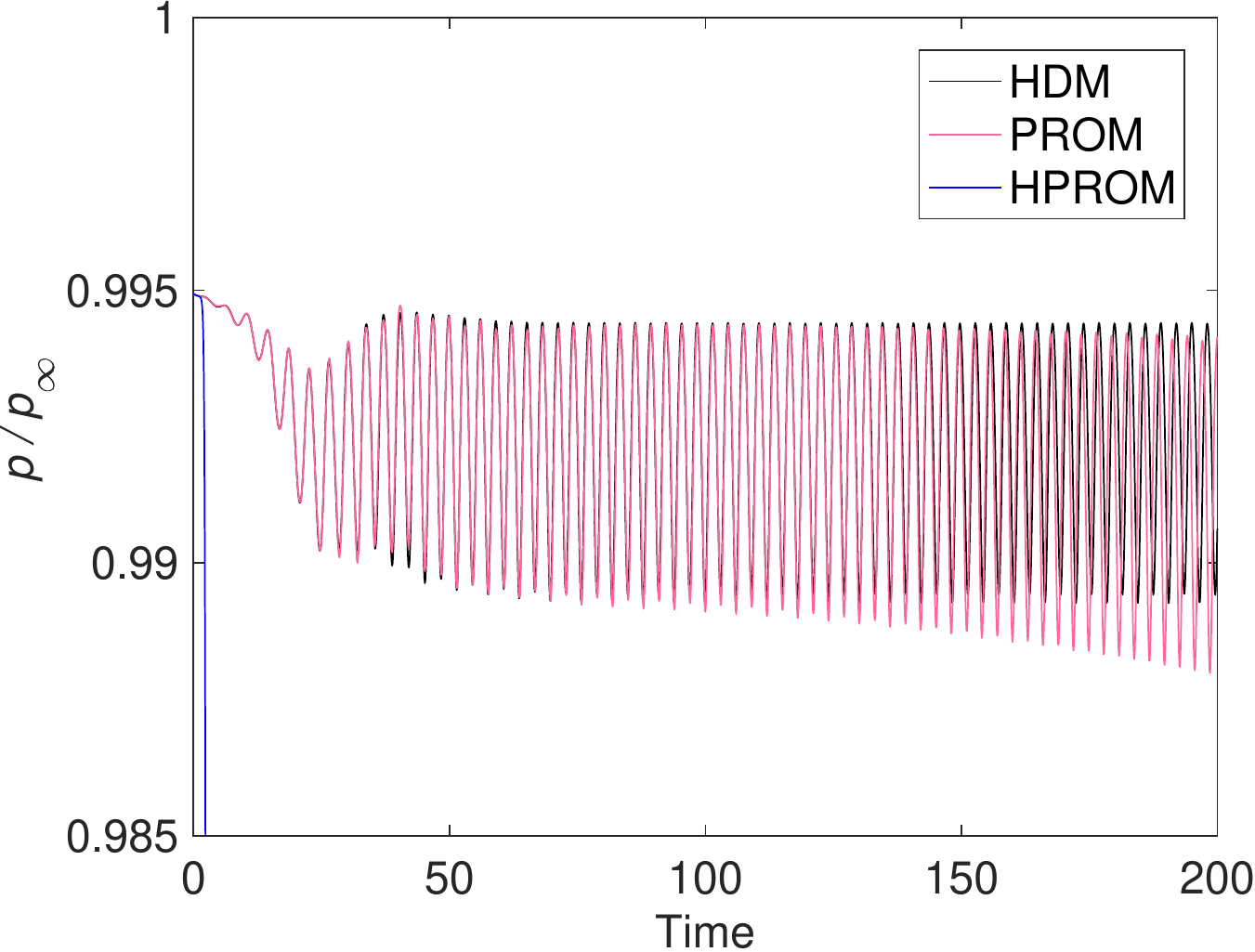}
	\caption{Galerkin, $n = 55$}
	\end{subfigure}%
	\caption{2D flow over a right circular cylinder: Time-histories of the pressure computed at a probe using the HDM and Galerkin PROMs and HPROMs.}
	\label{fig:cylinderprobepgal}
\end{figure}

Furthermore, all Galerkin HPROMs are found to be numerically unstable and to exhibit nonphysical pressures and densities that cause early termination of the simulations.
Interestingly, the numerical instability observed for a Galerkin HPROM is encountered at an earlier time when the dimension $n$ of the subspace approximation is increased.
These results demonstrate that -- at least for the chosen dimensions $n$ of the right ROBs -- the PROMs and HPROMs resulting from Galerkin projection do not deliver satisfactory performance for 
this problem. Given the presented trend and the knowledge of improved performance of Galerkin semi-discretizations under mesh refinement, one could expect that increasing further the dimension of
the right ROB could improve the numerical stability of the resulting PROM as well as its accuracy. However, this is not guaranteed, and if achieved, would sacrifice the desired performance of the
constructed PROM or HPROM.

\paragraph{Petrov-Galerkin reduced-order models}

The time-histories of the lift and drag coefficients computed using the HDM and Petrov-Galerkin PROMs and HPROMs are compared in Figures \ref{fig:cylinderliftlspg} and \ref{fig:cylinderdraglspg} for 
each chosen dimension $n$ of the right ROB. Figures \ref{fig:cylinderprobevxlspg} and \ref{fig:cylinderprobeplspg} compare the time-histories of the streamwise velocity and pressure computed at the 
selected probe location for the Petrov-Galerkin PROMs and HPROMs. The reader can observe that unlike the Galerkin PROMs and HPROMs, the Petrov-Galerkin counterparts do not exhibit a drift away from the 
HDM-based reference solution. For $n=35$, the observed mean value of the oscillations predicted for the QoIs $c_D$ and $v_x$ using the PROM is slightly larger than its counterpart predicted by
the HDM: However, it remains stationary after $t = 150$. In fact, the Petrov-Galerkin PROMs are found to be numerically stable at all considered values of $n$ and to deliver more than satisfactory 
accuracy for all QoIs. Hyperreduction is observed to impact neither the numerical stability nor the accuracy of the constructed PROMs: surprisingly however, the HPROMs constructed for this problem 
are occasionally found to improve the accuracy of their underlying PROMs.

\begin{figure}[h!]
	\centering
	\begin{subfigure}[c]{0.3\textwidth}%
	\centering
	\includegraphics[width=\linewidth]{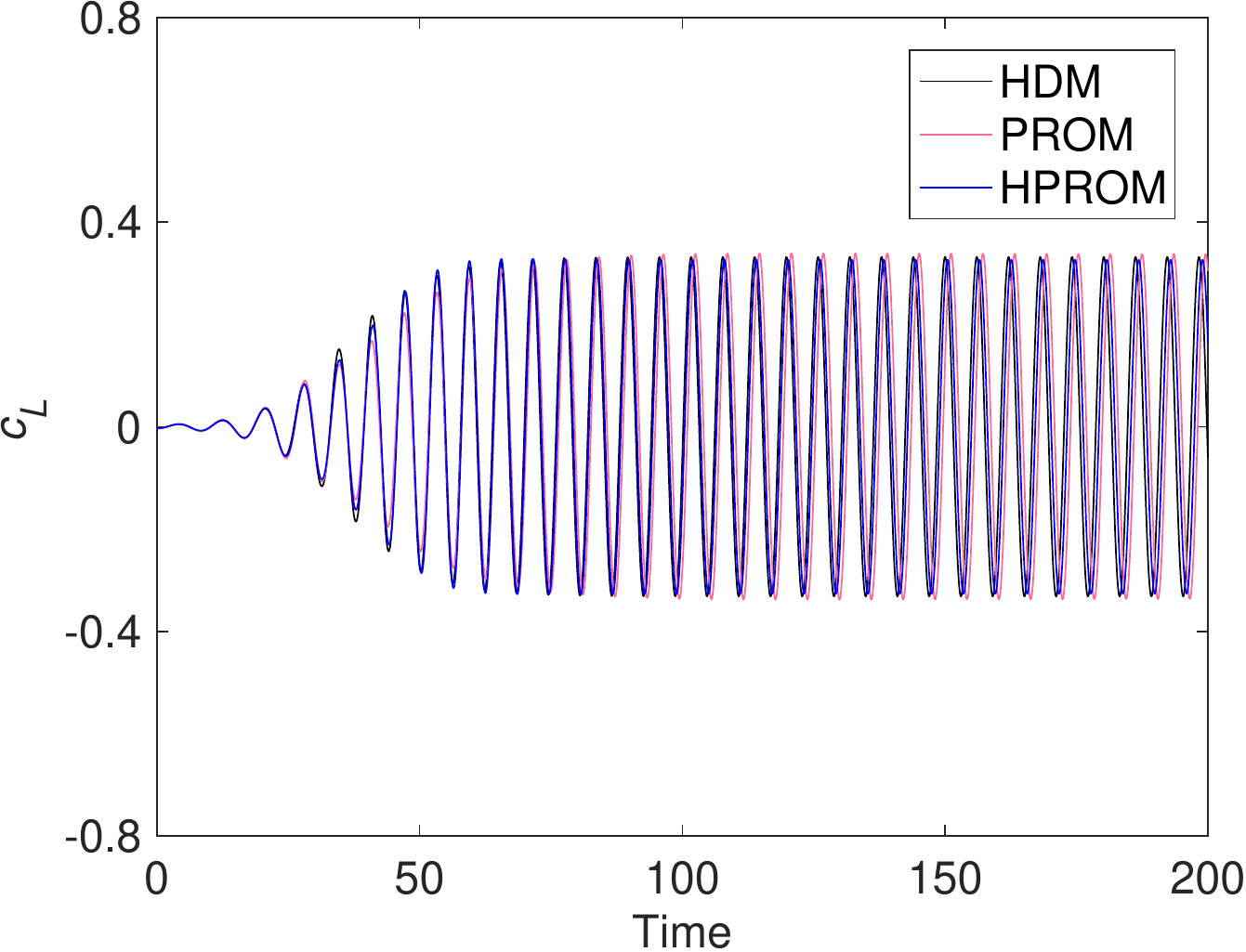}
	\caption{LSPG, $n = 20$}
	\end{subfigure}%
	\hspace{1em}
	\begin{subfigure}[c]{0.3\textwidth}%
	\centering
	\includegraphics[width=\linewidth]{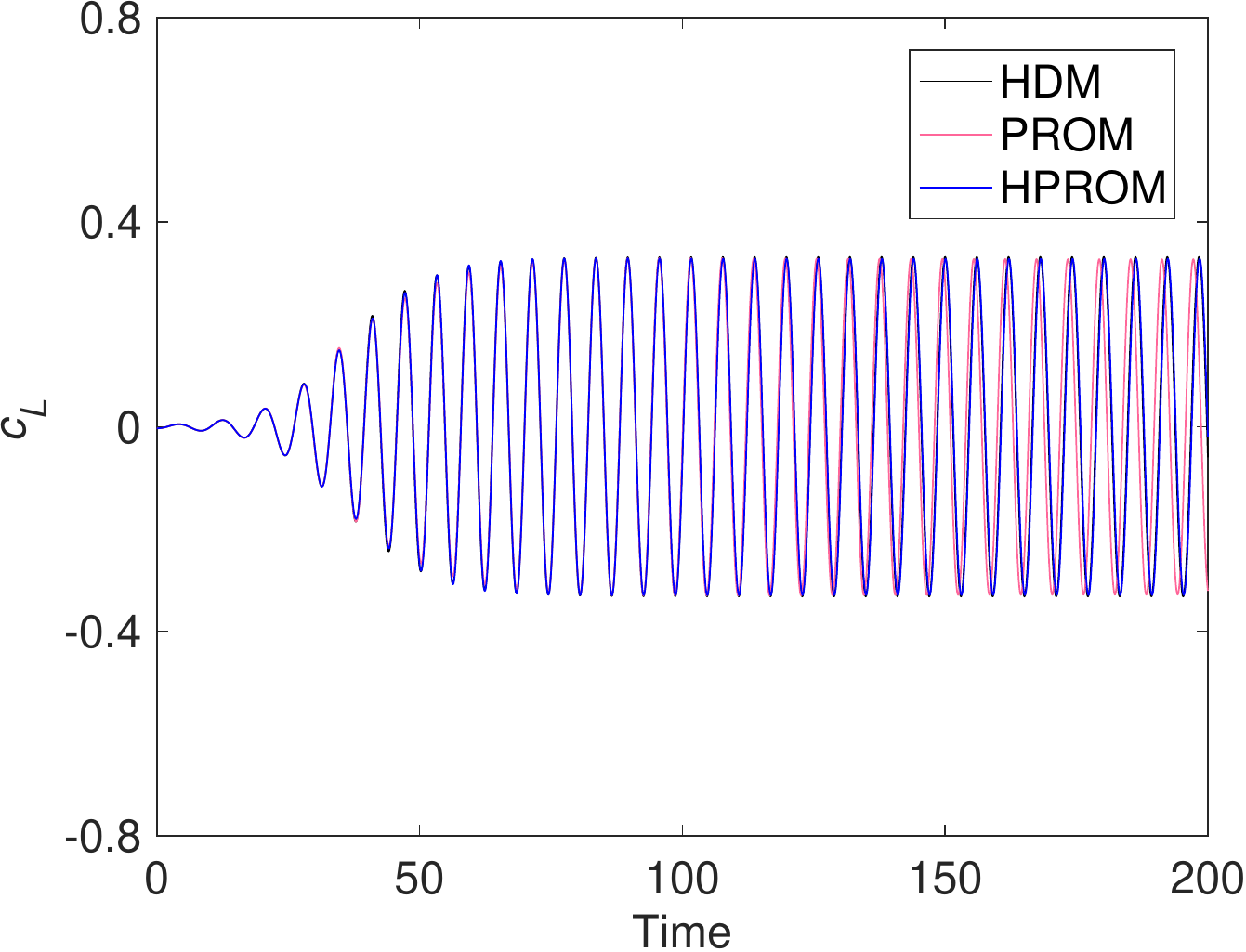}
	\caption{LSPG, $n = 35$}
	\end{subfigure}%
	\hspace{1em}
	\begin{subfigure}[c]{0.3\textwidth}%
	\centering
	\includegraphics[width=\linewidth]{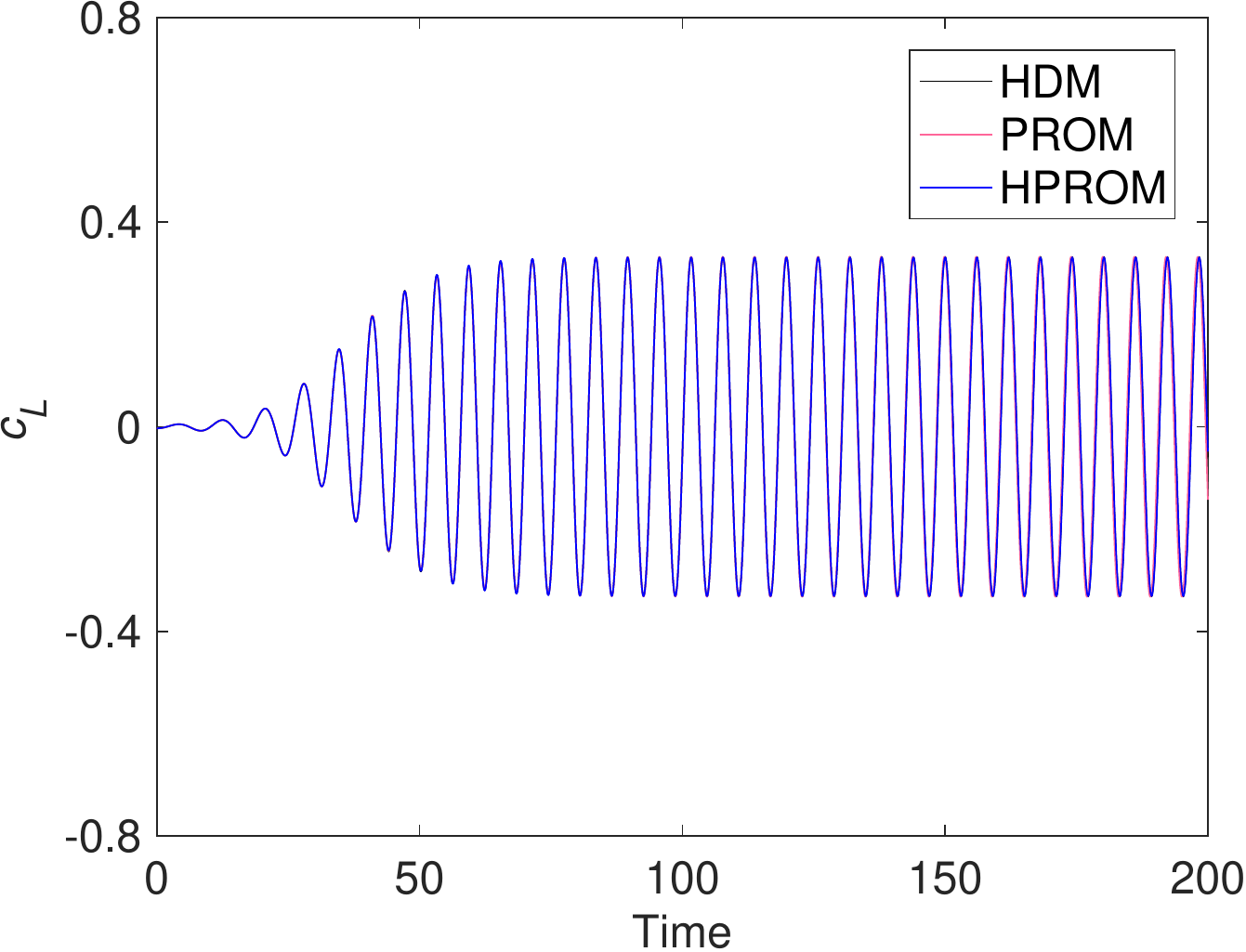}
	\caption{LSPG, $n = 55$}
	\end{subfigure}%
	\caption{2D flow over a right circular cylinder: Time-histories of the lift coefficient computed using the HDM and LSPG-based Petrov-Galerkin PROMs and HPROMs.}
	\label{fig:cylinderliftlspg}
\end{figure}

\begin{figure}[h!]
	\centering
	\begin{subfigure}[c]{0.3\textwidth}%
	\centering
	\includegraphics[width=\linewidth]{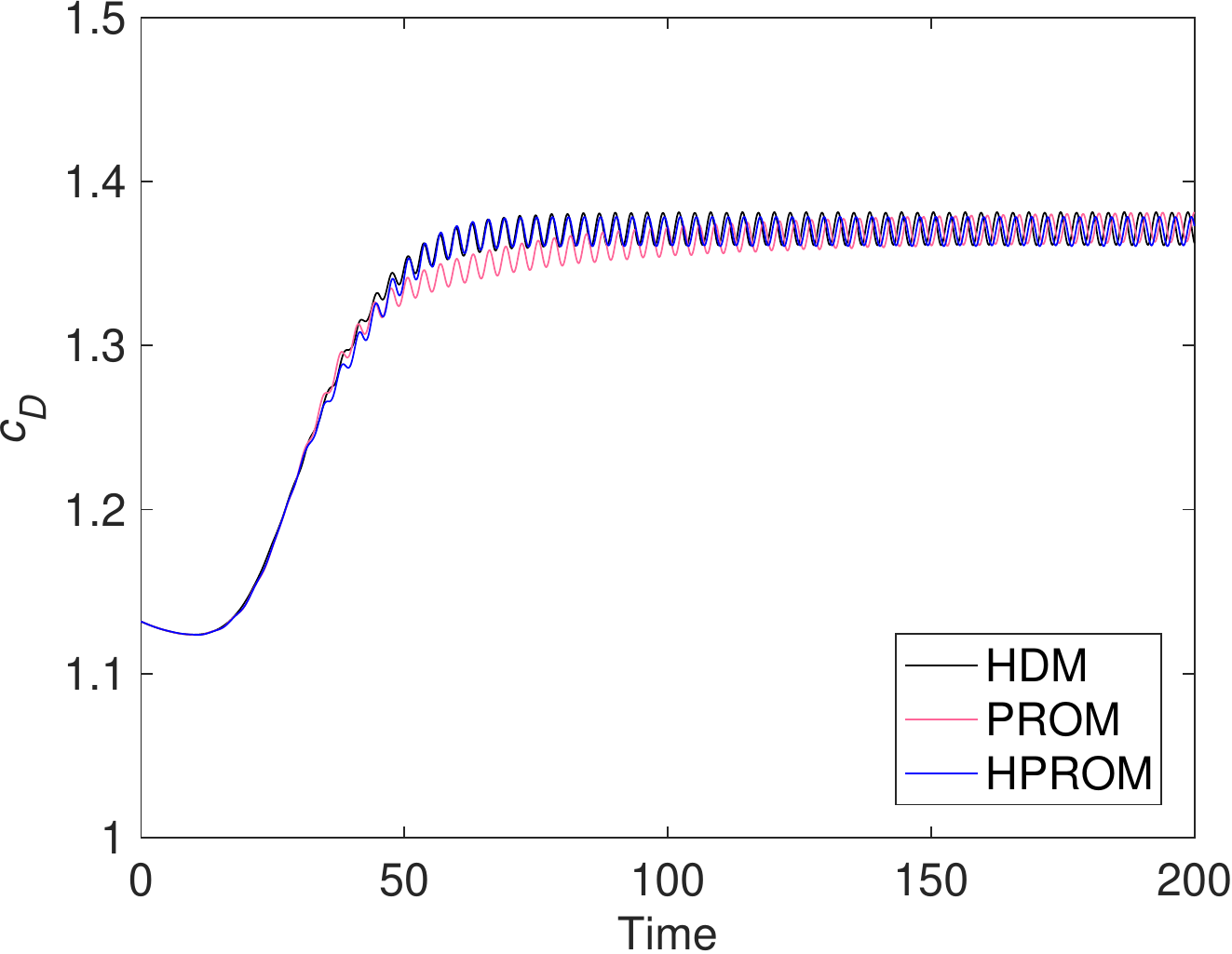}
	\caption{LSPG, $n = 20$}
	\end{subfigure}%
	\hspace{1em}
	\begin{subfigure}[c]{0.3\textwidth}%
	\centering
	\includegraphics[width=\linewidth]{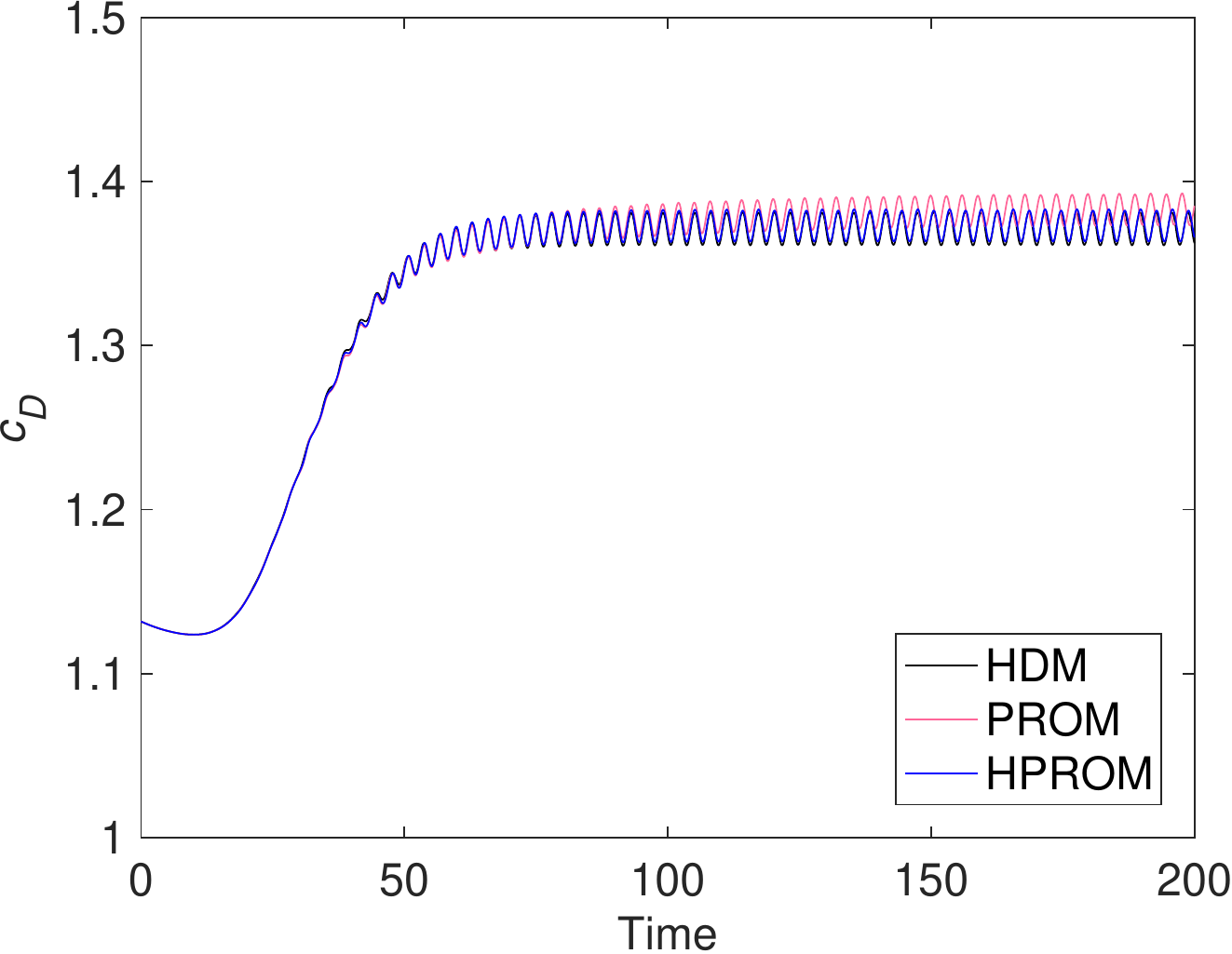}
	\caption{LSPG, $n = 35$}
	\end{subfigure}%
	\hspace{1em}
	\begin{subfigure}[c]{0.3\textwidth}%
	\centering
	\includegraphics[width=\linewidth]{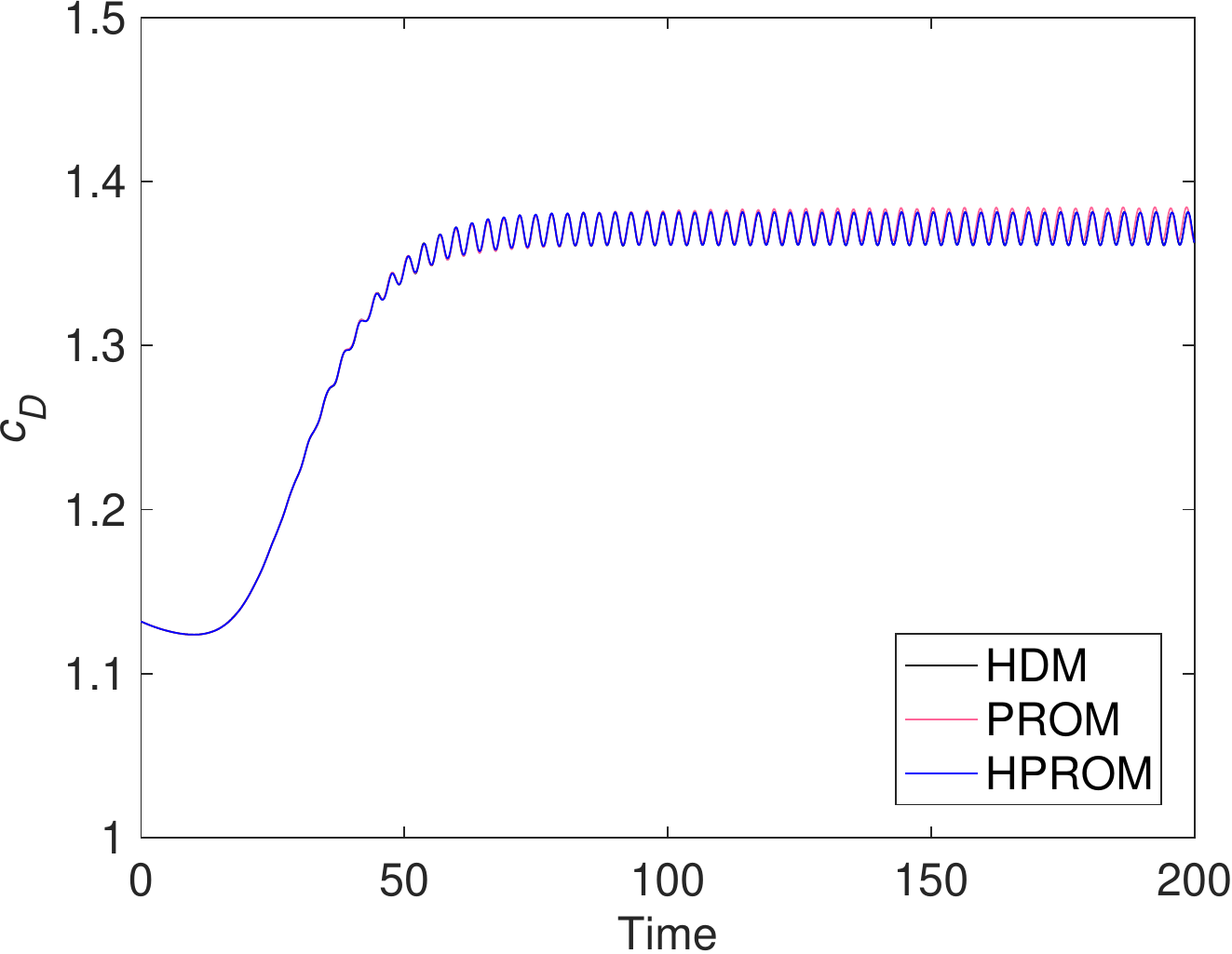}
	\caption{LSPG, $n = 55$}
	\end{subfigure}%
	\caption{2D flow over a right circular cylinder: Time-histories of the drag coefficient computed using the HDM and LSPG-based Petrov-Galerkin PROMs and HPROMs.}
	\label{fig:cylinderdraglspg}
\end{figure}

\begin{figure}[h!]
	\centering
	\begin{subfigure}[c]{0.3\textwidth}%
	\centering
	\includegraphics[width=\linewidth]{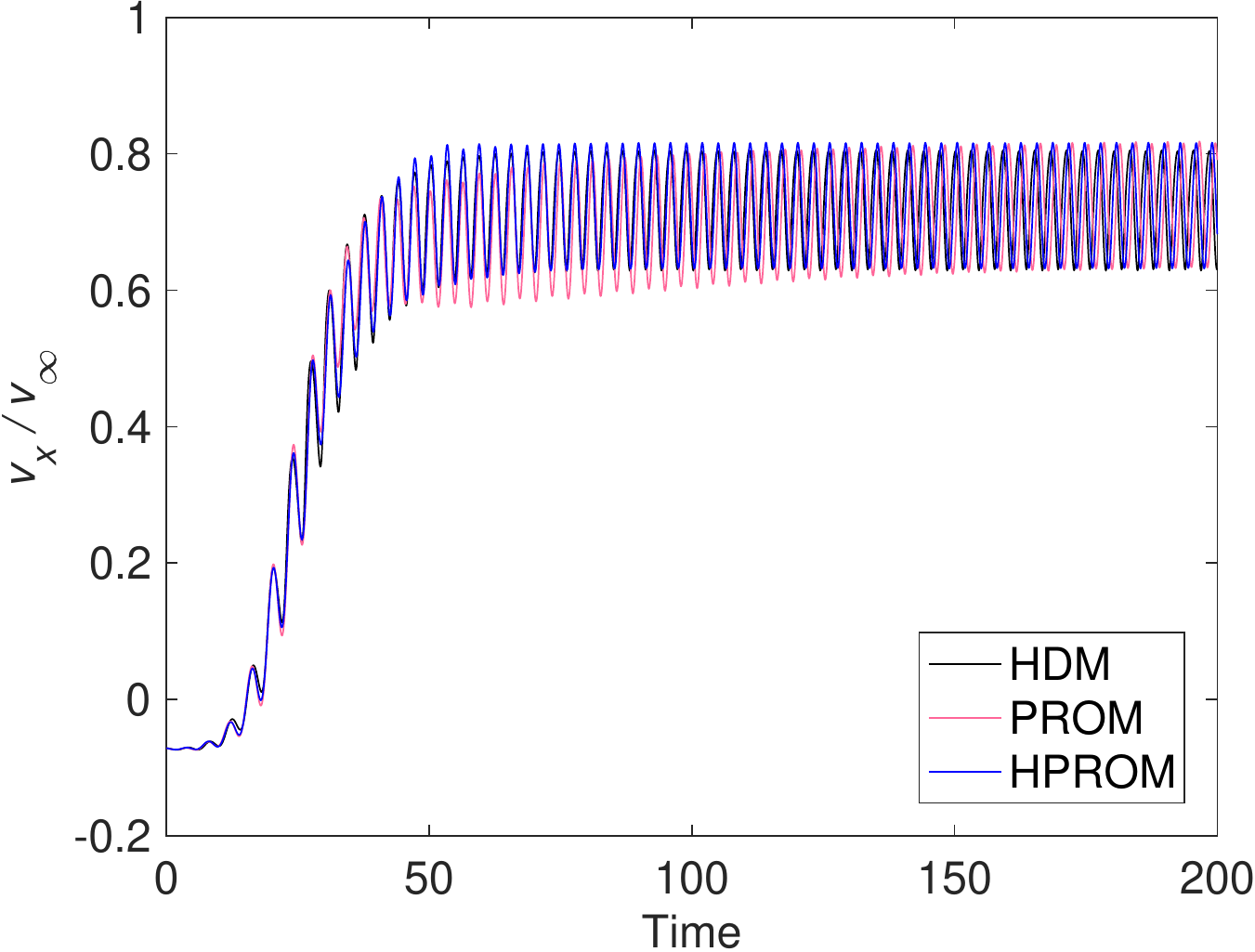}
	\caption{LSPG, $n = 20$}
	\end{subfigure}%
	\hspace{1em}
	\begin{subfigure}[c]{0.3\textwidth}%
	\centering
	\includegraphics[width=\linewidth]{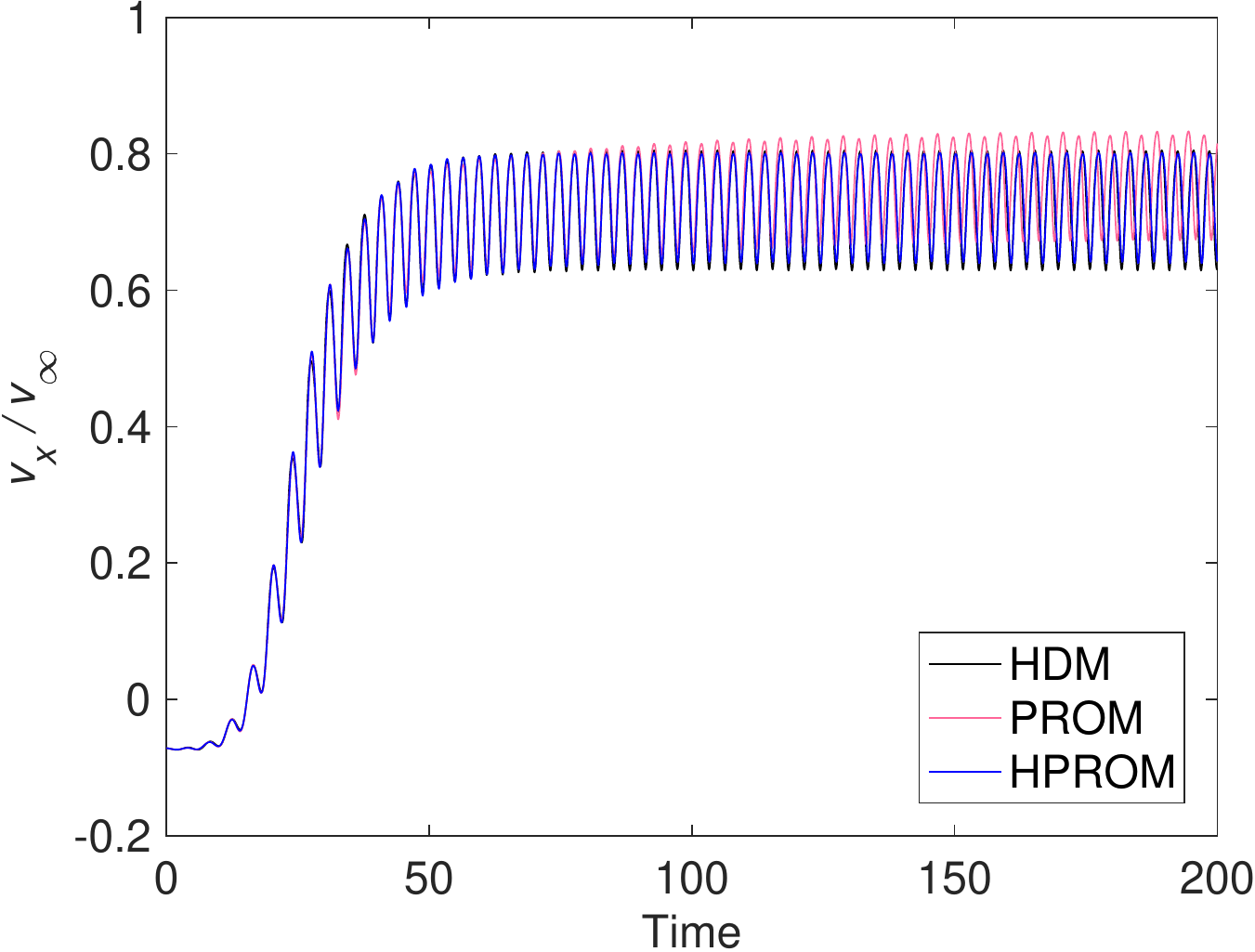}
	\caption{LSPG, $n = 35$}
	\end{subfigure}%
	\hspace{1em}
	\begin{subfigure}[c]{0.3\textwidth}%
	\centering
	\includegraphics[width=\linewidth]{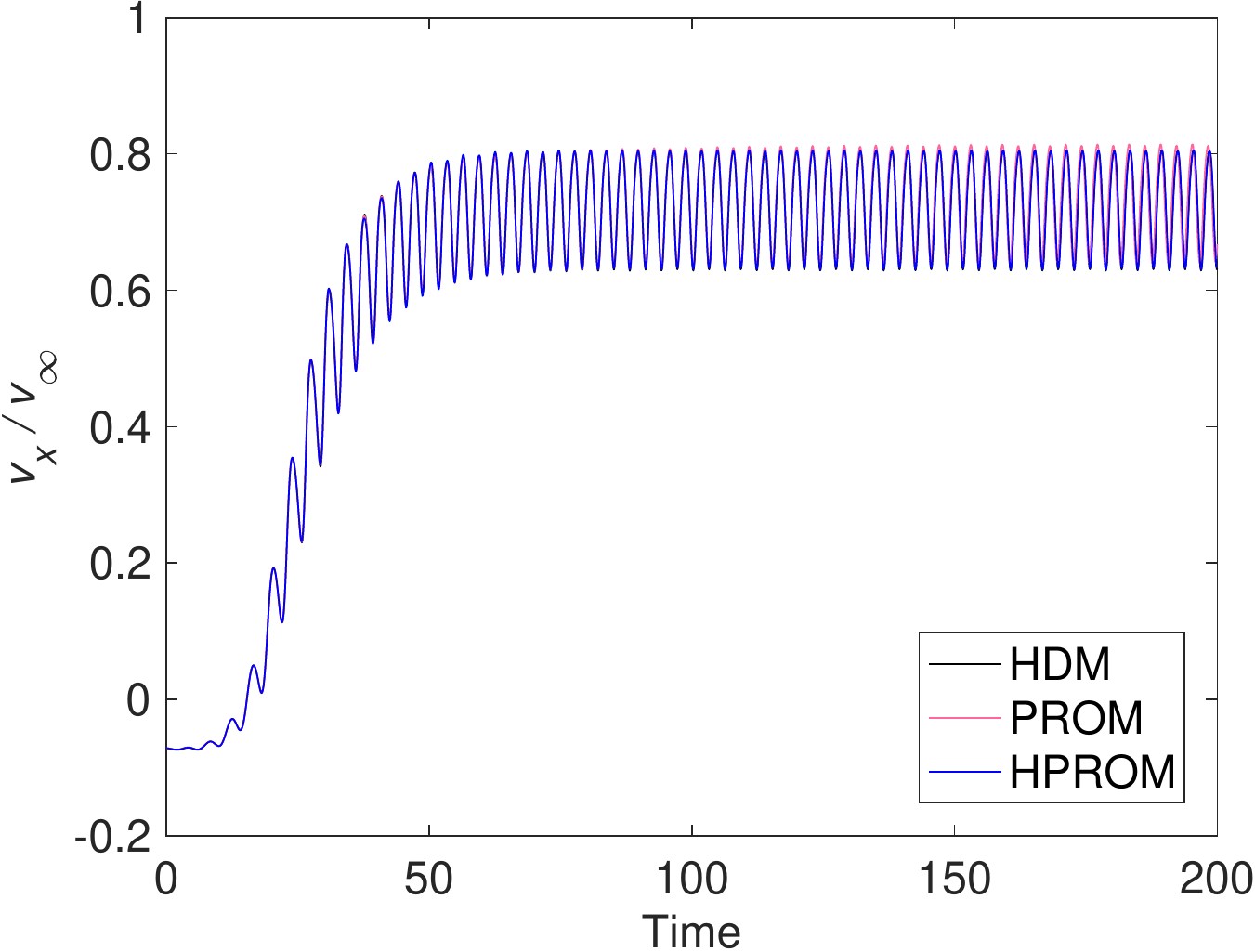}
	\caption{LSPG, $n = 55$}
	\end{subfigure}%
	\caption{2D flow over a right circular cylinder: Time-histories of the streamwise velocity computed at a probe using the HDM and LSPG-based Petrov-Galerkin PROMs and HPROMs.}
	\label{fig:cylinderprobevxlspg}
\end{figure}

\begin{figure}[h!]
	\centering
	\begin{subfigure}[c]{0.3\textwidth}%
	\centering
	\includegraphics[width=\linewidth]{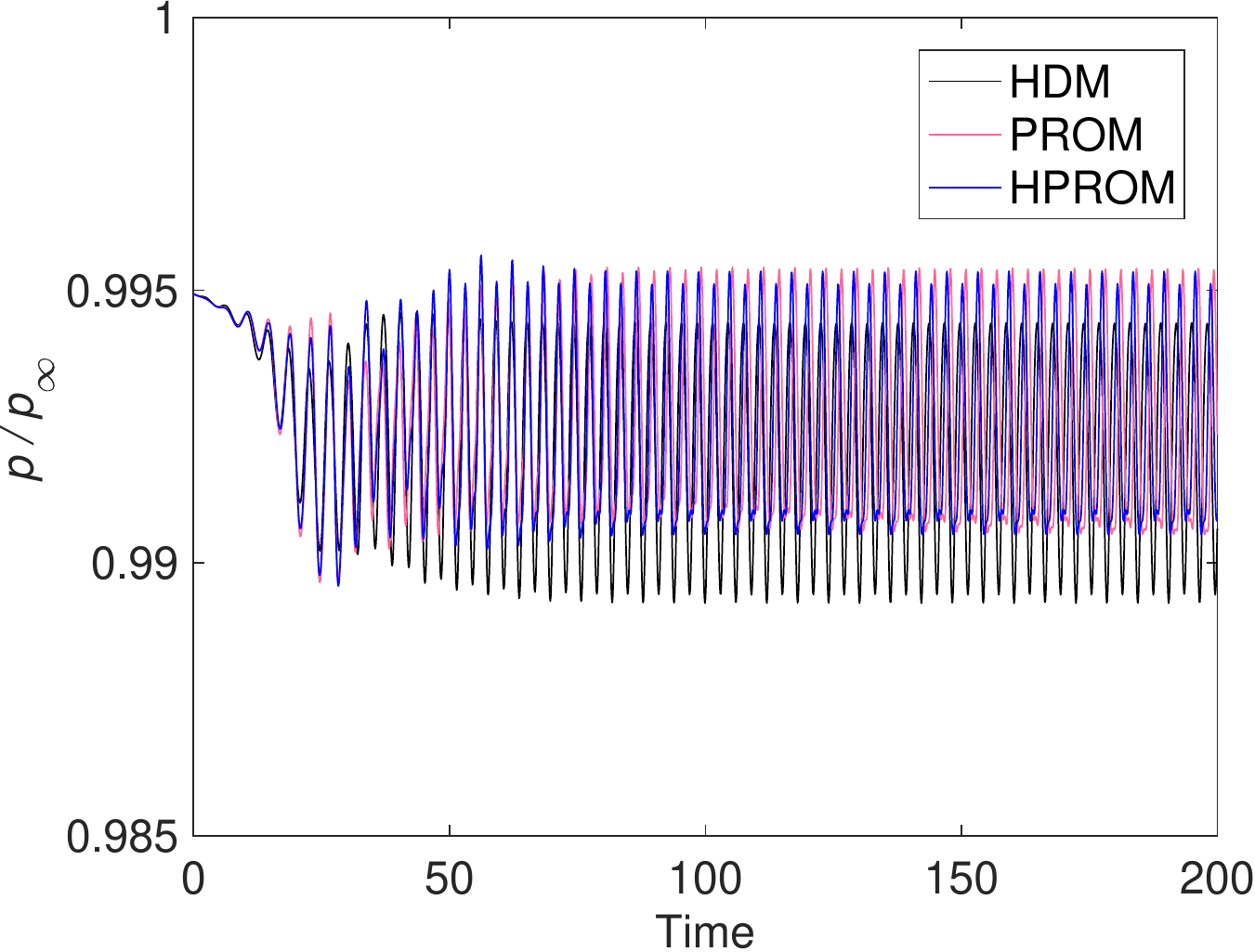}
	\caption{LSPG, $n = 20$}
	\end{subfigure}%
	\hspace{1em}
	\begin{subfigure}[c]{0.3\textwidth}%
	\centering
	\includegraphics[width=\linewidth]{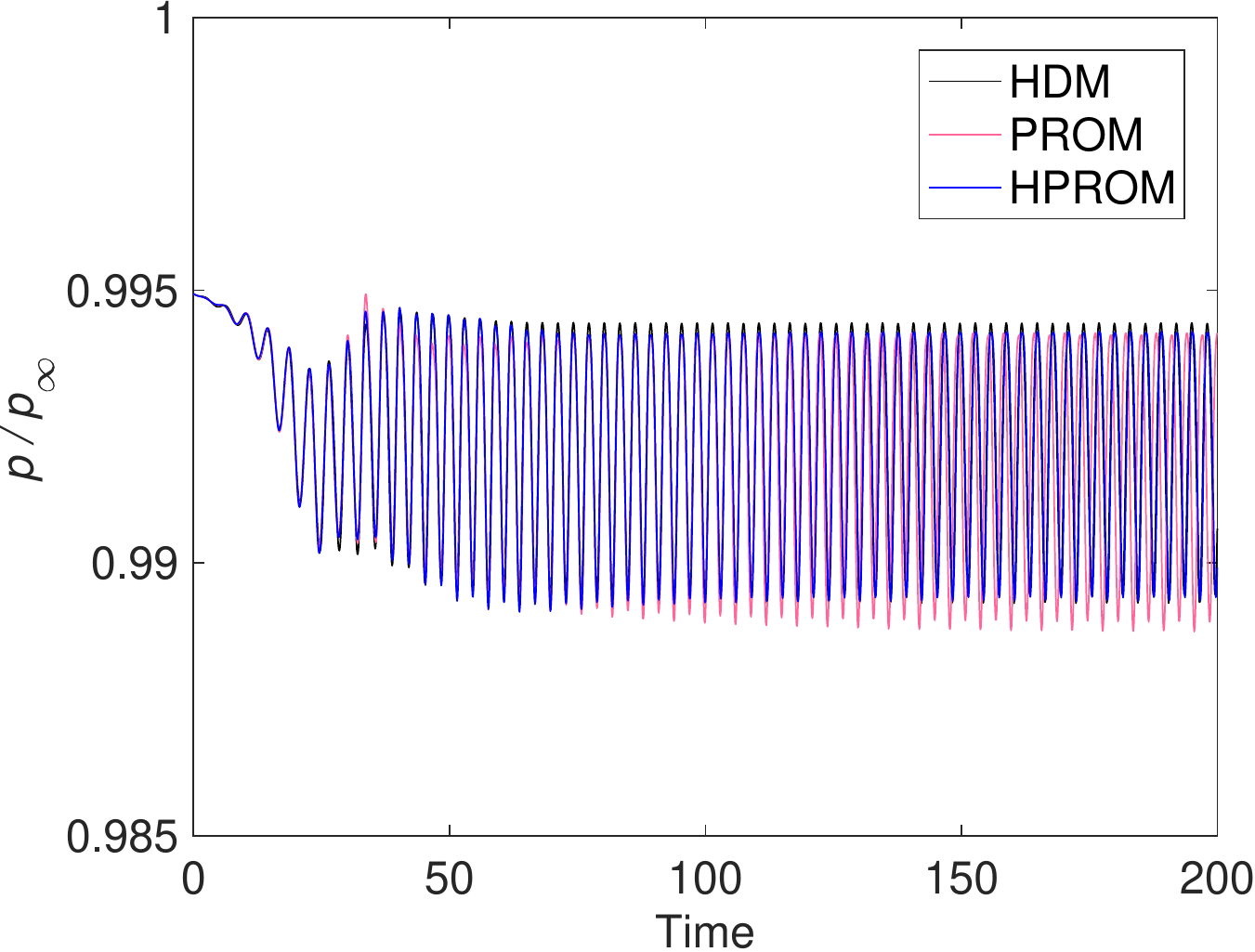}
	\caption{LSPG, $n = 35$}
	\end{subfigure}%
	\hspace{1em}
	\begin{subfigure}[c]{0.3\textwidth}%
	\centering
	\includegraphics[width=\linewidth]{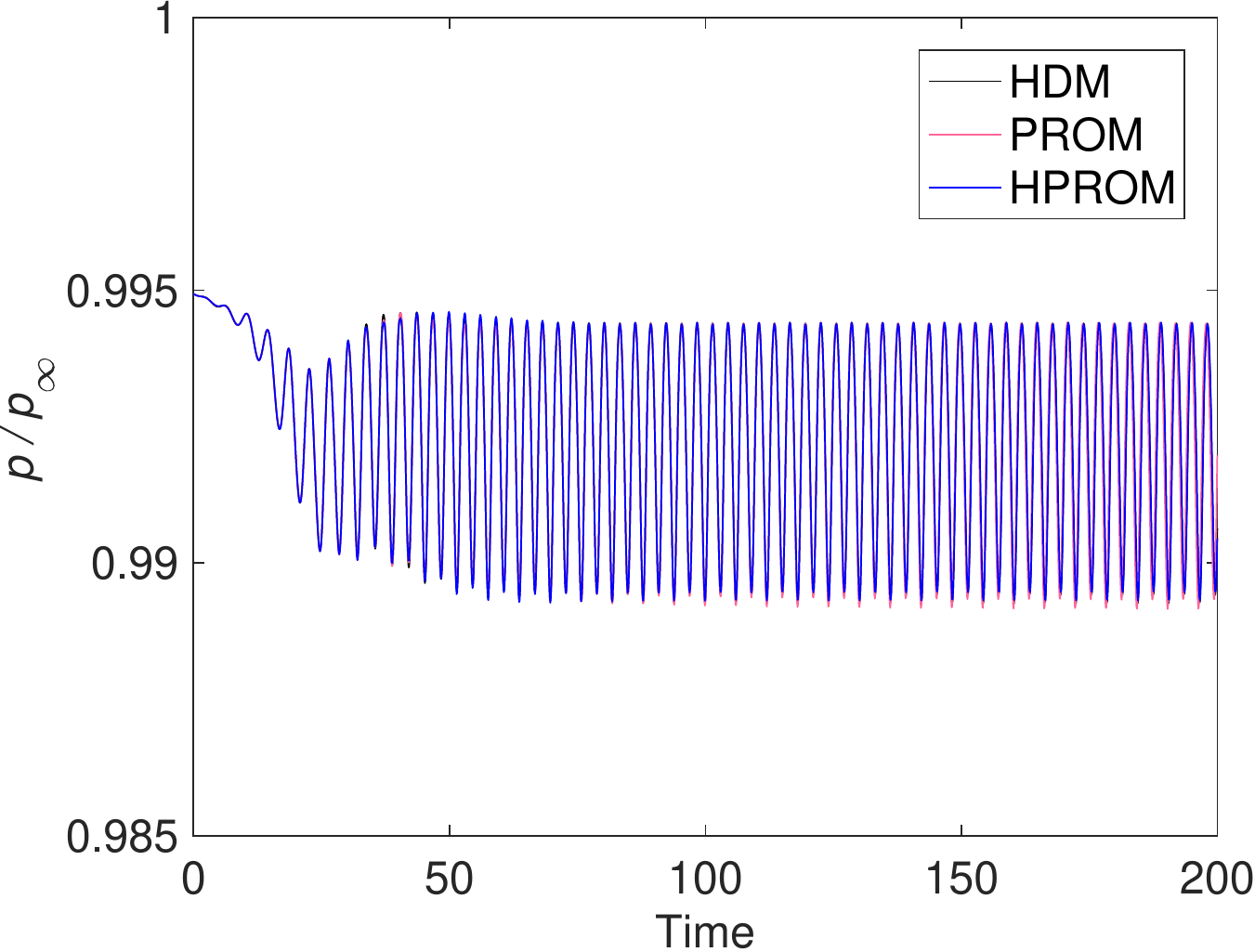}
	\caption{LSPG, $n = 55$}
	\end{subfigure}%
	\caption{2D flow over a right circular cylinder: Time-histories of the pressure computed at a probe using the HDM and LSPG-based Petrov-Galerkin PROMs and HPROMs.}
	\label{fig:cylinderprobeplspg}
\end{figure}

Finally, Table \ref{tab:cylindererr} reports for this problem on the relative errors incurred using each computational model, and Table \ref{tab:cylinderspeed} reports for each Galerkin and 
Petrov-Galerkin HPROM the size of the sample cell set computed by ECSW as well as the wall-clock time speed-up factor delivered by the HPROM. The errors and speed-up factors for the Galerkin HPROMs are 
not reported as these reduced-order models fail to complete the simulation due to nonlinear numerical instabilities encountered before reaching $t_{max}$. From these results, one can conclude 
that the LSPG-based Petrov-Galerkin PMOR approach is more appropriate for constructing PROMs for the class of flow problems represented by this example than its Galerkin counterpart, and offers improved 
numerical stability and accuracy properties. Furthermore, the Petrov-Galerkin HPROMs constructed using ECSW are found to deliver exceptional speed-up factors. Lastly, the results reported herein 
demonstrate that as long as the solution is periodic, a PROM or HPROM that is well-trained in a given time-interval can predict (as expected) the HDM-based solution outside of this time-interval with high accuracy.

\begin{table}[h!]
	\small
	\centering
	\caption{2D flow over a right circular cylinder: Variations with the dimension $n$ of the relative errors for the Galerkin and LSPG-based Petrov-Galerkin PROMs and HPROMs.}
	\begin{tabular}{lccccc} 
		\toprule
		Model & $n$ & $\mathbb{RE}_{c_D}$ ($\%$) & $\mathbb{RE}_{c_L}$ ($\%$) & $\mathbb{RE}_{v_x}$ ($\%$) & $\mathbb{RE}_{p}$ ($\%$) \\ \midrule
		Galerkin PROM & $20$ & $5.40$ & $156$ & $53.4$ & $0.836$ \\
		& $35$ & $1.85$ & $110.$ & $13.7$ & $0.206$ \\
		& $55$ & $0.711$ & $39.7$ & $6.52$ & $0.119$ \\ \midrule
		LSPG-based Petrov-Galerkin PROM & $20$ & $0.907$ & $92.1$ & $13.1$ & $0.244$ \\
		& $35$ & $0.685$ & $50.6$ & $7.55$ & $0.148$ \\
		& $55$ & $0.164$ & $11.4$ & $1.94$ & $0.034$ \\ \midrule
		LSPG-based Petrov-Galerkin HPROM & $20$ & $0.415$ & $41.1$ & $5.97$ & $0.130$ \\
		& $35$ & $0.130$ & $6.94$ & $1.65$ & $0.030$ \\
		& $55$ & $0.022$ & $1.53$ & $0.296$ & $0.007$ \\ \bottomrule
	\end{tabular}
	\label{tab:cylindererr}
\end{table}

\begin{table}[h!]
	\small
	\centering
	\caption{2D flow over a right circular cylinder: Variations with the dimension $n$ of the number of mesh cells sampled by ECSW and of the speed-up factor delivered by the 
	ECSW-based HPROM.}
	\begin{tabular}{lcccc}
		\toprule
		\multirow{2}{*}{Model} & \multirow{2}{*}{$n$} & \# of sampled & Fraction of HDM & Wall-clock time \\
		& & mesh cells & mesh cells ($\%$) & speed-up factor \\ \midrule
		Galerkin HPROM & $20$ & $149$ & $0.152$ & -- \\
		& $35$ & $278$ & $0.283$ & -- \\
		& $55$ & $423$ & $0.431$ & -- \\ \midrule
		LSPG-based Petrov-Galerkin HPROM & $20$ & $199$ & $0.203$ & $2.72 \times 10^3$ \\
		& $35$ & $321$ & $0.327$ & $1.41 \times 10^3$ \\
		& $55$ & $481$ & $0.490$ & $8.81 \times 10^2$ \\ \bottomrule
	\end{tabular}
	\label{tab:cylinderspeed}
\end{table}

\subsection{Direct numerical simulation of the Taylor-Green vortex}
\label{sec:tgv}

The Taylor-Green vortex is a canonical flow that is often used as a benchmark problem to test the ability of numerical methods to accurately compute in 3D the development and decay of 
incompressible homogeneous isotropic turbulence \cite{debonis2013, wang2013}. The problem geometry is chosen here to be a triply-periodic cube of side length $2\pi L$, and the 
initial velocity field is set to
\begin{align*}
v_x(x,y,z,t=0) &= V_0 \sin\left(\frac{x}{L}\right) \cos\left(\frac{y}{L}\right) \cos\left(\frac{z}{L}\right) \\
v_y(x,y,z,t=0) &= -V_0 \cos\left(\frac{x}{L}\right) \sin\left(\frac{y}{L}\right) \cos\left(\frac{z}{L}\right) \\
v_z(x,y,z,t=0) &= 0
\end{align*}
The Reynolds number, based on the length scale $L$ and velocity $V_0$, is set to $Re = 1,600$.
The resulting flow is computed in the nondimensional time-interval $t \in [0 , 20]$.

This problem involves resolving a large range of spatial scales as the flow transitions from its initial laminar condition to a fully turbulent regime. However, it features a zero
mean convection velocity. It is typically used to study the effects of severe basis truncation on PROM stability and accuracy.

\subsubsection{High-dimensional model}

For this problem, the nondimensional rotational form of the incompressible Navier-Stokes equations are semi-discretized using a pseudo-spectral Fourier-Galerkin method. The resulting DNS HDM
is compact and further described in \cite{mortensen2016}. The associated mesh has 512 equally-spaced vertices in each spatial direction, which provides adequate resolution for all turbulent length 
scales that develop at the considered Reynolds number. It leads to a very large-scale HDM of dimension $N = 402,653,184$.

For convenience (see below), time-discretization of the above HDM is performed using the fourth-order explicit Runge-Kutta scheme, and the fixed nondimensional time-step $\Delta t  = 1 \times 10^{-3}$.
Figure \ref{fig:tgv} illustrates the time-evolution of the HDM-based flow solution: it displays isosurfaces of the vorticity magnitude colored by the velocity magnitude at four different time-instances.
Using 128 cores of the Linux cluster, the HDM-based simulation is performed in 34.9 hours wall-clock time.

\begin{figure}[h!]
	\centering
	\begin{subfigure}[c]{0.35\textwidth}%
	\centering
	\includegraphics[width=\linewidth]{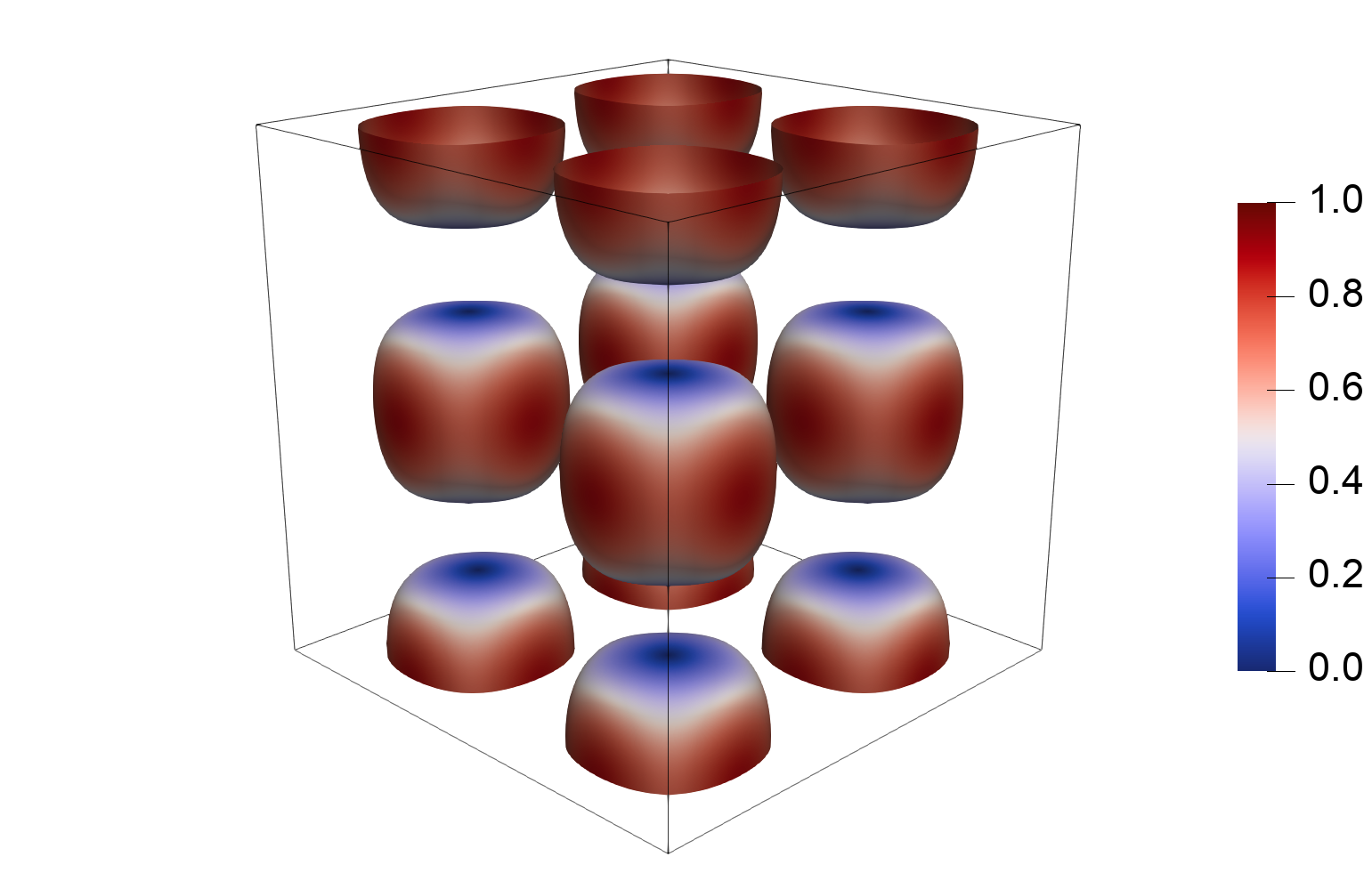}
	\caption{$t = 0$}
	\end{subfigure}%
	\begin{subfigure}[c]{0.35\textwidth}%
	\centering
	\includegraphics[width=\linewidth]{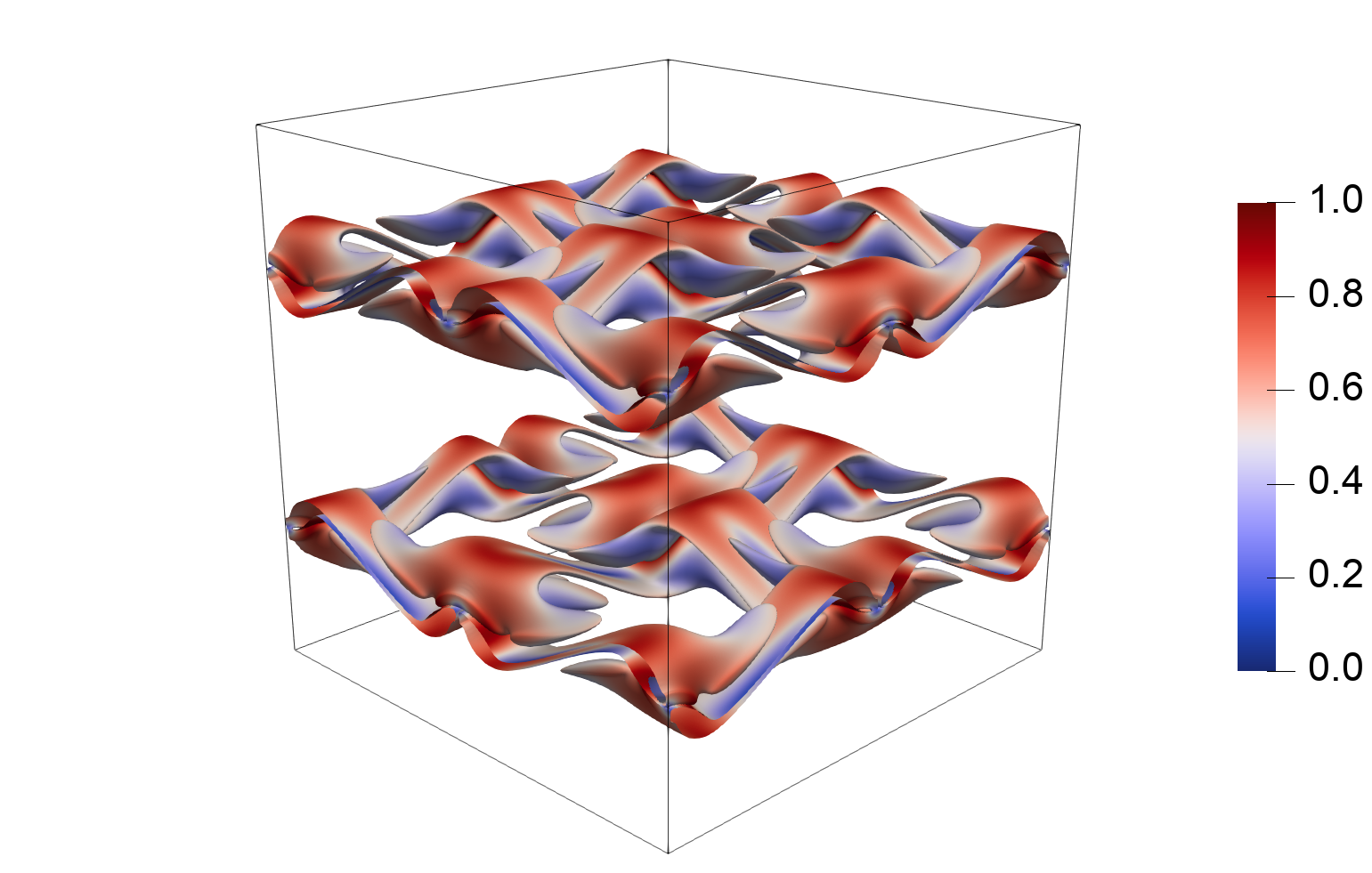}
	\caption{$t = 4$}
	\end{subfigure}%
	
	\medskip
	\begin{subfigure}[c]{0.35\textwidth}%
	\centering
	\includegraphics[width=\linewidth]{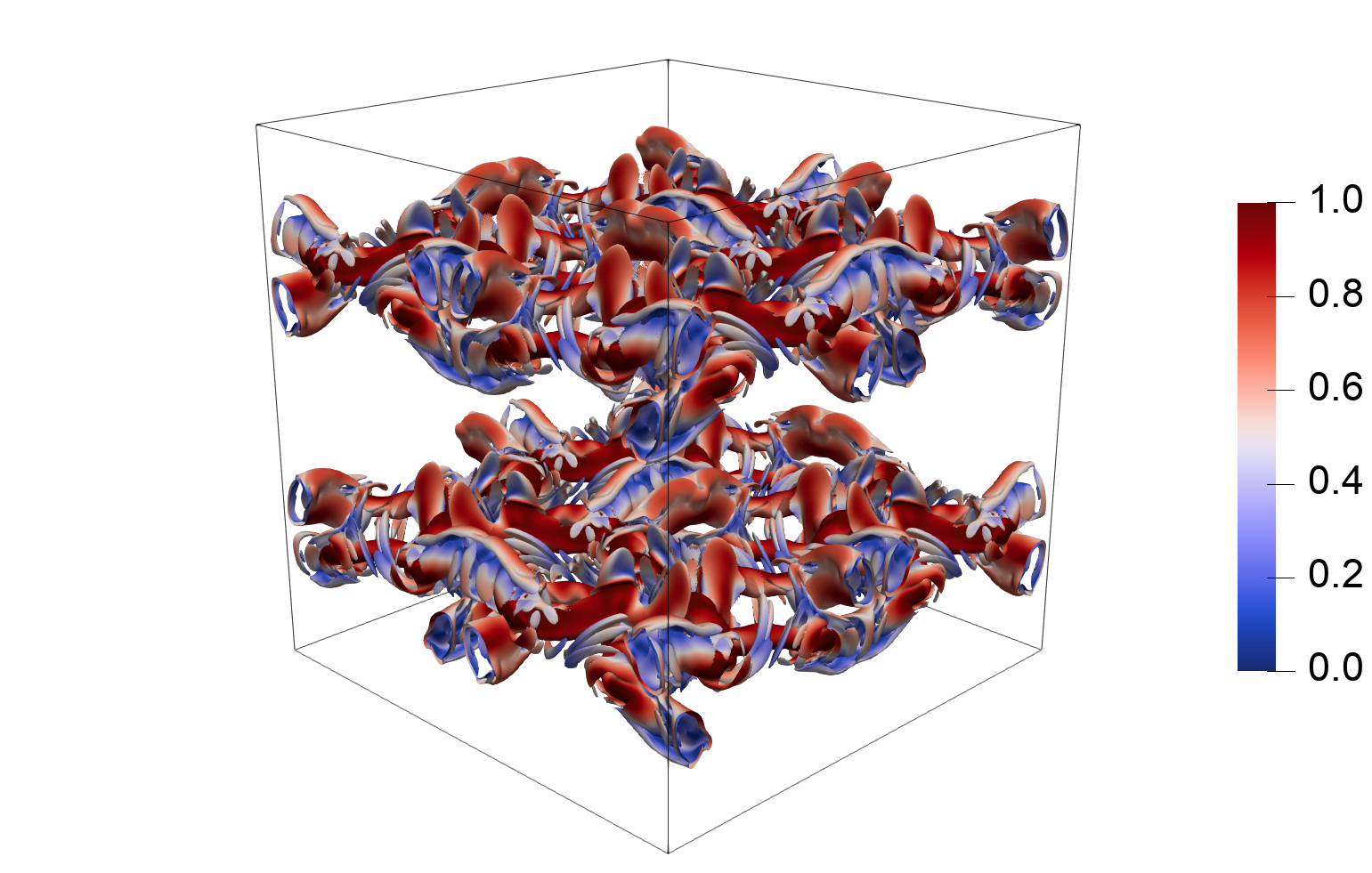}
	\caption{$t = 8$}
	\end{subfigure}%
	\begin{subfigure}[c]{0.35\textwidth}%
	\centering
	\includegraphics[width=\linewidth]{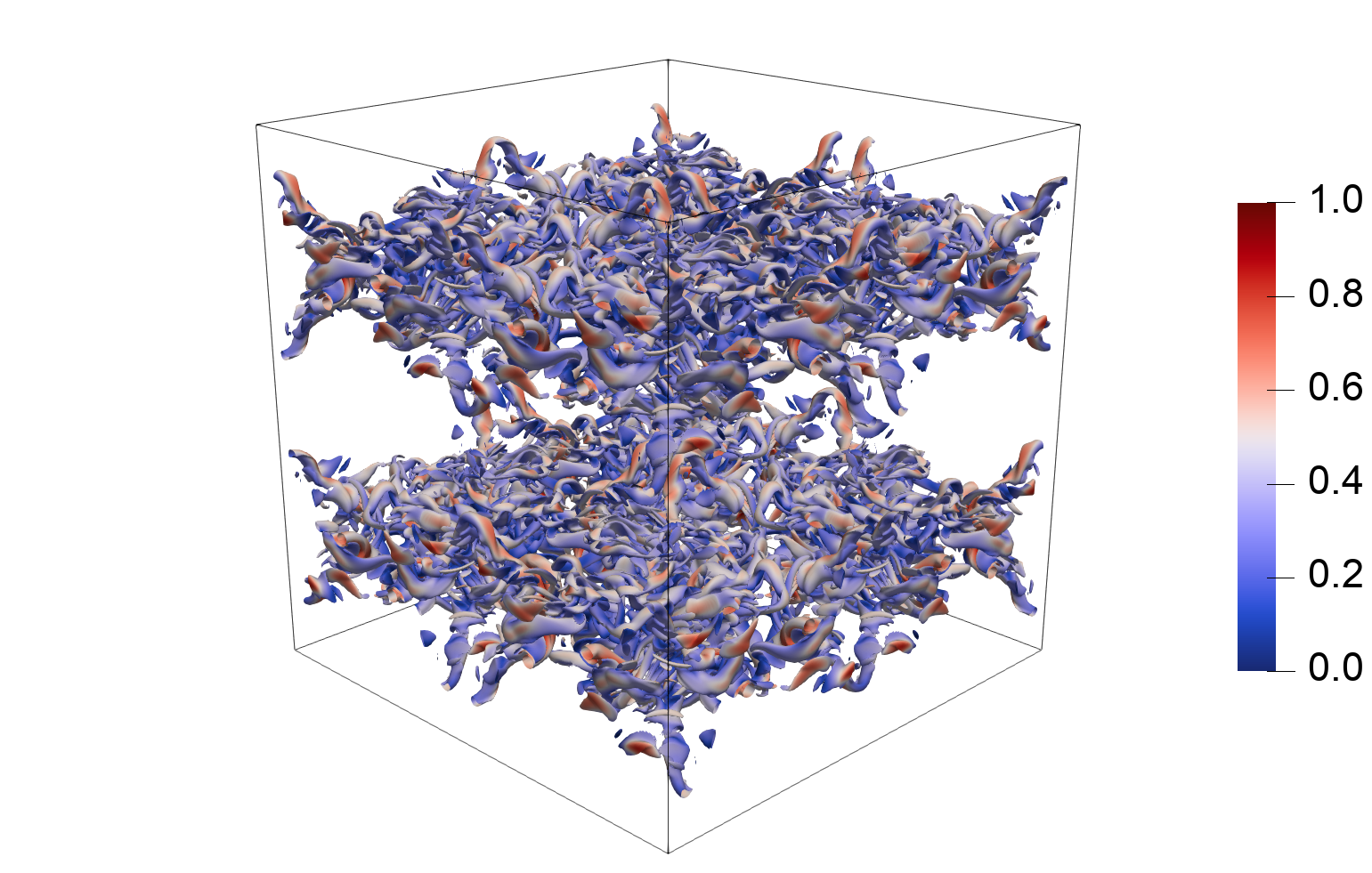}
	\caption{$t = 12$}
	\end{subfigure}%
	\caption{Taylor-Green vortex: Isosurfaces of the HDM-based solution vorticity magnitude colored by the velocity magnitude at four different time-instances.}
	\label{fig:tgv}
\end{figure}

\subsubsection{Projection-based reduced-order models with pre-computation}

Snapshots of the HDM-based solution in the wavenumber domain are collected at the sampling rate defined by $\Delta s = 1 \times 10^{-1}$, leading to a total number of 201 collected complex solution 
snapshots. These are transformed into the physical space, compressed using an SVD, and transformed back into the wavenumber space yielding four complex right ROBs of dimension $n = 6$, $n = 22$, $n = 47$,
and $n = 81$ that capture $90\%$, $99\%$, $99.9\%$, and $99.99\%$ of the energy of the singular values of the snapshot matrix, respectively.
The right ROBs could have been just as easily constructed by applying POD directly to the complex wavenumber space snapshots. Here, the affine offset $\mathbf{u}_0$ is set to zero.

The HDM for this problem includes only a quadratic nonlinearity in the state vector. Thus, the projected reduced-order quantities associated with Galerkin and Petrov-Galerkin PROMs can be pre-computed 
in this case during an offline phase, allowing for the complexity of the online computations to scale only with the dimension $n$ of the PROM. Specifically, for each constructed LSPG-based 
Petrov-Galerkin PROM, the online storage requirement and computational complexity is in this case $O(n^4)$; for a counterpart Galerkin PROM, they are in this case $O(n^3)$.

As mentioned above, the HDM constructed for this problem is time-discretized using the fourth-order explicit scheme. This, in order to avoid the burden associated with forming, storing, and solving 
dense, $N$-dimensional nonlinear systems of equations arising from implicit time-discretization. Given however that the counterpart nonlinear systems 
of equation associated with the PROMs are of dimension 
$n \ll N$ -- and therefore are significantly smaller -- the PROMs are time-discretized using the implicit, four-point, third-order BDF scheme. An additional advantage in this case for the discrete
PROMs is their increased maximum time-step which is not limited by the CFL condition associated with the mesh underlying the HDM.
Indeed, it is found that for all constructed PROMs, increasing the nondimensional time-step to $\Delta t = 2 \times 10^{-3}$ -- that is, a time-step that is twice as large as that afforded and used
by the HDM-based simulation -- yields stable and accurate numerical results.

For this example, all PROM simulations are performed on a single core of the Linux cluster, and four QoIs are considered for assessing the accuracy of the obtained numerical results: 
\begin{itemize}
	\item The time-dependent volume-averaged turbulent kinetic energy, 
		\begin{equation*} 
			E_k = \frac{1}{\lVert \Omega \rVert} \int_{\Omega} \frac{1}{2} (v_x^2 + v_y^2 +v_z^2) d\Omega 
		\end{equation*} 
		where $\lVert \Omega \rVert$ denotes the volume of the computational domain $\Omega$. 
	\item The time-dependent enstrophy-based dissipation rate, 
		\begin{equation*} \epsilon = \frac{2 \nu}{\lVert \Omega \rVert} \int_{\Omega} \frac{1}{2} (\omega_x^2 + \omega_y^2 + \omega_z^2) d\Omega 
		\end{equation*} 
		where $\nu$ is the kinematic viscosity and $\omega_x$, $\omega_y$, $\omega_z$ are the components of the flow vorticity.  
	\item The $y$ and $z$ components of the velocity, $v_y$ and $v_z$, computed at the probe location $x = \pi L$, $y = \pi L / 4$, $z = \pi L / 2$. At this location, $v_x$ is zero 
		during the entire simulation and therefore is not a particularly interesting QoI.
\end{itemize}
All relative errors associated with the above QoIs are computed using the nondimensional sampling rate defined by $\Delta s_{\mathbb{RE}} = 1 \times 10^{-1}$.

\paragraph{Galerkin reduced-order models}

The time-histories of the QoIs $E_k$, $\epsilon$, $v_y$, and $v_z$ computed using the HDM and Galerkin PROMs of increasing dimension $n$ are compared in Figures \ref{fig:tgvkgal}, \ref{fig:tgvepsgal}, 
\ref{fig:tgvprobevygal}, and \ref{fig:tgvprobevzgal}, respectively. For this problem, the Galerkin PROMs are found to be numerically stable for every value of $n$, and to exhibit a level of accuracy for 
each QoI that increases with $n$. Accuracy is achieved by the Galerkin PROMs for the integral QoIs as well as the pointwise QoIs, during the initial laminar regime as well as during the transition to 
turbulence. Even for the lowest-order Galerkin PROM of dimension $n = 6$, the computed kinetic energy still correctly decays with time, despite the fact that the physical length scales representing the 
dissipation range of the turbulent energy spectrum are not adequately resolved by this PROM. The reason why a Galerkin PROM is found to be numerically stable here, whereas it was found to be 
numerically unstable in the previous example may be due to the absence in this case of a mean convective velocity in the flow. However, it is certainly shown here that the inability of the Galerkin
PROM to resolve the smallest turbulent length scales captured by the DNS HDM does not lead to numerical instability, which supports the claim made in this paper that it is not the truncation error
suffered by a Galerkin PROM that necessarily causes its numerical instability.

\begin{figure}[h!]
	\centering
	\begin{subfigure}[c]{0.3\textwidth}%
	\centering
	\includegraphics[width=\linewidth]{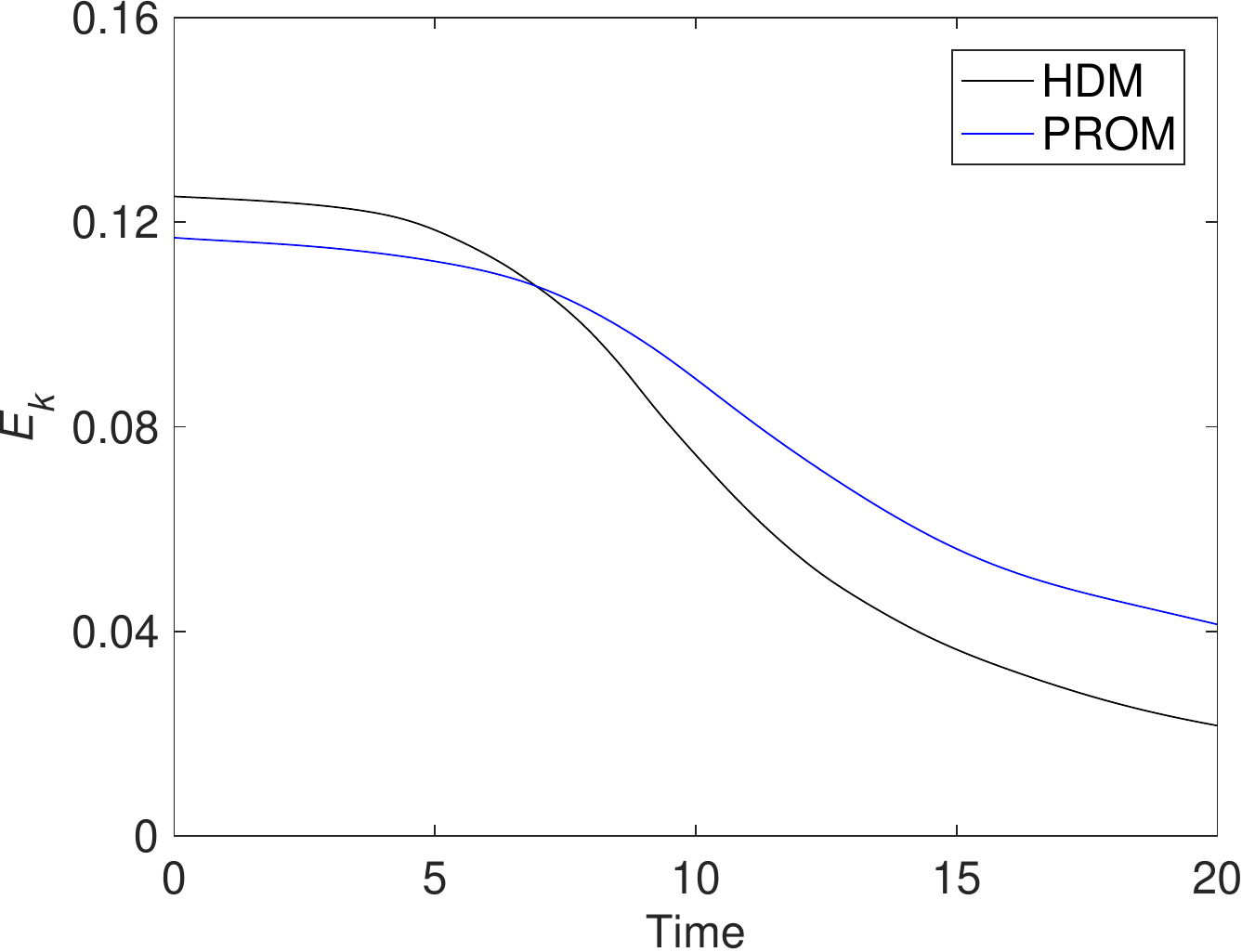}
	\caption{Galerkin, $n = 6$}
	\end{subfigure}%
	\hspace{1em}
	\begin{subfigure}[c]{0.3\textwidth}%
	\centering
	\includegraphics[width=\linewidth]{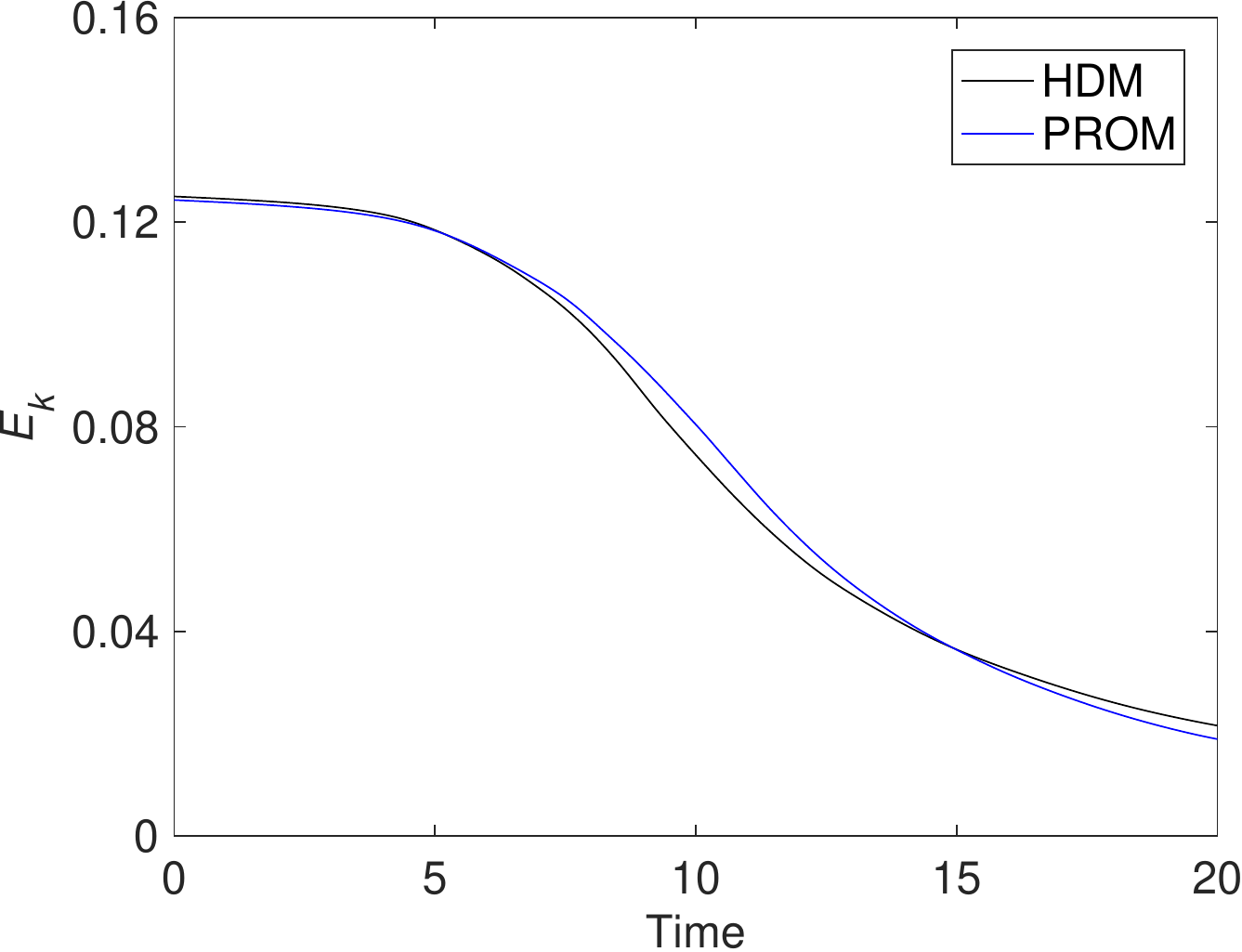}
	\caption{Galerkin, $n = 22$}
	\end{subfigure}%
	
	\medskip
	\begin{subfigure}[c]{0.3\textwidth}%
	\centering
	\includegraphics[width=\linewidth]{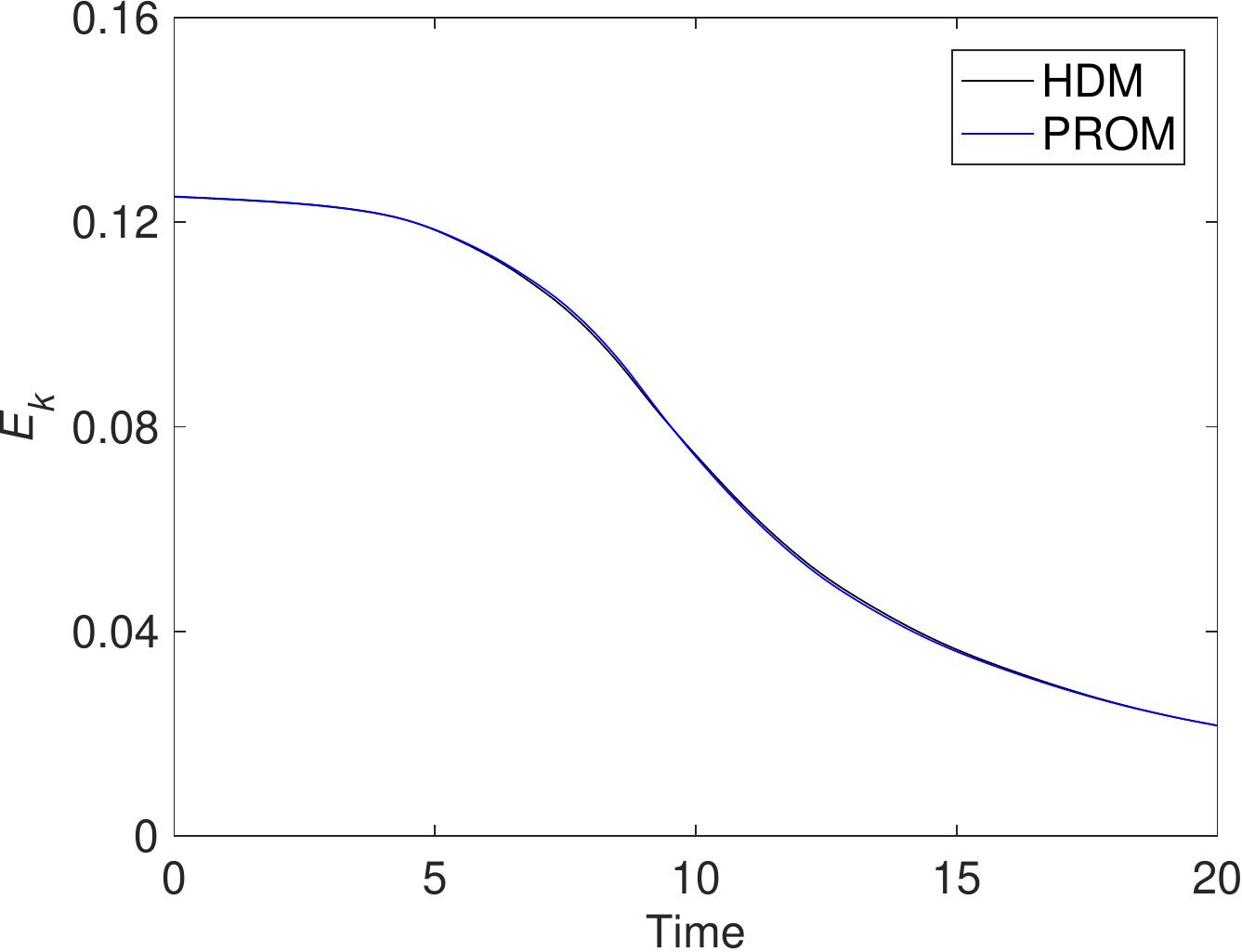}
	\caption{Galerkin, $n = 47$}
	\end{subfigure}%
	\hspace{1em}
	\begin{subfigure}[c]{0.3\textwidth}%
	\centering
	\includegraphics[width=\linewidth]{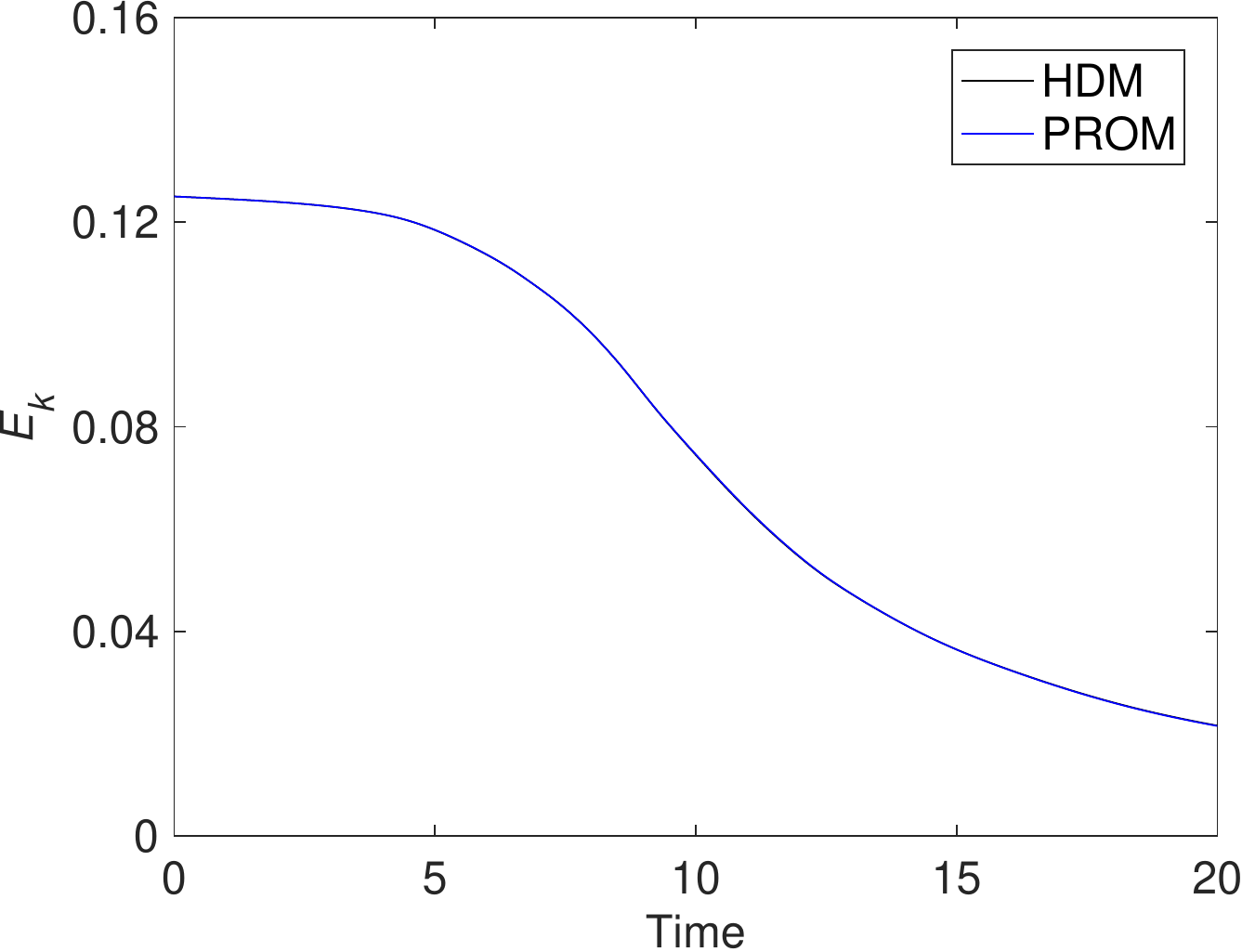}
	\caption{Galerkin, $n = 81$}
	\end{subfigure}%
	\caption{Taylor-Green vortex: Time-histories of the turbulent kinetic energy computed using the HDM and Galerkin PROMs.}
	\label{fig:tgvkgal}
\end{figure}

\begin{figure}[h!]
	\centering
	\begin{subfigure}[c]{0.3\textwidth}%
	\centering
	\includegraphics[width=\linewidth]{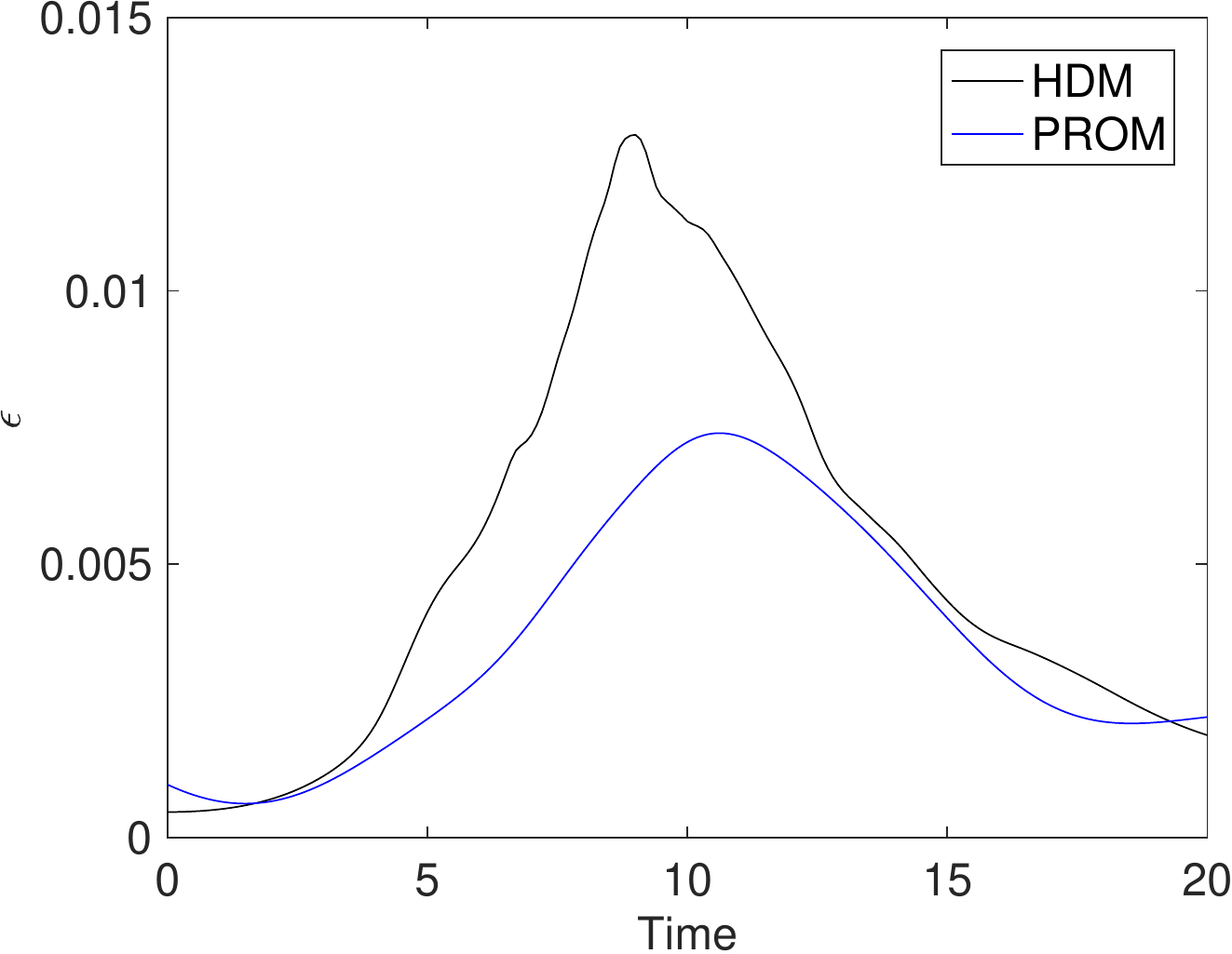}
	\caption{Galerkin, $n = 6$}
	\end{subfigure}%
	\hspace{1em}
	\begin{subfigure}[c]{0.3\textwidth}%
	\centering
	\includegraphics[width=\linewidth]{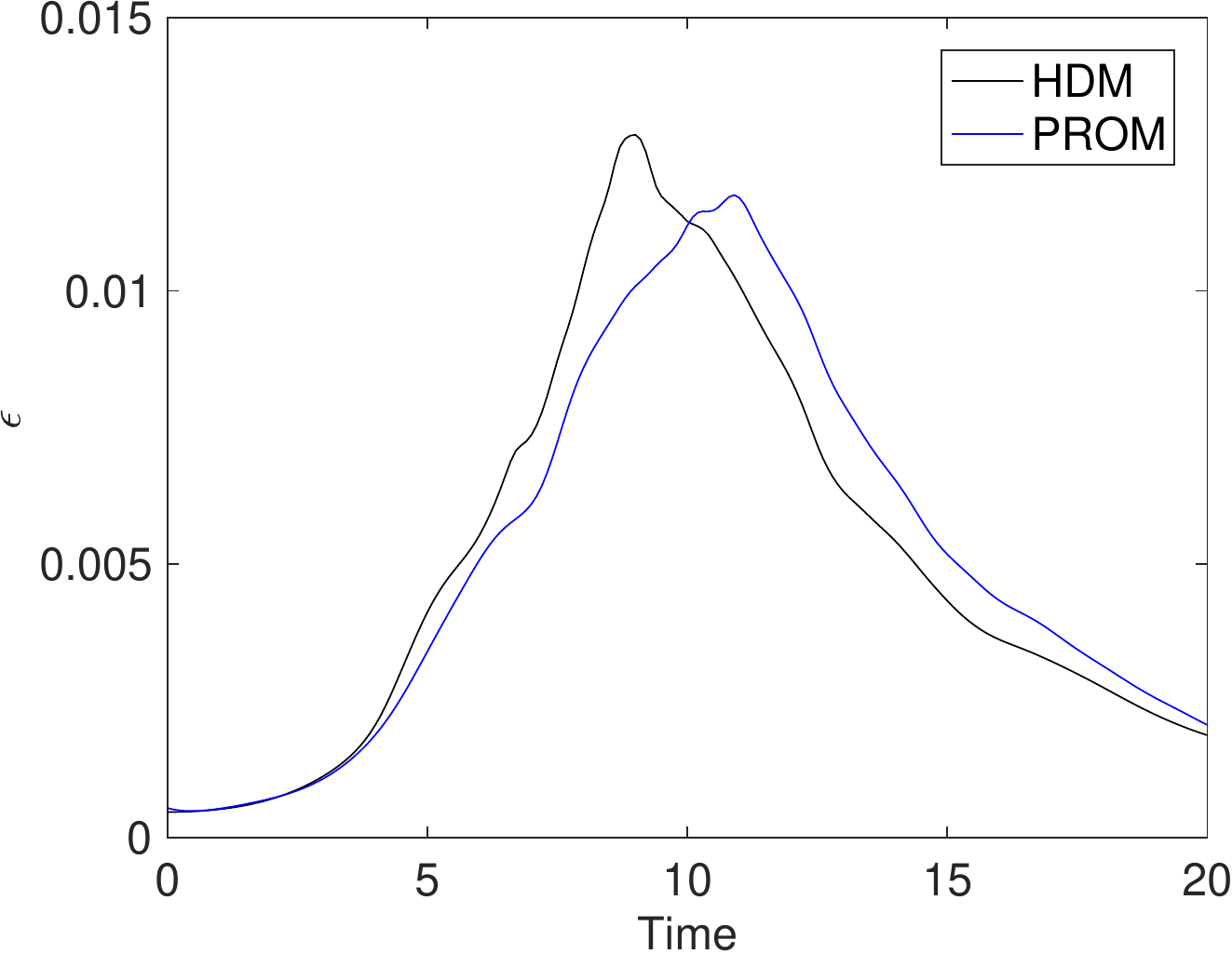}
	\caption{Galerkin, $n = 22$}
	\end{subfigure}%
	
	\medskip
	\begin{subfigure}[c]{0.3\textwidth}%
	\centering
	\includegraphics[width=\linewidth]{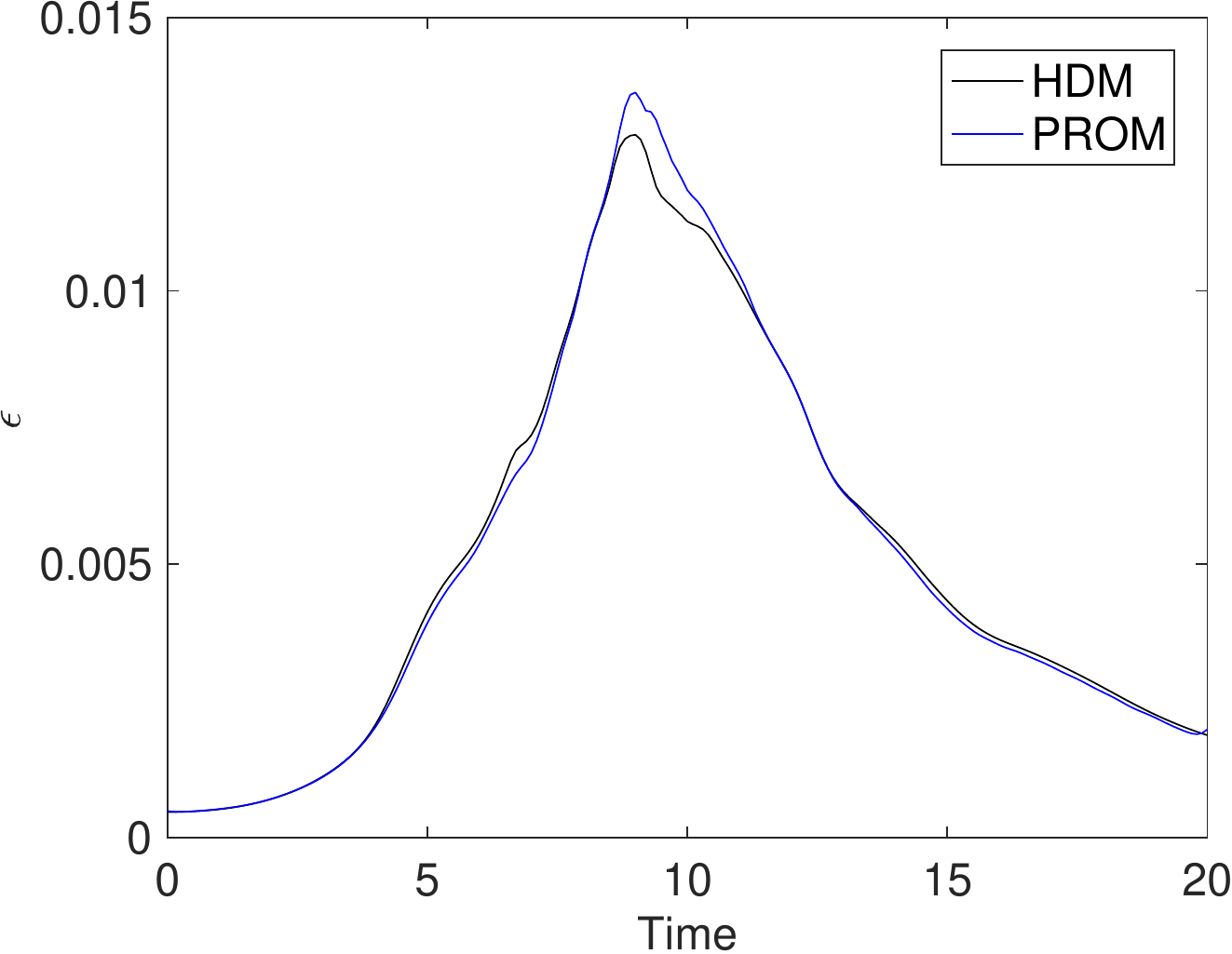}
	\caption{Galerkin, $n = 47$}
	\end{subfigure}%
	\hspace{1em}
	\begin{subfigure}[c]{0.3\textwidth}%
	\centering
	\includegraphics[width=\linewidth]{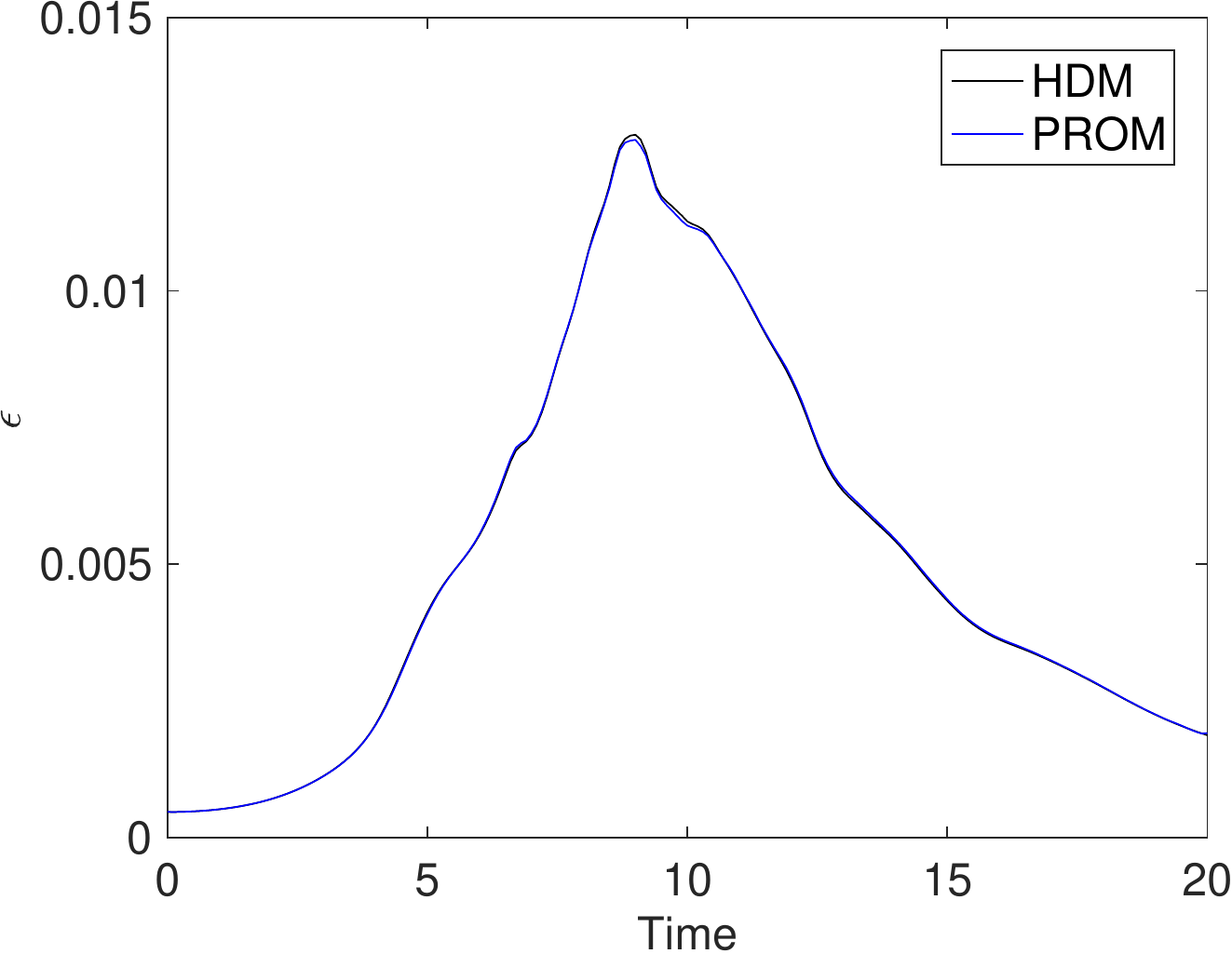}
	\caption{Galerkin, $n = 81$}
	\end{subfigure}%
	\caption{Taylor-Green vortex: Time-histories of the enstrophy-based dissipation rate computed using the HDM and Galerkin PROMs.}
	\label{fig:tgvepsgal}
\end{figure}

\begin{figure}[h!]
	\centering
	\begin{subfigure}[c]{0.3\textwidth}%
	\centering
	\includegraphics[width=\linewidth]{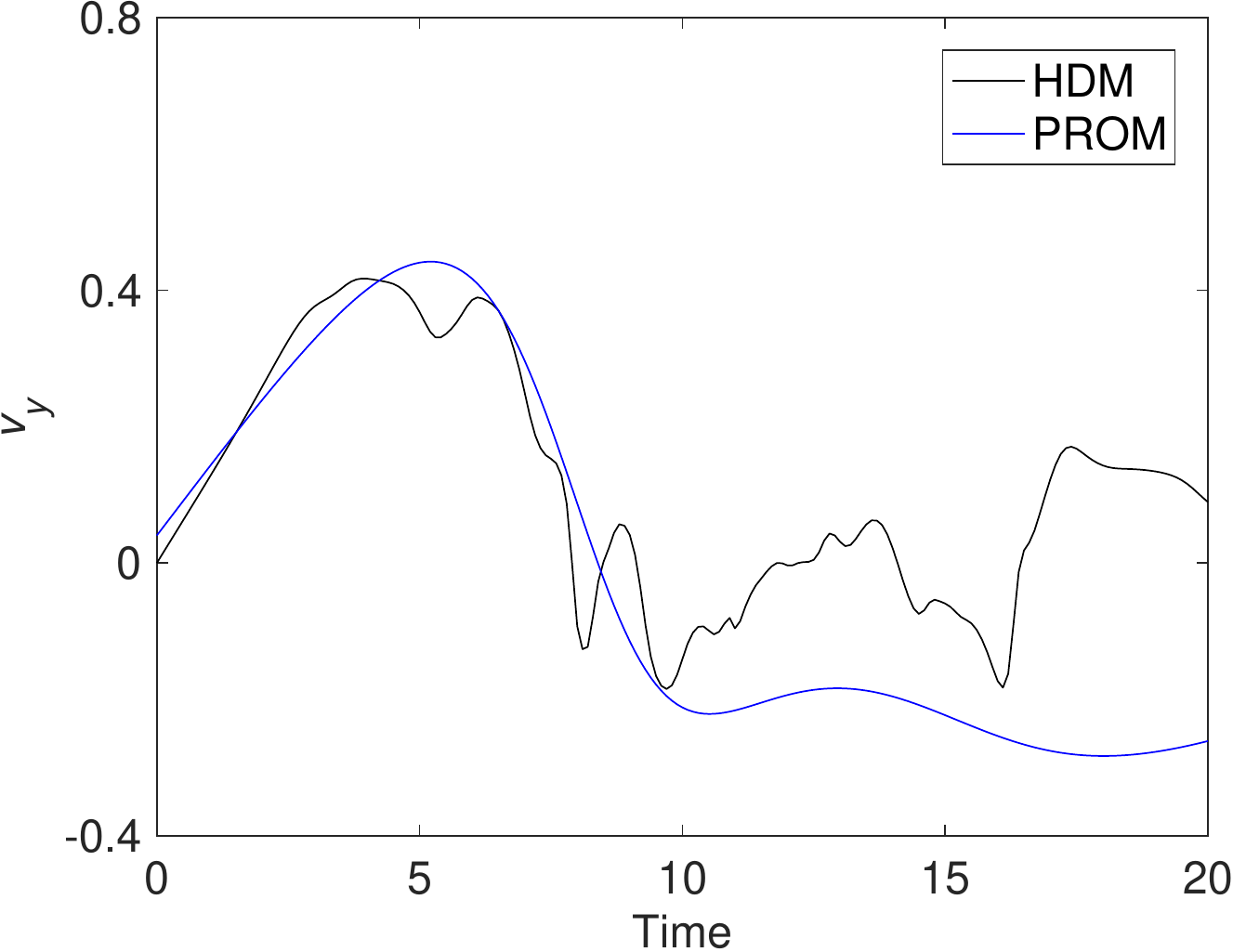}
	\caption{Galerkin, $n = 6$}
	\end{subfigure}%
	\hspace{1em}
	\begin{subfigure}[c]{0.3\textwidth}%
	\centering
	\includegraphics[width=\linewidth]{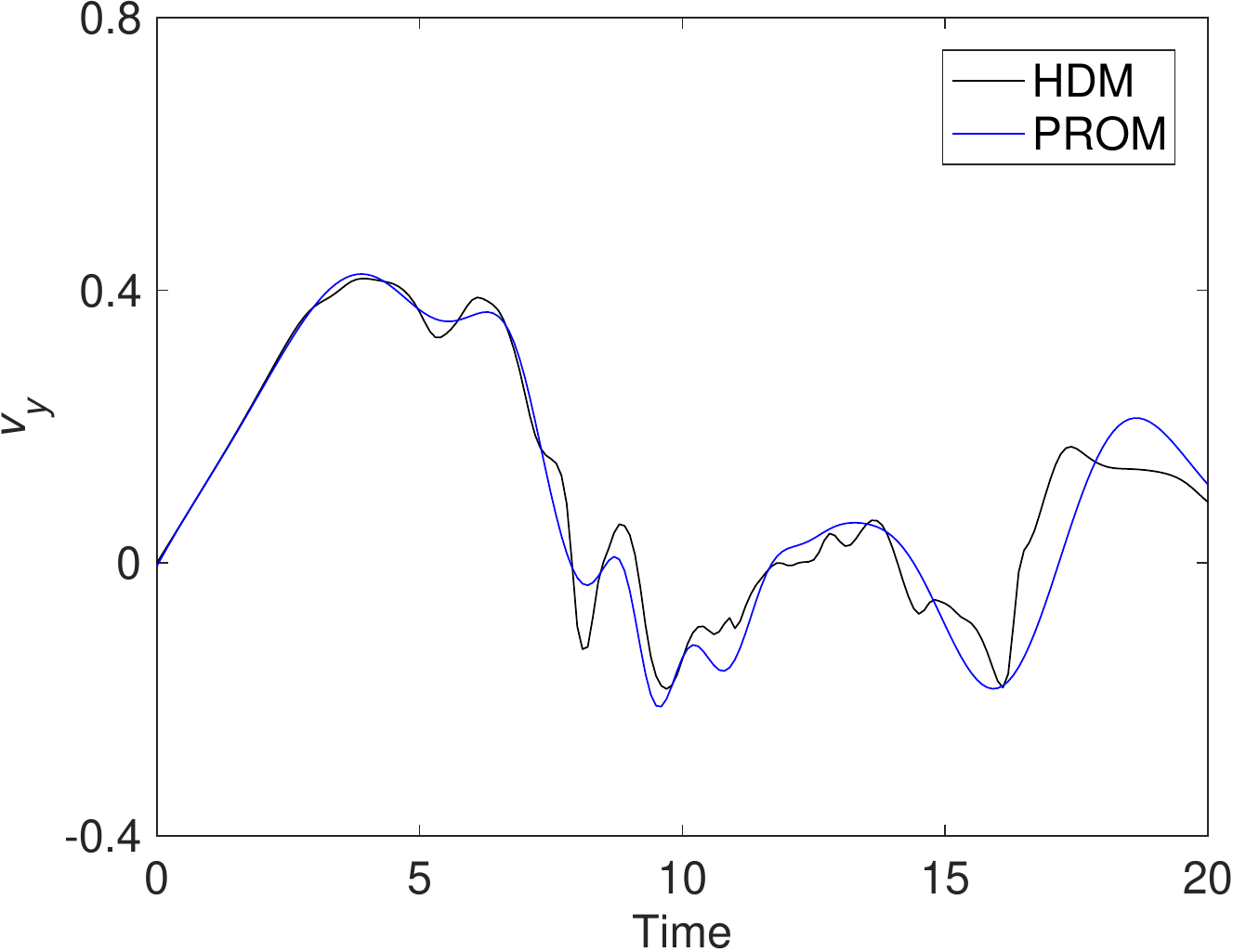}
	\caption{Galerkin, $n = 22$}
	\end{subfigure}%
	
	\medskip
	\begin{subfigure}[c]{0.3\textwidth}%
	\centering
	\includegraphics[width=\linewidth]{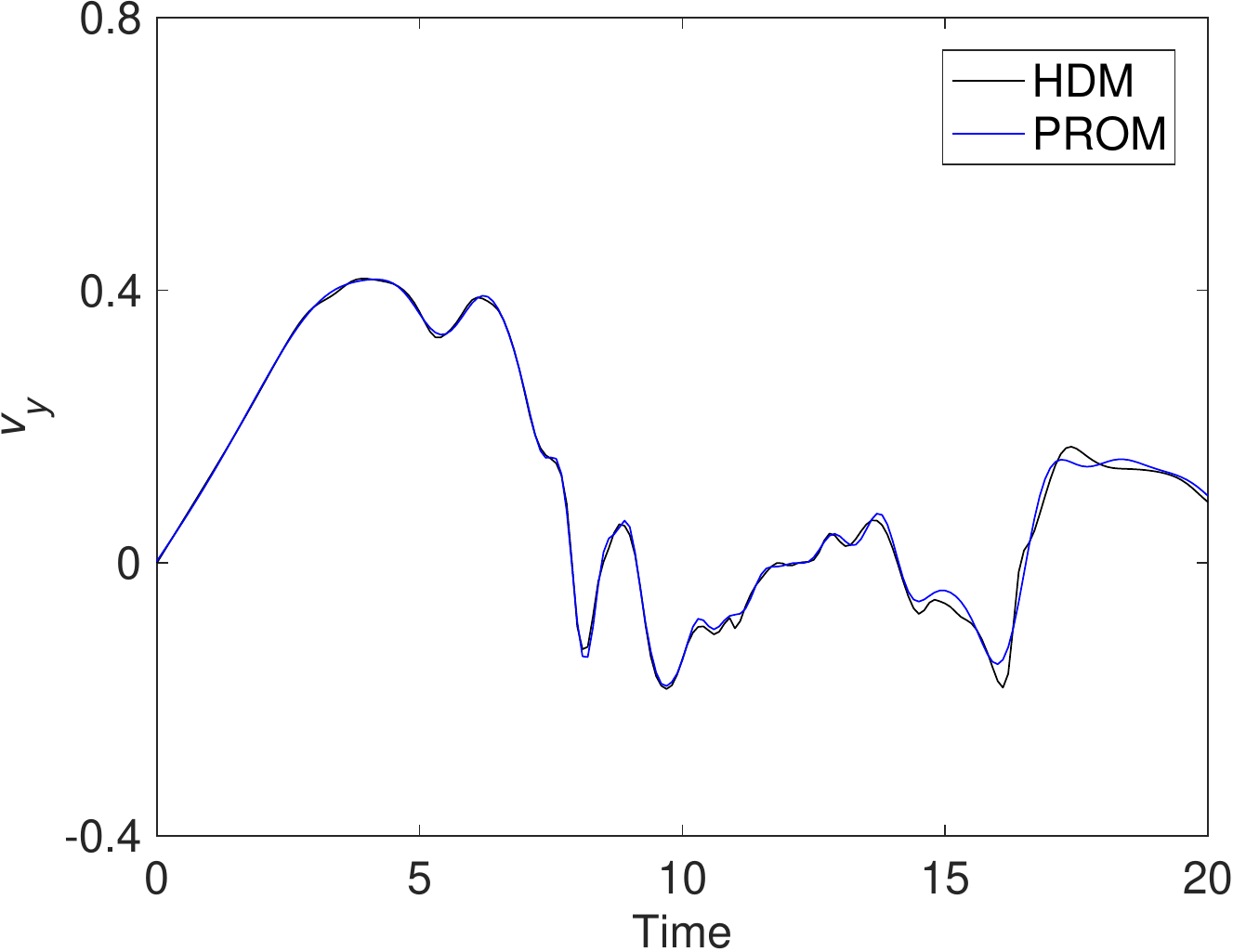}
	\caption{Galerkin, $n = 47$}
	\end{subfigure}%
	\hspace{1em}
	\begin{subfigure}[c]{0.3\textwidth}%
	\centering
	\includegraphics[width=\linewidth]{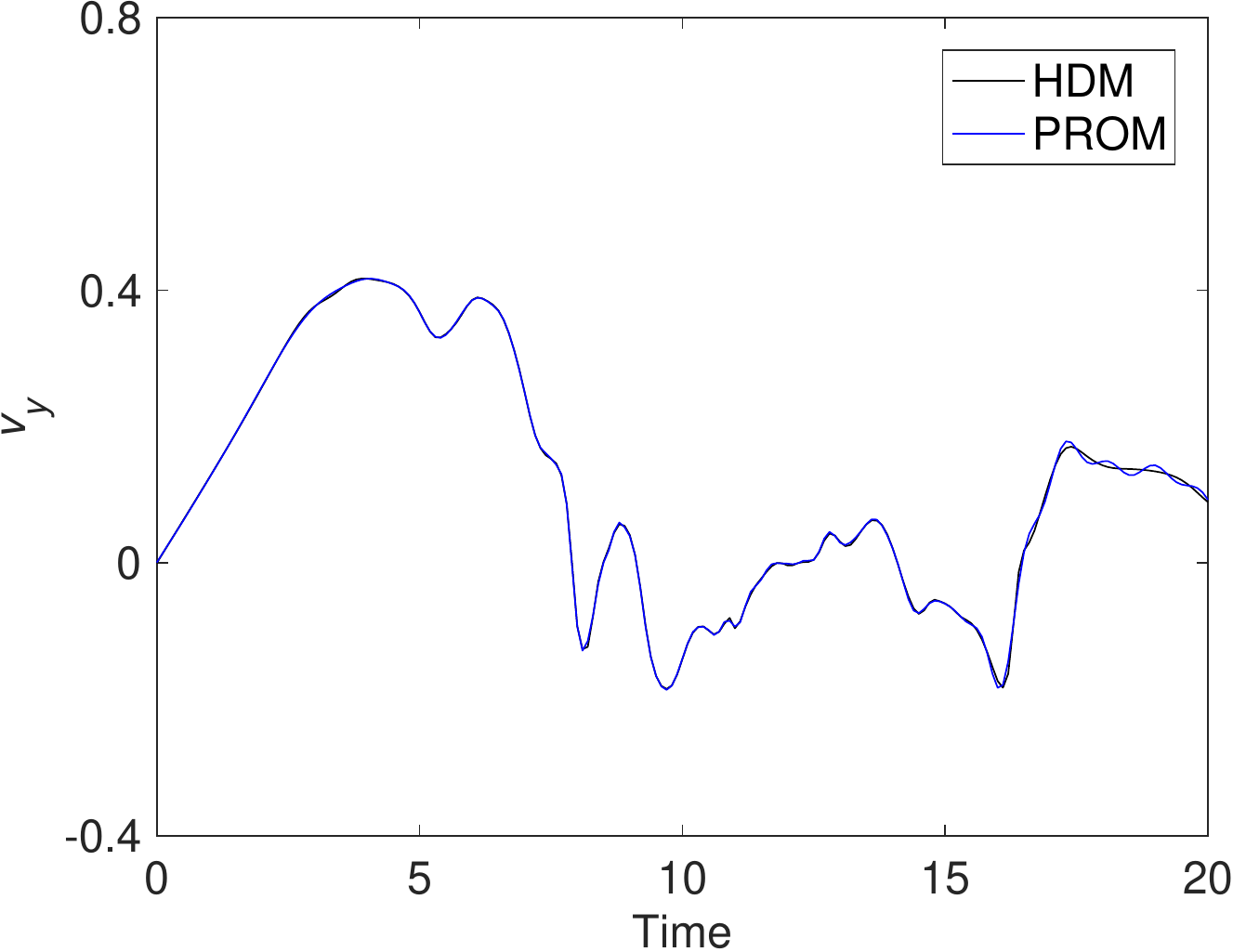}
	\caption{Galerkin, $n = 81$}
	\end{subfigure}%
	\caption{Taylor-Green vortex: Time-histories of the velocity component $v_y$ computed at a probe using the HDM and Galerkin PROMs.}
	\label{fig:tgvprobevygal}
\end{figure}

\begin{figure}[h!]
	\centering
	\begin{subfigure}[c]{0.3\textwidth}%
	\centering
	\includegraphics[width=\linewidth]{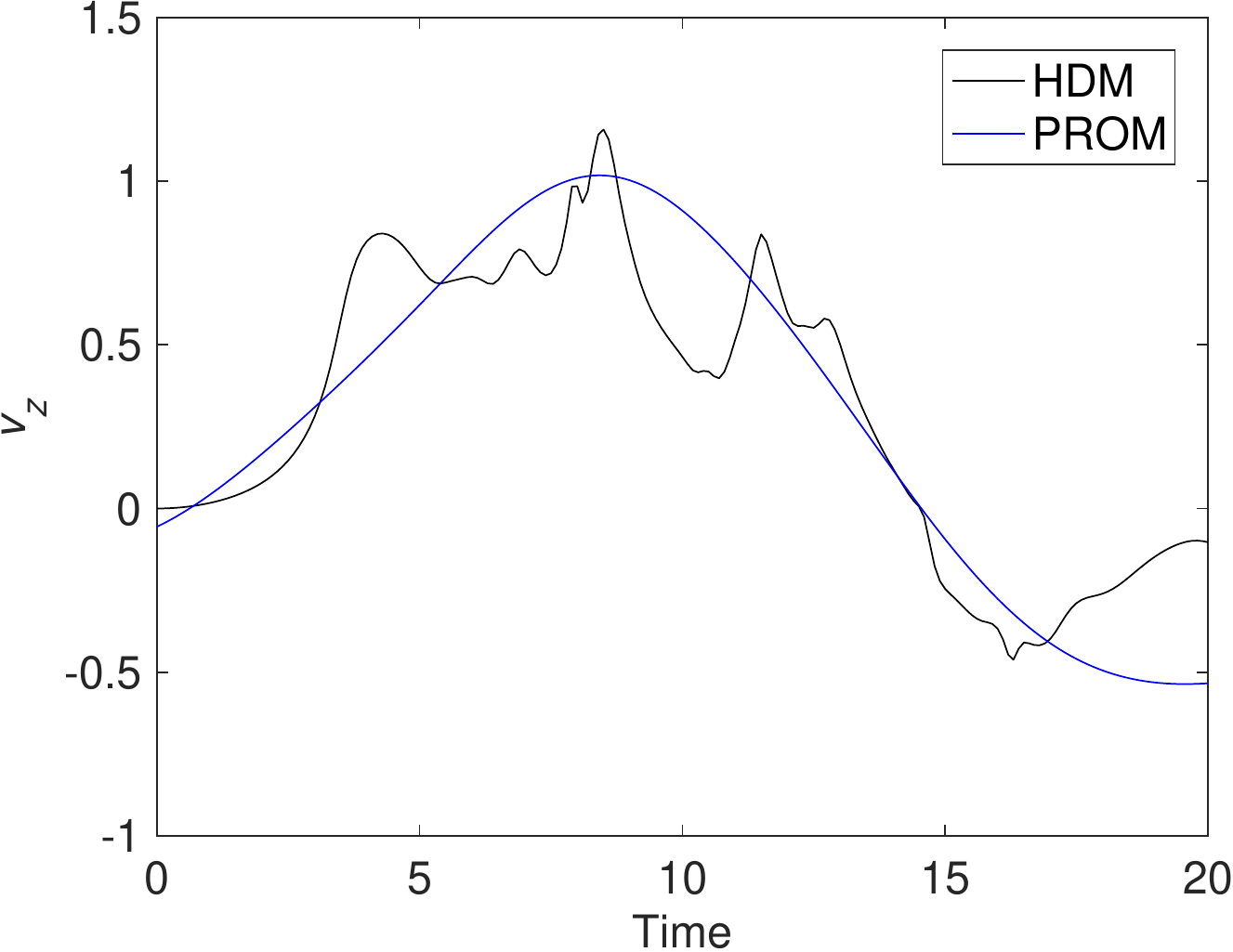}
	\caption{Galerkin, $n = 6$}
	\end{subfigure}%
	\hspace{1em}
	\begin{subfigure}[c]{0.3\textwidth}%
	\centering
	\includegraphics[width=\linewidth]{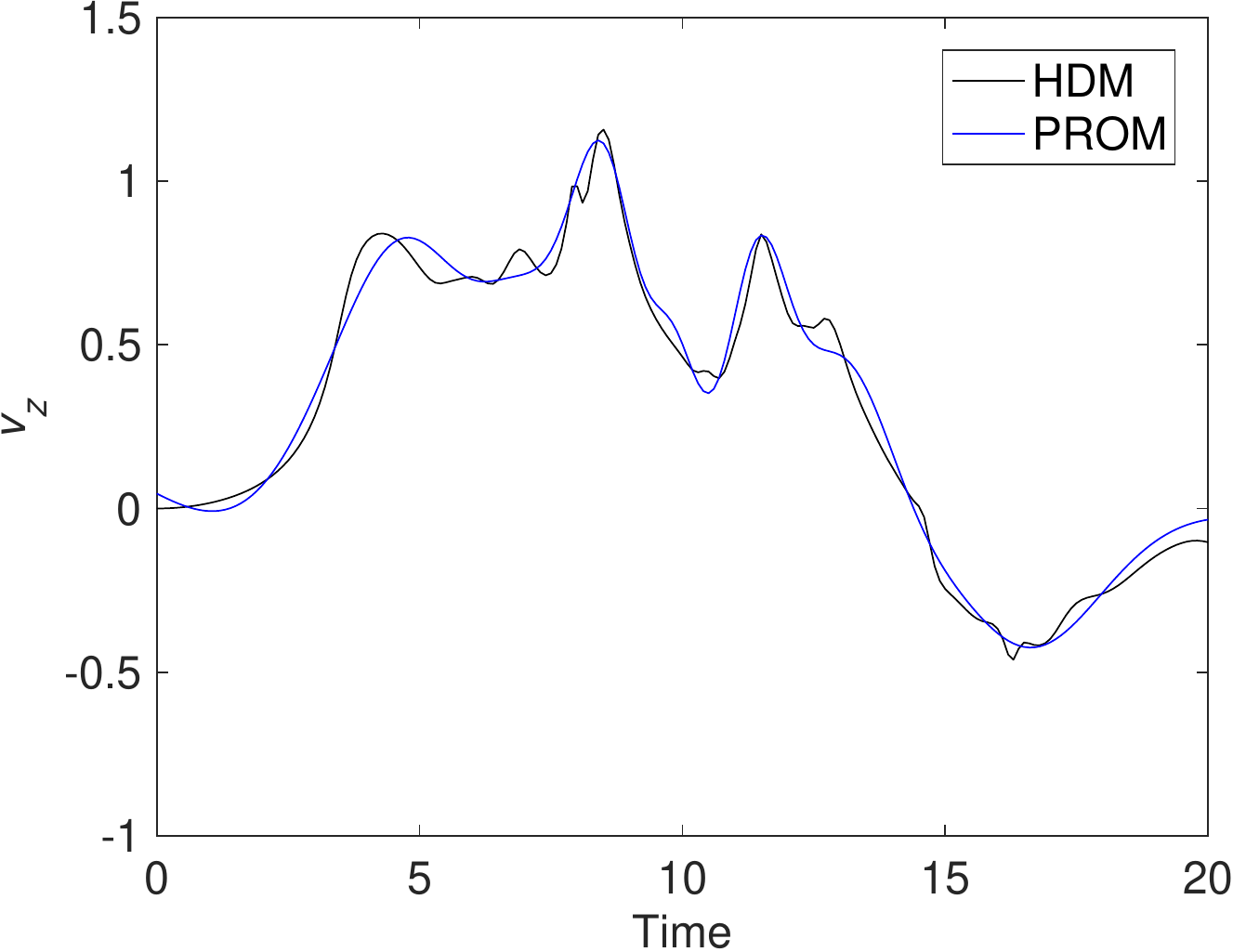}
	\caption{Galerkin, $n = 22$}
	\end{subfigure}%
	
	\medskip
	\begin{subfigure}[c]{0.3\textwidth}%
	\centering
	\includegraphics[width=\linewidth]{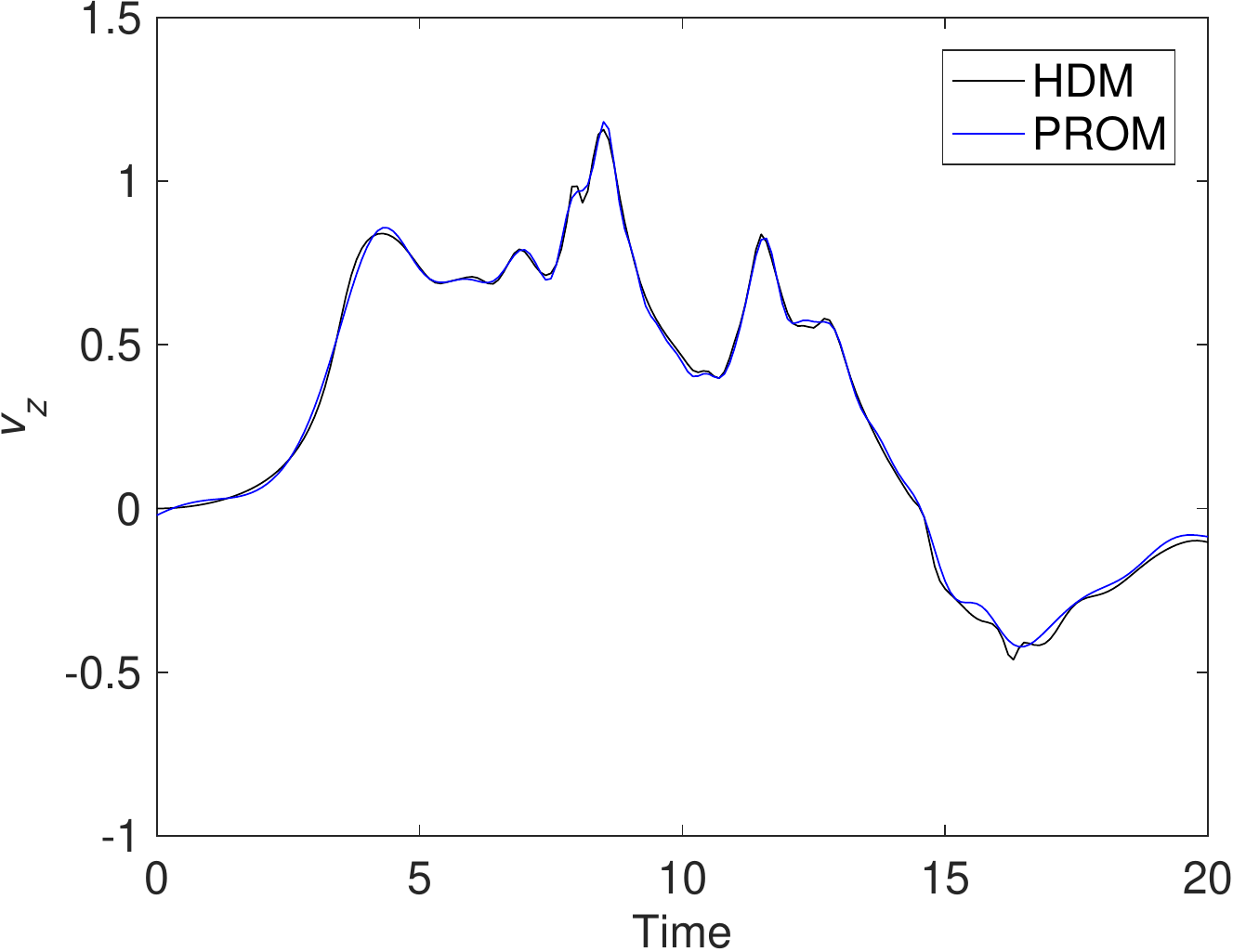}
	\caption{Galerkin, $n = 47$}
	\end{subfigure}%
	\hspace{1em}
	\begin{subfigure}[c]{0.3\textwidth}%
	\centering
	\includegraphics[width=\linewidth]{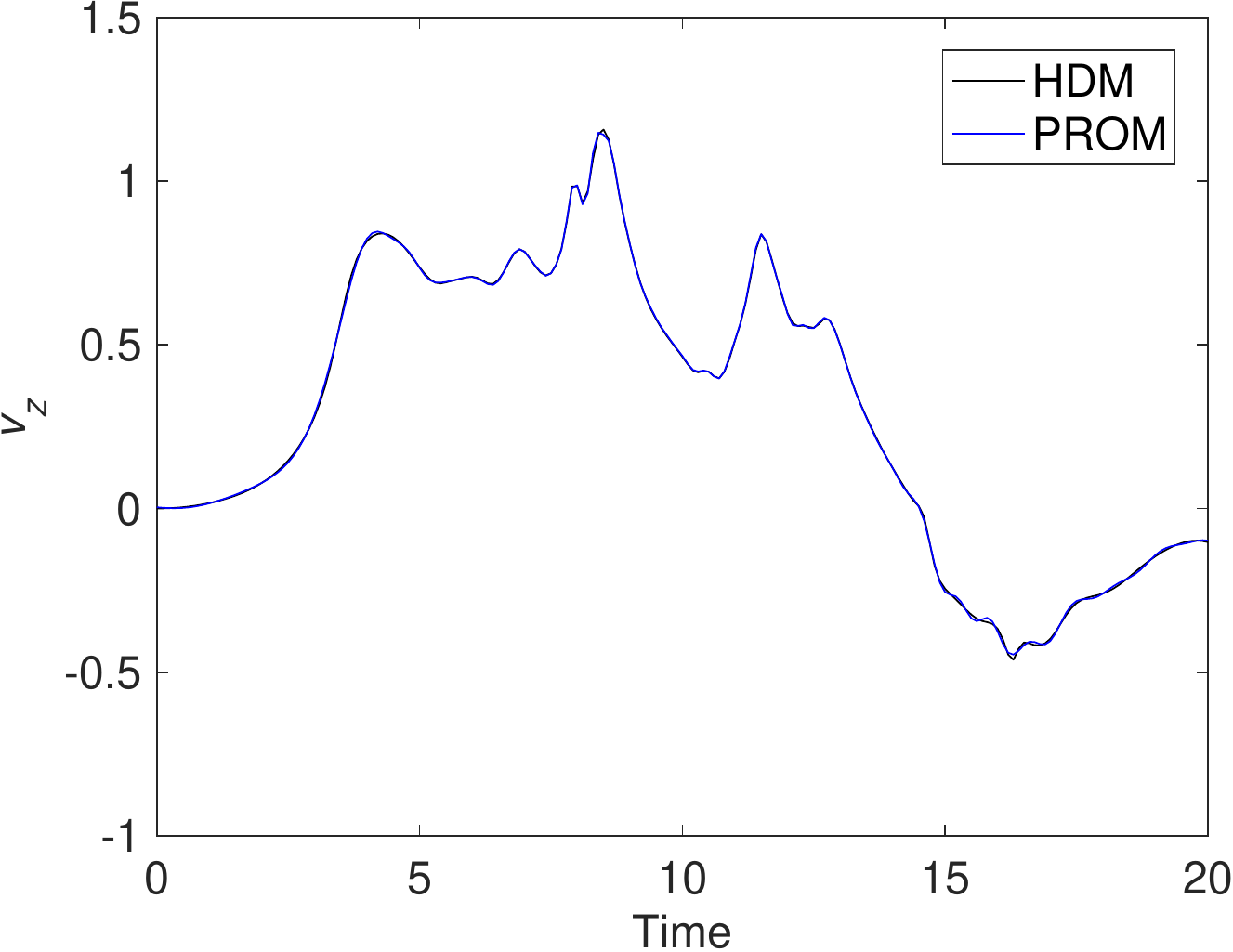}
	\caption{Galerkin, $n = 81$}
	\end{subfigure}%
	\caption{Taylor-Green vortex: Time-histories of the velocity component $v_z$ computed at a probe using the HDM and Galerkin PROMs.}
	\label{fig:tgvprobevzgal}
\end{figure}

\clearpage
\paragraph{Petrov-Galerkin reduced-order models}

Figures \ref{fig:tgvklspg}, \ref{fig:tgvepslspg}, \ref{fig:tgvprobevylspg}, and \ref{fig:tgvprobevzlspg} compare the time-histories of the QoIs $E_k$, $\epsilon$, $v_y$, and $v_z$ computed using the HDM 
and LSPG-based Petrov-Galerkin PROMs of various dimensions. The performance of each LSPG-based Petrov-Galerkin PROM is found to be very similar to that of the Galerkin PROM of the same dimension $n$ 
discussed above. 

\begin{figure}[h!]
	\centering
	\begin{subfigure}[c]{0.3\textwidth}%
	\centering
	\includegraphics[width=\linewidth]{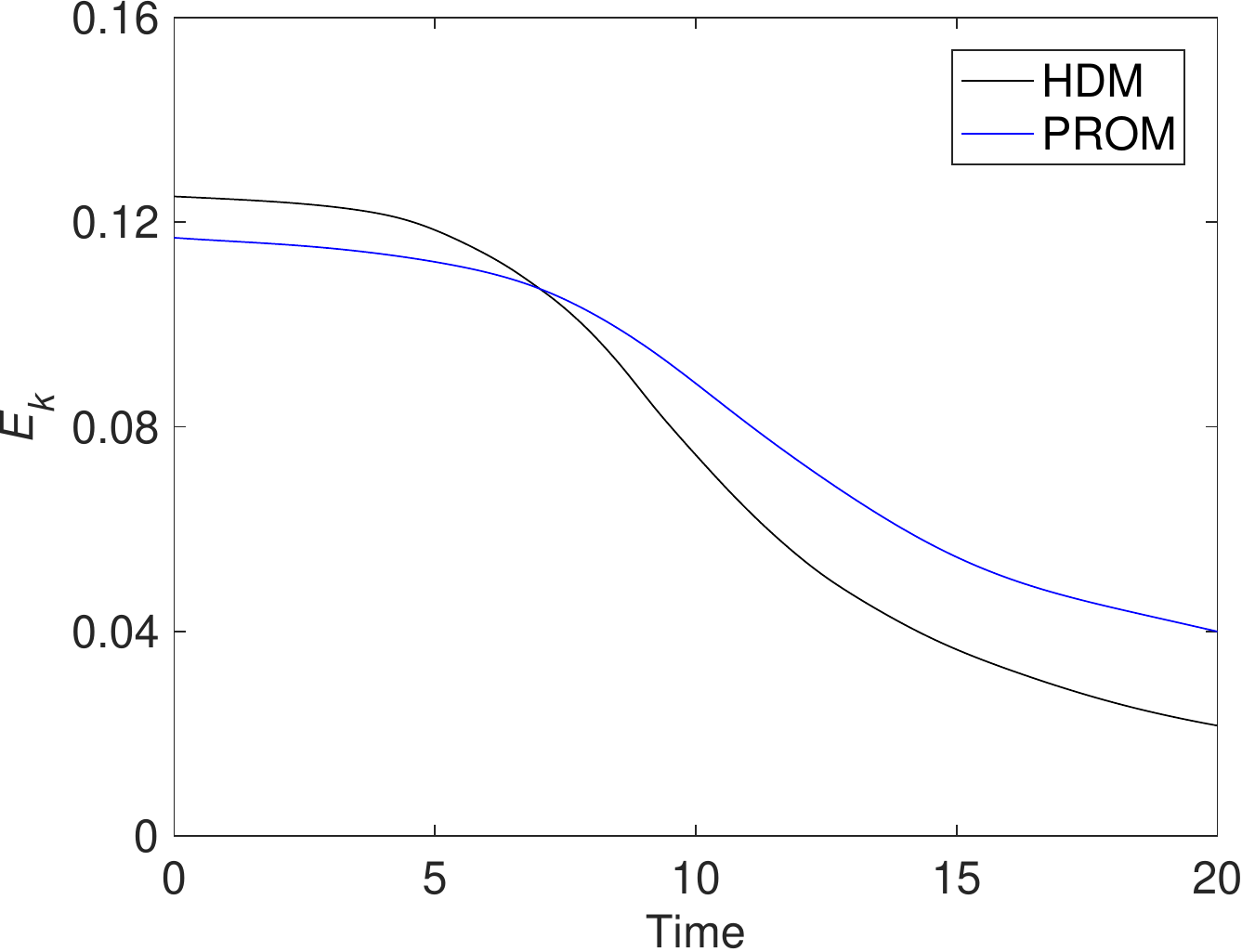}
	\caption{LSPG, $n = 6$}
	\end{subfigure}%
	\hspace{1em}
	\begin{subfigure}[c]{0.3\textwidth}%
	\centering
	\includegraphics[width=\linewidth]{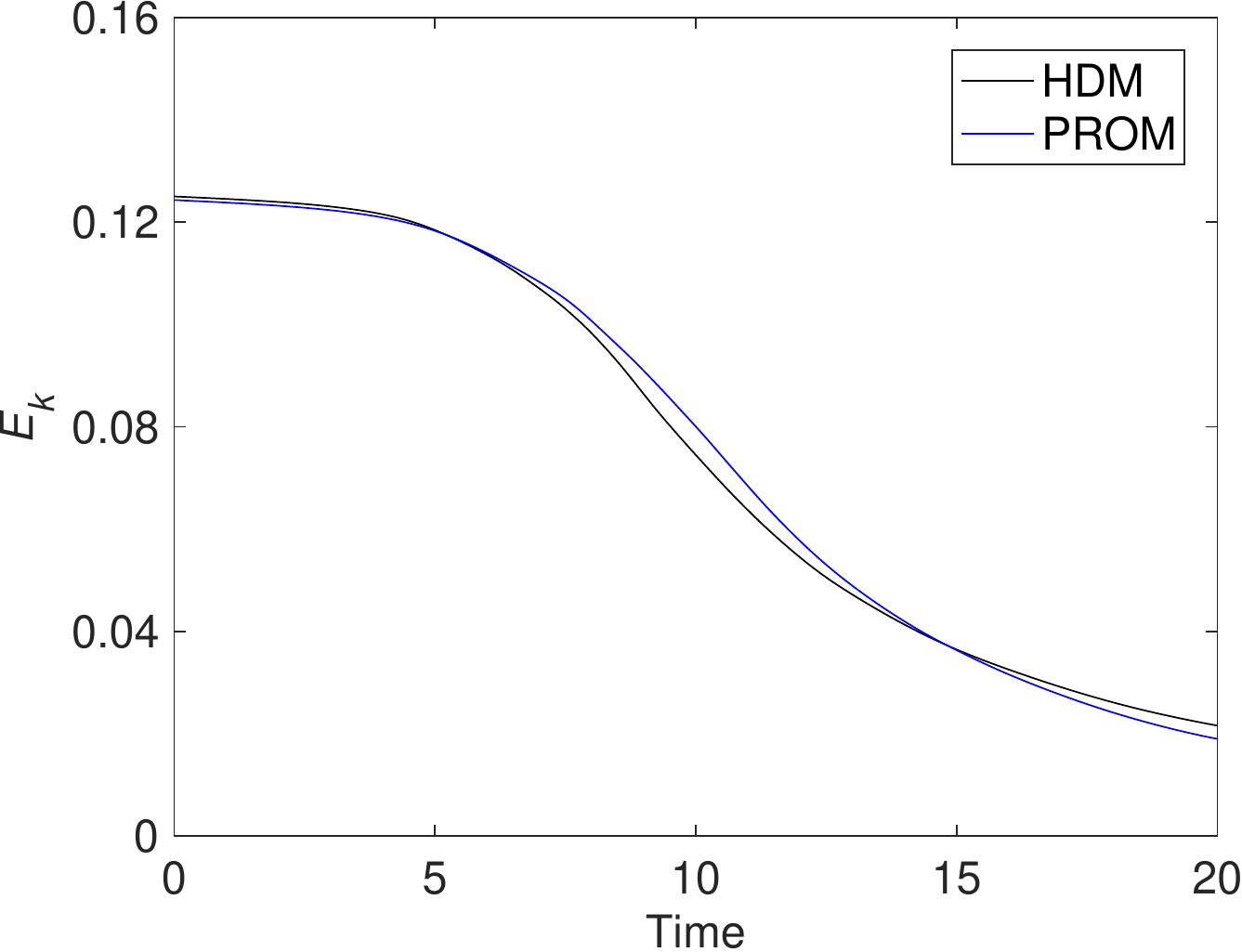}
	\caption{LSPG, $n = 22$}
	\end{subfigure}%
	
	\medskip
	\begin{subfigure}[c]{0.3\textwidth}%
	\centering
	\includegraphics[width=\linewidth]{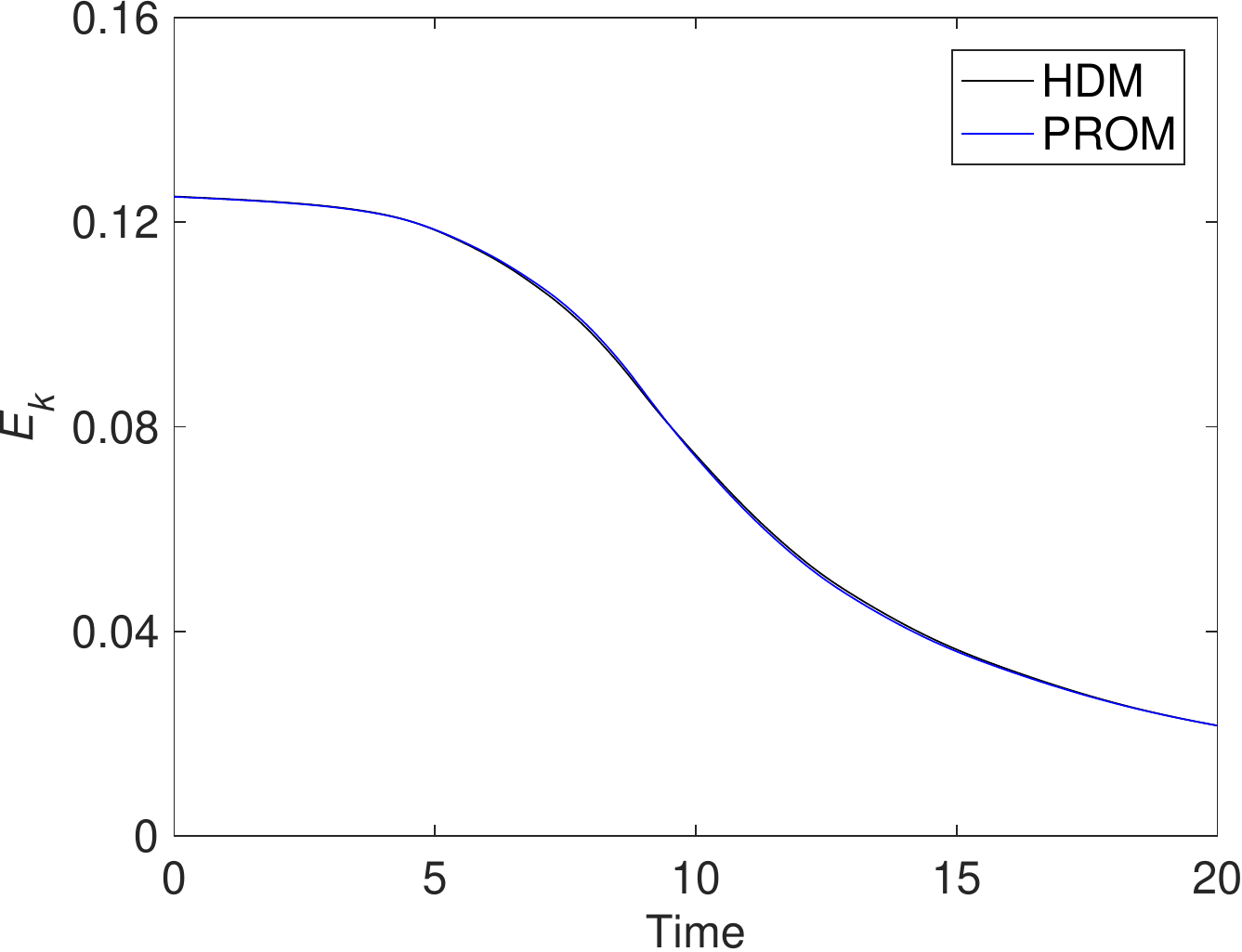}
	\caption{LSPG, $n = 47$}
	\end{subfigure}%
	\hspace{1em}
	\begin{subfigure}[c]{0.3\textwidth}%
	\centering
	\includegraphics[width=\linewidth]{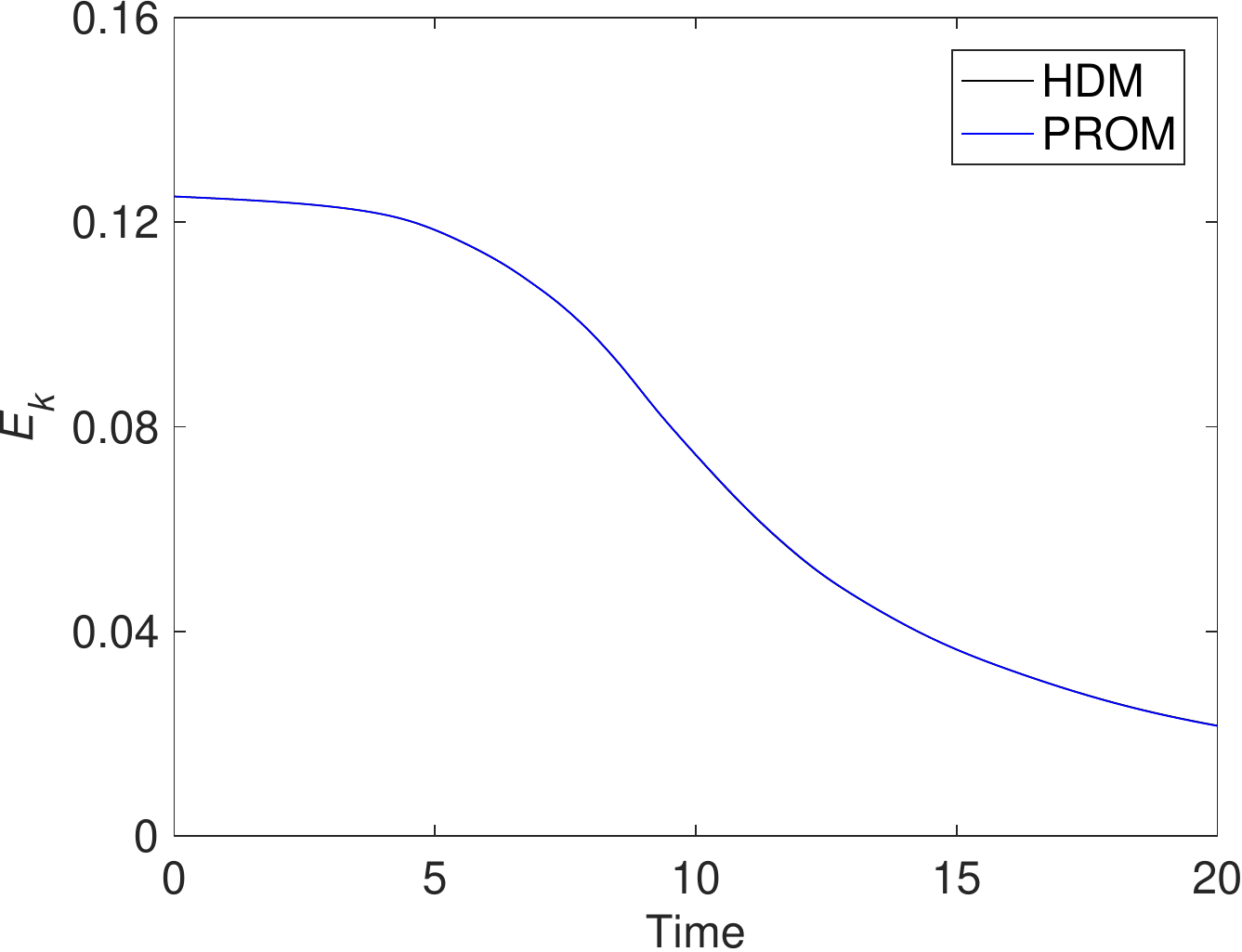}
	\caption{LSPG, $n = 81$}
	\end{subfigure}%
	\caption{Taylor-Green vortex: Time-histories of the turbulent kinetic energy computed using the HDM and LSPG-based Petrov-Galerkin PROMs.}
	\label{fig:tgvklspg}
\end{figure}

\begin{figure}[h!]
	\centering
	\begin{subfigure}[c]{0.3\textwidth}%
	\centering
	\includegraphics[width=\linewidth]{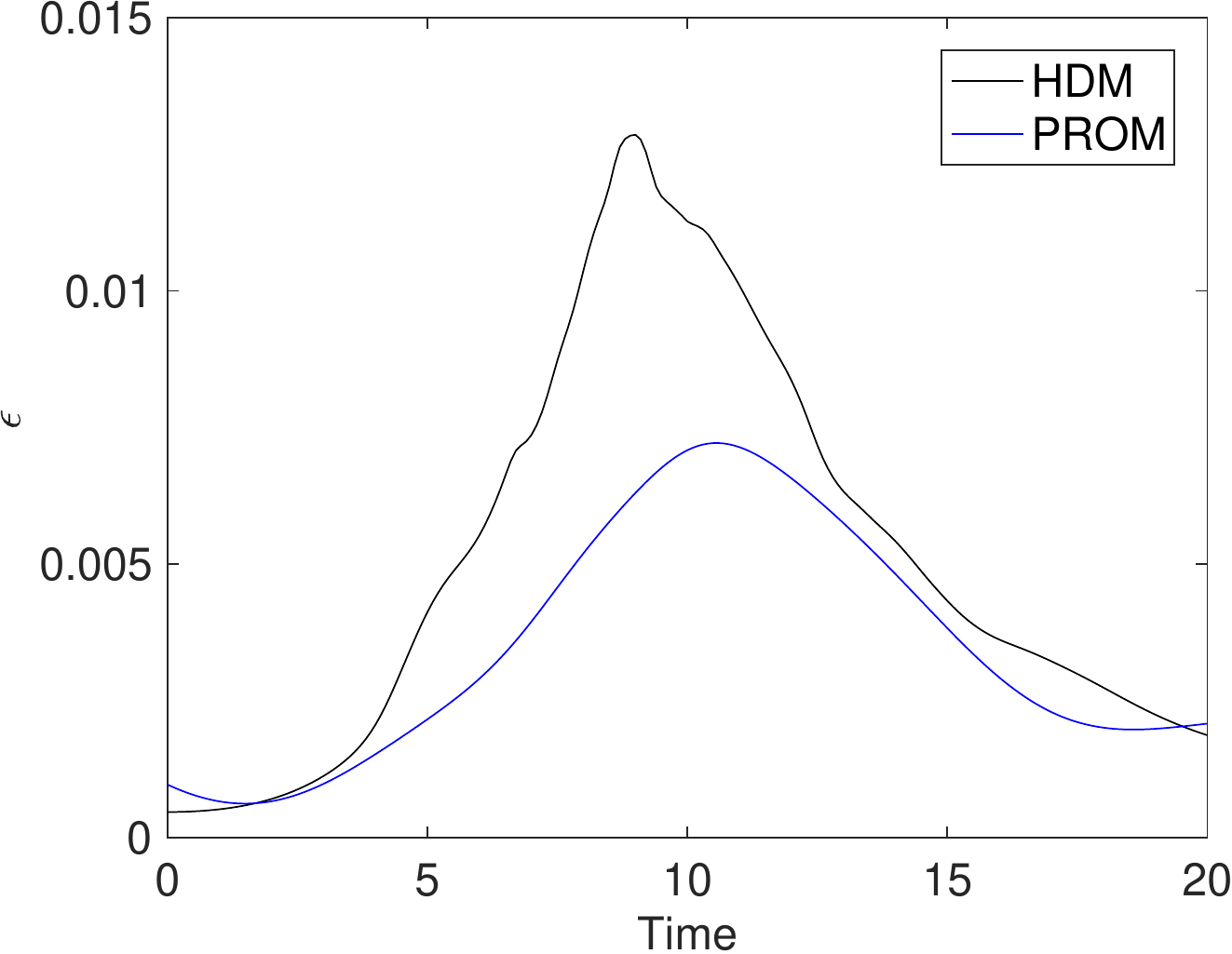}
	\caption{LSPG, $n = 6$}
	\end{subfigure}%
	\hspace{1em}
	\begin{subfigure}[c]{0.3\textwidth}%
	\centering
	\includegraphics[width=\linewidth]{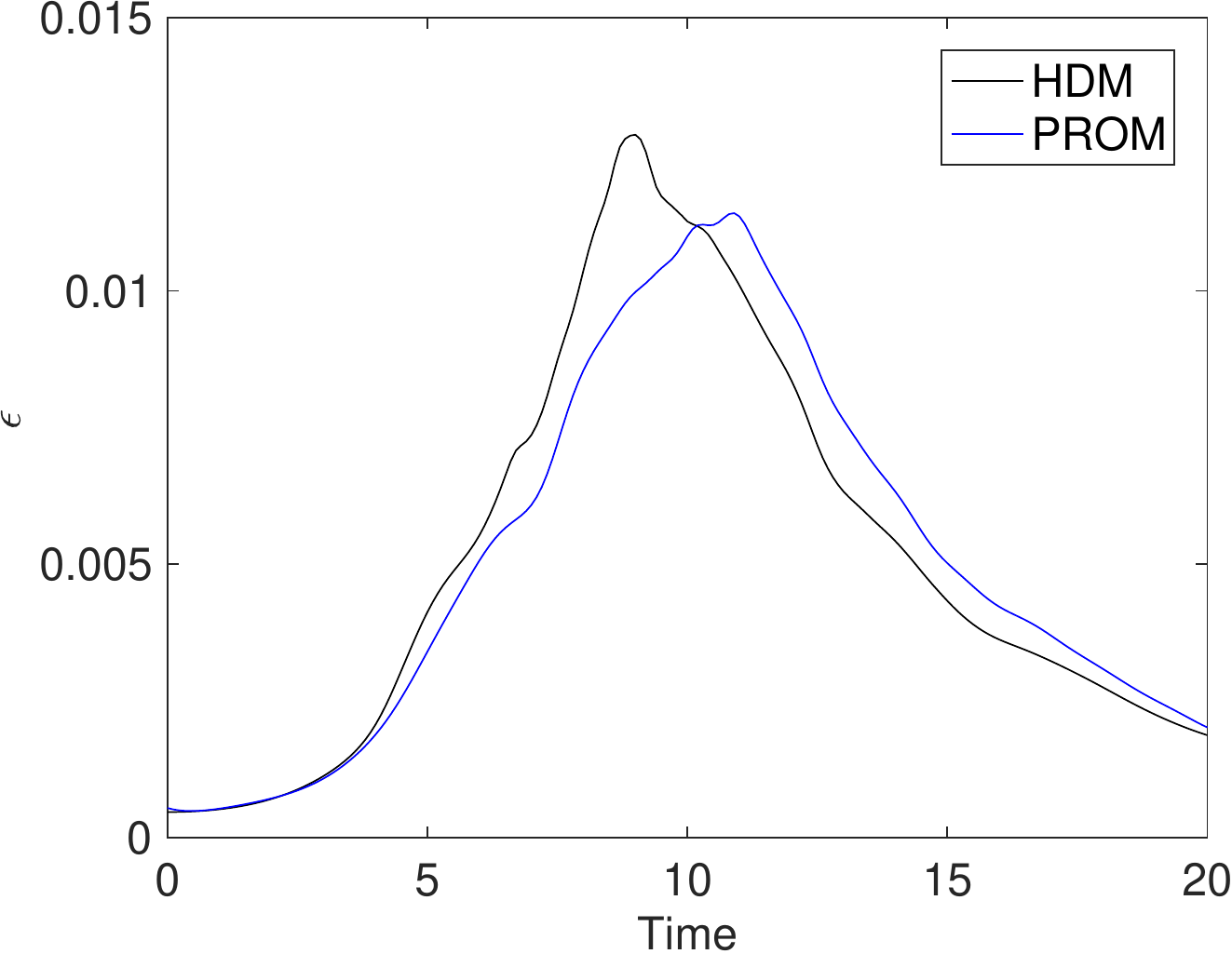}
	\caption{LSPG, $n = 22$}
	\end{subfigure}%
	
	\medskip
	\begin{subfigure}[c]{0.3\textwidth}%
	\centering
	\includegraphics[width=\linewidth]{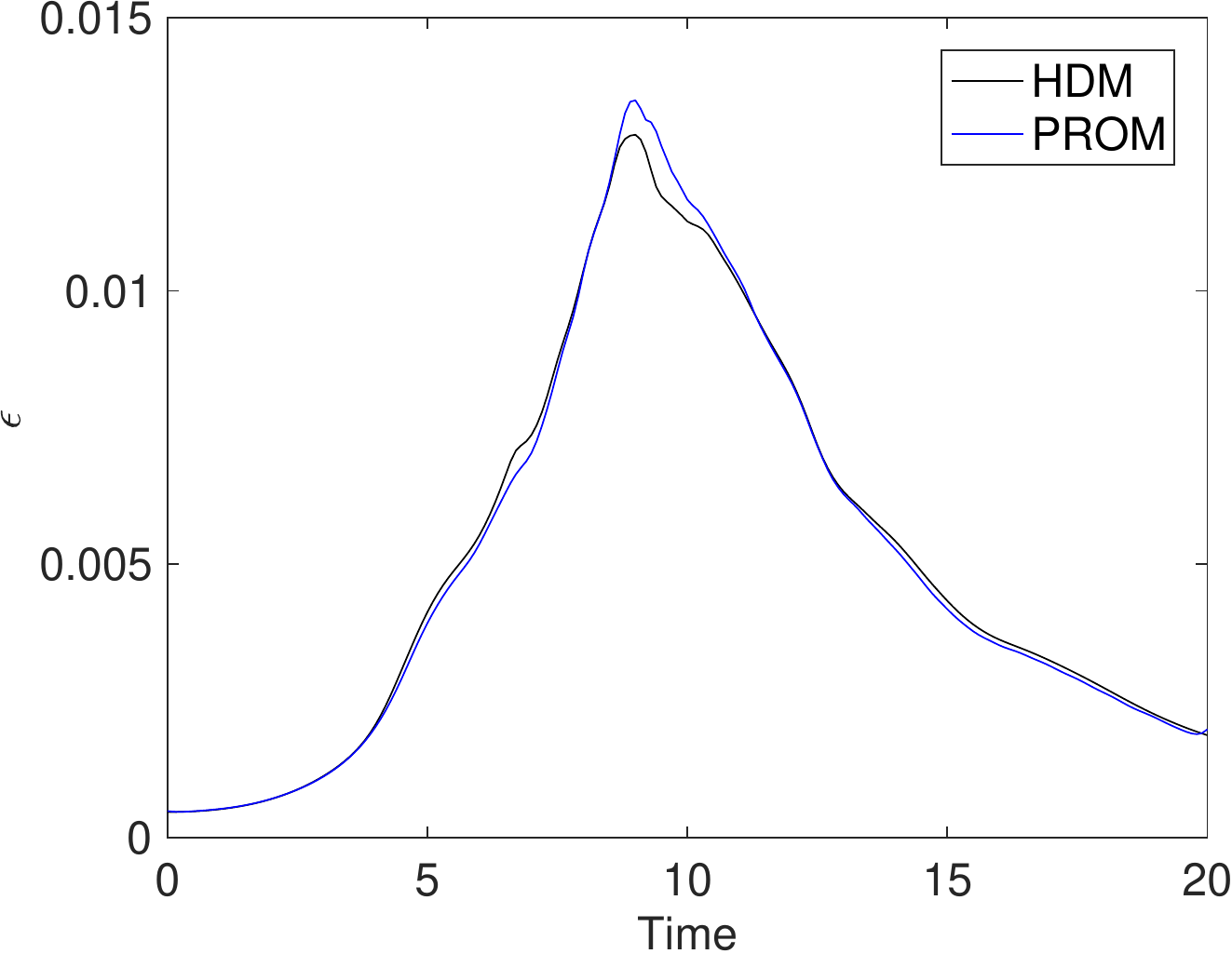}
	\caption{LSPG, $n = 47$}
	\end{subfigure}%
	\hspace{1em}
	\begin{subfigure}[c]{0.3\textwidth}%
	\centering
	\includegraphics[width=\linewidth]{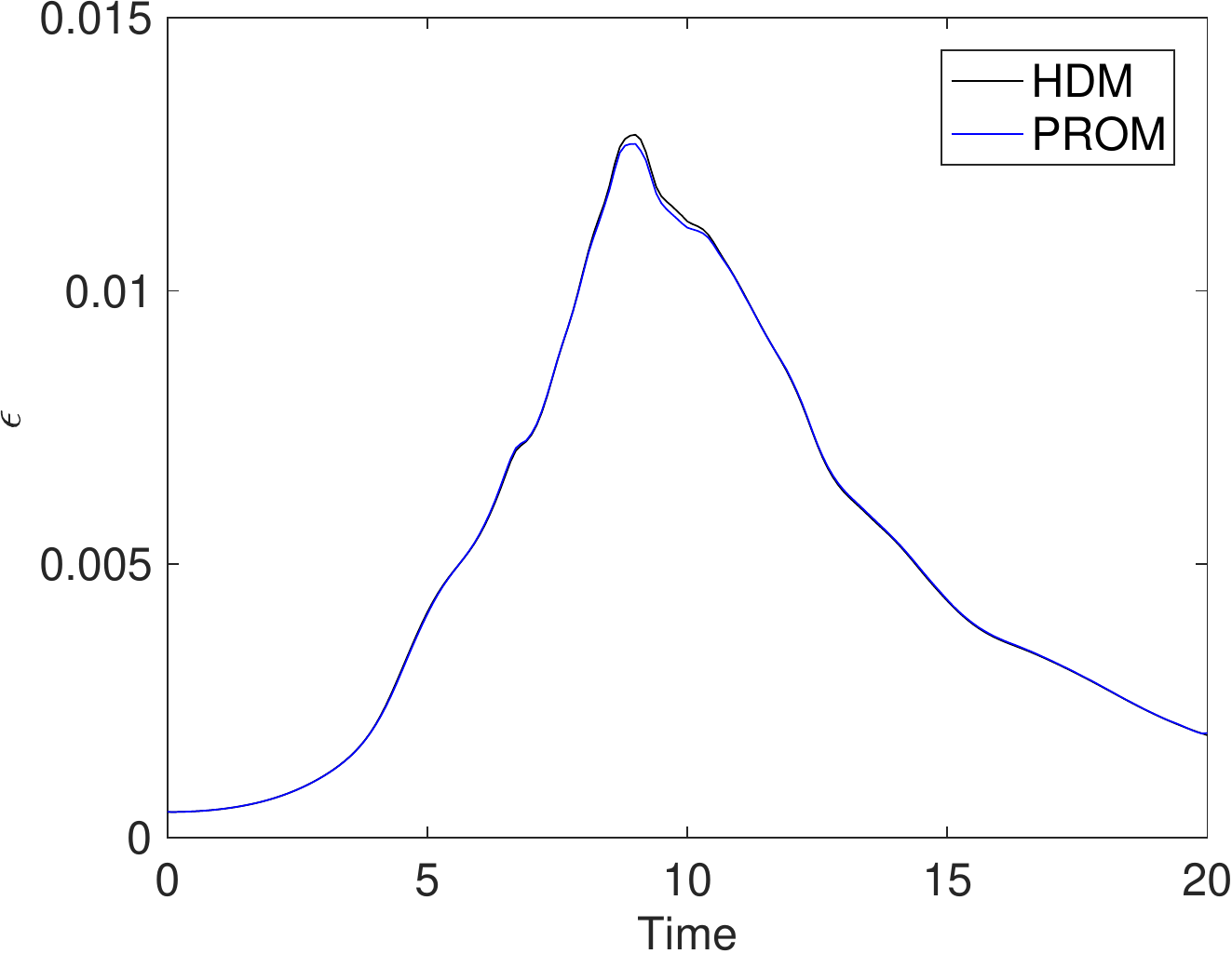}
	\caption{LSPG, $n = 81$}
	\end{subfigure}%
	\caption{Taylor-Green vortex: Time-histories of the enstrophy-based dissipation rate computed using the HDM and LSPG-based Petrov-Galerkin PROMs.}
	\label{fig:tgvepslspg}
\end{figure}

\begin{figure}[h!]
	\centering
	\begin{subfigure}[c]{0.3\textwidth}%
	\centering
	\includegraphics[width=\linewidth]{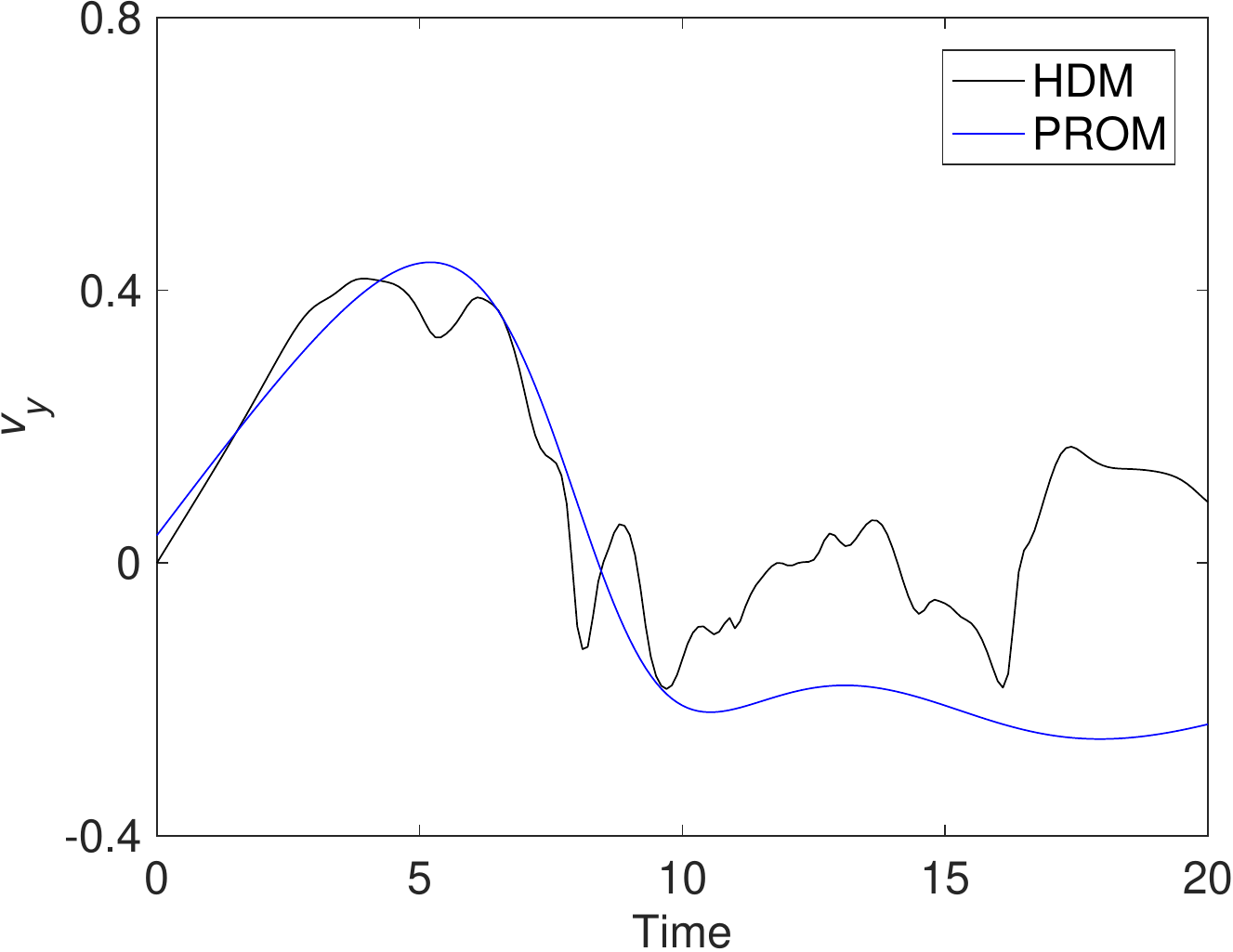}
	\caption{LSPG, $n = 6$}
	\end{subfigure}%
	\hspace{1em}
	\begin{subfigure}[c]{0.3\textwidth}%
	\centering
	\includegraphics[width=\linewidth]{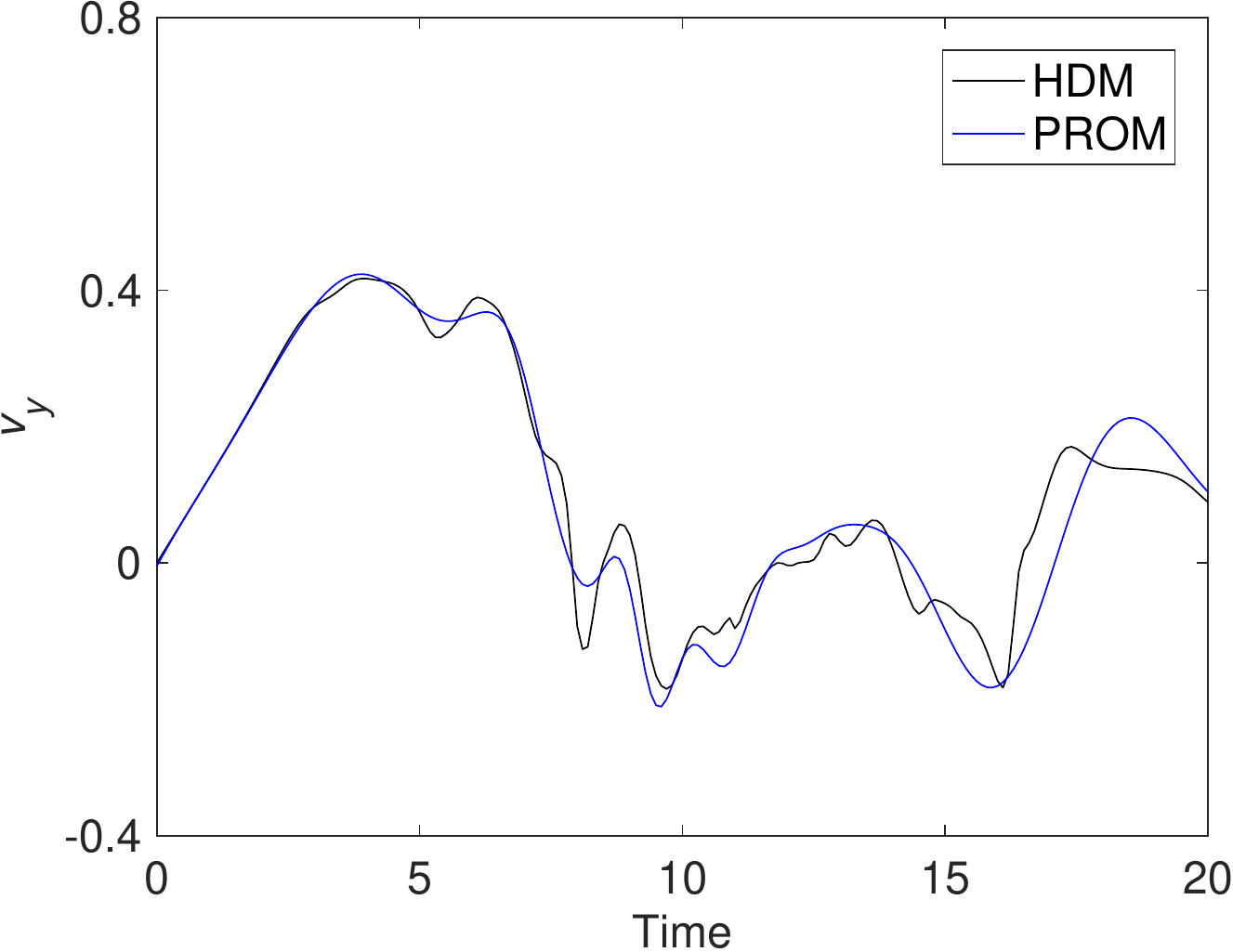}
	\caption{LSPG, $n = 22$}
	\end{subfigure}%
	
	\medskip
	\begin{subfigure}[c]{0.3\textwidth}%
	\centering
	\includegraphics[width=\linewidth]{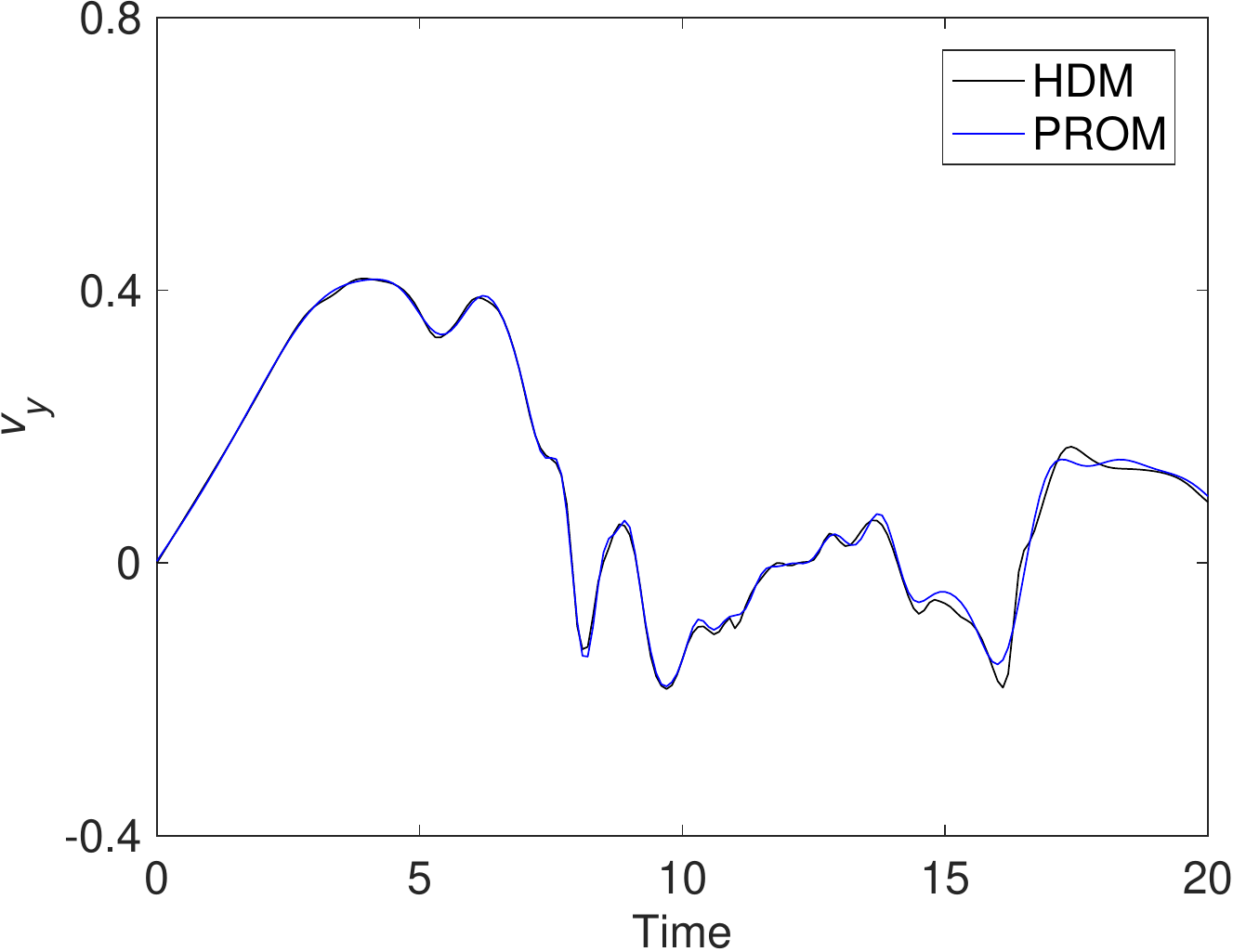}
	\caption{LSPG, $n = 47$}
	\end{subfigure}%
	\hspace{1em}
	\begin{subfigure}[c]{0.3\textwidth}%
	\centering
	\includegraphics[width=\linewidth]{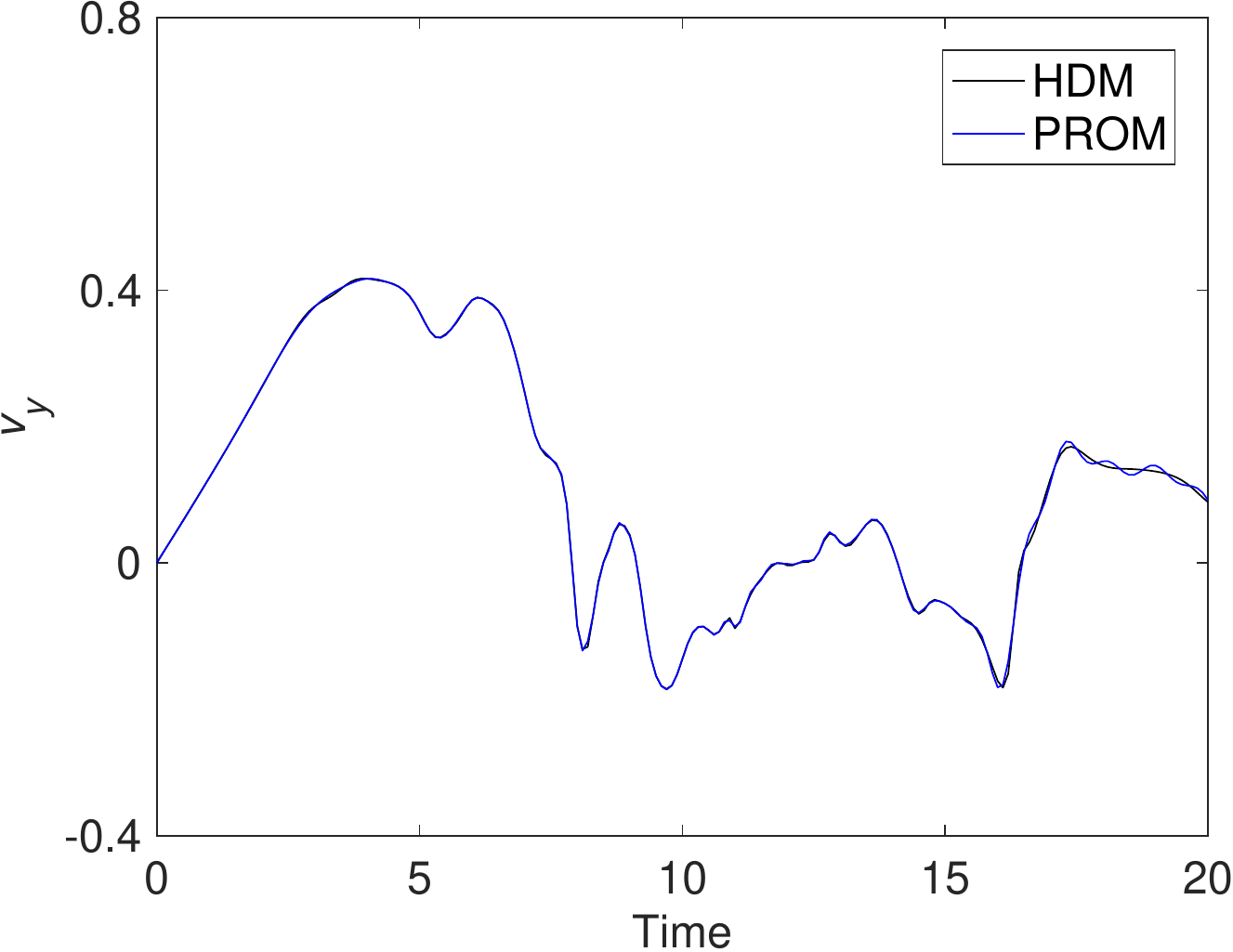}
	\caption{LSPG, $n = 81$}
	\end{subfigure}%
	\caption{Taylor-Green vortex: Time-histories of the velocity component $v_y$ computed at a probe using the HDM and LSPG-based Petrov-Galerkin PROMs.}
	\label{fig:tgvprobevylspg}
\end{figure}

\begin{figure}[h!]
	\centering
	\begin{subfigure}[c]{0.3\textwidth}%
	\centering
	\includegraphics[width=\linewidth]{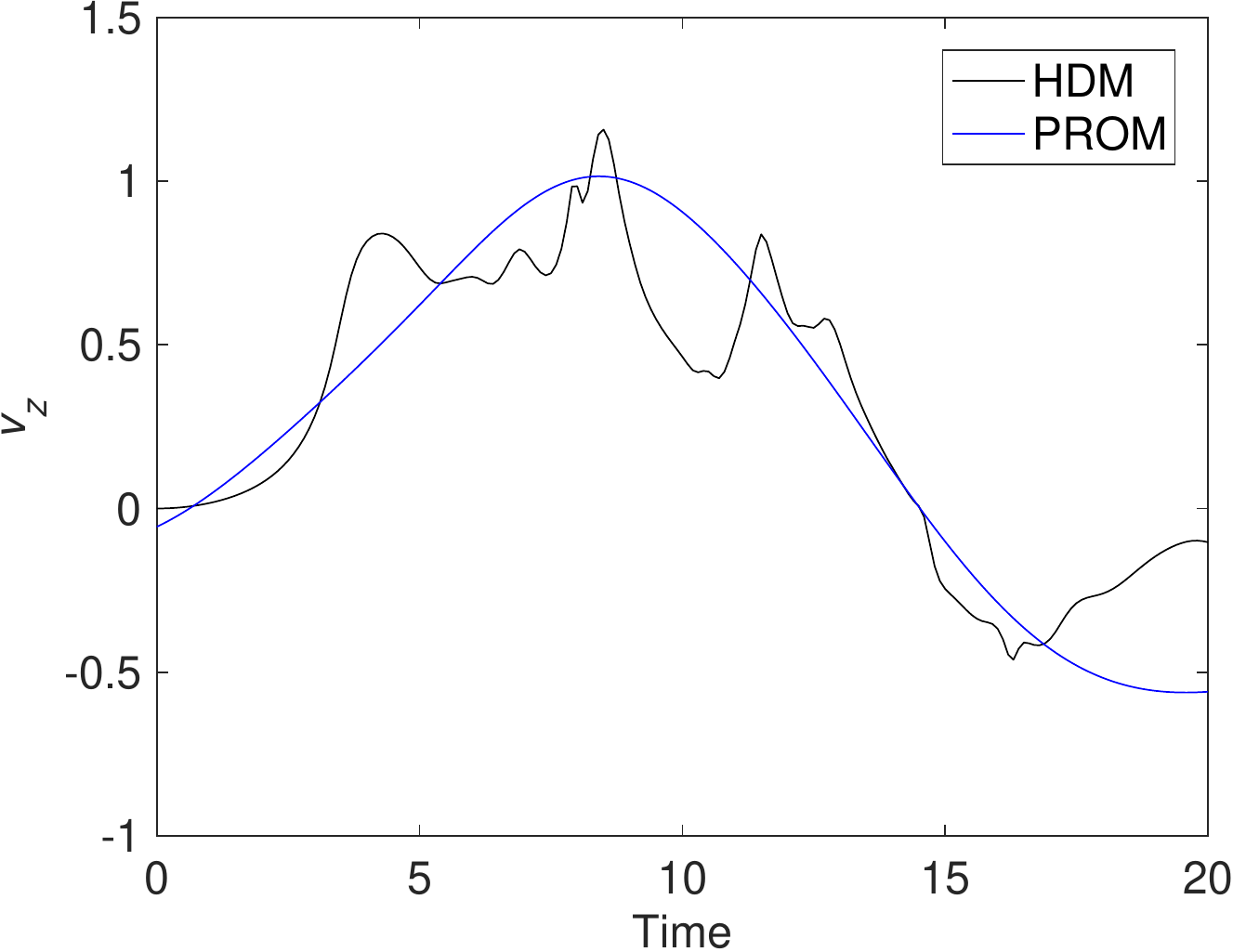}
	\caption{LSPG, $n = 6$}
	\end{subfigure}%
	\hspace{1em}
	\begin{subfigure}[c]{0.3\textwidth}%
	\centering
	\includegraphics[width=\linewidth]{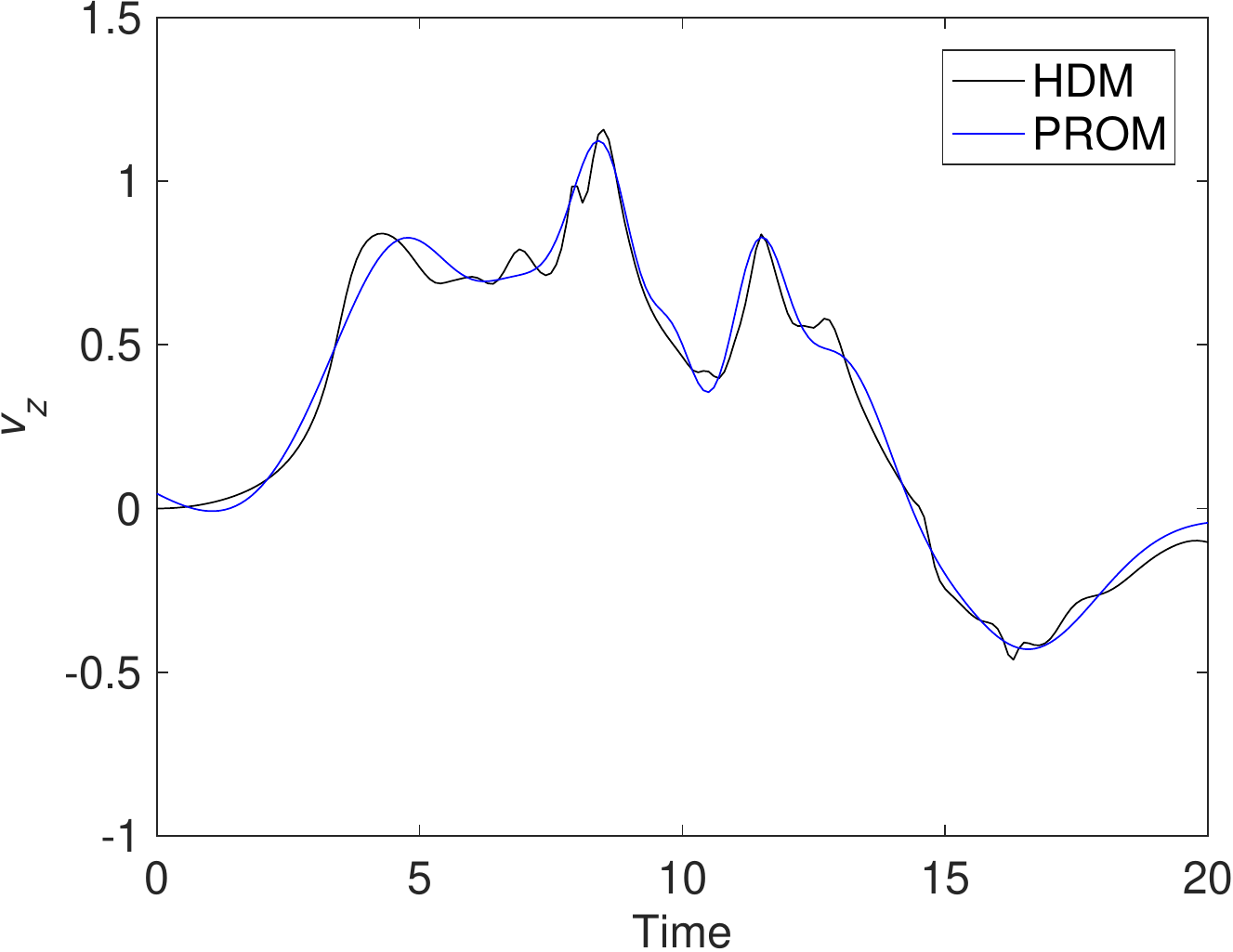}
	\caption{LSPG, $n = 22$}
	\end{subfigure}%
	
	\medskip
	\begin{subfigure}[c]{0.3\textwidth}%
	\centering
	\includegraphics[width=\linewidth]{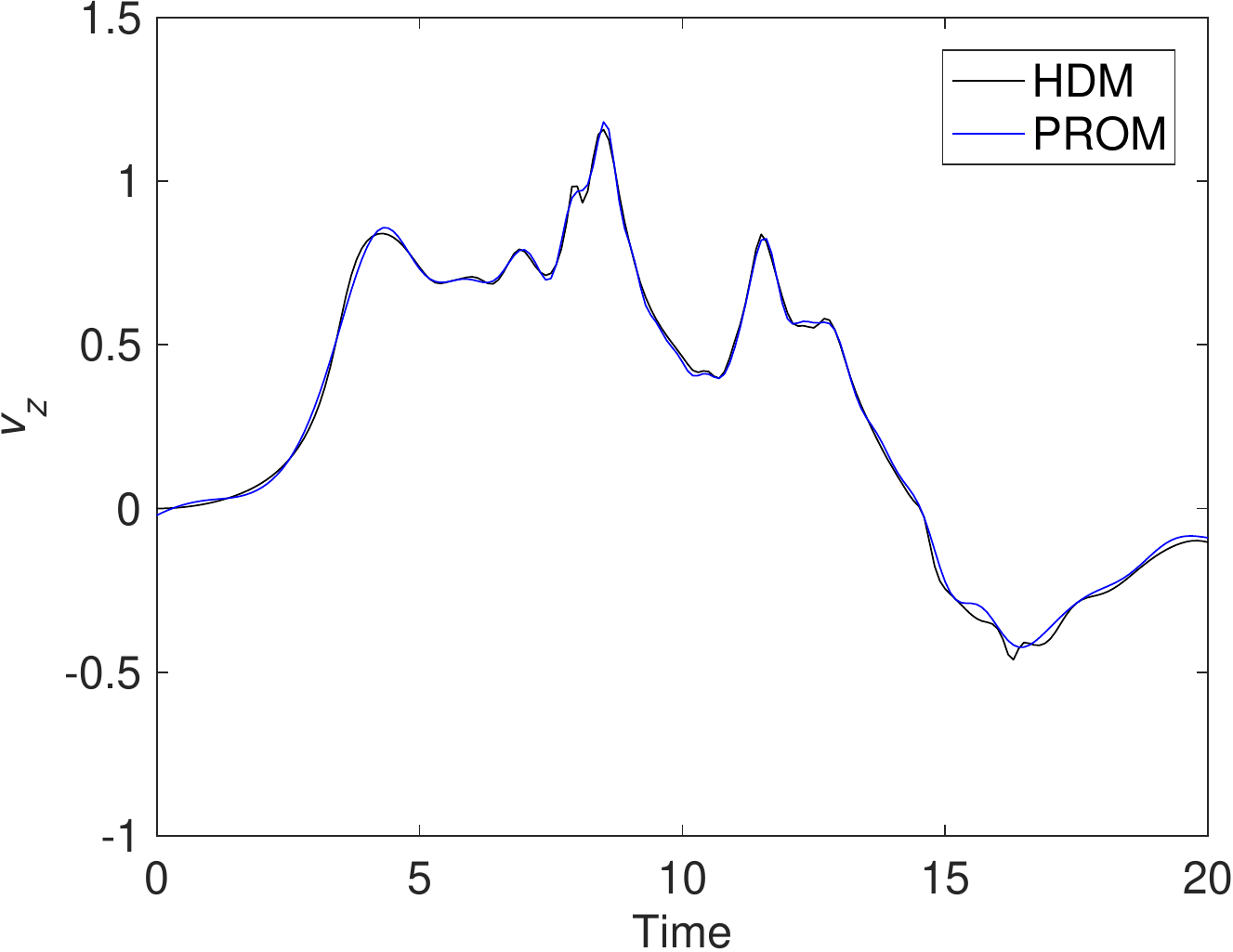}
	\caption{LSPG, $n = 47$}
	\end{subfigure}%
	\hspace{1em}
	\begin{subfigure}[c]{0.3\textwidth}%
	\centering
	\includegraphics[width=\linewidth]{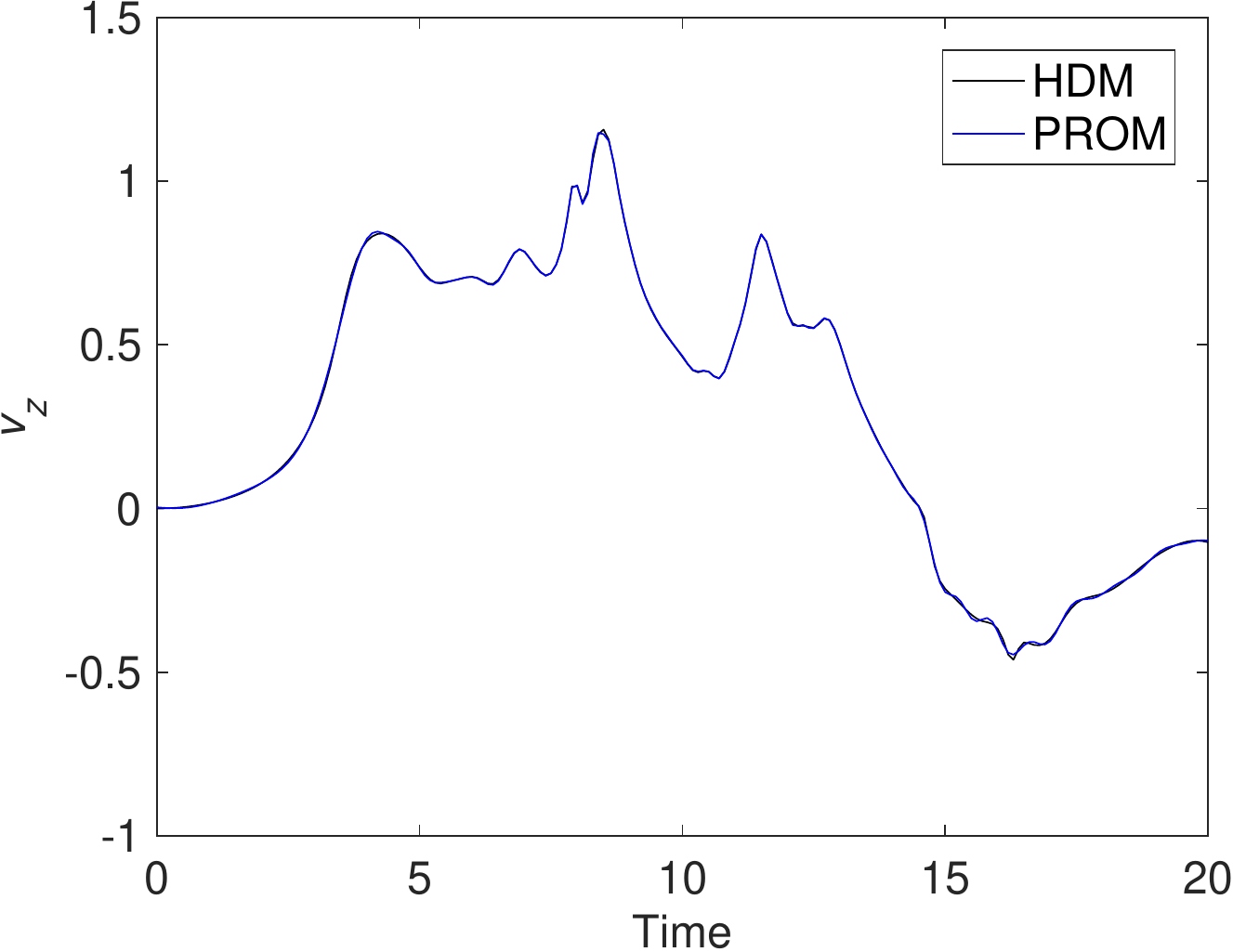}
	\caption{LSPG, $n = 81$}
	\end{subfigure}%
	\caption{Taylor-Green vortex: Time-histories of the velocity component $v_z$ computed at a probe using the HDM and LSPG-based Petrov-Galerkin PROMs.}
	\label{fig:tgvprobevzlspg}
\end{figure}

Tables \ref{tab:tgverr} and \ref{tab:tgvspeed} report on the relative errors and speed-up factors delivered by the various constructed PROMs. These tables show that all PROMs
deliver exceptional wall-clock time and CPU time speed-up factors that exceed in some cases six orders of magnitude. They also show that for a dimension as small as $n = 81$, both Galerkin and
LSPG-based Petrov-Galerkin PROMs achieve for all QoIs relative errors that are as low as $1\%$ (recall that for this problem, the dimension of the DNS HDM is $N = 402,653,184$).

\begin{table}[tb]
	\small
	\centering
	\caption{Taylor Green vortex: Variations with the dimension $n$ of the relative errors for the Galerkin and LSPG-based Petrov-Galerkin PROMs.}
	\begin{tabular}{lccccc}
		\toprule
		Model & $n$ & $\mathbb{RE}_{E_k}$ ($\%$) & $\mathbb{RE}_{\epsilon}$ ($\%$) & $\mathbb{RE}_{v_y}$ ($\%$) & $\mathbb{RE}_{v_z}$ ($\%$) \\ \midrule
		Galerkin PROM & $6$ & $17.3$ & $38.2$ & $96.2$ & $39.3$ \\
		& $22$ & $2.97$ & $17.0$ & $24.4$ & $9.44$ \\
		& $47$ & $0.441$ & $4.06$ & $5.04$ & $3.30$ \\
		& $81$ & $0.067$ & $0.520$ & $2.00$ & $0.892$ \\ \midrule
		LSPG-based Petrov-Galerkin PROM & $6$ & $16.1$ & $39.2$ & $91.2$ & $40.0$ \\
		& $22$ & $2.78$ & $15.8$ & $22.6$ & $9.16$ \\
		& $47$ & $0.437$ & $3.40$ & $4.90$ & $3.15$ \\
		& $81$ & $0.045$ & $0.706$ & $1.97$ & $0.875$ \\ \bottomrule
	\end{tabular}
	\label{tab:tgverr}
\end{table}

\begin{table}[tb]
	\small
	\centering
	\caption{Taylor Green vortex: Variations with the dimension $n$ of the speed-up factors delivered by the Galerkin and LSPG-based Petrov-Galerkin PROMs.}
	\begin{tabular}{lccc}
		\toprule
		\multirow{2}{*}{Model} & \multirow{2}{*}{$n$} & Wall-clock time & CPU time \\ 
		& & speed-up factor & speed-up factor \\ \midrule
		Galerkin PROM & $6$ & $2.50 \times 10^4$ & $3.19 \times 10^6$ \\
		& $22$ & $1.08 \times 10^4$ & $1.38 \times 10^6$ \\
		& $47$ & $4.09 \times 10^3$ & $5.24 \times 10^5$ \\
		& $81$ & $1.80 \times 10^3$ & $2.31 \times 10^5$ \\ \midrule
		LSPG-based Petrov-Galerkin PROM & $6$ & $2.11 \times 10^4$ & $2.70 \times 10^6$ \\
		& $22$ & $8.48 \times 10^3$ & $1.09 \times 10^6$ \\
		& $47$ & $1.82 \times 10^3$ & $2.33 \times 10^5$ \\
		& $81$ & $2.58 \times 10^2$ & $3.31 \times 10^4$ \\ \bottomrule
	\end{tabular}
	\label{tab:tgvspeed}
\end{table}

\subsection{Large eddy simulation of a turbulent compressible flow over a NACA airfoil}
\label{sec:naca}

The final example considered here demonstrates the PMOR of a compressible LES model for a flow over a NACA 0012 airfoil at $M_{\infty} = 0.2$, $30^\circ$ angle of attack, and $Re = 10,000$.
At these flow conditions, massive flow separation and highly turbulent wake dynamics occur, which poses a formidable challenge for PMOR.

\subsubsection{High-dimensional model}

As in the case of the example problem discussed in Section \ref{sec:cylinder}, the HDM is constructed here using a mixed finite volume/finite element semi-discretization of the nondimensional form
of the 3D compressible Navier-Stokes equations. However, the semi-discretization of the convective fluxes is performed here using a low-diffusion, fifth-order, vertex-based finite volume scheme 
\cite{debiez2000, camarri2004} that has more desirable dissipation properties for LES than lower-order counterpart schemes. The Vreman subgrid-scale model \cite{vreman2004} is chosen due to its 
improved treatment of laminar and transitional regions and its demonstrated ability to perform nearly as well as more complex and computationally intensive dynamic models. The semi-discretization of 
the LES subgrid-scale model and that of the standard Navier-Stokes diffusive fluxes are performed using a piecewise linear Galerkin FE method.

For this problem, the computational domain is constructed by extruding the NACA 0012 airfoil profile one chord length in the spanwise direction and surrounding it by a right circular cylinder
(see Figure \ref{fig:naca}(a)). Periodic boundary conditions are applied on the spanwise faces of this domain which are flushed with the lateral faces of the extruded airfoil
in order to allow for sufficient flow three-dimensionality and not inhibit the development 
of turbulent structures, and a no-slip adiabatic wall boundary condition is applied on the surface of the airfoil. The computational domain is discretized by an unstructured mesh with $2,079,546$ 
vertices and $11,909,406$ tetrahedral elements, which leads to an LES HDM of dimension $N = 10,397,730$. Figure \ref{fig:naca}(a) shows a cross-section of the discretized computational domain.
For the fully developed, statistically stationary flow, the mean value of $y^+$ computed for the first nearest layer of cells to the airfoil surface is approximately $0.68$, and its worst case value
is $1.92$. This suggests that the unstructured mesh has an appropriate resolution in the wall-normal direction within the boundary layer. The first layer of elements in the boundary layer 
has a worst case aspect ratio of 5, which indicates that the unstructured mesh has also the appropriate mesh resolution in the tangential direction to the wall surface.

Time-discretization of the HDM described above is performed using the $L$-stable third-order DIRK scheme defined by the Butcher tableau given in \ref{app:dirk}, and the fixed time-step
$\Delta t = 2 \times 10^{-2}$ (CFL $\approx 900$ for the given mesh). This time-integration scheme was shown in \cite{pazner2017} to perform favorably for a variety of fluid flow problems.
The nonlinear system of equations resulting at each Runge-Kutta stage is solved in the same manner as in the example problem discussed in Section \ref{sec:cylinder}, except that an ILU(0) factorization 
is used in this case for constructing the RAS preconditioner because this setting turns out in this case to be more computationally efficient. 

For all simulations discussed below, the initial flow condition is computed by impulsively starting the flow from a uniform state then time-integrating its dynamics until the flow becomes fully 
developed and reaches a statistical steady state. Using this initial condition, all simulations reported below are performed in the nondimensional time-interval $t \in [0, 30]$.
Figure \ref{fig:naca}(b) shows instantaneous isosurfaces of the HDM-based vorticity solution magnitude colored by the velocity magnitude at the end of the simulation time-interval $t = 30$.
The HDM-based solution of this problem is completed in 39.1 hours wall-clock time on 240 cores of the Linux cluster.

\begin{figure}[h!]
	\centering
	\begin{subfigure}[c]{0.275\textwidth}%
	\centering
	\includegraphics[width=\linewidth]{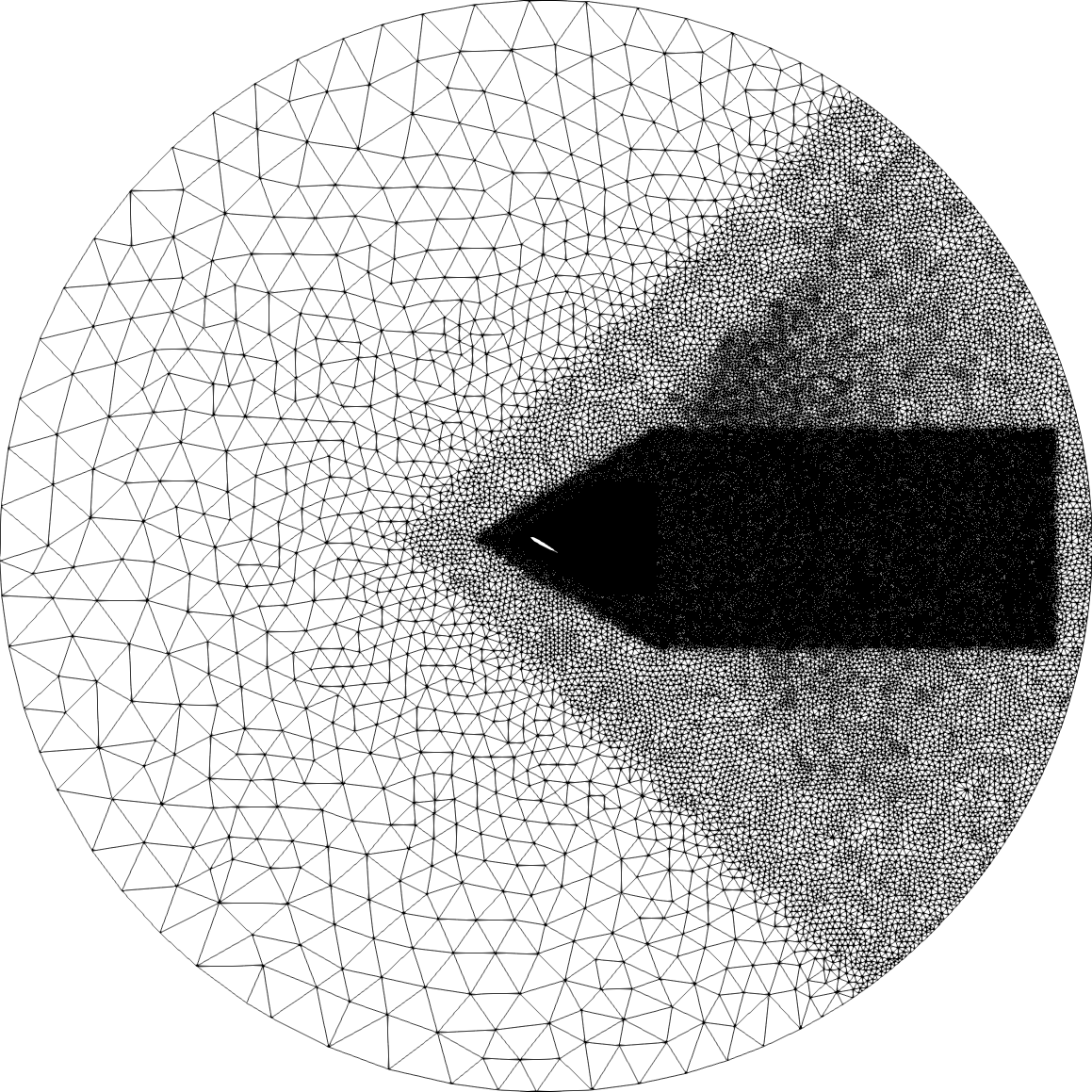}
	\caption{}
	\end{subfigure}%
	\hspace{2em}
	\begin{subfigure}[c]{0.60\textwidth}%
	\centering
	\includegraphics[width=\linewidth]{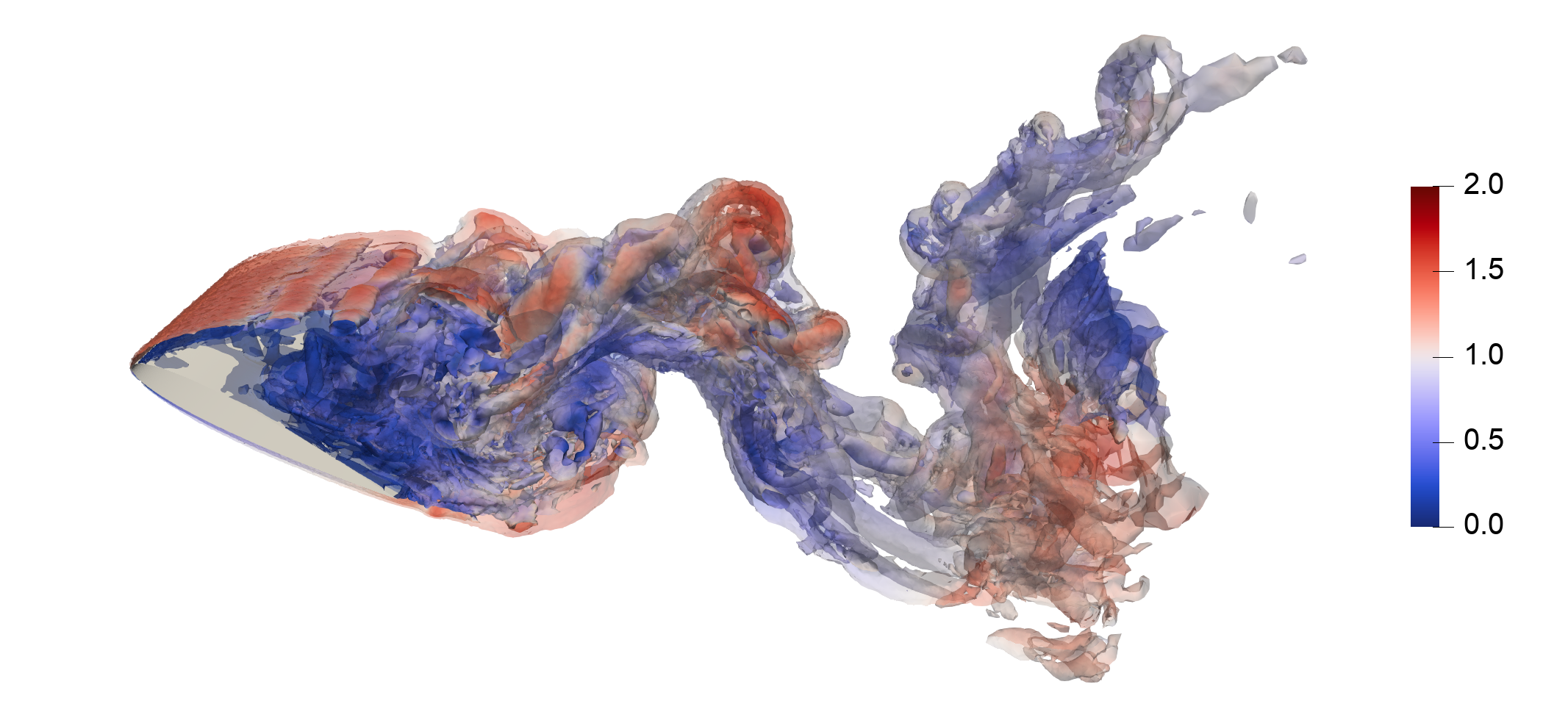}
	\caption{}
	\end{subfigure}%
	\caption{LES of a flow over a NACA airfoil: (a) Cross-section of the discretized computational domain; (b) HDM-based isosurfaces of the vorticity magnitude colored by the velocity magnitude 
	at $t = 30$.}
	\label{fig:naca}
\end{figure}

\subsubsection{Projection-based reduced-order models with hyperreduction}

Using the nondimensional sampling rate defined by $\Delta s = 6 \times 10^{-2}$, 501 solution snapshots are collected during the HDM-based simulation and compressed by SVD.
The offset of the affine subspace approximation $\mathbf{u}_0$ is chosen as the initial flow condition. Three right ROBs of dimension $n = 41$, $n = 83$, and $n = 213$ corresponding
to capturing $90\%$, $95\%$, and $99\%$ of the energy of the singular values of the snapshot matrix, respectively, and corresponding Galerkin and Petrov-Galerkin PROMs are constructed. 
As in the case of the example discussed in Section \ref{sec:cylinder}, these PROMs must be hyperreduced to achieve computational efficiency. Again, ECSW is employed for this purpose
equipped with every 10-$th$ collected solution snapshot for training the mesh sampling, and with the training tolerance $\varepsilon = 1\times10^{-2}$.

All constructed PROMs and HPROMs are time-discretized using the same third-order DIRK scheme used for time-discretizing the HDM, and the same fixed nondimensional time-step $\Delta t = 2 \times 10^{-2}$.
The PROM-based simulations are performed on 240 cores across 20 compute nodes of the Linux cluster -- as in the case of the HDM-based simulation -- but the HPROM-based simulations are 
performed on just 12 cores of a single compute node.

For this example application, the accuracy of each constructed PROM and HPROM is assessed using the time-dependent lift, drag, and side force coefficients $c_L$, $c_D$, and $c_Y$, respectively, as well
as the streamwise velocity $v_x$ and pressure $p$ computed at a probe in the wake of the flow. The probe is located 1.5 chord lengths downstream of the airfoil half-chord, in the spanwise center of the 
computational domain. The relative errors associated with the predictions of each of the aforementioned QoIs using the constructed PROMs and HPROMs are computed using
the nondimensional sampling rate defined by $\Delta s_{\mathbb{RE}} = \Delta t = 2 \times 10^{-2}$.

\paragraph{Galerkin reduced-order models}

In Figures \ref{fig:nacaliftgal}, \ref{fig:nacadraggal}, and \ref{fig:nacacrossgal}, the time-histories of the lift, drag, and side force coefficients computed using the HDM and the various 
constructed Galerkin PROMs and HPROMs are compared. Figures \ref{fig:nacaprobevxgal} and \ref{fig:nacaprobepgal} report on similar comparisons for the predictions of the streamwise velocity and pressure 
computed at the probe location. These figures show that every constructed Galerkin PROM and HPROM fails to complete the simulation over the entire specified nondimensional time-interval, due to
encounters with nonlinear numerical instabilities in the form of a negative pressure or density before reaching $t_{max}$. Increasing the size of the PROM dimension appears to improve numerical
stability: However, even the largest constructed Galerkin PROM ($n = 213$) fails to maintain numerical stability across the entirety of the specified simulation time-interval.

\begin{figure}[h!]
	\centering
	\begin{subfigure}[c]{0.3\textwidth}%
	\centering
	\includegraphics[width=\linewidth]{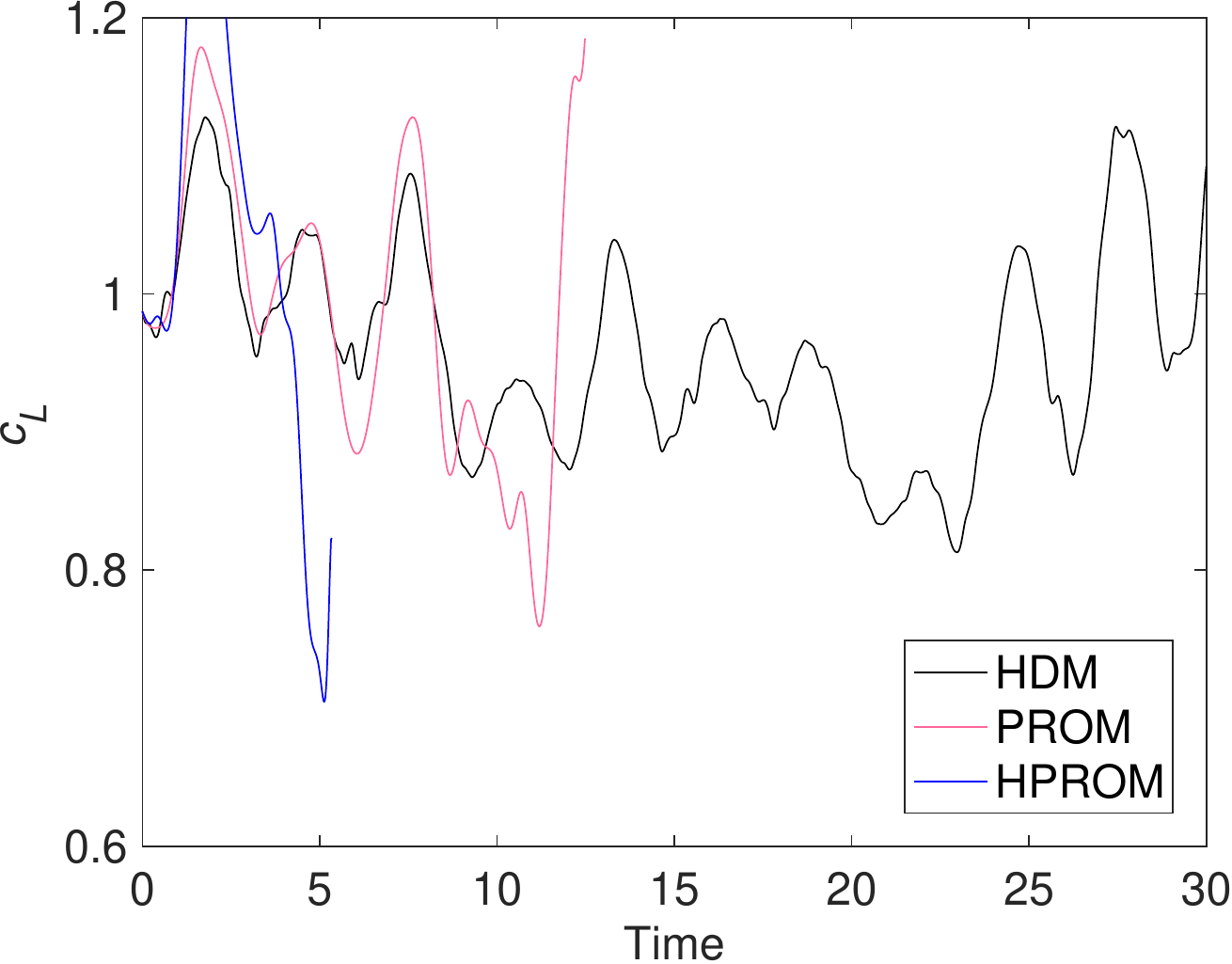}
	\caption{Galerkin, $n = 41$}
	\end{subfigure}%
	\hspace{1em}
	\begin{subfigure}[c]{0.3\textwidth}%
	\centering
	\includegraphics[width=\linewidth]{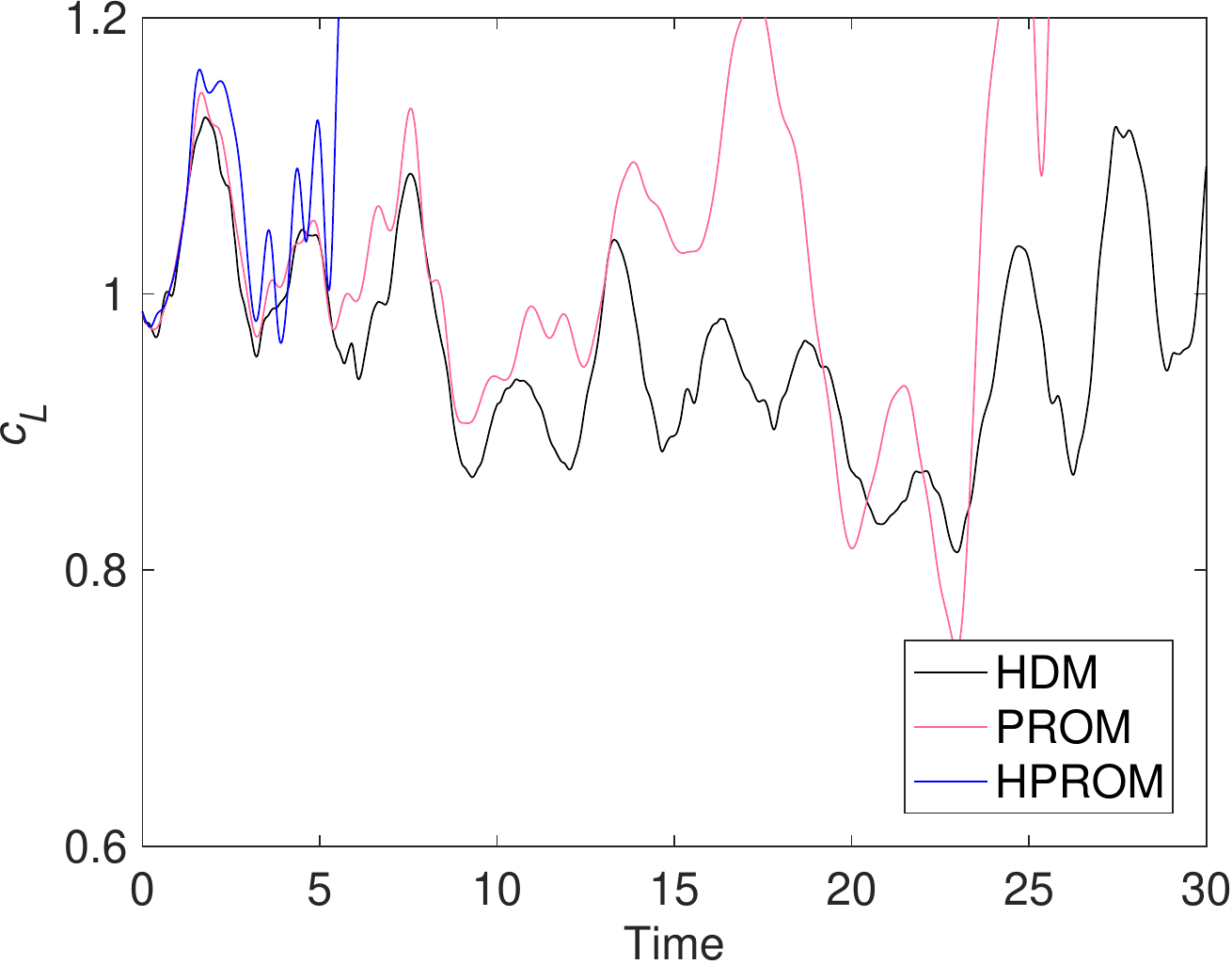}
	\caption{Galerkin, $n = 83$}
	\end{subfigure}%
	\hspace{1em}
	\begin{subfigure}[c]{0.3\textwidth}%
	\centering
	\includegraphics[width=\linewidth]{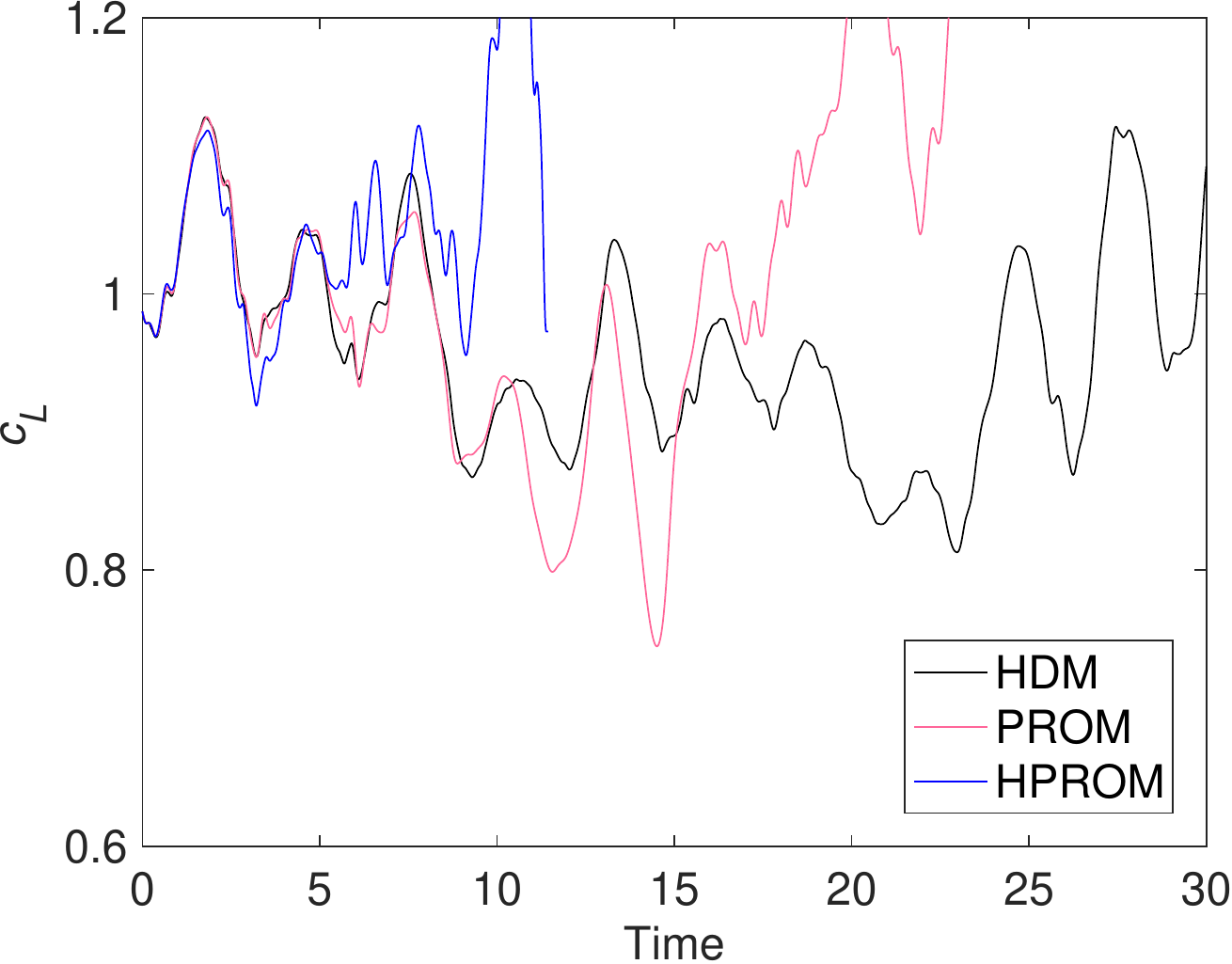}
	\caption{Galerkin, $n = 213$}
	\end{subfigure}%
	\caption{LES of a flow over a NACA airfoil: Time-histories of the lift coefficient computed using the HDM and Galerkin PROMs and HPROMs.}
	\label{fig:nacaliftgal}
\end{figure}

\begin{figure}[h!]
	\centering
	\begin{subfigure}[c]{0.3\textwidth}%
	\centering
	\includegraphics[width=\linewidth]{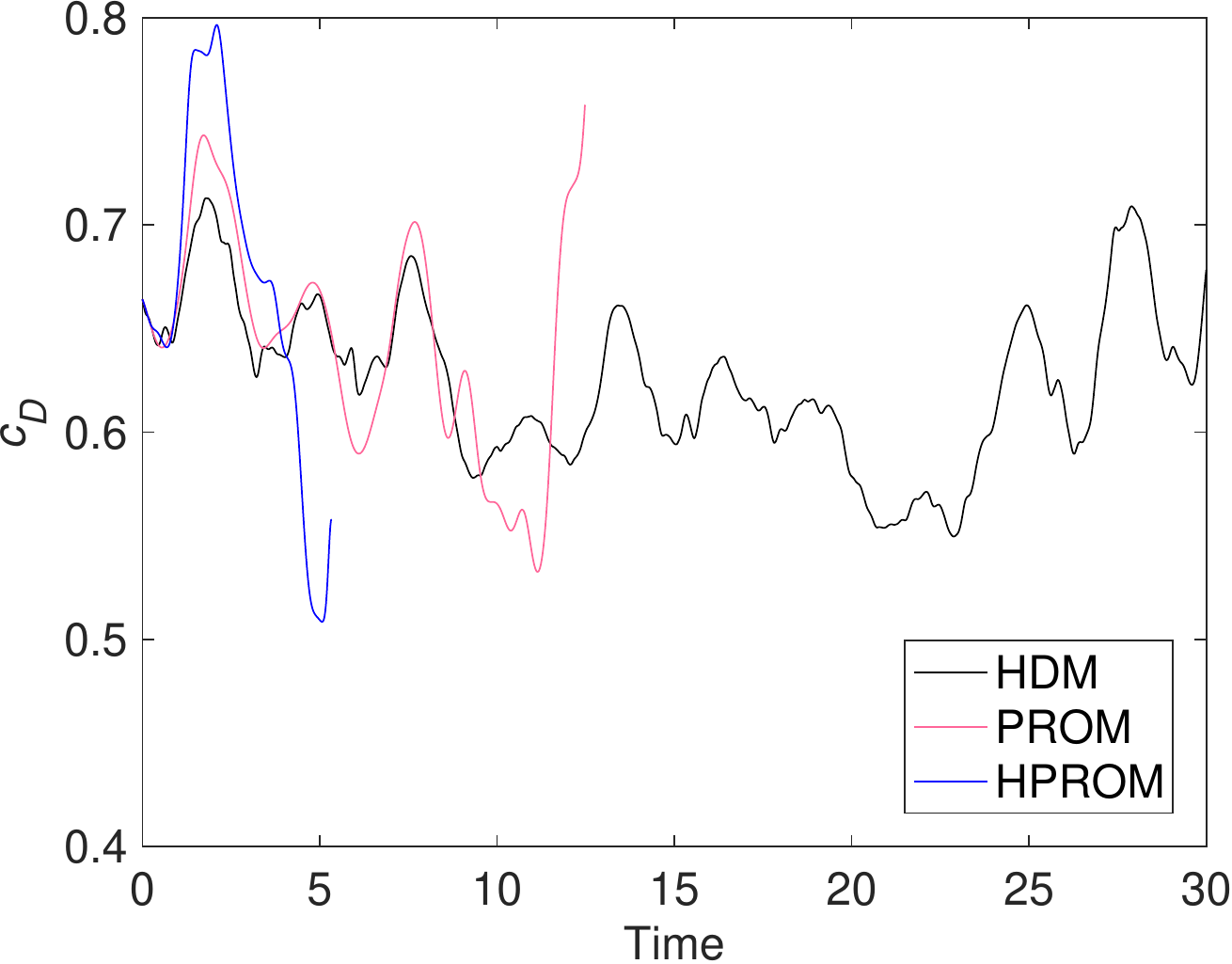}
	\caption{Galerkin, $n = 41$}
	\end{subfigure}%
	\hspace{1em}
	\begin{subfigure}[c]{0.3\textwidth}%
	\centering
	\includegraphics[width=\linewidth]{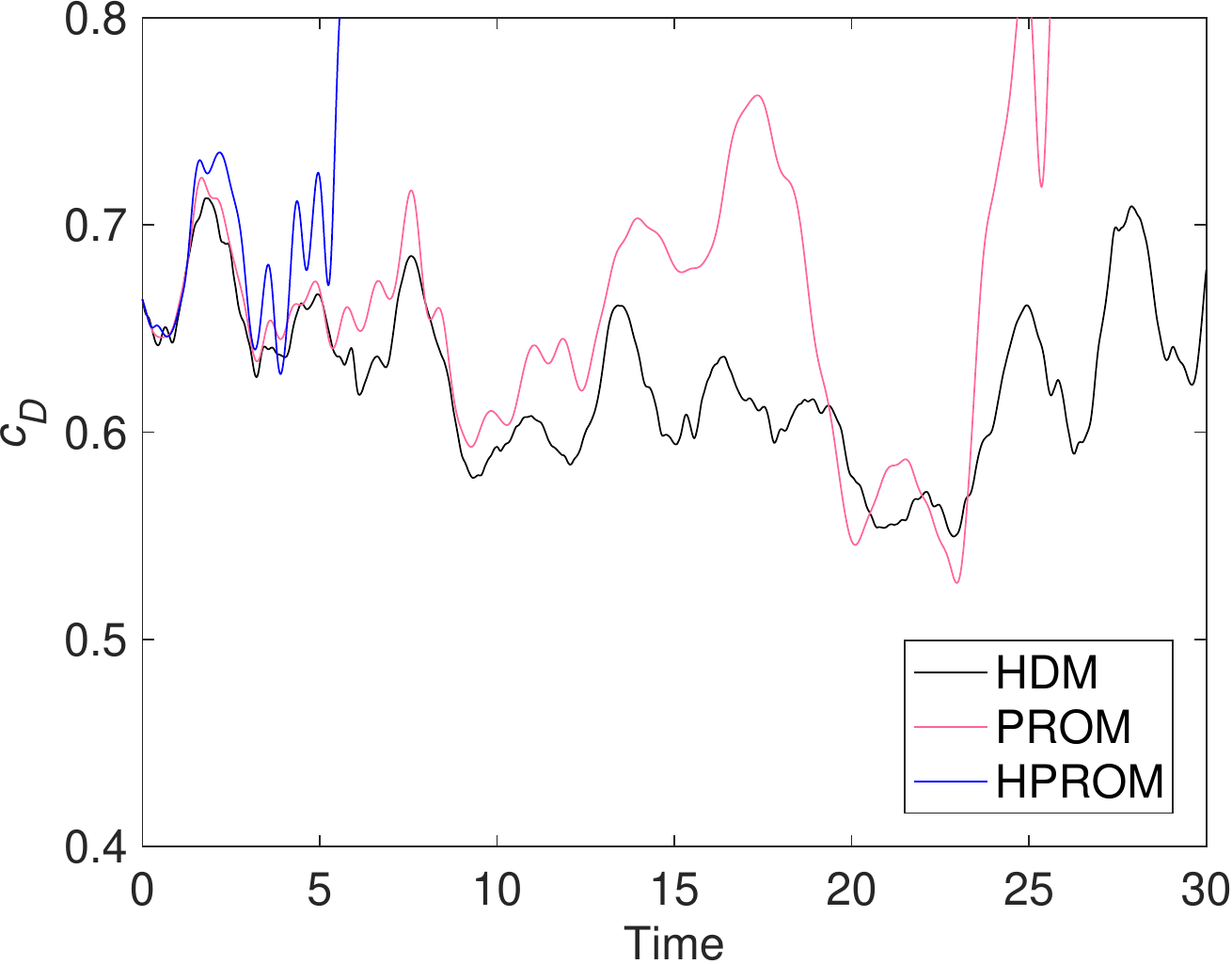}
	\caption{Galerkin, $n = 83$}
	\end{subfigure}%
	\hspace{1em}
	\begin{subfigure}[c]{0.3\textwidth}%
	\centering
	\includegraphics[width=\linewidth]{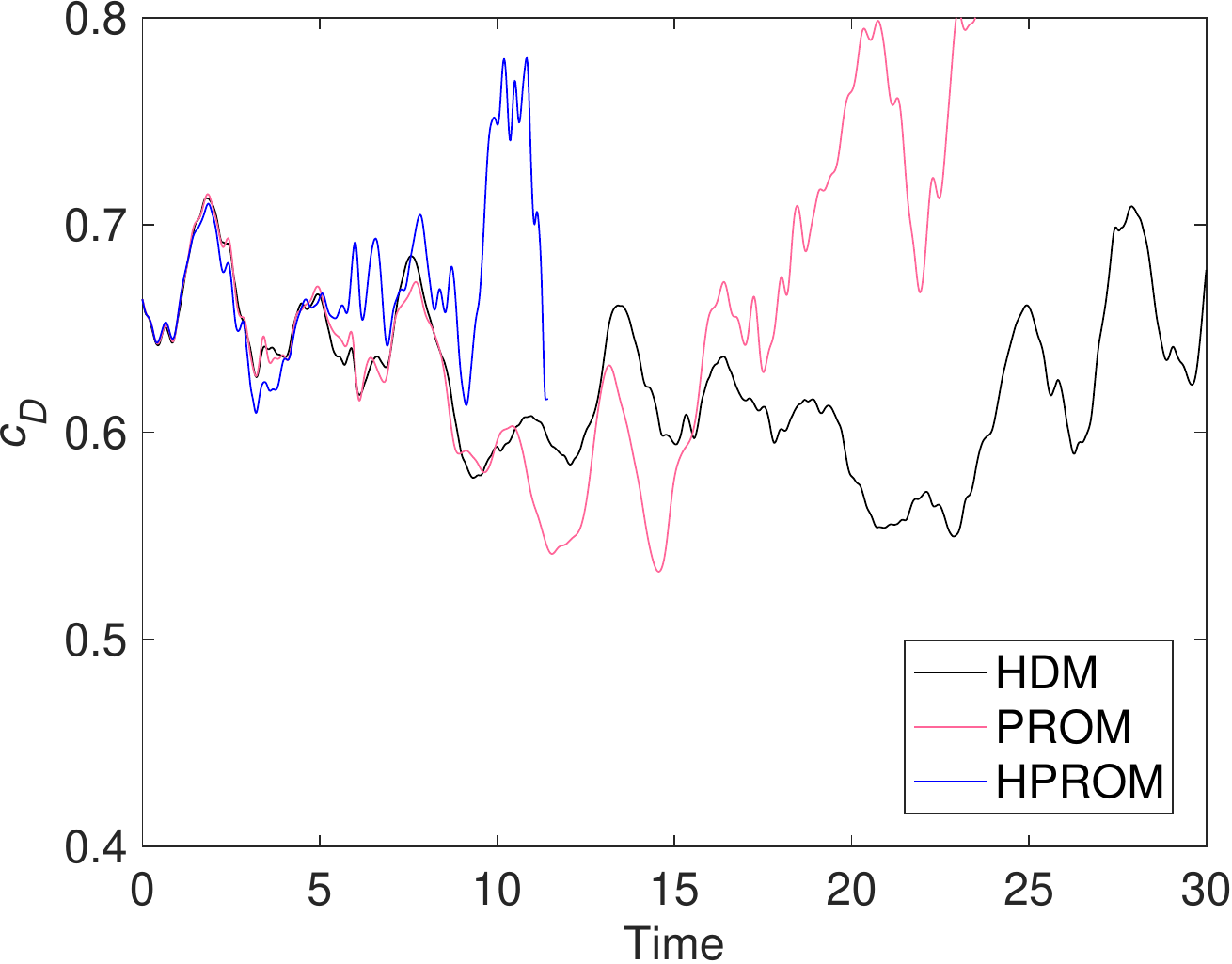}
	\caption{Galerkin, $n = 213$}
	\end{subfigure}%
	\caption{LES of a flow over a NACA airfoil: Time-histories of the drag coefficient computed using the HDM and Galerkin PROMs and HPROMs.}
	\label{fig:nacadraggal}
\end{figure}

\begin{figure}[h!]
	\centering
	\begin{subfigure}[c]{0.3\textwidth}%
	\centering
	\includegraphics[width=\linewidth]{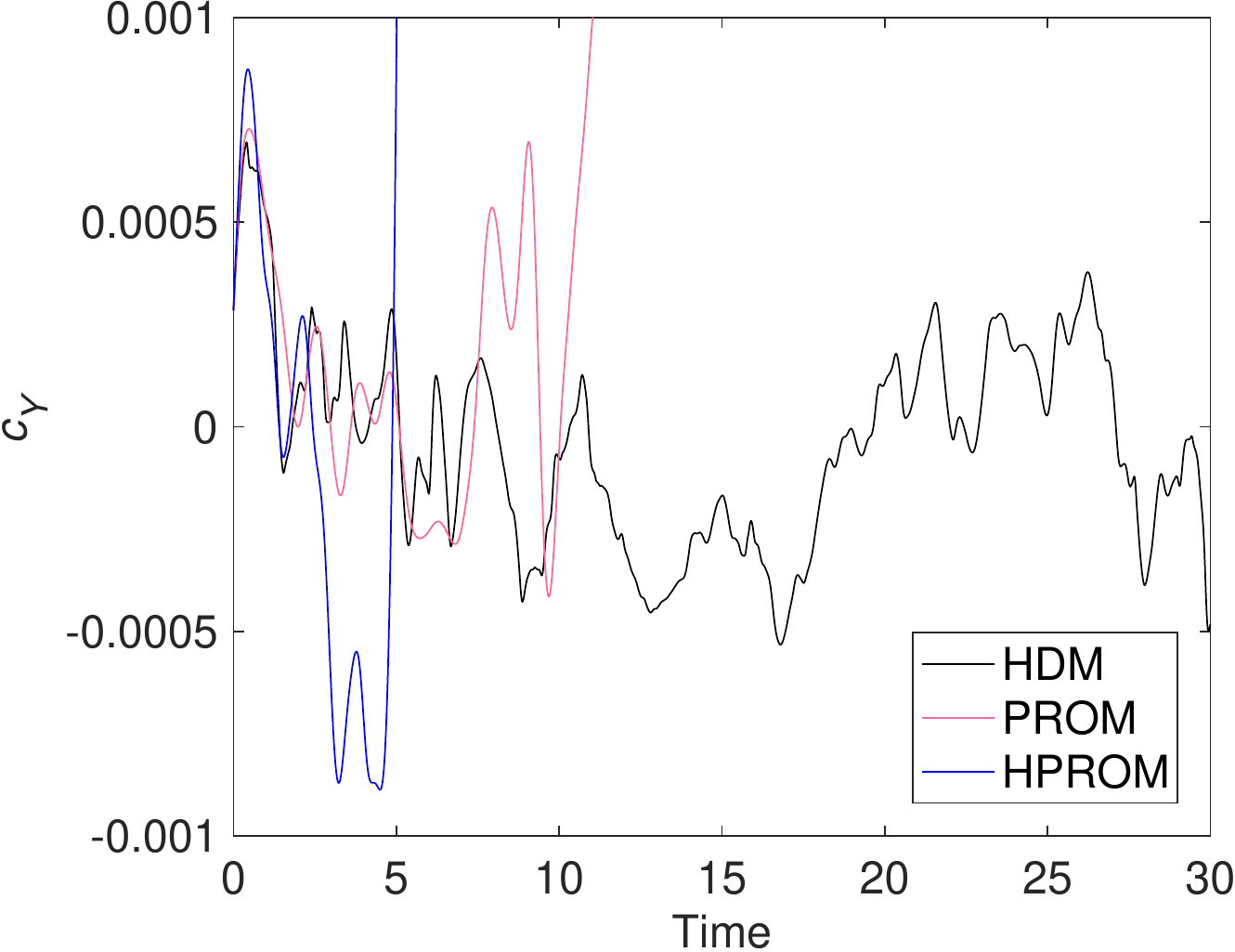}
	\caption{Galerkin, $n = 41$}
	\end{subfigure}%
	\hspace{1em}
	\begin{subfigure}[c]{0.3\textwidth}%
	\centering
	\includegraphics[width=\linewidth]{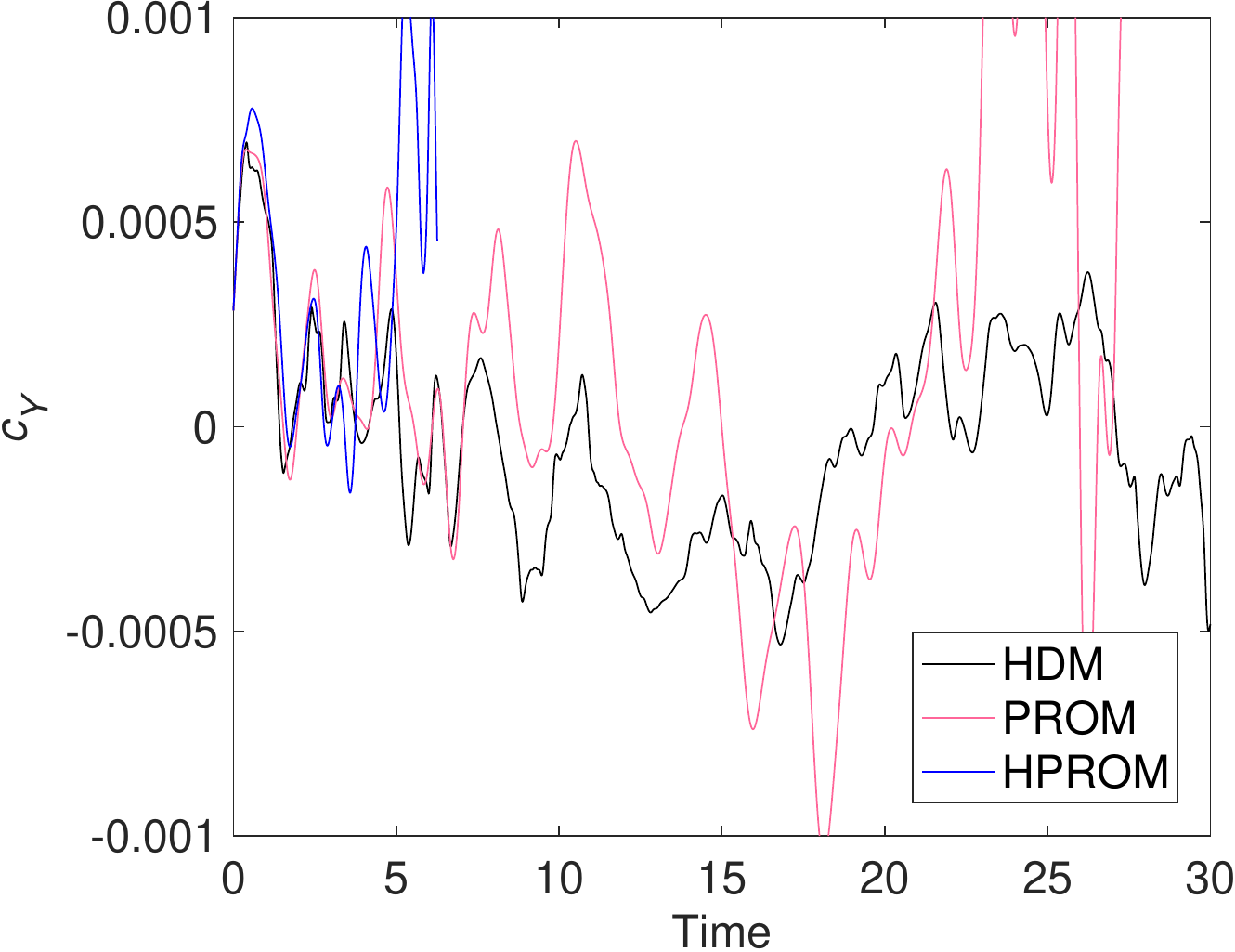}
	\caption{Galerkin, $n = 83$}
	\end{subfigure}%
	\hspace{1em}
	\begin{subfigure}[c]{0.3\textwidth}%
	\centering
	\includegraphics[width=\linewidth]{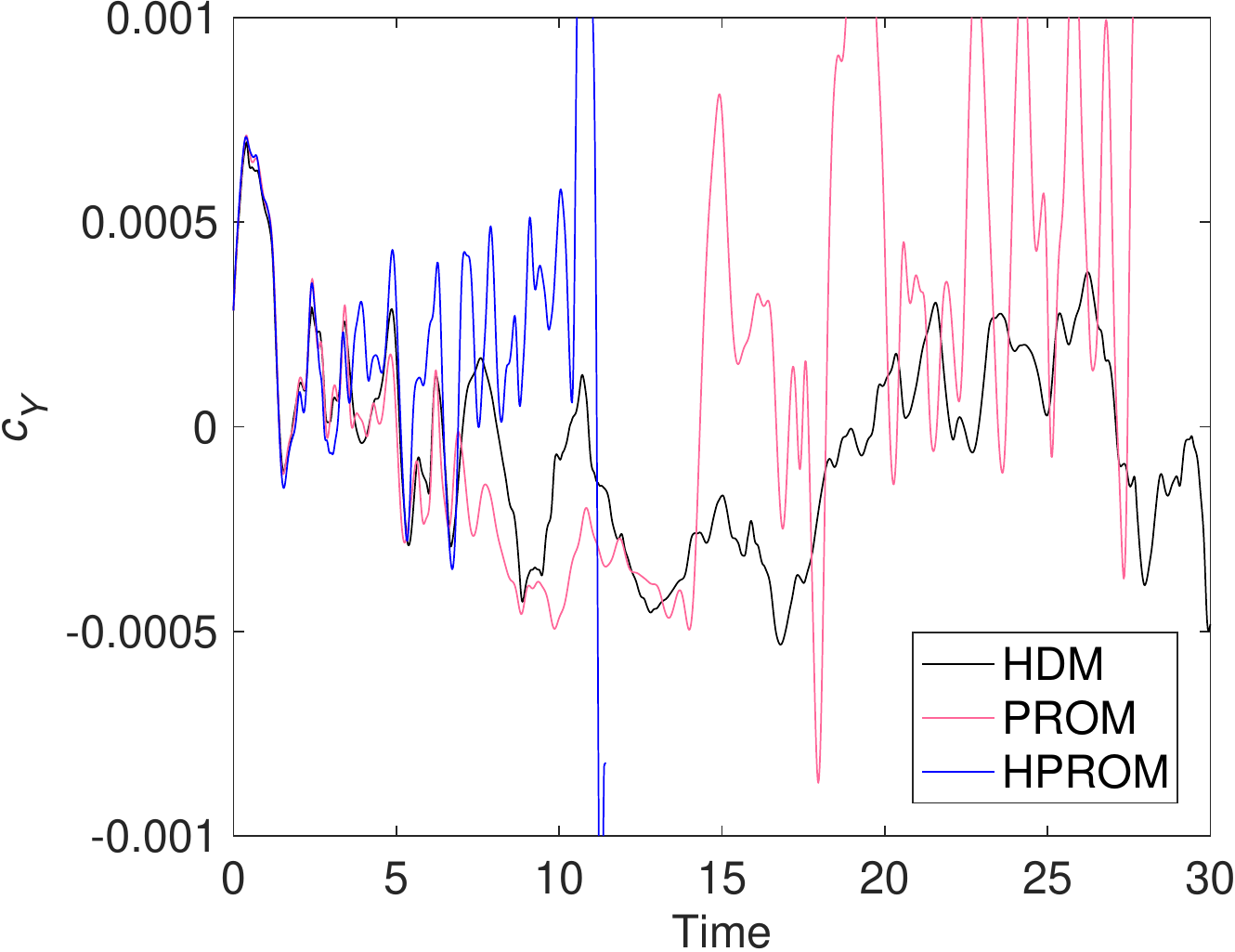}
	\caption{Galerkin, $n = 213$}
	\end{subfigure}%
	\caption{LES of a flow over a NACA airfoil: Time-histories of the side force coefficient computed using the HDM and Galerkin PROMs and HPROMs.}
	\label{fig:nacacrossgal}
\end{figure}

\begin{figure}[h!]
	\centering
	\begin{subfigure}[c]{0.3\textwidth}%
	\centering
	\includegraphics[width=\linewidth]{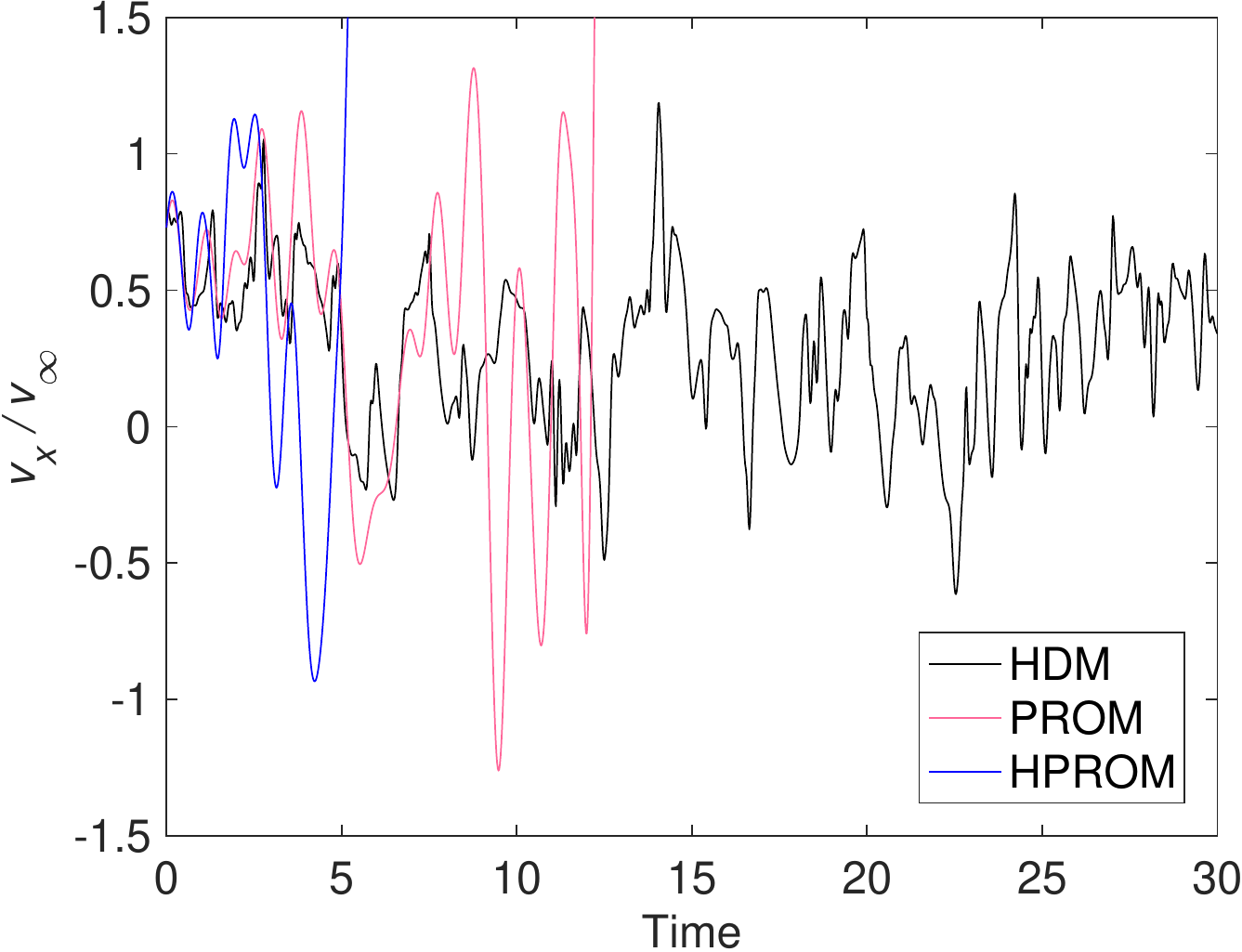}
	\caption{Galerkin, $n = 41$}
	\end{subfigure}%
	\hspace{1em}
	\begin{subfigure}[c]{0.3\textwidth}%
	\centering
	\includegraphics[width=\linewidth]{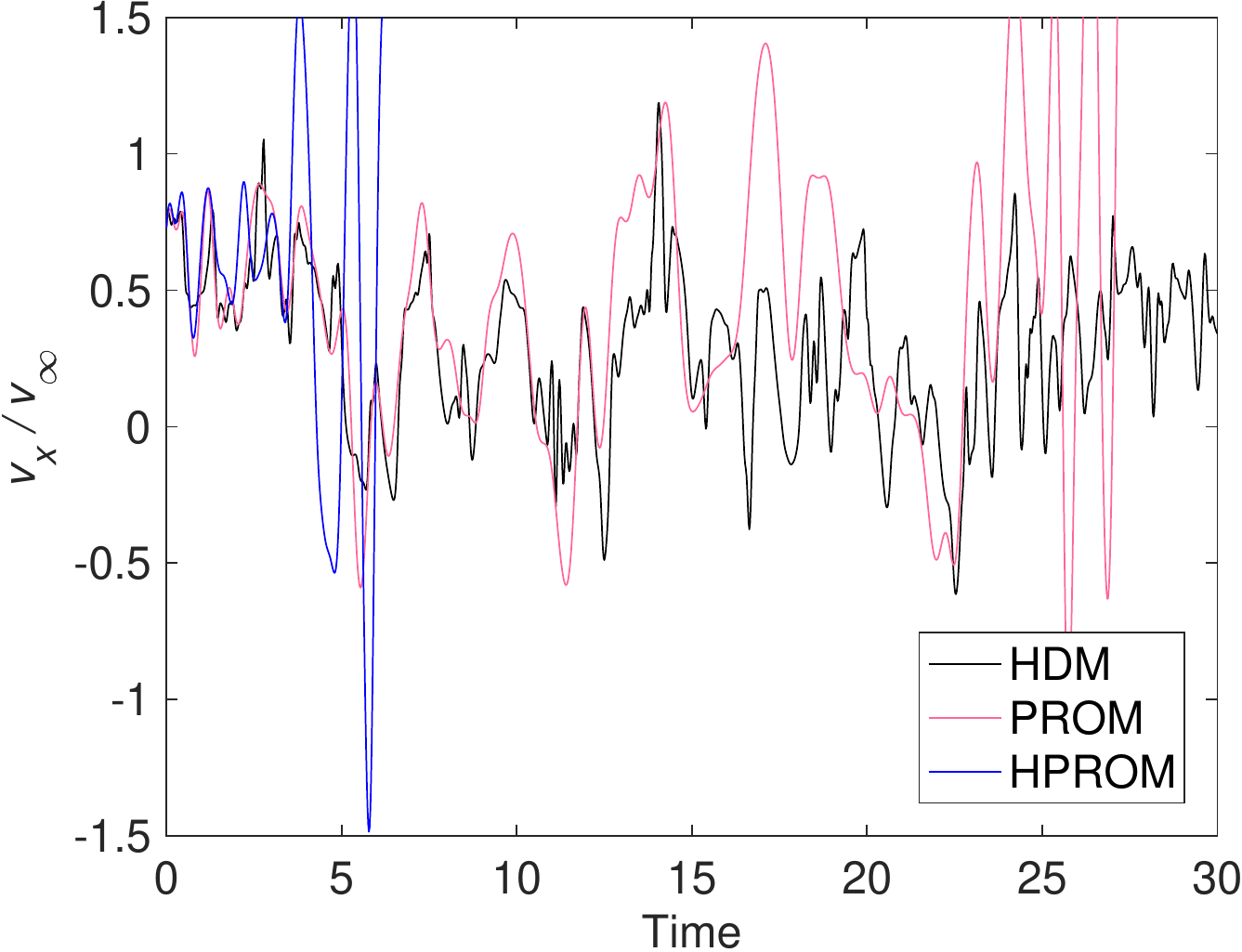}
	\caption{Galerkin, $n = 83$}
	\end{subfigure}%
	\hspace{1em}
	\begin{subfigure}[c]{0.3\textwidth}%
	\centering
	\includegraphics[width=\linewidth]{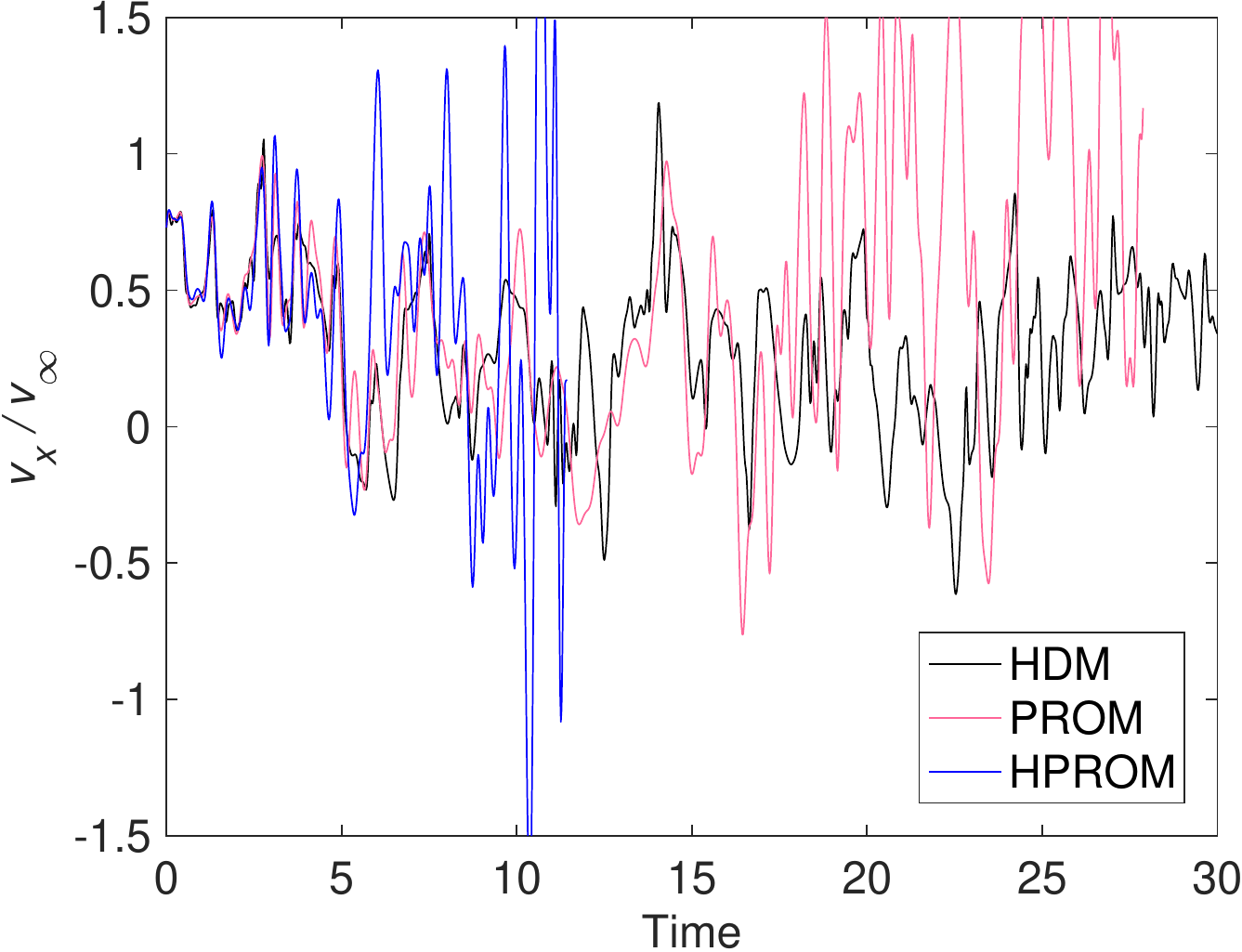}
	\caption{Galerkin, $n = 213$}
	\end{subfigure}%
	\caption{LES of a flow over a NACA airfoil: Time-histories of the streamwise velocity computed using the HDM and Galerkin PROMs and HPROMs.}
	\label{fig:nacaprobevxgal}
\end{figure}

\begin{figure}[h!]
	\centering
	\begin{subfigure}[c]{0.3\textwidth}%
	\centering
	\includegraphics[width=\linewidth]{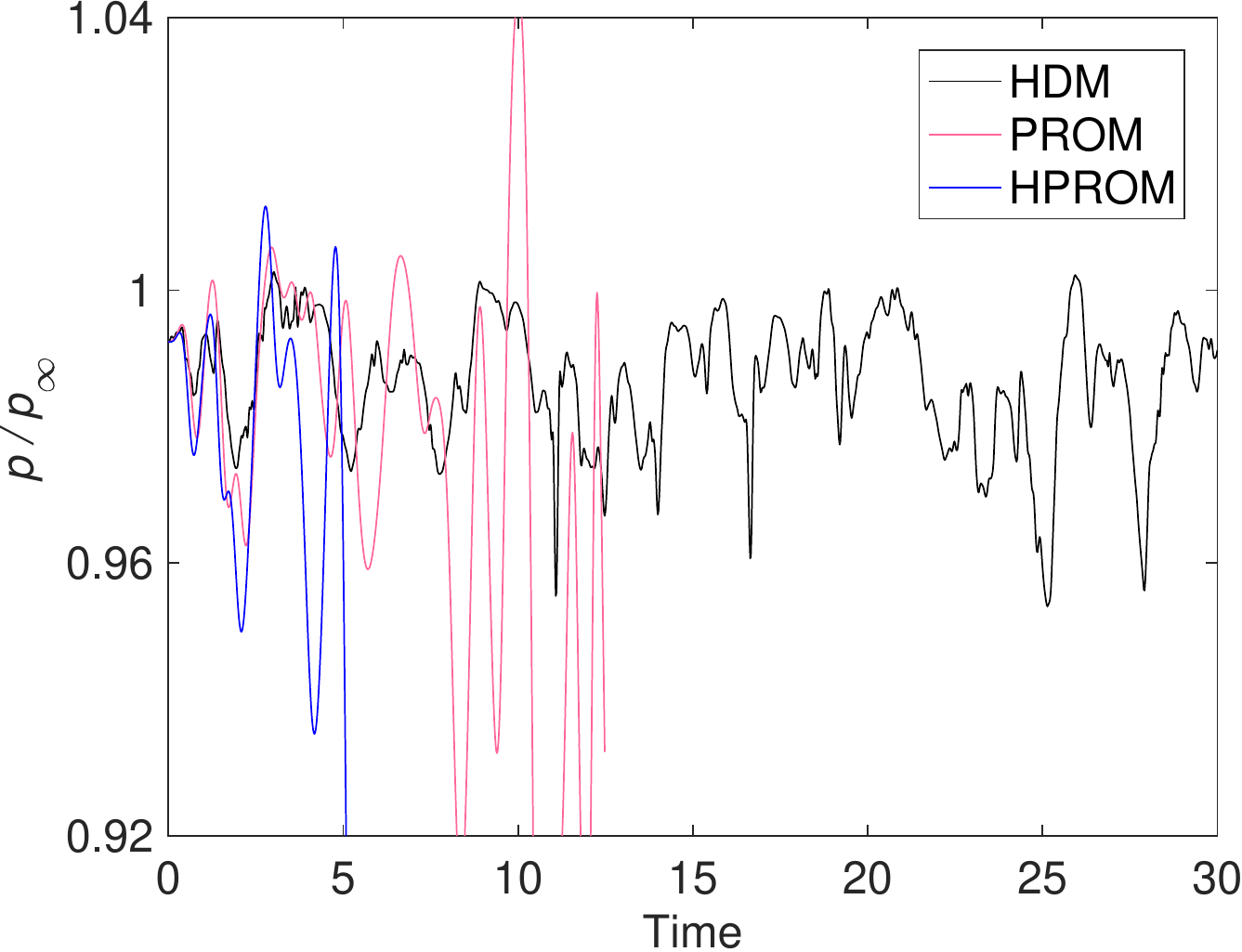}
	\caption{Galerkin, $n = 41$}
	\end{subfigure}%
	\hspace{1em}
	\begin{subfigure}[c]{0.3\textwidth}%
	\centering
	\includegraphics[width=\linewidth]{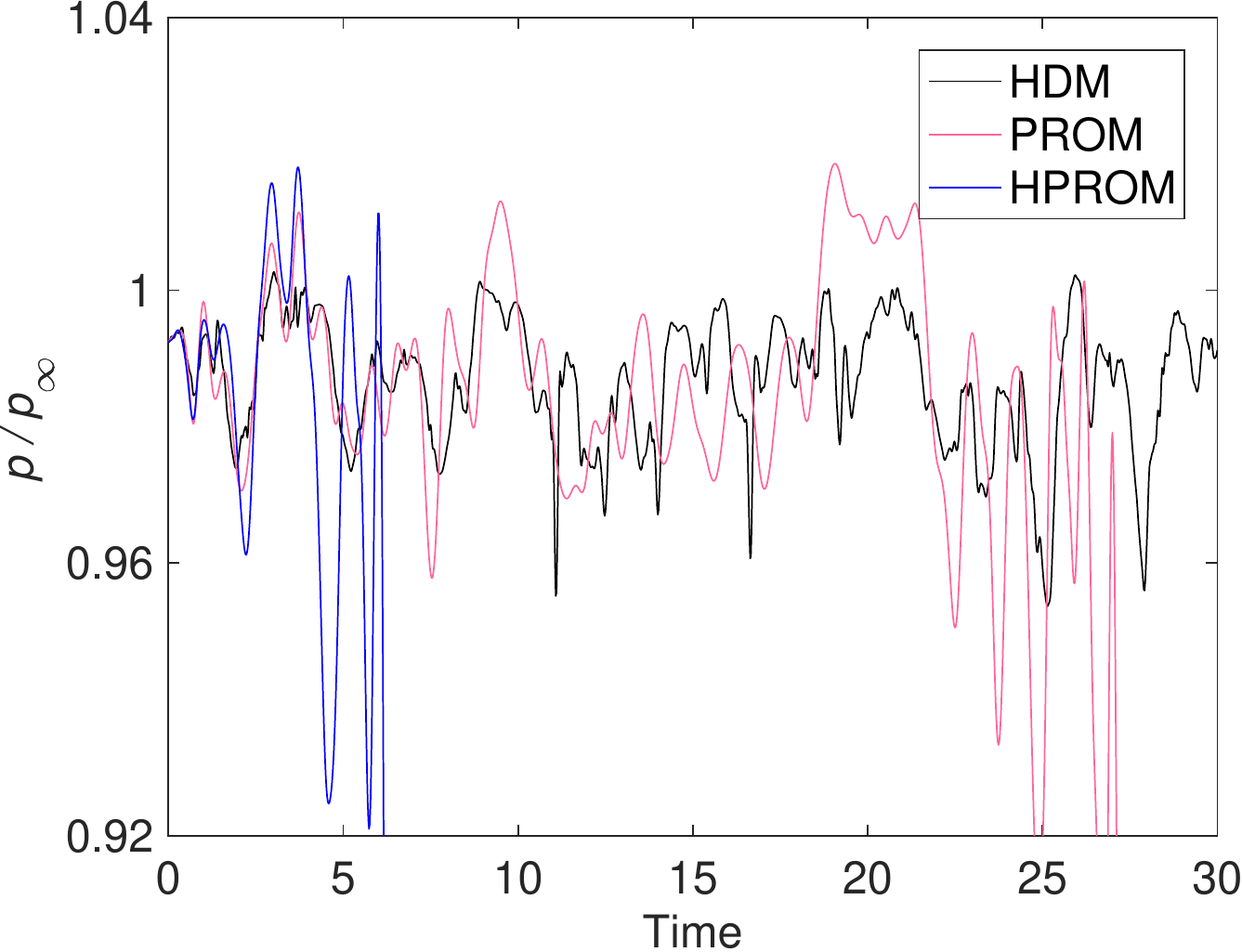}
	\caption{Galerkin, $n = 83$}
	\end{subfigure}%
	\hspace{1em}
	\begin{subfigure}[c]{0.3\textwidth}%
	\centering
	\includegraphics[width=\linewidth]{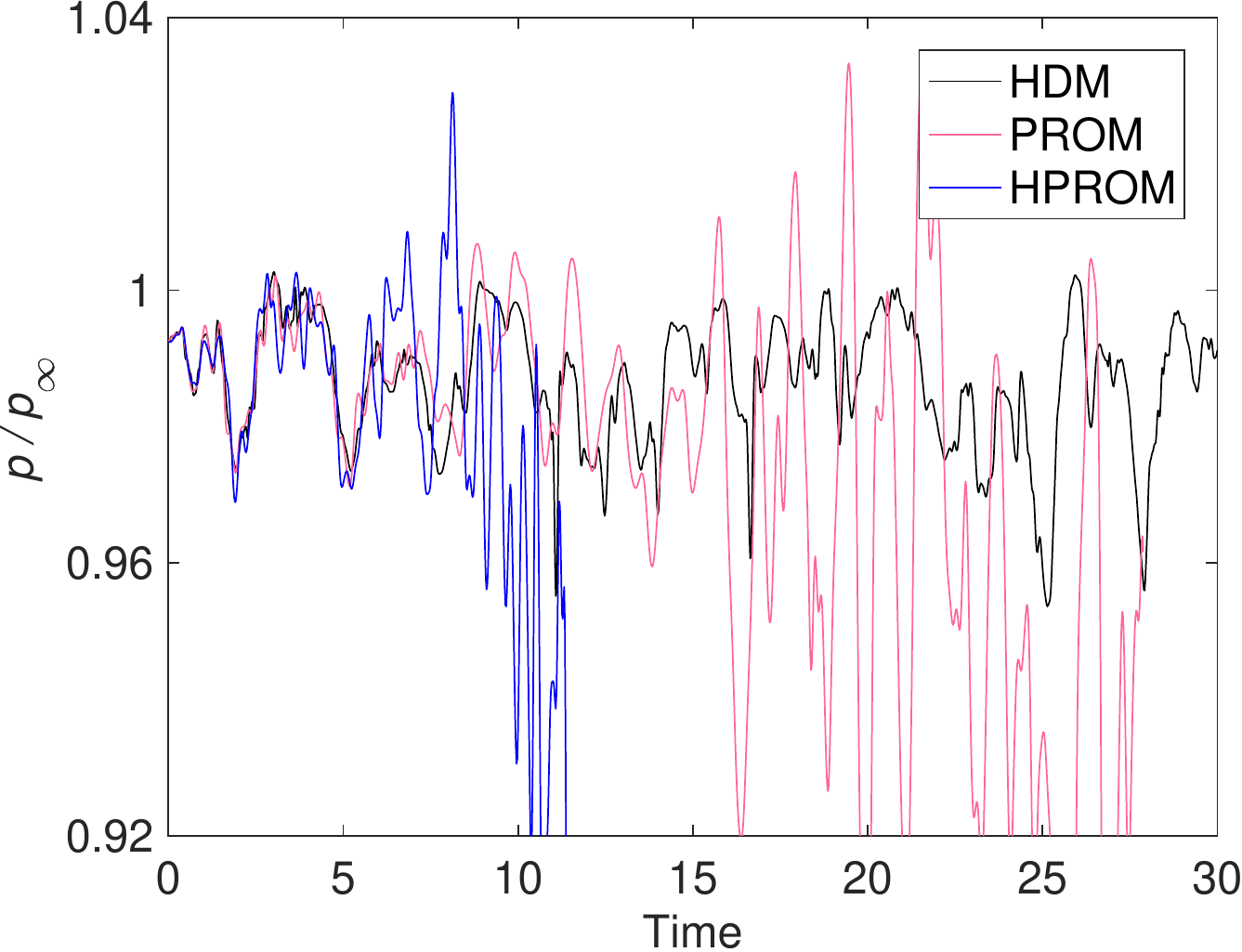}
	\caption{Galerkin, $n = 213$}
	\end{subfigure}%
	\caption{LES of a flow over a NACA airfoil: Time-histories of the pressure computed at a probe using the HDM and Galerkin PROMs and HPROMs.}
	\label{fig:nacaprobepgal}
\end{figure}

\clearpage
\paragraph{Petrov-Galerkin reduced-order models}

The time-histories of each QoI computed using the HDM and LSPG-based PROMs and HPROMs are compared in Figures \ref{fig:nacaliftlspg}, \ref{fig:nacadraglspg}, and \ref{fig:nacacrosslspg} for $c_L$, $c_D$, and $c_Y$, and in Figures \ref{fig:nacaprobevxlspg} and \ref{fig:nacaprobeplspg} for $v_x$ and $p$ computed at the probe location. These time-histories show that all constructed LSPG-based PROMs and 
HPROMs are numerically stable in the entirety of the specified simulation time-interval. Increasing the dimension $n$ of the subspace approximation is also shown to increase the level of accuracy of 
these Petrov-Galerkin PROMs and HPROMs, as expected.

\begin{figure}[h!]
	\centering
	\begin{subfigure}[c]{0.3\textwidth}%
	\centering
	\includegraphics[width=\linewidth]{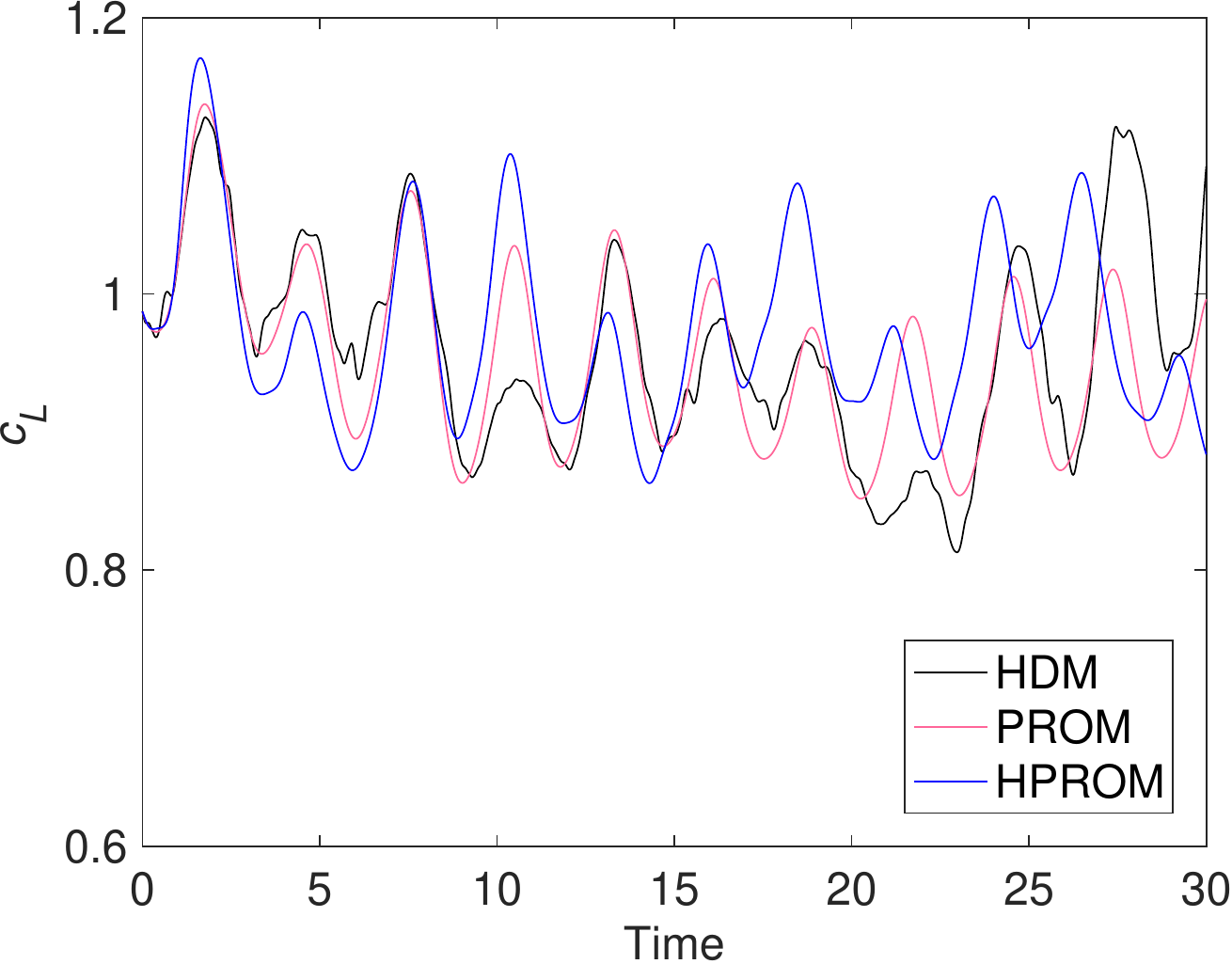}
	\caption{LSPG, $n = 41$}
	\end{subfigure}%
	\hspace{1em}
	\begin{subfigure}[c]{0.3\textwidth}%
	\centering
	\includegraphics[width=\linewidth]{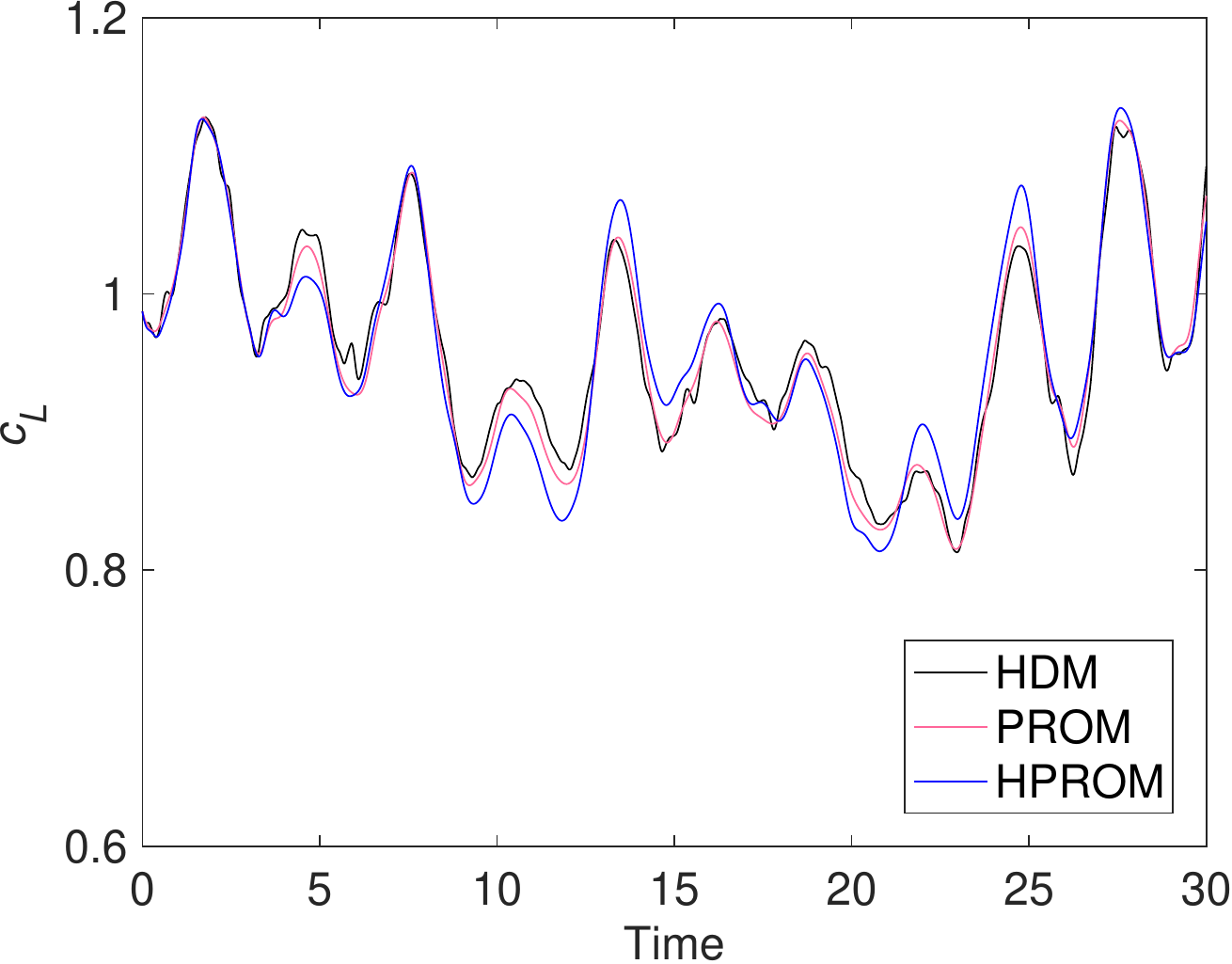}
	\caption{LSPG, $n = 83$}
	\end{subfigure}%
	\hspace{1em}
	\begin{subfigure}[c]{0.3\textwidth}%
	\centering
	\includegraphics[width=\linewidth]{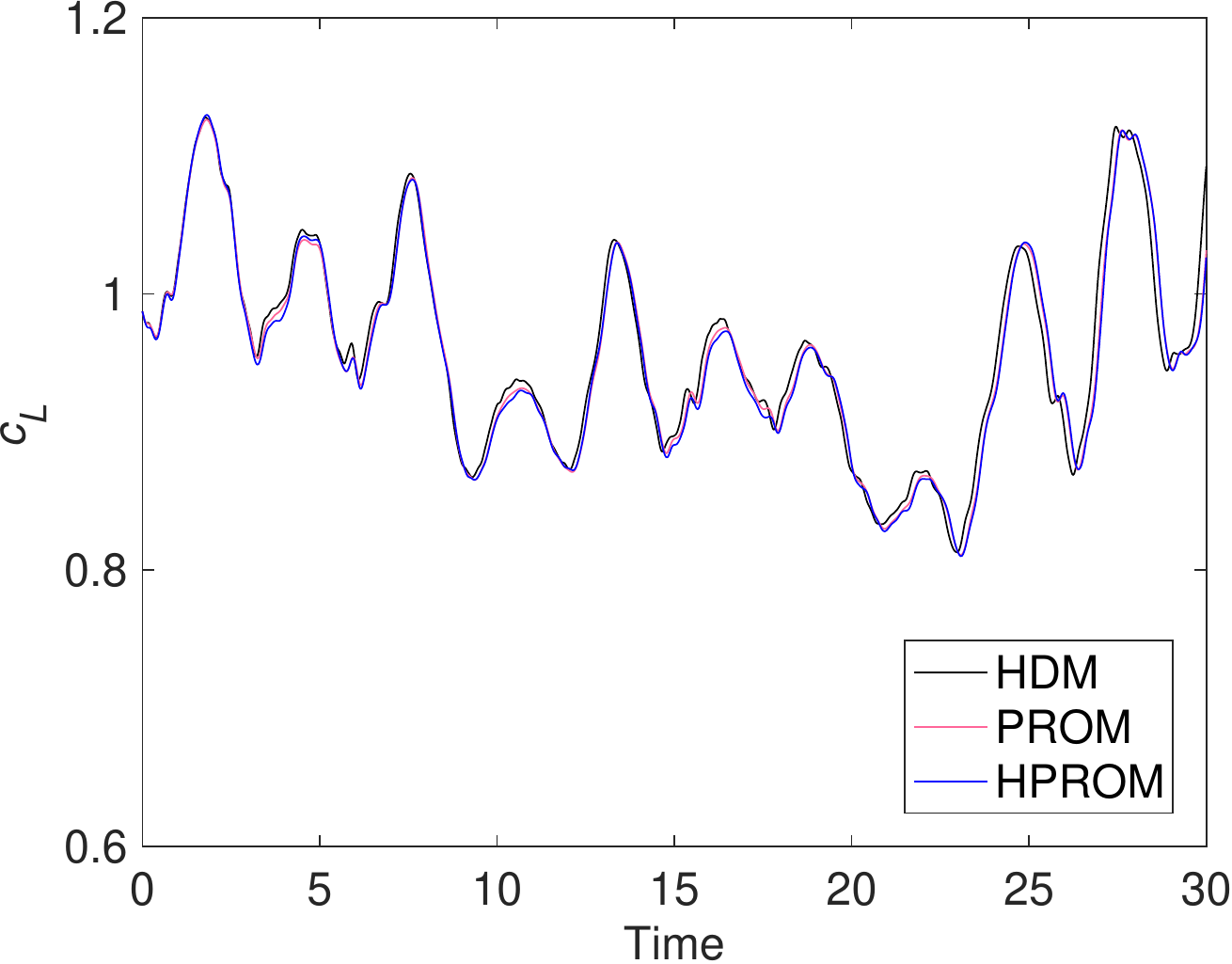}
	\caption{LSPG, $n = 213$}
	\end{subfigure}%
	\caption{LES of a flow over a NACA airfoil: Time-histories of the lift coefficient computed using the HDM and LSPG-based Petrov-Galerkin PROMs and HPROMs.}
	\label{fig:nacaliftlspg}
\end{figure}

\begin{figure}[h!]
	\centering
	\begin{subfigure}[c]{0.3\textwidth}%
	\centering
	\includegraphics[width=\linewidth]{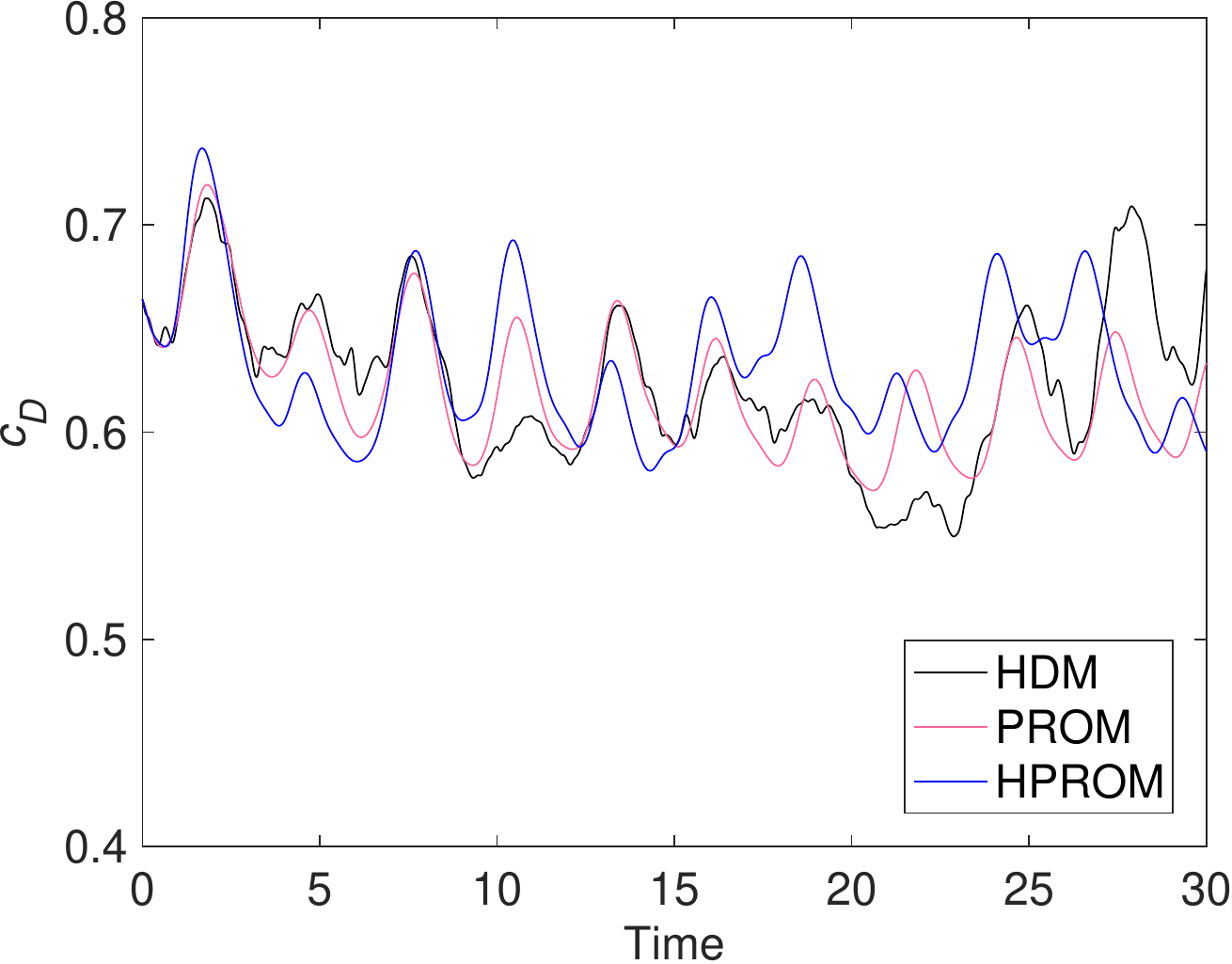}
	\caption{LSPG, $n = 41$}
	\end{subfigure}%
	\hspace{1em}
	\begin{subfigure}[c]{0.3\textwidth}%
	\centering
	\includegraphics[width=\linewidth]{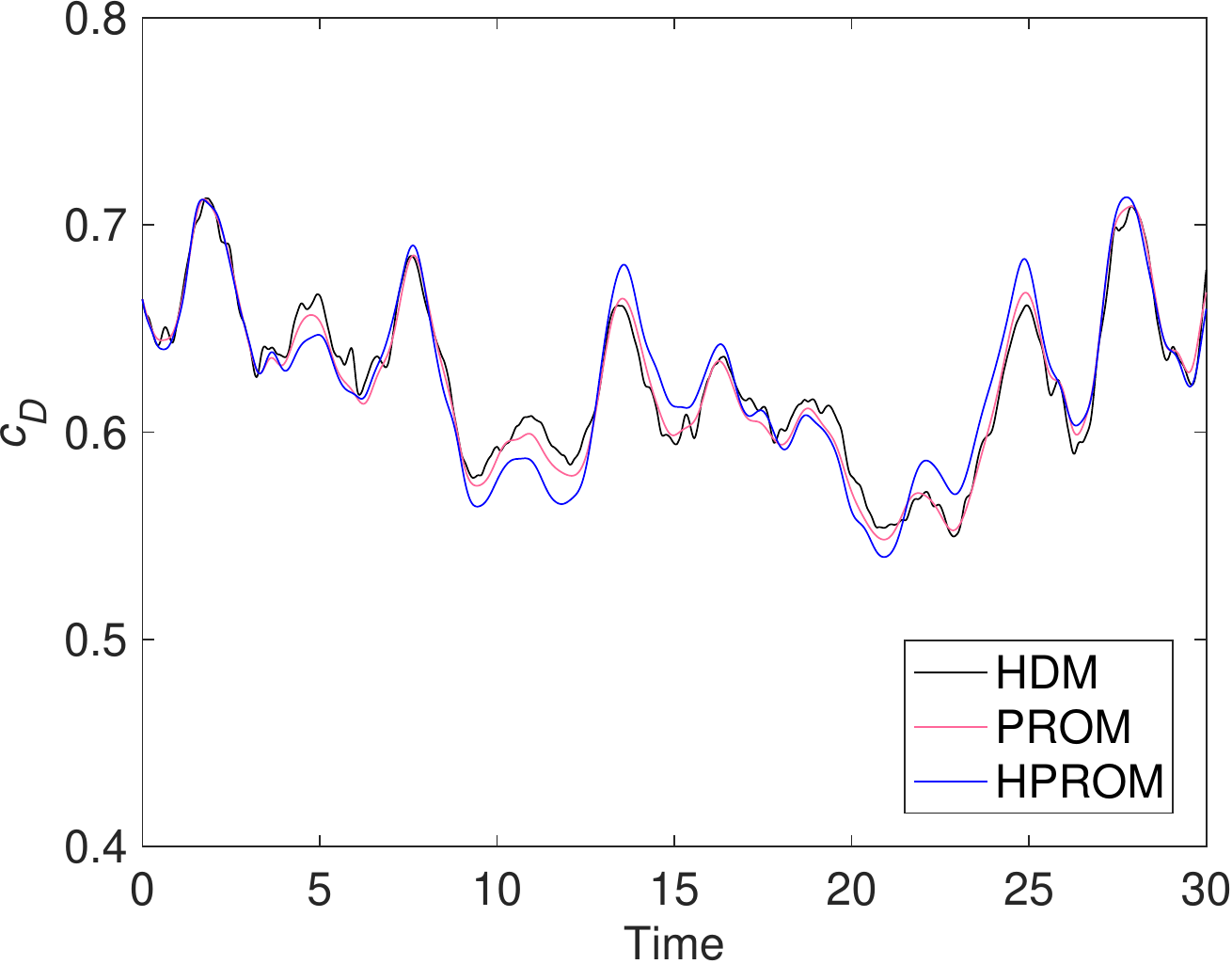}
	\caption{LSPG, $n = 83$}
	\end{subfigure}%
	\hspace{1em}
	\begin{subfigure}[c]{0.3\textwidth}%
	\centering
	\includegraphics[width=\linewidth]{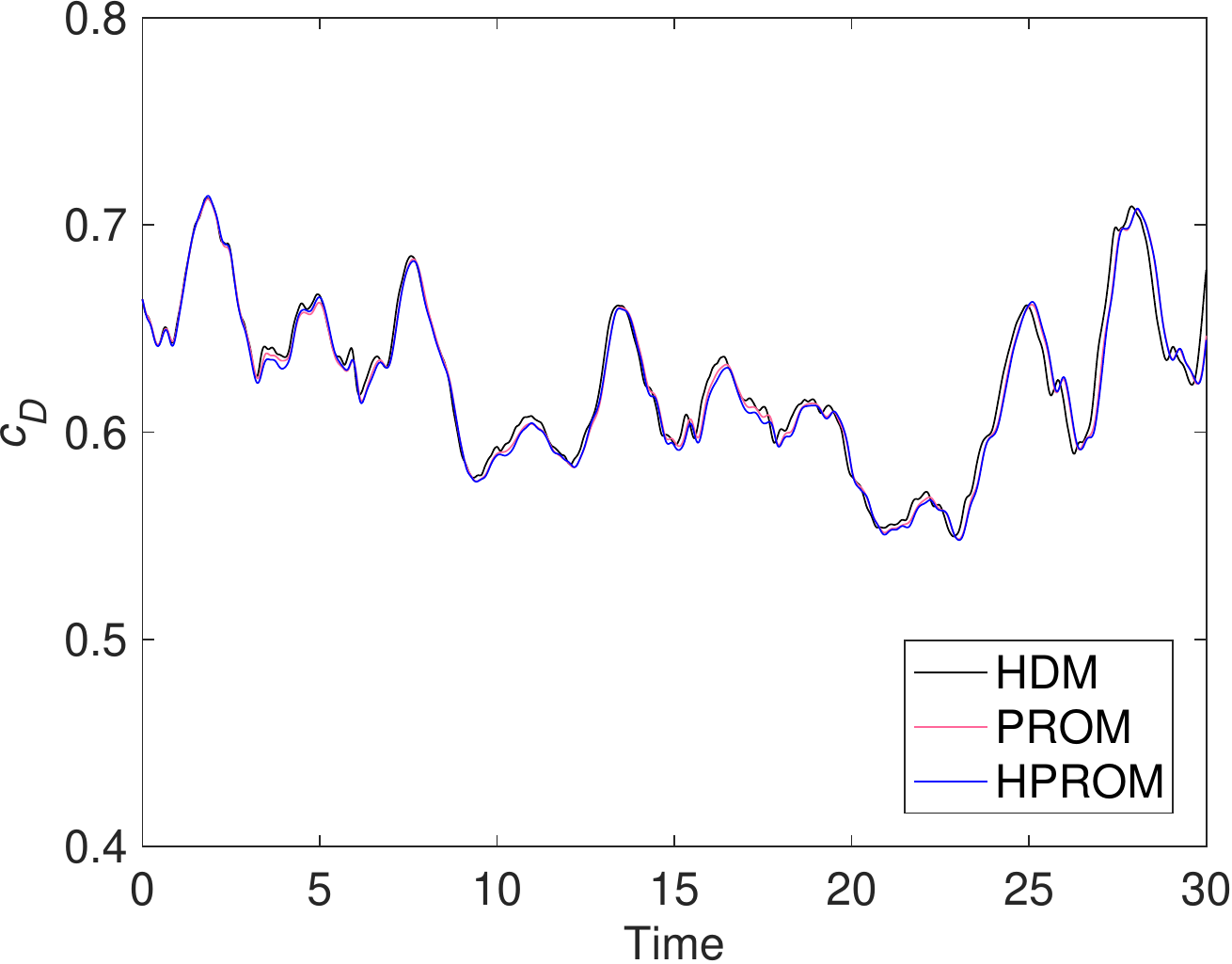}
	\caption{LSPG, $n = 213$}
	\end{subfigure}%
	\caption{LES of a flow over a NACA airfoil: Time-histories of the drag coefficient computed using the HDM and LSPG-based Petrov-Galerkin PROMs and HPROMs.}
	\label{fig:nacadraglspg}
\end{figure}

\begin{figure}[h!]
	\centering
	\begin{subfigure}[c]{0.3\textwidth}%
	\centering
	\includegraphics[width=\linewidth]{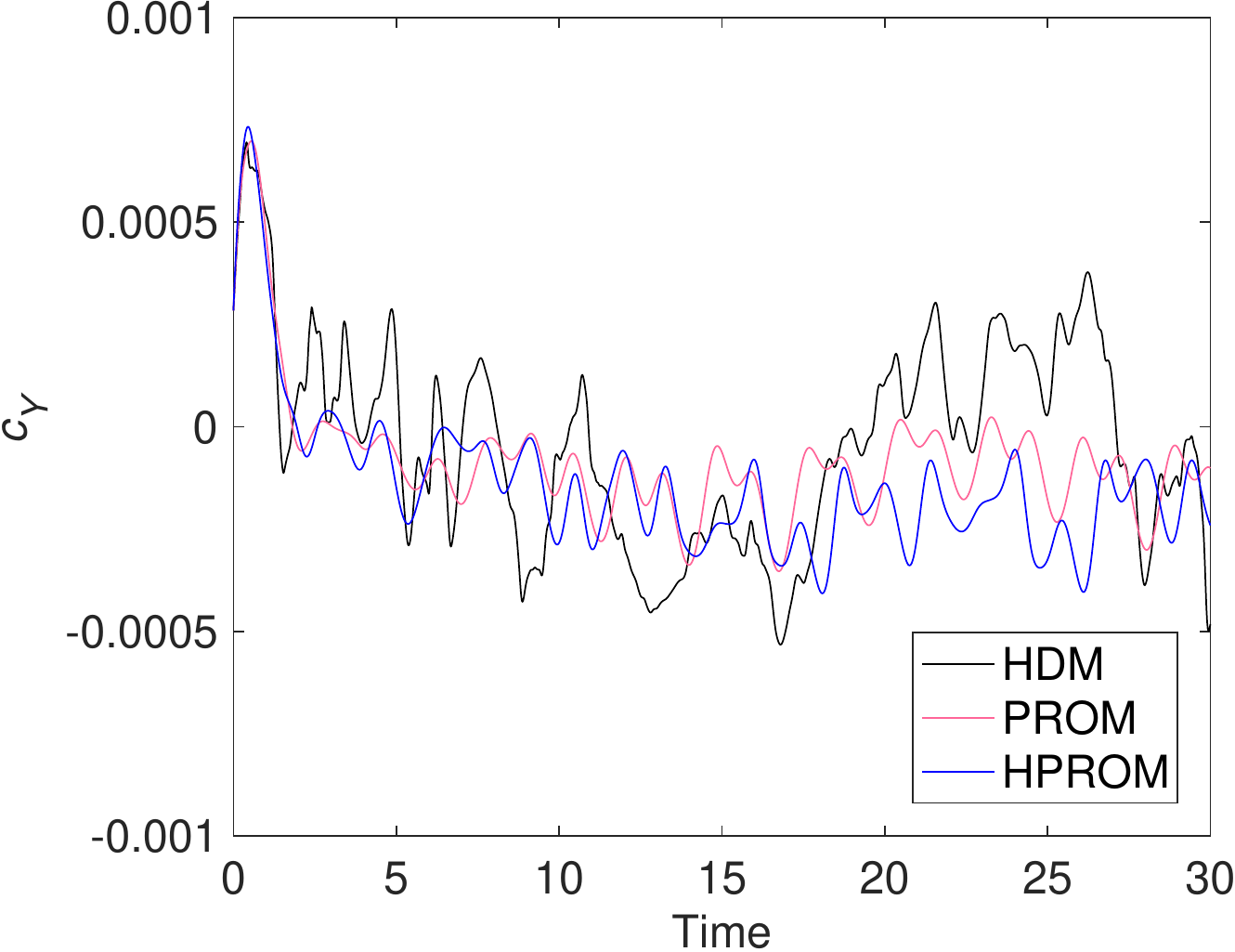}
	\caption{LSPG, $n = 41$}
	\end{subfigure}%
	\hspace{1em}
	\begin{subfigure}[c]{0.3\textwidth}%
	\centering
	\includegraphics[width=\linewidth]{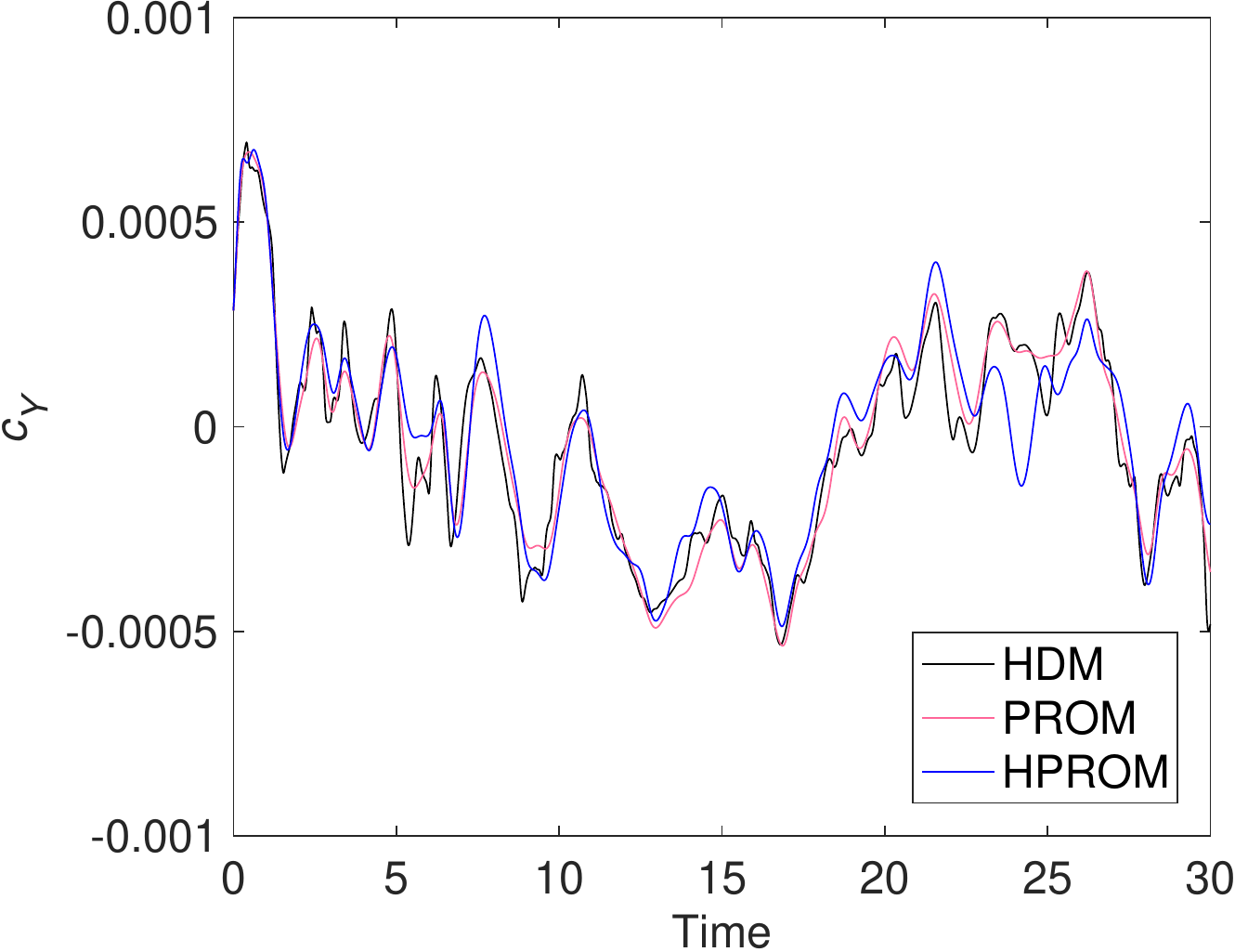}
	\caption{LSPG, $n = 83$}
	\end{subfigure}%
	\hspace{1em}
	\begin{subfigure}[c]{0.3\textwidth}%
	\centering
	\includegraphics[width=\linewidth]{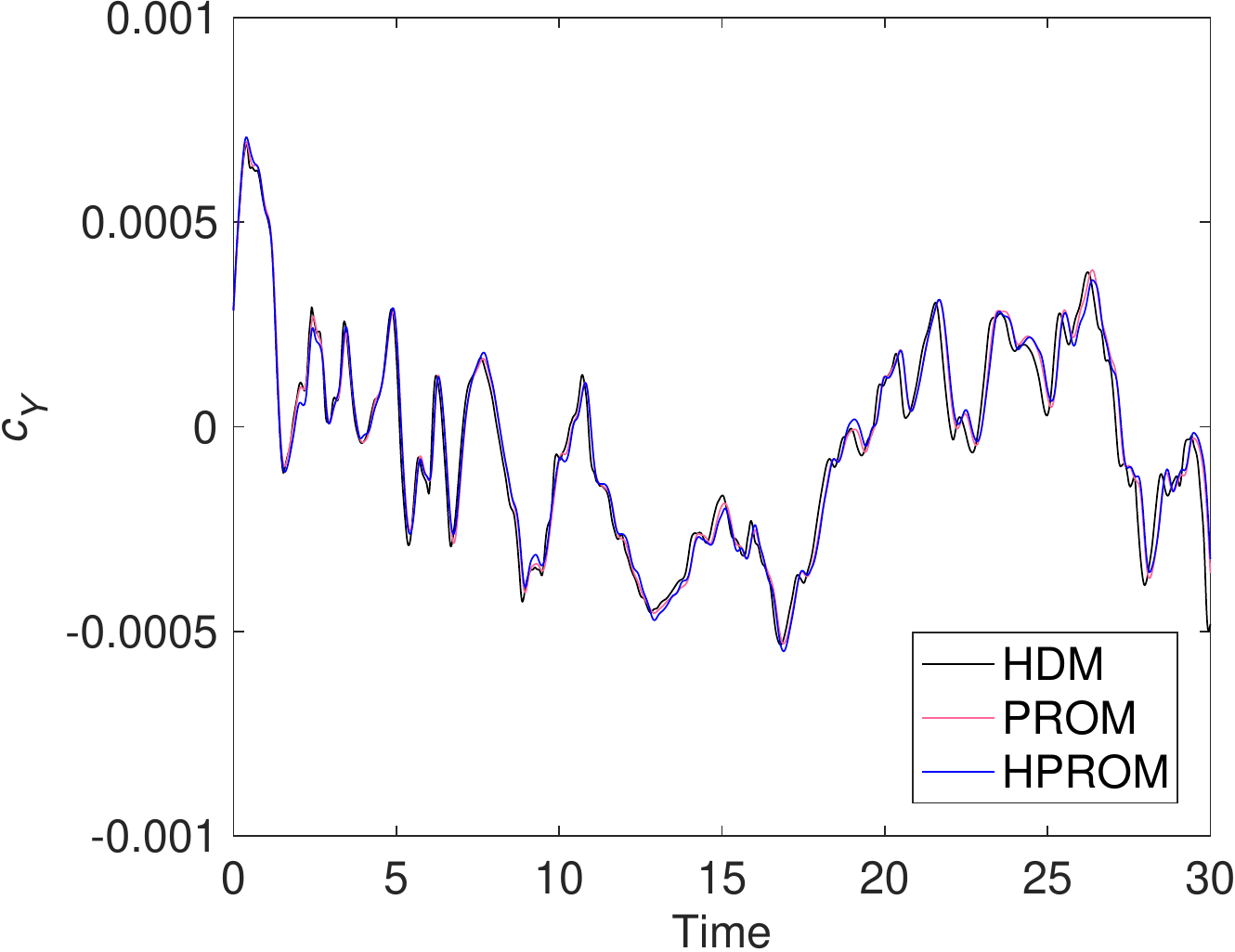}
	\caption{LSPG, $n = 213$}
	\end{subfigure}%
	\caption{LES of a flow over a NACA airfoil: Time-histories of the side force coefficient computed using the HDM and LSPG-based Petrov-Galerkin PROMs and HPROMs.}
	\label{fig:nacacrosslspg}
\end{figure}

\begin{figure}[h!]
	\centering
	\begin{subfigure}[c]{0.3\textwidth}%
	\centering
	\includegraphics[width=\linewidth]{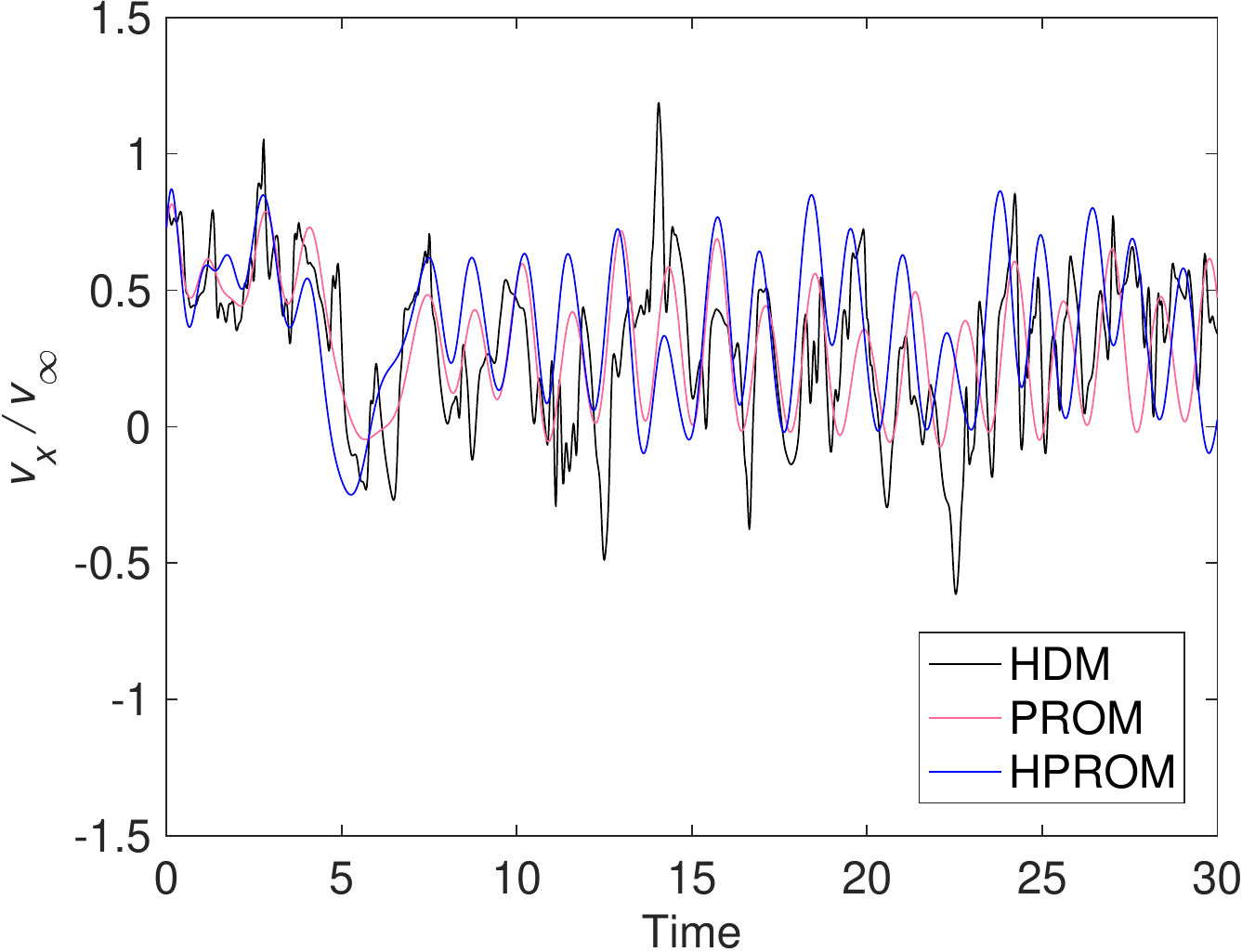}
	\caption{LSPG, $n = 41$}
	\end{subfigure}%
	\hspace{1em}
	\begin{subfigure}[c]{0.3\textwidth}%
	\centering
	\includegraphics[width=\linewidth]{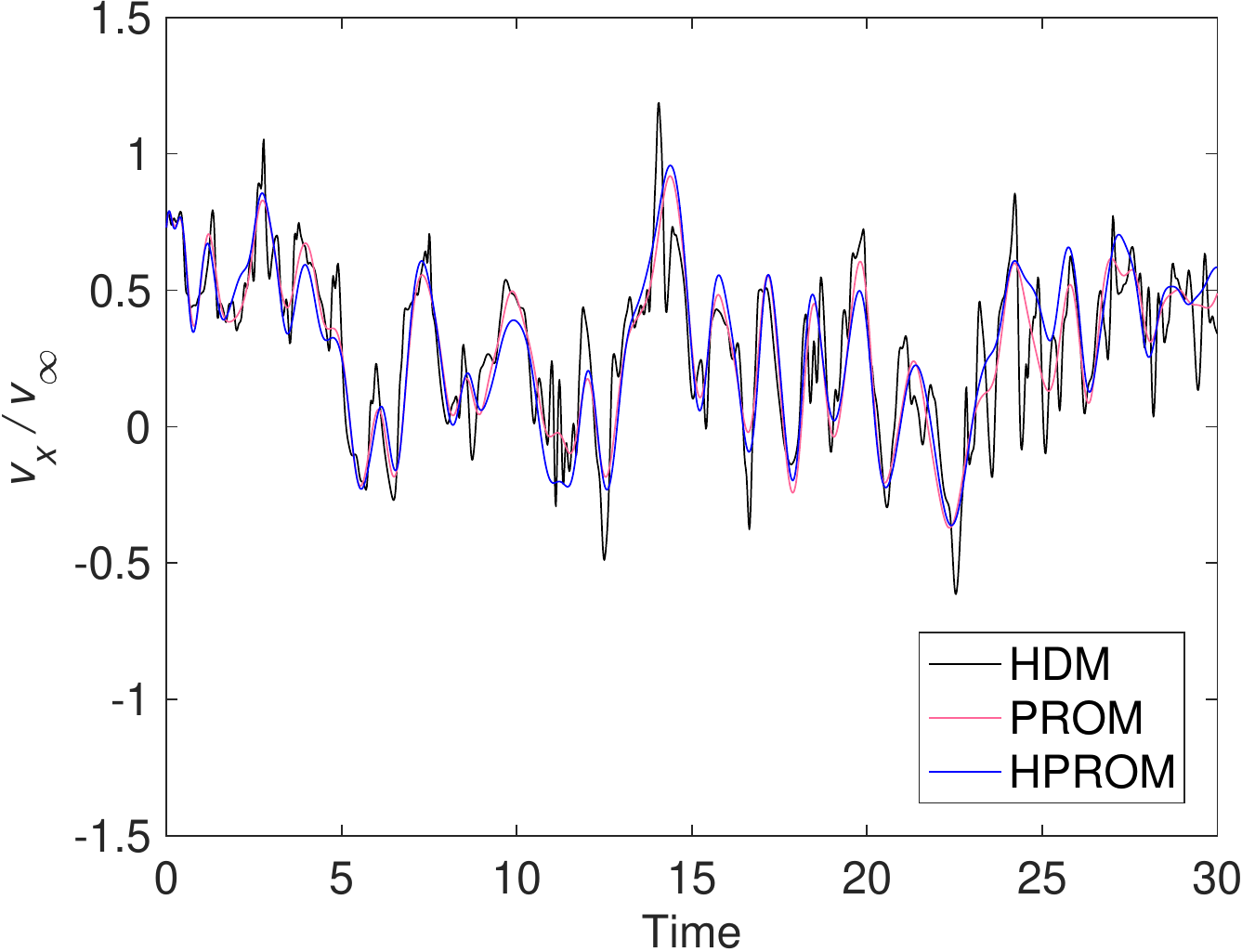}
	\caption{LSPG, $n = 83$}
	\end{subfigure}%
	\hspace{1em}
	\begin{subfigure}[c]{0.3\textwidth}%
	\centering
	\includegraphics[width=\linewidth]{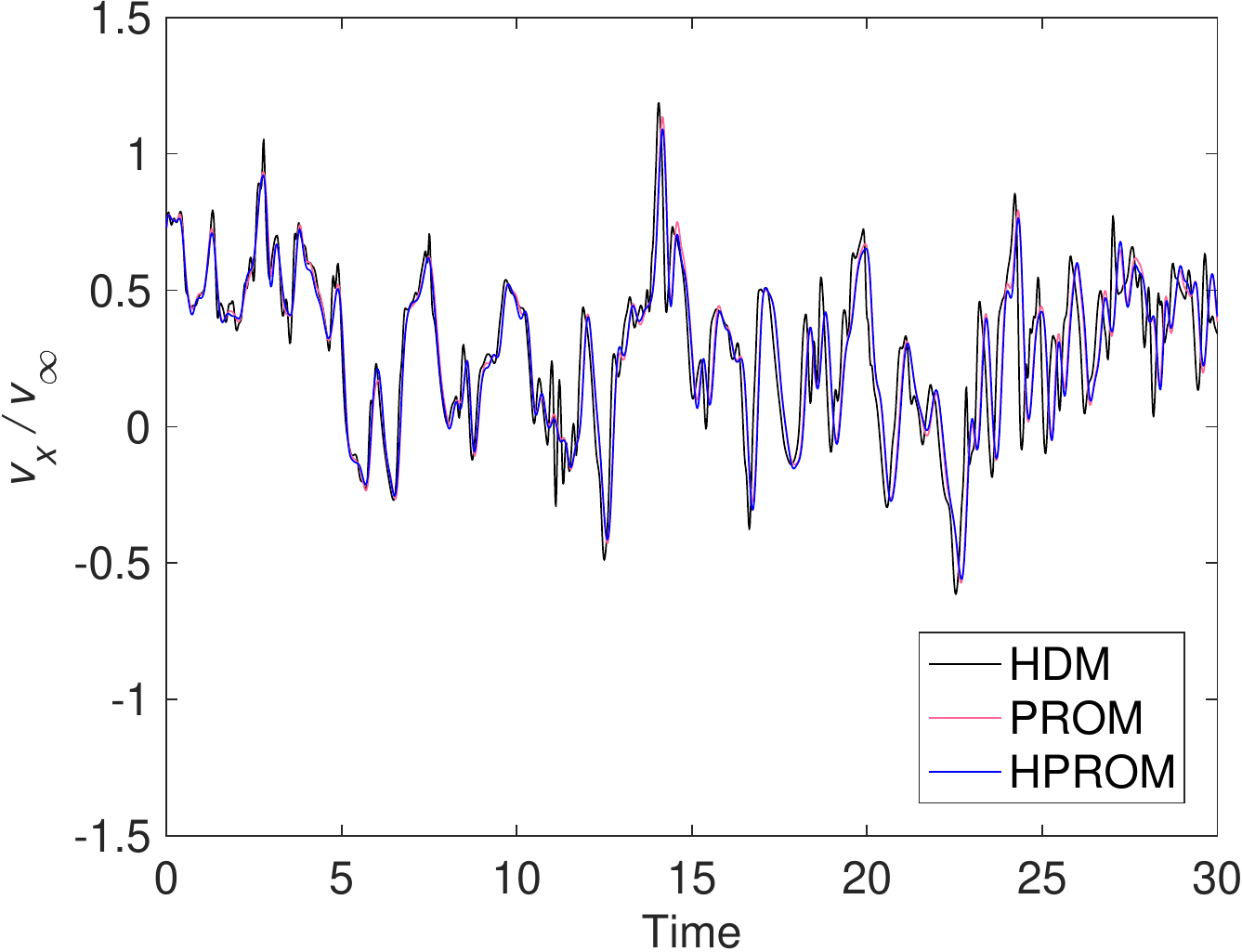}
	\caption{LSPG, $n = 213$}
	\end{subfigure}%
	\caption{LES of a flow over a NACA airfoil: Time-histories of the streamwise velocity computed at a probe using the HDM and LSPG-based Petrov-Galerkin PROMs and HPROMs.}
	\label{fig:nacaprobevxlspg}
\end{figure}

\begin{figure}[h!]
	\centering
	\begin{subfigure}[c]{0.3\textwidth}%
	\centering
	\includegraphics[width=\linewidth]{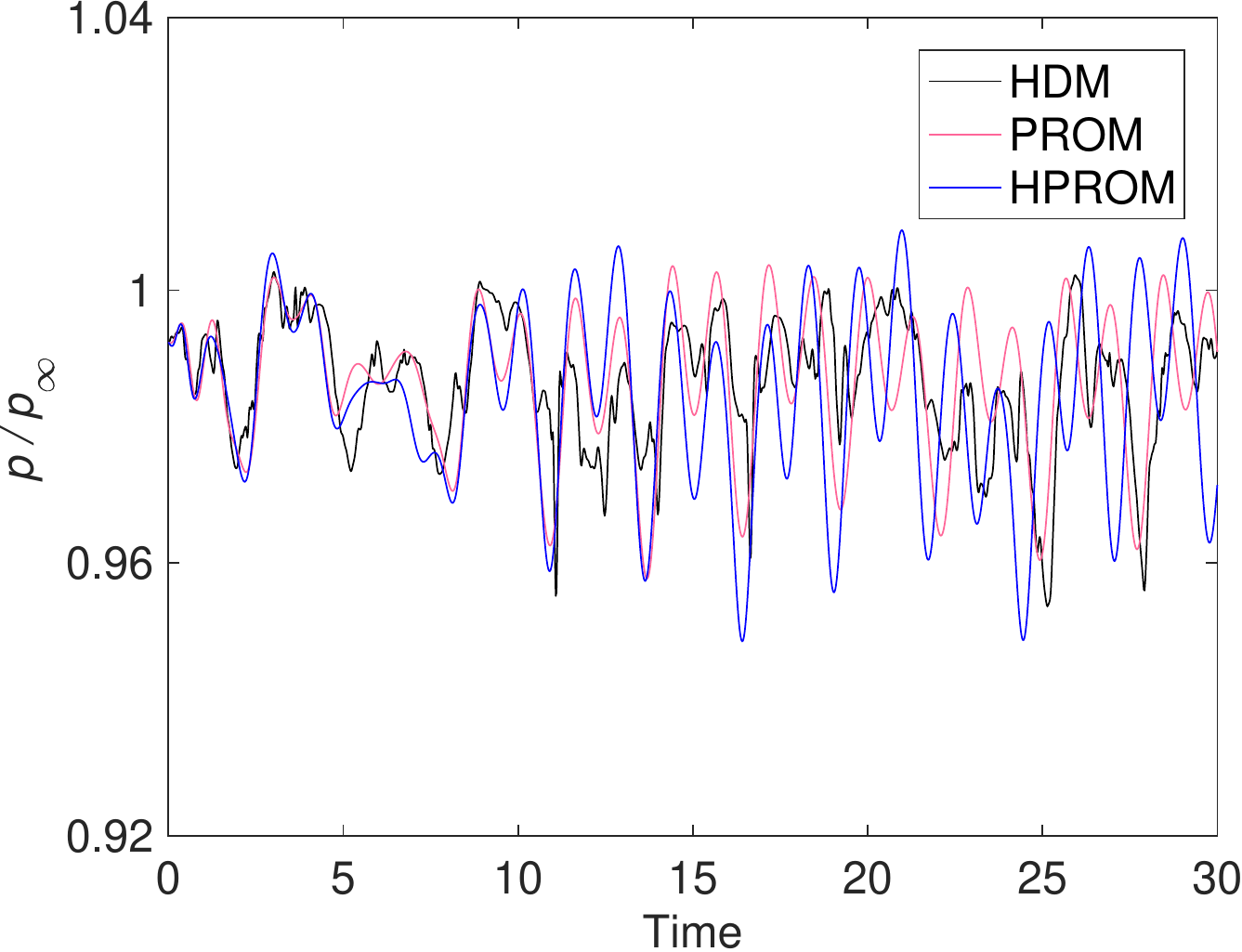}
	\caption{LSPG, $n = 41$}
	\end{subfigure}%
	\hspace{1em}
	\begin{subfigure}[c]{0.3\textwidth}%
	\centering
	\includegraphics[width=\linewidth]{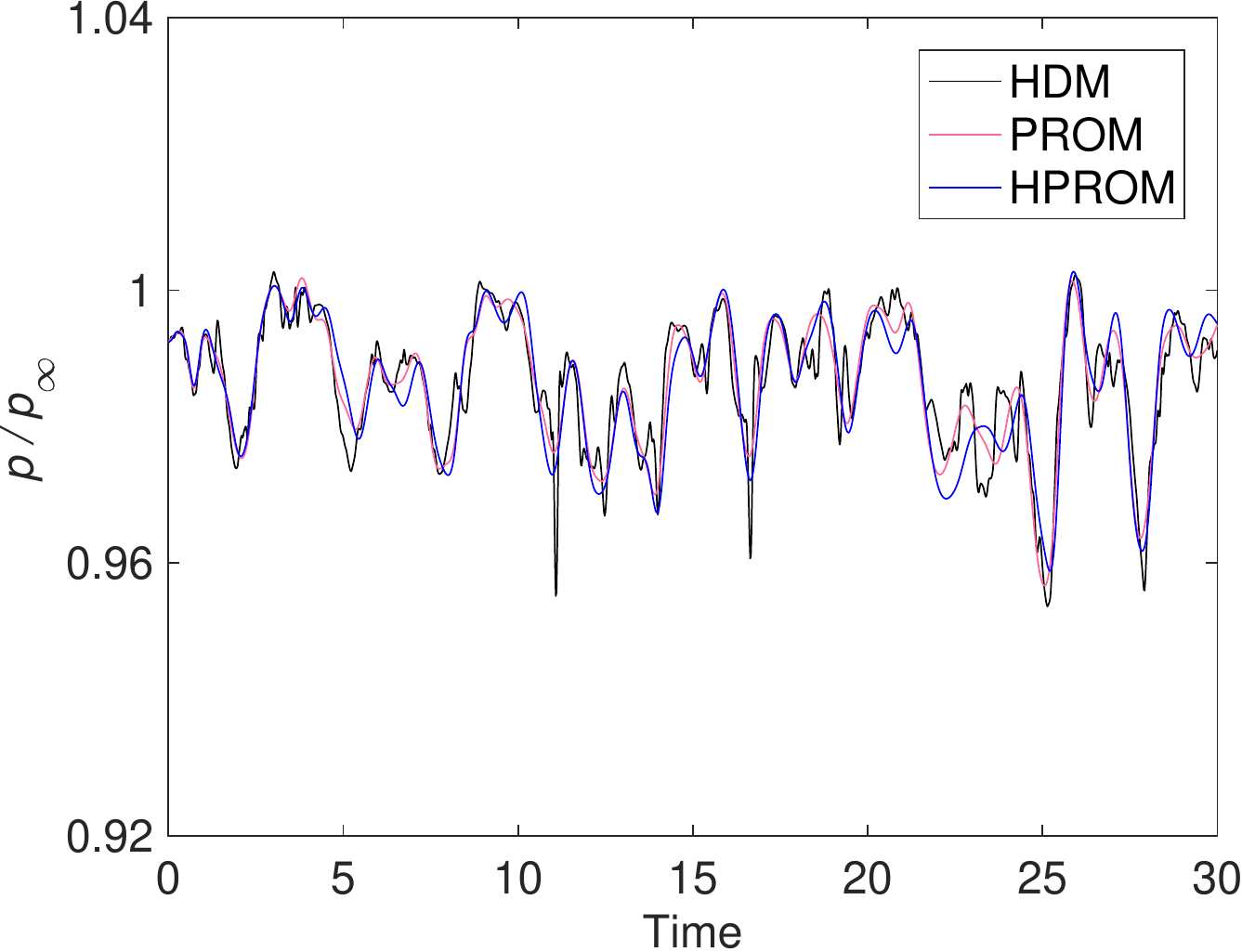}
	\caption{LSPG, $n = 83$}
	\end{subfigure}%
	\hspace{1em}
	\begin{subfigure}[c]{0.3\textwidth}%
	\centering
	\includegraphics[width=\linewidth]{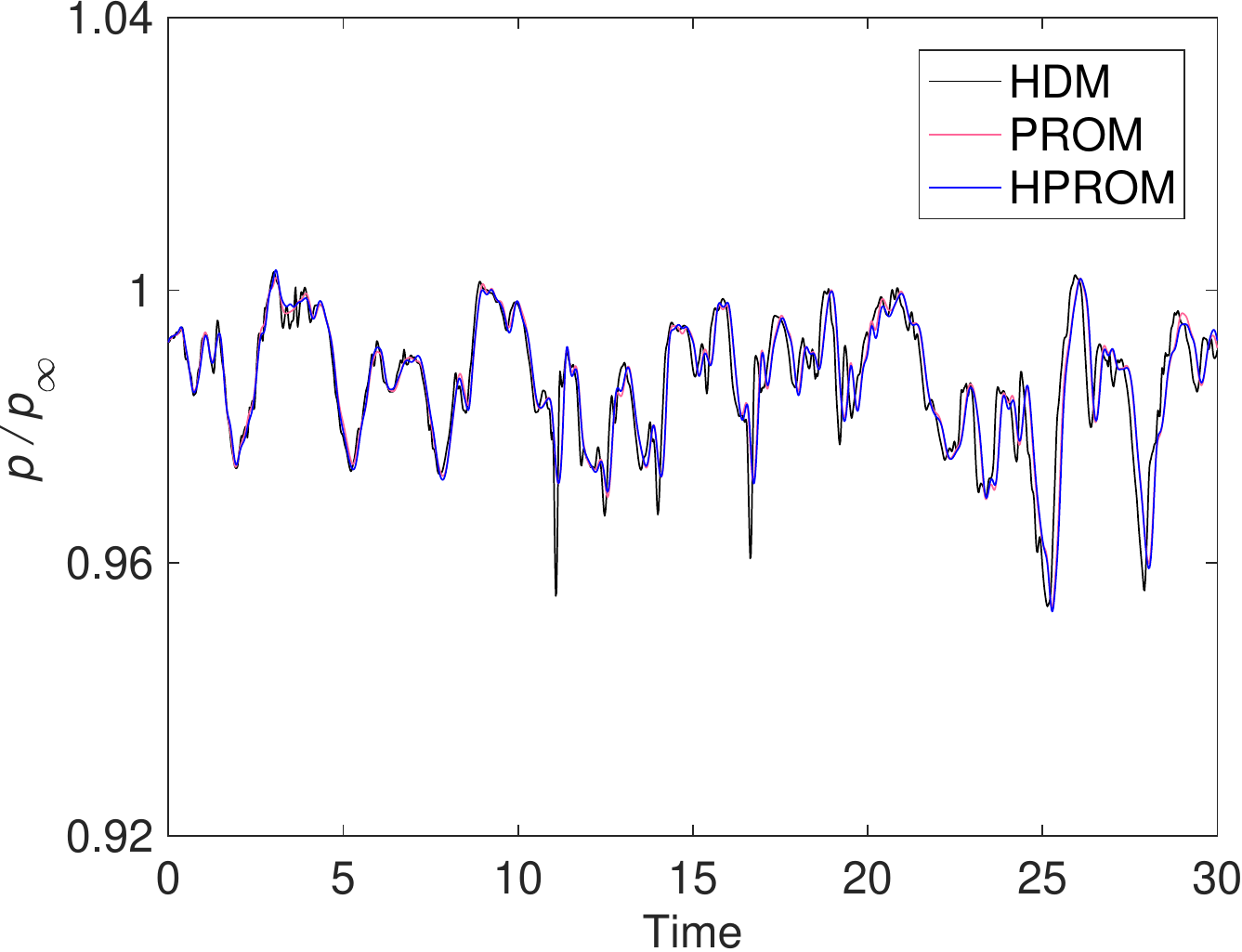}
	\caption{LSPG, $n = 213$}
	\end{subfigure}%
	\caption{LES of a flow over a NACA airfoil: Time-histories of the pressure computed at a probe using the HDM and LSPG-based Petrov-Galerkin PROMs and HPROMs.}
	\label{fig:nacaprobeplspg}
\end{figure}

\clearpage

Table \ref{tab:nacaerr} reports on the relative errors incurred using each LSPG-based Petrov-Galerkin PROM and HPROM constructed for this problem (the counterpart relative errors for the constructed
Galerkin PROMs and HPROMs are not reported due to their failure to complete the target simulations). Note that for some of the QoIs, the accuracy delivered by the LSPG-based Petrov-Galerkin PROM
of dimension $n = 213$ is slightly lower than that delivered by its counterpart of dimension $n = 83$. This seemingly odd observation can be attributed to the fact that for these QoIs,
Figures \ref{fig:nacaliftlspg}, \ref{fig:nacaprobevxlspg}, and \ref{fig:nacaprobeplspg} show that the PROM of dimension $n = 213$ exhibits a slightly larger phase error than its counterpart of dimension $n = 83$; on the other hand it reproduces more accurately the
peaks of the time-histories of these QoIs.

Finally, Table \ref{tab:nacaspeed} reports for each LSPG-based Petrov-Galerkin HPROM the size of the reduced mesh computed by ECSW, as well as the achieved wall-clock time and CPU time speed-up 
factors.  Even for this complex flow problem, the HPROM-based simulations are found to deliver excellent speed-up factors while achieving numerical stability and accuracy,
including for small values of the dimension $n$.

\begin{table}[h!]
	\small
	\centering
	\caption{LES of a flow over a NACA airfoil: Variations with the dimension $n$ of the relative errors for the LSPG-based Petrov-Galerkin PROMs and HPROMs.}
	\begin{tabular}{lcccccc} 
		\toprule
		Model & $n$ & $\mathbb{RE}_{c_D}$ & $\mathbb{RE}_{c_L}$ & $\mathbb{RE}_{c_Y}$ & $\mathbb{RE}_{v_x}$ & $\mathbb{RE}_{p}$ \\ \midrule
		LSPG-based Petrov-Galerkin PROM & $41$ & $4.07$ & $4.90$ & $75.9$ & $61.8$ & $0.878$ \\
		& $83$ & $0.995$ & $1.15$ & $24.4$ & $33.8$ & $0.391$ \\
		& $213$ & $0.979$ & $1.32$ & $18.3$ & $36.4$ & $0.427$ \\ \midrule
		LSPG-based Petrov-Galerkin HPROM & $41$ & $6.82$ & $8.44$ & $96.8$ & $80.1$ & $1.43$ \\
		& $83$ & $2.04$ & $2.35$ & $40.1$ & $38.4$ & $0.484$ \\
		& $213$ & $1.11$ & $1.48$ & $20.2$ & $37.8$ & $0.451$ \\ \bottomrule
	\end{tabular}
	\label{tab:nacaerr}
\end{table}

\begin{table}[h!]
	\small
	\centering
	\caption{LES of a flow over a NACA airfoil: Variations with the dimension $n$ of the number of mesh cells sampled by ECSW and of the speed-up factor delivered by the ECSW-based HPROM.}
	\begin{tabular}{lccccc}
		\toprule
		\multirow{2}{*}{Model} & \multirow{2}{*}{$n$} & \# of sampled & Fraction of HDM & Wall-clock time & CPU time \\
		& & mesh cells & mesh cells ($\%$) & speed-up factor & speed-up factor \\ \midrule
		LSPG-based Petrov-Galerkin HPROM & $41$ & $1,167$ & $0.056$ & $2.72 \times 10^2$ & $5.44 \times 10^3$ \\
		& $83$ & $2,390$ & $0.115$ & $7.12 \times 10^1$ & $1.42 \times 10^3$ \\
		& $213$ & $6,199$ & $0.298$ & $1.31 \times 10^1$ & $2.62 \times 10^2$ \\ \bottomrule
	\end{tabular}
	\label{tab:nacaspeed}
\end{table}

\section{Conclusions}
\label{sec:conclude}

Projection-based model order reduction for the fast numerical simulation of turbulence and turbulent flows can be challenging from an accuracy perspective due to the large Kolmogorov $n$-width of the high-dimensional solution manifold 
for such problems, which leads to the truncation of non-negligible higher-order modes when constructing a low-dimensional subspace approximation. In this paper, it is argued that modal truncation, which certainly influences the 
accuracy of a projection-based reduced-order model (PROM), does not inherently lead to the numerical instability of a PROM-based simulation. Specifically, it is not the often stated inability of a PROM to resolve the dissipative regime 
of the turbulent energy cascade that leads to its numerical instability observed in the context of convection-dominated turbulent flows. Instead, it is the reliance on the Galerkin projection framework for constructing PROMs for such 
applications.

The above statements are supported in this paper by numerical examples which reveal that even in the absence of turbulence, PROMs constructed using the Galerkin projection framework fail to preserve the numerical stability of the
underlying high-dimensional model (HDM) for convection-dominated problems. This is consistent with decades of literature studying the development stable and accurate spatial discretization schemes for such problems.
On the other hand, Galerkin PROMs are shown to be numerically stable and accurate, even in the presence of highly turbulent phenomena and severe modal truncation, for an example problem involving vortex decay without mean convection: 
In other words, it is shown through a benchmark problem that the inability of a PROM to resolve the small-scale dissipative modes of the underlying HDM does not necessarily lead to excessive energy levels.

Using other supporting numerical examples, it is also shown that alternatively, the numerical stability of an HDM for a convection-dominated laminar or turbulent flow problem can be preserved by a PROM constructed using the 
Petrov-Galerkin projection framework by appropriately constructing the left reduced-order basis, without resorting to any additional explicit modeling. 

Finally, it is also shown that when properly hyperreduced, Petrov-Galerkin PROMs can not only achieve numerical stability for the simulation of turbulence and convection-dominated turbulent flows, but also reduce the computational
resources associated with such simulations by up to six orders of magnitude.

\section{Acknowledgments}

Sebastian Grimberg and Charbel Farhat acknowledge partial support by the Air Force Office of Scientific Research under grant FA9550-17-1-0182,
partial support by the Office of Naval Research under Grant N00014-17-1-2749, partial support by a research grant from the King
Abdulaziz City for Science and Technology (KACST), and partial support by The Boeing Company under Contract Sponsor Ref. 45047.
Noah Youkilis acknowledges the support by an NSF GRFP fellowship under Grant number 2018260970.
This document does not necessarily reflect the position of these institutions, and no official endorsement should be inferred.

\appendix
\section{Equivalence between nonlinear residual minimization and Petrov-Galerkin projection}
\label{app:proofnl}

\begin{proposition}
\label{prop:resmin}
Given an SPD matrix $\boldsymbol{\Theta} \in \mathbb{R}^{N\times N}$, computing the solution $\mathbf{x}^*$ of the unconstrained residual minimization problem
\begin{equation}
\label{appeqn:resmin}
\min_{\mathbf{x} \, \in \, \mathbb{R}^n} \big \lVert \mathbf{r}(\mathbf{u}_0 + \mathbf{V}\mathbf{x}; \boldsymbol{\mu})\big\rVert_{\boldsymbol{\Theta}}^2
\end{equation}
is equivalent to solving the Petrov-Galerkin reduced-order problem
\begin{equation}
\label{appeqn:respg}
\mathbf{W}^T \mathbf{r}(\mathbf{u}_0 + \mathbf{V}\mathbf{x}^*; \boldsymbol{\mu}) = 0
\end{equation}
with
\begin{equation}
\label{appeqn:w}
\mathbf{W} = \boldsymbol{\Theta} \mathbf{J}(\mathbf{u}_0 + \mathbf{V}\mathbf{x}; \boldsymbol{\mu}) \mathbf{V}
\end{equation}
provided that the methods and initial guesses for solving the generally nonconvex problem (\ref{appeqn:resmin}) and nonlinear system of equations (\ref{appeqn:respg}) are chosen to ensure convergence to 
the same local minimum.
\end{proposition}
\begin{proof}
The necessary and sufficient first-order condition characterizing a local minimum of the generally nonconvex objective function
\begin{equation*}
f(\mathbf{x}) = \big\lVert \mathbf{r}(\mathbf{u}_0 + \mathbf{V}\mathbf{x}; \boldsymbol{\mu}) \big\rVert_{\boldsymbol{\Theta}}^2
\end{equation*}
of (\ref{appeqn:resmin}) is $\nabla f(\mathbf{x}^*) = 0$.
Here,
\begin{equation*}
\begin{split}
\nabla f(\mathbf{x}) &= \nabla (\mathbf{r}(\mathbf{u}_0 + \mathbf{V}\mathbf{x}; \boldsymbol{\mu})^T \boldsymbol{\Theta} \mathbf{r}(\mathbf{u}_0 + \mathbf{V}\mathbf{x}; \boldsymbol{\mu})) \\
&= 2 \mathbf{V}^T \mathbf{J}^T(\mathbf{u}_0 + \mathbf{V}\mathbf{x}; \boldsymbol{\mu}) \boldsymbol{\Theta} \mathbf{r}(\mathbf{u}_0 + \mathbf{V}\mathbf{x}; \boldsymbol{\mu})
\end{split}
\end{equation*}
which yields
\begin{equation*}
\mathbf{V}^T \mathbf{J}^T(\mathbf{u}_0 + \mathbf{V}\mathbf{x}^*; \boldsymbol{\mu}) \boldsymbol{\Theta} \mathbf{r}(\mathbf{u}_0 + \mathbf{V}\mathbf{x}^*; \boldsymbol{\mu}) = 0
\end{equation*}
This expression is equivalent to (\ref{appeqn:respg}) with the definition of $\mathbf{W}$ given in (\ref{appeqn:w}).
\end{proof}

\begin{corollary}
If $\mathbf{J}(\mathbf{u}_0 + \mathbf{V}\mathbf{x}; \boldsymbol{\mu})$ and thus $\mathbf{J}^{-1}(\mathbf{u}_0 + \mathbf{V}\mathbf{x}; \boldsymbol{\mu})$ are SPD and therefore define a norm, the solution $\mathbf{x}^*$ of (\ref{appeqn:resmin}) with
\begin{equation*}
\boldsymbol{\Theta} = \mathbf{J}^{-1}(\mathbf{u}_0 + \mathbf{V}\mathbf{x}; \boldsymbol{\mu})
\end{equation*}
is given by solving the Galerkin projection-based system of equations
\begin{equation}
\label{appeqn:resg}
\mathbf{V}^T \mathbf{r}(\mathbf{u}_0 + \mathbf{V}\mathbf{x}^*; \boldsymbol{\mu}) = 0
\end{equation}
\end{corollary}
\begin{proof}
Following Proposition \ref{prop:resmin}, substituting this definition of $\boldsymbol{\Theta}$ into the expression (\ref{appeqn:w}) for $\mathbf{W}$ yields $\mathbf{W} = \mathbf{V}$ and results in the Galerkin projection-based system of equations (\ref{appeqn:resg})
\end{proof}

\section{Equivalence between iterate step-direction error minimization and Petrov-Galerkin projection}
\label{app:prooflin}

\begin{proposition}
\label{prop:resminl}
Computing the solution $\mathbf{x}^*$ of the unconstrained minimization problem
\begin{equation}
\label{appeqn:resminl}
\min_{\mathbf{x} \, \in \, \mathbb{R}^n} \big \lVert \mathbf{V}\mathbf{x} -\boldsymbol{\Delta}\mathbf{u}^{(p)} \big \rVert_{{\mathbf{J}^{(p)}}^T \boldsymbol{\Theta} \mathbf{J}^{(p)}}^2
\end{equation}
where $\mathbf{J}^{(p)}\boldsymbol{\Delta}\mathbf{u}^{(p)} = -\mathbf{r}^{(p)}$ is equivalent to computing the solution of the Petrov-Galerkin reduced-order problem
\begin{equation}
\label{appeqn:respgl}
\mathbf{W}^T \mathbf{J}^{(p)} \mathbf{V} \mathbf{x}^* = -\mathbf{W}^T \mathbf{r}^{(p)}
\end{equation}
with
\begin{equation}
\label{appeqn:wl}
\mathbf{W} = \boldsymbol{\Theta} \mathbf{J}^{(p)} \mathbf{V}
\end{equation}
\end{proposition}
\begin{proof}
The necessary and sufficient first-order condition characterizing a global minimum of the convex objective function
\begin{equation*}
f(\mathbf{x}) = \big \lVert \mathbf{V}\mathbf{x} -\boldsymbol{\Delta}\mathbf{u}^{(p)} \big \rVert_{{\mathbf{J}^{(p)}}^T \boldsymbol{\Theta} \mathbf{J}^{(p)}}^2
\end{equation*}
of (\ref{appeqn:resminl}) is $\nabla f(\mathbf{x}^*) = 0$.
Here,
\begin{equation*}
\begin{split}
\nabla f(\mathbf{x}) &= \nabla \left(\left(\mathbf{V}\mathbf{x} -\boldsymbol{\Delta}\mathbf{u}^{(p)}\right)^T {\mathbf{J}^{(p)}}^T \boldsymbol{\Theta} \mathbf{J}^{(p)} \left(\mathbf{V}\mathbf{x} -\boldsymbol{\Delta}\mathbf{u}^{(p)}\right)\right) \\
&= 2 \mathbf{V}^T {\mathbf{J}^{(p)}}^T \boldsymbol{\Theta} \mathbf{J}^{(p)} \mathbf{V} \mathbf{x} - 2 \mathbf{V}^T {\mathbf{J}^{(p)}}^T \boldsymbol{\Theta} \mathbf{J}^{(p)} \boldsymbol{\Delta}\mathbf{u}^{(p)} \\
&= 2 \mathbf{V}^T {\mathbf{J}^{(p)}}^T \boldsymbol{\Theta} \mathbf{J}^{(p)} \mathbf{V} \mathbf{x} + 2 \mathbf{V}^T {\mathbf{J}^{(p)}}^T \boldsymbol{\Theta} \mathbf{r}^{(p)}\\
\end{split}
\end{equation*}
which yields
\begin{equation*}
\mathbf{V}^T {\mathbf{J}^{(p)}}^T \boldsymbol{\Theta} \mathbf{J}^{(p)} \mathbf{V} \mathbf{x}^* = -\mathbf{V}^T {\mathbf{J}^{(p)}}^T \boldsymbol{\Theta} \mathbf{r}^{(p)}
\end{equation*}
The above expression is equivalent to (\ref{appeqn:respgl}) with the definition of $\mathbf{W}$ given in (\ref{appeqn:wl}).
\end{proof}

\begin{corollary}
If $\mathbf{J}^{(p)}$ and thus ${\mathbf{J}^{(p)}}^{-1}$ are SPD and therefore define a norm, computing the solution $\mathbf{x}^*$ of \textrm{(\ref{appeqn:resminl})} with
\begin{equation*}
\boldsymbol{\Theta} = {\mathbf{J}^{(p)}}^{-1}
\end{equation*}
is equivalent to solving the Galerkin reduced-order problem
\begin{equation}
\label{appeqn:resgl}
\mathbf{V}^T \mathbf{J}^{(p)} \mathbf{V} \mathbf{x}^* = -\mathbf{V}^T \mathbf{r}^{(p)}
\end{equation}
\end{corollary}
\begin{proof}
From Proposition \ref{prop:resminl} and the substitution of the above choice for $\boldsymbol{\Theta}$ in the expression (\ref{appeqn:wl}) for $\mathbf{W}$ yields $\mathbf{W} = \mathbf{V}$ and results 
in the Galerkin reduced-order problem (\ref{appeqn:resgl}).
\end{proof}

\section{Butcher tableaux for second- and third-order DIRK schemes}
\label{app:dirk}

\noindent The Butcher tableau for the two-stage, second-order DIRK scheme is
\begin{equation*}
	\begin{aligned}[c]
	\renewcommand\arraystretch{1.2}
	\begin{array}{c|cc}
		\alpha & \alpha \\ 
		1 & 1 - \alpha & \alpha \\ \hline
		& 1 - \alpha & \alpha
	\end{array}
	\end{aligned}
	\qquad \qquad
	\begin{aligned}[c]
	\alpha = 1 - \frac{\sqrt{2}}{2}
	\end{aligned}
\end{equation*} \\

\noindent The Butcher tableau for the three-stage, third-order DIRK scheme is
\begin{equation*}
	\begin{aligned}[c]
	\renewcommand\arraystretch{1.2}
	\begin{array}{c|ccc}
		\alpha & \alpha \\ 
		\tau_2 & \tau_2 - \alpha & \alpha \\
		1 & b_1 & b_2 & \alpha \\ \hline
		& b_1 & b_2 & \alpha
	\end{array}
	\end{aligned}
	\qquad \qquad
	\begin{aligned}[c]
	&\alpha = 1 + \frac{\sqrt{6}}{2} \sin\left(\frac{1}{3}\arctan\left(\frac{\sqrt{2}}{4}\right)\right) - \frac{\sqrt{2}}{2} \cos\left(\frac{1}{3}\arctan\left(\frac{\sqrt{2}}{4}\right)\right) \\
	&\tau_2 = \frac{1+\alpha}{2} \\
	&b_1 = -\frac{6\alpha^2 -16\alpha +1}{4} \\
	&b_2 = \frac{6\alpha^2 - 20\alpha + 5}{4}
	\end{aligned}
\end{equation*}


\bibliographystyle{elsarticle-num-names}
\bibliography{main}

\end{document}